\RequirePackage{fix-cm}
\documentclass{svjour3}                     
\smartqed  
\usepackage[table,xcdraw,dvipsnames]{xcolor}

\usepackage[utf8x]{inputenc}
\usepackage{graphicx}
\usepackage{indentfirst}
\usepackage{cmap}
\usepackage{ifthen}
\usepackage{tikz}
\usepackage{algorithm}
\usepackage{pifont}
\usepackage{url}
\usepackage{mathtools}
\usepackage{hyperref}
\usepackage{multirow}
\usepackage{amsfonts}
\usepackage{amssymb}
\usepackage{booktabs}
\usepackage{cite}
\usepackage[dvipsnames]{xcolor}
\usepackage{algorithmic}
\usepackage{caption}

\newcommand{\R}{{\mathbb R}} 
\newcommand{\E}{{\mathbb E}}

\newcommand{\la}{\langle}
\newcommand{\ra}{\rangle}
\newcommand\xqed[1]{%
  \leavevmode\unskip\penalty9999 \hbox{}\nobreak\hfill
  \quad\hbox{#1}}
\newcommand\demo{\xqed{$\triangle$}}
\DeclareMathOperator{\argmin}{argmin}
\DeclareMathOperator{\argmax}{argmax}

\def\vp{\varphi}

\usepackage{subfigure}
\usepackage{xr}
\usepackage{textcomp}

\def\pd#1{{\color{ForestGreen}#1}} 
\def\todo#1{{\color{red}TODO: #1}} 

\makeatletter
\newcommand*{\addFileDependency}[1]{
  \typeout{(#1)}
  \@addtofilelist{#1}
  \IfFileExists{#1}{}{\typeout{No file #1.}}
}
\makeatother

\newcommand*{\myexternaldocument}[1]{%
    \externaldocument{#1}%
    \addFileDependency{#1.tex}%
    \addFileDependency{#1.aux}%
}
\myexternaldocument{inexact_oracle/Appendix_A}
\myexternaldocument{inexact_oracle/Appendix_B}
\myexternaldocument{inexact_oracle/Appendix_C}
\newtheorem{assumption}{Assumption}
\newcommand{\cmark}{\color[HTML]{3166FF}{\ding{51} }\color{black} }%
\newcommand{\xmark}{\color[HTML]{CB0000}{\ding{55} }\color{black} }%
\usepackage{enumitem}
\usepackage{pifont}
\numberwithin{equation}{section}

\begin{document}
\title{On Accelerated Methods for Saddle-Point Problems with Composite Structure
}

\titlerunning{Accelerated Methods for Saddle-Point Problems}        

\author{Vladislav Tominin  $^1$      \and
        Yaroslav Tominin  $^2$     \and
        Ekaterina Borodich  $^3$     \and
        Dmitry Kovalev $^4$         \and
        Alexander Gasnikov  $^5$    \and
        Pavel Dvurechensky  $^6$     
}


\institute{
$^1$ Moscow Institute of Physics and Technology, Dolgoprudny, Russia. \email{tominin.vd@phystech.edu}
\and
$^2$ Moscow Institute of Physics and Technology, Dolgoprudny, Russia. \email{tominin.yad@phystech.edu}\and
$^3$ Moscow Institute of Physics and Technology, Dolgoprudny, Russia. \email{borodich.ed@phystech.edu}
\and
$^4$ King Abdullah University of Science and Technology, Thuwal, Saudi Arabia. \email{dmitry.kovalev@kaust.edu.sa} \and
$^5$ Moscow Institute of Physics and Technology, Dolgoprudny, Russia, and Higher school of economics, Moscow, Russia, and  Institute for Information Transmission Problems RAS, Moscow, Russia \email{gasnikov@yandex.ru} \and
$^6$  Weierstrass Institute for Applied Analysis and Stochastics, Berlin
\email{ pavel.dvurechensky@wias-berlin.de}
}

\date{Received: date / Accepted: date}

\maketitle

\begin{abstract}
We consider strongly-convex-strongly-concave saddle-point problems with general non-bilinear objective and different condition numbers with respect to the primal and the dual variables. First, we consider such problems with smooth composite terms, one of which having finite-sum structure. For this setting we propose a variance reduction algorithm with complexity estimates superior to the existing bounds in the literature. Second, we consider finite-sum saddle-point problems with composite terms and propose several algorithms depending on the properties of the composite terms. When the composite terms are smooth we obtain better complexity bounds than the ones in the literature, including the bounds of a recently proposed nearly-optimal algorithms which do not consider the composite structure of the problem. 
If the composite terms are prox-friendly, we propose a variance reduction algorithm that, on the one hand, is accelerated compared to existing variance reduction algorithms and, on the other hand, provides in the composite setting similar complexity bounds to the nearly-optimal algorithm which is designed for non-composite setting.
Besides that, our algorithms allow to separate the complexity bounds, i.e. estimate, for each part of the objective separately, the number of oracle calls  that is sufficient to achieve a given accuracy. This is important since different parts can have different arithmetic complexity of the oracle, and it is desired to call expensive oracles less often than cheap oracles.
The key thing to all these results is our general framework for saddle-point problems, which may be of independent interest. 
This framework, in turn is based on our proposed Accelerated Meta-Algorithm for composite optimization with probabilistic inexact oracles and probabilistic inexactness in the proximal mapping, which may be of independent interest as well.


\end{abstract}

\section{Introduction}
\label{S:Intro}

Saddle-point optimization problems have many applications in different areas of modelling an optimization. The most classical example is, perhaps, two-player zero-sum games \cite{morgenstern1953theory,nash1950bargaining}, including differential games \cite{isaacs1999differential}. More recent examples include imaging problems \cite{chambolle2011first-order} and machine learning problems \cite{shalev-shwartz2014accelerated}, where  primal-dual saddle-point representations of  large-scale optimization problems are constructed and primal-dual methods are used. Many non-smooth optimization problems, such as $\ell_{\infty}$ or $\ell_{1}$ regression admit a saddle-point representation, which allows one to propose methods \cite{nesterov2005smooth,nemirovski2004prox} having faster convergence than the standard subgradient scheme.
Recently saddle-point problems started to attract more attention from the machine learning community in application to generative adversarial networks training, where the training process consists of a competition of a generator of non-real images and a discriminator which tries to distinguish between real and artificial images. Another application examples are equilibrium problems in two-stage congested traffic flow models \cite{GasnikovTransport}.

From the algorithmic viewpoint the most studied setting deals with saddle-point problems having bilinear structure \cite{nesterov2005smooth,nemirovski2004prox,carmon2019variance,song2021variance},
 where the cross term between the primal and dual variable is linear with respect to each variable. The extensions include bilinear problems with  prox-friendly (i.e. admitting a proximal operator in closed form) composite terms \cite{chambolle2011first-order,lan2019lectures}. 
 A related line of research studies variational inequalities 
 \cite{nemirovski2004prox,lan2019lectures} since any convex-concave saddle-point problem can be reformulated as a variational inequality problem with monotone operator. In this area lower bounds for first-order methods are known \cite{nemirovsky1983problem} and optimal methods exist
 \cite{nemirovski2004prox,nesterov2007dual,nesterov2006solving,chen2017accelerated,lan2019lectures}. Notably, these works do not rely on the bilinear structure and allow to solve convex-concave saddle-point problems with Lipschitz-continuous gradients, including differential games \cite{dvurechensky2015primal-dual}.
 An alternative approach, which mostly inspired this paper, is based on representation of a saddle-point problem $\min_x \max_y G(x,y)$ as either a primal minimization problem with an implicitly given objective $g(x)=\max_y G(x,y)$ or a dual maximization problem with an implicitly given objective $\tilde{g}(y)=\min_x G(x,y)$. This approach was used in \cite{nesterov2005smooth,nesterov2005excessive} for problems with bilinear structure and later extended in \cite{hien2020inexact} for general saddle-point problems.
 Such connection with optimization turned out to be quite productive since it allows to exploit accelerated optimization methods. In particular, recent advances in this direction are due to an observation \cite{gasnikov2016stochastic,alkousa2020accelerated,lowerbounds} that primal and dual problems can have different condition numbers which opens up a possibility to obtain faster algorithms.

In this paper we focus on strongly-convex-strongly-concave saddle-point problems with different condition numbers $\kappa_x$, $\kappa_y$ of the primal and dual problems respectively. The classical upper bound  $\tilde{O}(\kappa_x+\kappa_y)$ for this setting is achieved by the algorithm of \cite{nesterov2006solving}. 
Recently, \cite{lowerbounds} proved 
a lower complexity bound $\tilde{\Omega}(\sqrt{\kappa_x\kappa_y})$ for first-order methods, which raised a question of whether first-order methods can be accelerated for this setting.
Independently \cite{Alkousa2019} proposed accelerated methods with improved, yet suboptimal complexity bounds. In \cite{pmlr-v125-lin20a} the authors improved the bounds of \cite{alkousa2020accelerated} and proposed an algorithm with an optimal up to a polylogarithmic factor complexity bound $\tilde{O}(\sqrt{\kappa_x\kappa_y})$. 
Subsequently, the logarithmic factors have been improved independently in the papers (we cite them in chronological order) \cite{dvinskikh2020accelerated,wang2020improved,yang2020catalyst}. The listed papers consider large-scale regime when primal and dual problem have large dimension and use gradient-type methods. If, say, the dimension of the primal variable $x$ is moderate, one can use cutting-plane methods \cite{gladin2020solving,gladin2021solving} in combination with gradient-type methods.
We also mention the following papers which are related, but consider different from ours setting of convex-concave saddle-point problems \cite{zhu2020accelerated}, strongly-convex-concave and nonconvex-concave \cite{NEURIPS2019_05d0abb9}, nonconvex-concave \cite{ostrovskii2020efficient,xu2020unified}.

When an optimization problem has a special structure of finite-sum, also known as empirical risk minimization problems, variance reduction \cite{lan2020first,pmlr-v125-lin20a} techniques are often exploited to reduce the complexity bounds. We are interested also in application of such techniques for saddle-point problems. Variance reduction methods for saddle-point problems were proposed in \cite{NIPS2016_1aa48fc4} and recently improved in \cite{alacaoglu2021stochastic}, yet without distinguishing between primal and dual condition numbers.

In this paper we continue the line of research  \cite{alkousa2020accelerated,dvinskikh2020accelerated} 
by exploring additional structure of the problem, such as finite-sum form and presence of composite terms. We also develop algorithms which allow to separate the complexity bounds for different parts of the problem. 
The latter, in particular, means that for each part of the objective we estimate separately the number of its gradient evaluations.
This allows to obtain further acceleration if the smoothness constants and complexities of an oracle call for different parts are different since more expensive oracles are called less frequently than it would be required by existing methods. Next we consider two main problem formulations which have additional structure and which we explore in the paper. We also give detailed explanation of the difference of our setting and bounds with the literature.


The first problem formulation, we are interested in, is strongly-convex-strongly-concave saddle-point problem of the form 
\begin{equation} 
    \label{eq:problem}
    \min _{x \in \mathbb{R}^{d_x}} \max _{y \in \mathbb{R}^{d_y}} \left\{f(x)+ G(x, y) - h(y)\right\}, \;\;\; h(y):=\frac{1}{m_h} \sum_{i = 1}^{m_h} h_i(y). 
\end{equation}
where $G(x, y)$ is convex in $x$ and concave in $y$ and is $L_G$-smooth in each variable, $f(x)$ is $\mu_x$-strongly convex and $L_f$-smooth, $h(y)$ is $\mu_y$-strongly convex and $L_h$-smooth. We refer to the functions $f$ and $h$ as composite terms. In this setting it is natural to define condition numbers $\kappa_x=L_G/\mu_x$ and $\kappa_y=L_G/\mu_y$ for the primal minimization and dual maximization problems respectively.
As it was already mentioned, the most studied \cite{chambolle2011first-order, lan2019lectures} setting corresponds to a particular case of $m_h=1$ and bilinear function $G(x, y) = \langle Ax, y \rangle$ for some linear operator $A$ and the functions $f,g$ being prox-friendly, i.e. admit a tractable proximal operator \cite{BSMF_1965__93__273_0}, e.g. evaluation of the point $\arg \min_x\{f(x)+\frac{1}{2}\|x-\bar{x}\|_2^2 \}$ in the case of $f$. 
Existing algorithms \cite{NIPS2016_1aa48fc4,Alkousa2019,alkousa2020accelerated,pmlr-v125-lin20a,dvinskikh2020accelerated,wang2020improved,yang2020catalyst} for problem \eqref{eq:problem} with non-bilinear structure do not exploit the finite-sum structure of the function $h$ and when it is smooth require to calculate the gradient of the whole sum, which may be expensive when $m_h\gg1$.
Unlike them we incorporate variance reduction technique to make the number of evaluations of $\nabla h_i(y)$ smaller than by the existing methods. Unlike  \cite{NIPS2016_1aa48fc4,pmlr-v125-lin20a,wang2020improved,yang2020catalyst} we separate the complexity estimates for each part of the objective, i.e. we estimate separately a sufficient number of evaluations of $\nabla f(x)$, $\nabla_x G(x,y)$, $\nabla_y G(x,y)$, $\nabla h_i(y)$ to achieve a given accuracy. 
This allows to call each oracle less number of times than it is required by existing methods and is important since evaluation of each gradient can have different arithmetic operations complexity, and it is desired to call expensive oracles less often than cheap oracles. Compared to \cite{Alkousa2019,alkousa2020accelerated}, where the complexities are also separated, we obtain better complexity bounds for each part of the objective. Moreover, for the particular case when $f=h=0$, our bounds are the same to the best known bounds \cite{wang2020improved,yang2020catalyst} and are optimal up to logarithmic factors. Otherwise, when $m_h > 1$ and/or $f, h$ are nonzero we obtain the best, to our knowledge complexity bounds.
We summarize comparison of ours and existing results for the case $m_h > 1$ in Table \ref{table:h-sum-m_h=1} and for the particular case $m_h = 1$ in Table \ref{table:h-sum}.

\begingroup
\setlength{\tabcolsep}{2pt} 
\renewcommand{\arraystretch}{2.5} 
\begin{table}\centering
\begin{tabular}{|c|l|l|c|c|} 
\hline
\textbf{Referenses} & \multicolumn{2}{c|}{\textbf{Complexity}} & \begin{tabular}[c]{@{}c@{}}\bf{Variance}\\\bf{reduction} \end{tabular} & \begin{tabular}[c]{@{}c@{}}\bf{Complexity}\\\bf{separation} \end{tabular}\\ \hline


& $\nabla f: \; \tilde{O}\left(\kappa^{(f+G)}_x+\kappa^{(G+h)}_y\right)$             & $\nabla_x G: \; \tilde{O}\left(\kappa^{(f+G)}_x+\kappa^{(G+h)}_y\right)$             &          &         \\ \cline{2-3}

\multirow{-2}{*}{\cite{nesterov2006solving}} & $\nabla h_i:  \tilde{O}\left(m_h \kappa^{(f+G)}_x+m_h\kappa^{(G+h)}_y\right)$             & $\nabla_y G: \; \tilde{O}\left(\kappa^{(f+G)}_x+\kappa^{(G+h)}_y\right)$           & \multirow{-2}{*}{\xmark} & \multirow{-2}{*}{\xmark}  \\ \hline \hline

& $\nabla f: \; \tilde{O}\left(\sqrt{\kappa^{(f+G)}_x\kappa^{(G+h)}_y}\right)$             & $\nabla_x G: \; \tilde{O}\left(\sqrt{\kappa^{(f+G)}_x\kappa^{(G+h)}_y}\right)$             &       &           \\ \cline{2-3}

\multirow{-2}{*}{\cite{pmlr-v125-lin20a,wang2020improved,yang2020catalyst}} & $\nabla h_i:  \tilde{O}\left(m_h\sqrt{\kappa^{(f+G)}_x\kappa^{(G+h)}_y}\right)$             & $\nabla_y G: \; \tilde{O}\left(\sqrt{\kappa^{(f+G)}_x\kappa^{(G+h)}_y}\right)$     & \multirow{-2}{*}{\xmark}  & \multirow{-2}{*}{\xmark}        \\ \hline \hline

& $\nabla f: \; \tilde{O}\left(\sqrt{\kappa^{(f)}_x}\right)$             & $\nabla_x G: \; \tilde{O}\left(\sqrt{\kappa^{(G)}_x \kappa^{(G)}_y}\right)$             &       &             \\ \cline{2-3}

\multirow{-2}{*}{\cite{alkousa2020accelerated}} & $\nabla h_i:  \tilde{O}\left(m_h\sqrt{\kappa^{(G)}_x \kappa^{(G)}_y\kappa^{(h)}_y}\right)$             & $\nabla_y G: \; \tilde{O}\left(\kappa^{(G)}_y\sqrt{\kappa^{(G)}_x} \right)$     & \multirow{-2}{*}{\xmark}& \multirow{-2}{*}{\cmark}          \\ \hline \hline

\rowcolor[HTML]{C0C0C0} 
\cellcolor[HTML]{C0C0C0}                    & $\nabla f: \; \tilde{O}\left(\sqrt{\kappa^{(f)}_x \kappa^{(G)}_y}\right)$             & $\nabla_x G: \; \tilde{O}\left(\sqrt{\kappa^{(G)}_x \kappa^{(G)}_y}\right)$             & \cellcolor[HTML]{C0C0C0}         & \cellcolor[HTML]{C0C0C0}            \\ \cline{2-3}
\rowcolor[HTML]{C0C0C0} 
\multirow{-2}{*}{\cellcolor[HTML]{C0C0C0}\begin{tabular}[c]{@{}c@{}}\textbf{This paper} \\
\textbf{(Theorem~\ref{theorem:h-sum})} \end{tabular}} & $\nabla h_i:  \tilde{O}\left(\sqrt{m_h \kappa^{(G)}_x \kappa^{(h)}_y}\right)$             &  $\nabla_y G: \; \tilde{O}\left(\sqrt{\kappa^{(G)}_x \kappa^{(G)}_y}\right)$          & \multirow{-2}{*}{\cmark} & \multirow{-2}{*}{\cmark}  \\ \hline

\end{tabular}
\caption{Comparison of gradient complexities for problem \eqref{eq:problem} with $m_h>1$, i.e. the number of corresponding gradient evaluations, to find an $\varepsilon$-saddle point with probability at least $1-\sigma$. 
Notation $\tilde{O}(X)$ hides constant and polylogarithmic in $\varepsilon^{-1}$ and $\sigma^{-1}$ factors.
For a function $F$, we denote $\kappa_x^{(F)}=L_F/\mu_x$, $\kappa_y^{(F)}=L_F/\mu_y$. 
The results of Theorem~\ref{theorem:h-sum} are obtained under additional assumptions 
$m_h (4L_G + \mu_y) \leq L_h$, $ 2L_G + \mu_x   \leq L_f, \mu_y  \leq L_G$, $\mu_x  \leq L_G$. 
}
\label{table:h-sum-m_h=1}

\end{table}
\endgroup

\begingroup
\setlength{\tabcolsep}{1pt} 
\renewcommand{\arraystretch}{2.5} 
\begin{table}\centering
\begin{tabular}{|c|l|l|c|c|c|} 
\hline
\textbf{Referenses}                                  & \multicolumn{2}{c|}{\textbf{Complexity}} &  \textbf{Assumptions} &\begin{tabular}[c]{@{}c@{}}\bf{\;\;CS\;\;} \end{tabular}\\ \hline

& $\nabla f: \; \tilde{O}\left(\kappa^{(f+G)}_x+\kappa^{(G+h)}_y\right)$             & $\nabla_x G: \; \tilde{O}\left(\kappa^{(f+G)}_x+\kappa^{(G+h)}_y\right)$             &          &          \\ \cline{2-3}

\multirow{-2}{*}{\cite{nesterov2006solving}} & $\nabla h_i:  \tilde{O}\left(\kappa^{(f+G)}_x\kappa^{(G+h)}_y\right)$             & $\nabla_y G: \; \tilde{O}\left(\kappa^{(f+G)}_x+\kappa^{(G+h)}_y\right)$             & \multirow{-2}{*}{\begin{tabular}[c]{@{}c@{}}$f$ is $L_f$-smooth,\\ $h$ is $L_h$-smooth \end{tabular}}  & \multirow{-2}{*}{\xmark} \\ \hline \hline


& $\nabla f: \; \tilde{O}\left(\sqrt{\kappa^{(f+G)}_x\kappa^{(G+h)}_y}\right)$             & $\nabla_x G: \; \tilde{O}\left(\sqrt{\kappa^{(f+G)}_x\kappa^{(G+h)}_y}\right)$             &           &        \\ \cline{2-3}

\multirow{-2}{*}{\cite{pmlr-v125-lin20a,wang2020improved,yang2020catalyst}} & $\nabla h_i:  \tilde{O}\left(\sqrt{\kappa^{(f+G)}_x\kappa^{(G+h)}_y}\right)$             & $\nabla_y G: \; \tilde{O}\left(\sqrt{\kappa^{(f+G)}_x\kappa^{(G+h)}_y}\right)$             & \multirow{-2}{*}{\begin{tabular}[c]{@{}c@{}}$f$ is $L_f$-smooth,\\ $h$ is $L_h$-smooth \end{tabular}}  & \multirow{-2}{*}{\xmark} \\ \hline \hline


& $\nabla f: \; \tilde{O}\left(\sqrt{\kappa^{(f)}_x}\right)$             & $\nabla_x G: \; \tilde{O}\left(\sqrt{\kappa^{(G)}_x \kappa^{(G)}_y}\right)$             &       &             \\ \cline{2-3}

\multirow{-2}{*}{\cite{alkousa2020accelerated}} & $\nabla h_i:  \tilde{O}\left(\sqrt{\kappa^{(G)}_x \kappa^{(G)}_y\kappa^{(h)}_y}\right)$             & $\nabla_y G: \; \tilde{O}\left(\kappa^{(G)}_y\sqrt{\kappa^{(G)}_x} \right)$             & \multirow{-2}{*}{\begin{tabular}[c]{@{}c@{}}$f$ is $L_f$-smooth,\\ $h$ is $L_h$-smooth \end{tabular}}    & \multirow{-2}{*}{\cmark}\\ \hline \hline


\rowcolor[HTML]{C0C0C0} 
\cellcolor[HTML]{C0C0C0}                    & $\nabla f: \; \tilde{O}\left(\sqrt{\kappa^{(f)}_x \kappa^{(G)}_y}\right)$             & $\nabla_x G: \; \tilde{O}\left(\sqrt{\kappa^{(G)}_x \kappa^{(G)}_y}\right)$             & \cellcolor[HTML]{C0C0C0}      & \cellcolor[HTML]{C0C0C0}                 \\ \cline{2-3}
\rowcolor[HTML]{C0C0C0} 
\multirow{-2}{*}{\cellcolor[HTML]{C0C0C0}\begin{tabular}[c]{@{}c@{}}\textbf{This paper} \\
\textbf{(Corollary~\ref{theorem:h-not-sum})} \end{tabular}} & $\nabla h:  \tilde{O}\left(\max \left\{\sqrt{ \kappa^{(G)}_x \kappa^{(h)}_y}, \sqrt{ \kappa^{(G)}_x \kappa^{(G)}_y}\right\}\right)$             & $\nabla_y G: \; \tilde{O}\left(\sqrt{\kappa^{(G)}_x \kappa^{(G)}_y}\right)$             & \multirow{-2}{*}{\cellcolor[HTML]{C0C0C0}\begin{tabular}[c]{@{}c@{}}$f$ is $L_f$-smooth,\\ $h$ is $L_h$-smooth \end{tabular}}& \multirow{-2}{*}{\cellcolor[HTML]{C0C0C0} \cmark}\\ \hline \hline

\rowcolor[HTML]{C0C0C0} 
\cellcolor[HTML]{C0C0C0}                    & $\nabla f: \; \tilde{O}\left(\sqrt{\kappa^{(f)}_x \kappa^{(G)}_y}\right)$             & $\nabla_x G: \; \tilde{O}\left(\sqrt{ \kappa^{(G)}_x \kappa^{(G)}_y}\right)$             & \cellcolor[HTML]{C0C0C0}     & \cellcolor[HTML]{C0C0C0}               \\ \cline{2-3}
\rowcolor[HTML]{C0C0C0} 
\multirow{-2}{*}{\cellcolor[HTML]{C0C0C0}\begin{tabular}[c]{@{}c@{}}\textbf{This paper} \\
\textbf{(Theorem~\ref{theorem:G-sum_h_prox})} \end{tabular}} & $\nabla h:  \tilde{O}\left(\sqrt{\kappa^{(G)}_y}\right)$             & $\nabla_y G: \; \tilde{O}\left(\sqrt{ \kappa^{(G)}_x \kappa^{(G)}_y}\right)$             & \multirow{-2}{*}{\cellcolor[HTML]{C0C0C0}\begin{tabular}[c]{@{}c@{}}$f$ is $L_f$-smooth,\\ $h$ is $L_h$-smooth prox-friendly \end{tabular}}& \multirow{-2}{*}{\cellcolor[HTML]{C0C0C0} \cmark} \\ \hline \hline

\rowcolor[HTML]{C0C0C0} 
\cellcolor[HTML]{C0C0C0}                    & $\nabla f: \;0$             & $\nabla_x G: \; \tilde{O}\left(\sqrt{ \kappa^{(G)}_x \kappa^{(G)}_y}\right)$             & \cellcolor[HTML]{C0C0C0}   & \cellcolor[HTML]{C0C0C0}                 \\ \cline{2-3}
\rowcolor[HTML]{C0C0C0} 
\multirow{-2}{*}{\cellcolor[HTML]{C0C0C0}\begin{tabular}[c]{@{}c@{}}\textbf{This paper} \\
\textbf{(Theorem~\ref{theorem:section_6})} \end{tabular}} & $\nabla h: \; 0$             & $\nabla_y G: \; \tilde{O}\left(\sqrt{ \kappa^{(G)}_x \kappa^{(G)}_y}\right)$             & \multirow{-2}{*}{\cellcolor[HTML]{C0C0C0}\begin{tabular}[c]{@{}c@{}}$f$, $h$ prox-friendly \end{tabular}} & \multirow{-2}{*}{\cellcolor[HTML]{C0C0C0} \cmark}\\ \hline

\end{tabular}
\caption{
Comparison of gradient complexities for problem \eqref{eq:problem} with $m_h=1$, i.e. the number of corresponding gradient evaluations, to find an $\varepsilon$-saddle point for the problem. 
Notation $\tilde{O}(X)$ hides constant and polylogarithmic in $\varepsilon^{-1}$ factors. \textbf{ CS} stands for complexity separation.
For a function $F$, we denote $\kappa_x^{(F)}=L_F/\mu_x$, $\kappa_y^{(F)}=L_F/\mu_y$. 
}
\label{table:h-sum}

\end{table}
\endgroup

\begingroup
\setlength{\tabcolsep}{0.7pt} 
\renewcommand{\arraystretch}{2.5} 
\begin{table}\centering
\begin{tabular}{|c|l|l|c|c|c|c|}
\hline
\textbf{Referenses}                                  & \multicolumn{2}{c|}{\textbf{Complexity}} & \textbf{Prox-f} & \textbf{Prox-h} & \begin{tabular}[c]{@{}c@{}}\bf{VR} \end{tabular}  &\begin{tabular}[c]{@{}c@{}}\bf{CS} \end{tabular}                                     \\ \hline



& $\nabla f: \; \tilde{O}\left(\kappa^{(f+G)}_x+\kappa^{(G+h)}_y\right)$             & $\nabla_x G_i: \; \tilde{O}\left(m_G \kappa^{(f+G)}_x+m_G \kappa^{(G+h)}_y\right)$             &     & &       &       \\ \cline{2-3}

\multirow{-2}{*}{\cite{nesterov2006solving}} & $\nabla h:  \tilde{O}\left( \kappa^{(f+G)}_x+m_h\kappa^{(G+h)}_y\right)$             & $\nabla_y G_i: \; \tilde{O}\left(m_G \kappa^{(f+G)}_x+m_G \kappa^{(G+h)}_y\right)$             & \multirow{-2}{*}{\xmark} & \multirow{-2}{*}{\xmark}& \multirow{-2}{*}{\xmark}& \multirow{-2}{*}{\xmark}\\ \hline \hline

& $\nabla f: \; \tilde{O}\left(\sqrt{\kappa^{(f+G)}_x\kappa^{(G+h)}_y}\right)$             & $\nabla_x G_i: \; \tilde{O}\left(m_G \sqrt{\kappa^{(f+G)}_x\kappa^{(G+h)}_y}\right)$             &         &     &   &   \\ \cline{2-3}

\multirow{-2}{*}{\cite{pmlr-v125-lin20a}} & $\nabla h:  \tilde{O}\left(\sqrt{\kappa^{(f+G)}_x\kappa^{(G+h)}_y}\right)$             & $\nabla_y G_i: \; \tilde{O}\left(m_G \sqrt{\kappa^{(f+G)}_x\kappa^{(G+h)}_y}\right)$             & \multirow{-2}{*}{\xmark}& \multirow{-2}{*}{\xmark} & \multirow{-2}{*}{\xmark}& \multirow{-2}{*}{\xmark}\\ \hline \hline

& $\nabla f: \; \tilde{O}\left(\sqrt{\kappa^{(f)}_x}\right)$             & $\nabla_x G_i: \; \tilde{O}\left(m_G \sqrt{\kappa^{(G)}_x \kappa^{(G)}_y}\right)$             &        &      &   &   \\ \cline{2-3}

\multirow{-2}{*}{\cite{alkousa2020accelerated}} & $\nabla h:  \tilde{O}\left(\sqrt{\kappa^{(G)}_x \kappa^{(G)}_y\kappa^{(h)}_y}\right)$             & $\nabla_y G_i: \; \tilde{O}\left(m_G \kappa^{(G)}_y\sqrt{\kappa^{(G)}_x} \right)$             & \multirow{-2}{*}{\xmark}& \multirow{-2}{*}{\xmark} & \multirow{-2}{*}{\xmark} & \multirow{-2}{*}{\cmark}\\ \hline \hline


& $\nabla f: \; \tilde{O}\left( \sqrt{m_G}\frac{\max \left\{L_G+L_f,L_G+L_h \right\}}{\min \{\mu_x,\mu_y\}}\right)$             & $\nabla_x G_i: \; \tilde{O}\left( \sqrt{m_G}\frac{\max \left\{L_G+L_f,L_G+L_h \right\}}{\min \{\mu_x,\mu_y\}}\right)$             &      &      &   &     \\ \cline{2-3}

\multirow{-2}{*}{\cite{NIPS2016_1aa48fc4,alacaoglu2021stochastic}} & $\nabla h:  \tilde{O}\left( \sqrt{m_G}\frac{\max \left\{L_G+L_f,L_G+L_h \right\}}{\min \{\mu_x,\mu_y\}}\right)$             & $\nabla_y G_i: \; \tilde{O}\left( \sqrt{m_G}\frac{\max \left\{L_G+L_f,L_G+L_h \right\}}{\min \{\mu_x,\mu_y\}}\right)$             & \multirow{-2}{*}{\xmark}& \multirow{-2}{*}{\xmark}& \multirow{-2}{*}{\cmark}& \multirow{-2}{*}{\xmark} \\ \hline \hline

\rowcolor[HTML]{C0C0C0} 
\cellcolor[HTML]{C0C0C0}                    & $\nabla f: \; \tilde{O}\left(\sqrt{\kappa^{(f)}_x \kappa^{(G)}_y}\right)$             & $\nabla_x G_i: \; \tilde{O}\left(m_G\sqrt{ \kappa^{(G)}_x \kappa^{(G)}_y}\right)$             & \cellcolor[HTML]{C0C0C0}  & \cellcolor[HTML]{C0C0C0}   & \cellcolor[HTML]{C0C0C0}   & \cellcolor[HTML]{C0C0C0}            \\ \cline{2-3}
\rowcolor[HTML]{C0C0C0} 
\multirow{-2}{*}{\cellcolor[HTML]{C0C0C0}\begin{tabular}[c]{@{}c@{}}\textbf{This paper} \\
\textbf{(Theorem~\ref{theorem:G-sum_noprox})} \end{tabular}} & $\nabla h:  \tilde{O}\left(\max \left\{\sqrt{\kappa^{(G)}_x \kappa^{(h)}_y},\sqrt{ \kappa^{(G)}_x \kappa^{(G)}_y}\right\}\right)$             & $\nabla_y G_i: \; \tilde{O}\left(m_G\sqrt{ \kappa^{(G)}_x \kappa^{(G)}_y}\right)$             & \multirow{-2}{*}{\cellcolor[HTML]{C0C0C0}\xmark} & \multirow{-2}{*}{\cellcolor[HTML]{C0C0C0}\xmark} & \multirow{-2}{*}{\cellcolor[HTML]{C0C0C0}\xmark}& \multirow{-2}{*}{\cellcolor[HTML]{C0C0C0}\cmark}\\ \hline \hline

\rowcolor[HTML]{C0C0C0} 
\cellcolor[HTML]{C0C0C0}                    & $\nabla f: \; \tilde{O}\left(\sqrt{\kappa^{(f)}_x \kappa^{(G)}_y}\right)$             & $\nabla_x G_i: \; \tilde{O}\left(m_G\sqrt{ \kappa^{(G)}_x \kappa^{(G)}_y}\right)$             & \cellcolor[HTML]{C0C0C0}  & \cellcolor[HTML]{C0C0C0}   & \cellcolor[HTML]{C0C0C0}       & \cellcolor[HTML]{C0C0C0}        \\ \cline{2-3}
\rowcolor[HTML]{C0C0C0} 
\multirow{-2}{*}{\cellcolor[HTML]{C0C0C0}\begin{tabular}[c]{@{}c@{}}\textbf{This paper} \\
\textbf{(Theorem~\ref{theorem:G-sum_h_prox})} \end{tabular}} & $\nabla h:  \tilde{O}\left(\sqrt{\kappa^{(G)}_y}\right)$             & $\nabla_y G_i: \; \tilde{O}\left(m_G\sqrt{ \kappa^{(G)}_x \kappa^{(G)}_y}\right)$             & \multirow{-2}{*}{\cellcolor[HTML]{C0C0C0}\xmark} & \multirow{-2}{*}{\cellcolor[HTML]{C0C0C0}\cmark}&\multirow{-2}{*}{\cellcolor[HTML]{C0C0C0}\xmark}& \multirow{-2}{*}{\cellcolor[HTML]{C0C0C0}\cmark}\\ \hline  \hline

& $\nabla f: \; 0$             & $\nabla_x G_i: \; \tilde{O}\left( \sqrt{m_G}\max \left\{\kappa_x^{(G)},\kappa_y^{(G)} \right\}\right)$             &      &      &   &     \\ \cline{2-3}

\multirow{-2}{*}{\cite{NIPS2016_1aa48fc4,alacaoglu2021stochastic}} & $\nabla h:  0$             & $\nabla_y G_i: \; \tilde{O}\left( \sqrt{m_G}\max \left\{\kappa_x^{(G)},\kappa_y^{(G)} \right\}\right)$             & \multirow{-2}{*}{\cmark}& \multirow{-2}{*}{\cmark}& \multirow{-2}{*}{\cmark}& \multirow{-2}{*}{\cmark} \\ \hline \hline

\rowcolor[HTML]{C0C0C0} 
\cellcolor[HTML]{C0C0C0}                    & $\nabla f: \; 0$             & $\nabla_x G_i: \; \tilde{O}\left(\sqrt{m_G} \sqrt{\kappa_x^{(G)}\kappa_y^{(G)}}\right)$             & \cellcolor[HTML]{C0C0C0}  & \cellcolor[HTML]{C0C0C0}& \cellcolor[HTML]{C0C0C0}       & \cellcolor[HTML]{C0C0C0}            \\ \cline{2-3}
\rowcolor[HTML]{C0C0C0} 
\multirow{-2}{*}{\cellcolor[HTML]{C0C0C0}\begin{tabular}[c]{@{}c@{}}\textbf{This paper} \\
\bf{(Theorem~\ref{theorem:section_6})} \end{tabular}} & $\nabla h:  0$             & $\nabla_y G_i: \; \tilde{O}\left(\sqrt{m_G} \sqrt{\kappa_x^{(G)}\kappa_y^{(G)}}\right)$             & \multirow{-2}{*}{\cellcolor[HTML]{C0C0C0}\cmark}& \multirow{-2}{*}{\cellcolor[HTML]{C0C0C0}\cmark} & \multirow{-2}{*}{\cellcolor[HTML]{C0C0C0}\cmark}& \multirow{-2}{*}{\cellcolor[HTML]{C0C0C0}\cmark} \\ \hline
\end{tabular}
\caption{
Comparison of gradient complexities for problem \eqref{eq:problem_Gsum}, i.e. the number of corresponding gradient evaluations, to find $\varepsilon$-saddle point with probability at least $1-\sigma$. Notation $\tilde{O}(X)$ hides constant and polylogarithmic in $\varepsilon^{-1}$ and $\sigma^{-1}$ factors. For a function $F$, we denote $\kappa_x^{(F)}=L_F/\mu_x$, $\kappa_y^{(F)}=L_F/\mu_y$.
\textbf{Prox-f} (\textbf{Prox-h}) stands for $f$ ($h$) being prox-friendly.  
\textbf{ CS} stands for complexity separation. \textbf{VR} stands for variance reduction.}
\label{table:G-sum}

\end{table}
\endgroup

The second problem formulation, we are interested in, is strongly-convex-strongly-concave saddle-point problem of the form 
\begin{equation} 
    \label{eq:problem_Gsum}
    \min _{x \in \mathbb{R}^{d_x}} \max _{y \in \mathbb{R}^{d_y}} \left\{f(x)+ G(x, y) - h(y)\right\}, \;\;\; G(x,y):=\frac{1}{m_G} \sum_{i = 1}^{m_G} G_i(x,y).
\end{equation}
where each $G_i(x, y)$ is convex in $x$ and concave in $y$ and is $L_G^i$-smooth in each variable, $f(x)$ is $\mu_x$-strongly convex, $h(y)$ is $\mu_y$-strongly convex. In this setting it is natural to define condition numbers $\kappa_x=L_G/\mu_x$ and $\kappa_y=L_G/\mu_y$ for the primal minimization and dual maximization problems respectively, where $L_G = \frac{1}{m_G}\sum_{i = 1}^{m_G} L_G^i$. 
We consider this problem under three different additional assumptions: a) $f(x)$ is $L_f$-smooth, $h(y)$ is $L_h$-smooth; b) $f(x)$ is $L_f$-smooth, $h(y)$ is  smooth and prox-friendly; c) $f(x)$  and $h(y)$ are both prox-friendly. Under assumption a) and b), similarly to \cite{pmlr-v125-lin20a,wang2020improved,yang2020catalyst} we do not exploit the 
finite-sum structure of the function $G$. Yet, unlike these papers and \cite{NIPS2016_1aa48fc4}, where variance reduction methods are proposed, we separate the complexity bounds for the number of oracle calls for each part of the objective, i.e. we estimate a sufficient number of evaluations of $\nabla f(x)$, $\nabla_x G_{i}(x,y)$, $\nabla_y G_{i}(x,y)$, $\nabla h(y)$ to achieve a given accuracy. 
This allows to call each oracle less number of times than it is required by existing methods and is important since evaluation of each gradient can have different arithmetic operations complexity, and it is desired to call expensive oracles less often than cheap oracles. 
Compared to \cite{Alkousa2019,alkousa2020accelerated}, where the complexities are also separated, we obtain better complexity bounds for each part of the objective. Moreover, for the particular case when $f=h=0$, our bounds are the same to the best known bounds \cite{wang2020improved,yang2020catalyst}.

Under assumption c), similarly to \cite{NIPS2016_1aa48fc4}, we exploit the finite-sum structure of the function $G$ and propose an accelerated variance reduction method for problem \eqref{eq:problem_Gsum}.
The authors of \cite{NIPS2016_1aa48fc4} considered smooth $\mu_x$-strongly convex and $\mu_y$-strongly concave saddle-point problems in the form $\min _{x \in \mathbb{R}^{d_x}} \max _{y \in \mathbb{R}^{d_y}}G(x, y)+M(x,y)$, where $M(x,y)$ is prox-friendly in both variables. Under an additional assumption that $\mu_x=\mu_y=\mu$, i.e. $\kappa_x=\kappa_y=\kappa$ they obtain complexity $\widetilde{O}(\sqrt{m_G }\kappa)$.
Based on a combination of the Catalyst framework \cite{catalyst2015nips} and the algorithm of \cite{NIPS2016_1aa48fc4}, we propose a variance reduction algorithm with a better bound $\widetilde{O}(\sqrt{m_G   \kappa_x \kappa_y})$. Moreover, this composite setting is not covered by other existing algorithms \cite{pmlr-v125-lin20a,wang2020improved,yang2020catalyst} even in the case of $m_G=1$.
We summarize comparison of ours and existing results in Table \ref{table:G-sum}.


\subsection{Our approach}
To solve the described saddle-point problems under different assumptions we first propose a general framework and then specialize it to problem \eqref{eq:problem} or problem \eqref{eq:problem_Gsum}. Our approach to saddle-point problems is based on considering them as minimization problems with objective implicitly given as a solution to a maximization problem.
Thus, to develop our general framework, we first consider optimization problem of the form
\begin{align}\label{eq:typical}
\min\limits_{x \in \R^d }\{ F\left( x \right):=\varphi\left( x \right)+\psi\left( x \right)\},
\end{align}
and develop a novel inexact accelerated gradient method (Algorithm \ref{alg:highorder_inexact}) which uses inexact first-order information on $\varphi$ and $\psi$ and inexact proximal steps.
Then we note that the problems \eqref{eq:problem} or \eqref{eq:problem_Gsum} can be rewritten as
\begin{equation}\label{eq:3F1}
\min _{x \in  \R^{d_x}} \{ F(x):=f(x)+\underbrace{\max _{y \in \R^{d_y}}\{G(x, y)-h(y)\}}_{g(x)=G(x, y^*(x))-h(y^*(x))} \;\; \} ,
\end{equation}
which is consistent with the problem formulation \eqref{eq:typical}.
Using this representation we can apply our Algorithm \ref{alg:highorder_inexact} with $\varphi\left( x \right) = f(x)$ and $\psi\left( x \right) = g(x)$  to solve this problem. In each step we need to obtain a first-order information about the function $g$, which we can do inexactly by solving the inner maximization problem by the same Algorithm \ref{alg:highorder_inexact}, but now with $\varphi\left( y \right) = -G(x, y)$ and $\psi\left( y \right) = h(y)$. To obtain near-optimal upper complexity bounds and separate oracle complexity for different parts of the problem \eqref{eq:3F1} we introduce additional inner-outer cycles, which will be described in detail below.

As said, our framework is based on the system of inner-outer loops, where in each loop an accelerated gradient method is applied to obtain better complexity results. To implement our approach we then need a flexible accelerated method which can be applied in a number of different situations. In some sense we need an accelerated meta-algorithm, or an accelerated envelope, which uses any method in the lower level to solve an auxiliary problem of the upper level and, as a result, obtain an accelerated version of the method used in the lower level. Existing algorithms of this type \cite{catalyst2015nips, monteiro2013accelerated,aspremont2021acceleration}  are based on accelerated proximal point method that uses some algorithm on the lower level to implement inexact proximal mapping. Unfortunately, we can not use these existing methods in our case since in our system of inner-outer loops a loop in the lower level leads to inexact gradient information in the upper level. Moreover, if a randomized method is used in the lower level, one obtains stochastic inexactness in the upper level. These kinds of inexactnesses of the oracles for $\varphi, \psi$ are not account for in the existing general acceleration frameworks \cite{catalyst2015nips, monteiro2013accelerated,aspremont2021acceleration}. Motivated by this gap in the literature, we develop a generic accelerated meta-algorithm with probabilistic inexact oracles. Moreover, we also implement an adaptive stopping criterion for the method in the lower level which guarantees appropriate quality of the inexact proximal mapping and leads to the accelerated convergence rate on the upper level.



\subsection{Contributions}
To sum up, our contributions are as follows.
First, we  provide a general inexact accelerated meta-algorithm (AM) listed as Algorithm \ref{alg:highorder_inexact} for convex optimization problems of the form \eqref{eq:typical} with inexact oracles. We also obtain an accelerated linearly convergent version of this algorithm by employing the restart technique with the resulting algorithm listed as Algorithm \ref{alg:restarts_inexact_notconvex}. We provide theoretical analysis of this algorithm under stochastic inexactness in different parts of this algorithm, i.e. inexact oracle and inexact proximal step. Unlike existing accelerated proximal methods we consider composite problems \eqref{eq:typical} and use inexact proximal step only with respect to $\varphi$.
Next, we use this AM to construct a new general framework to systematically obtain new algorithms and complexity bounds for saddle-point problems with the structure \eqref{eq:problem} or \eqref{eq:problem_Gsum}. 
As a result, we obtain new accelerated methods for general saddle-point problems, including accelerated variance reduction methods, which leads to better complexity bounds than existing in the literature. Moreover, our algorithms allow to separate complexity bounds for the number of oracle calls for each part of the problem formulation, i.e., for problem \eqref{eq:problem} we estimate a sufficient number of evaluations of $\nabla f(x)$, $\nabla_x G(x,y)$, $\nabla_y G(x,y)$, $\nabla h_i(y)$ to achieve given accuracy.
For problem \eqref{eq:problem_Gsum} we estimate a sufficient number of evaluations of $\nabla f(x)$, $\nabla_x G_i(x,y)$, $\nabla_y G_i(x,y)$, $\nabla h(y)$ to achieve given accuracy.
This complexity separation is important since evaluation of each gradient can have different arithmetic operations complexity, and it is desired to call expensive oracles less often than cheap oracles. 

\subsection{Paper organization}
In Section \ref{section:inexactAM}, we propose an Accelerated Meta-Algorithm and extend it for strongly convex setting with probabilistic inexact oracle and probabilistic inexactness in the proximal step.  
Then, in Section \ref{section:saddleFramework}, by sequential applying of the Accelerated Meta-Algorithm, we obtain a general framework for solving saddle-point problems. This framework is based on two main assumptions for a possibility to solve two optimization problems.
In Section \ref{section:saddle} we specialize the general framework to solve problem \eqref{eq:problem} by showing how to satisfy its two main assumptions, and providing the resulting algorithm. In Section \ref{section:G_sum} we consider problem \eqref{eq:problem_Gsum} under additional assumptions: a) $f(x)$ is $L_f$-smooth, $h(y)$ is $L_h$-smooth; b) $f(x)$ is $L_f$-smooth, $h(y)$ is  smooth and prox-friendly. We specialize the general framework for this setting and propose accelerated algorithms.
Finally, in Section \ref{S:G_sum_prox-friendly} we consider problem \eqref{eq:problem_Gsum} under additional assumption c) both $f(x)$  and $h(y)$ are prox-friendly. In this case, since the Accelerated Meta-Algorithm can not be applied in this case, we develop a different approach based on a combination of the Catalyst framework \cite{catalyst2015nips} and the algorithm of \cite{NIPS2016_1aa48fc4}.

\subsection{Preliminaries}
We introduce some notation which we use throughout the paper. We denote by $\|x\|$ and $\|y\|$ the standard Euclidean norms for $x \in \mathbb{R}^{d_x}$ and $y \in  \mathbb{R}^{d_y}$ respectively. This leads to the Euclidean norm on $\mathbb{R}^{d_x} \times \mathbb{R}^{d_y}$ defined as $\| (x_1, x_2) - (y_1, y_2)\|^2 = \|x_1 - x_2\|^2 + \|y_1 - y_2\|^2$, $x_1, x_2, \in  \mathbb{R}^{d_x} , y_1, y_2 \in \mathbb{R}^{d_y}$.

We say that a function $f$ is $L_f$-smooth if it is differentiable and its gradient satisfies Lipshitz condition
\begin{equation}
    \|\nabla f (x_1)-\nabla f (x_2)\|\leq L_f\|x_1-x_2\|, \quad x_1,x_2 \in {\rm dom } f
\end{equation}
for some $L_f >0$.
We say that a function $f$ is $\mu_f$-strongly convex if, for some $\mu_f > 0$ and for any its subgradient it holds that
\begin{equation}
    f (x_2) \geq f (x_1) + \langle\nabla f (x_1), x_2-x_1 \rangle + \frac{\mu_f}{2}\|x_1-x_2\|^2, \quad x_1,x_2 \in {\rm dom } f.
\end{equation}
We say that a function $f$ is prox-friendly if it admits a tractable proximal operator \cite{BSMF_1965__93__273_0}. This means that the evaluation of the point 
\begin{equation}
\arg \min_x\left\{f(x)+\frac{1}{2}\|x-\bar{x}\|^2 \right\}
\end{equation}
for some fixed $\bar{x}$ can be made either in closed form or numerically very efficiently up to  machine precision.
Note that if a function is prox-friendly, then the problem, 
 \begin{equation} \label{f_x:prox-friendly}
 \underset{x }{\min}\left\{\la c_1, x\ra + f(x) + \frac{c_2}{2}\|x-\bar{x}\|^2\right\}
 \end{equation}
can be solved either in closed form or numerically very efficiently up to  machine precision for any fixed   $ c_1$, $\bar{x}$, and $c_2 > 0$.

For an optimization problem $\min_x f(x)$, we say that a random point $\hat{x}$ is an $(\varepsilon,\sigma)$-solution to this problem if $\mathbb{P}\{ f(\hat{x})-f^*\leq \varepsilon \} \geq 1-\sigma$.

For a function $\xi \left(\varepsilon\right)$, where $\varepsilon  \in \mathbb{R}$ we write $\xi \left(\varepsilon  \right)=\bf{poly}\left(\varepsilon \right)$ if $\xi \left( \cdot  \right)$ is a polynomial function of $\varepsilon$.
For a function $\xi \left(\varepsilon, \sigma \right)$, where $\varepsilon, \sigma \in \mathbb{R}$ we write $\xi \left(\varepsilon, \sigma \right)=\bf{poly}\left(\varepsilon, \sigma \right)$ if $\xi \left( \cdot,\sigma \right)$ is a polynomial function of $\varepsilon$ and $\xi \left(\varepsilon,\cdot  \right)$ is a polynomial function of $\sigma$.

\section{Inexact Accelerated Meta-algorithm}\label{section:inexactAM}
As it was described above, our approach is based on an accelerated method for a general optimization problem with the objective given as a sum of two functions
\begin{equation}
\label{eq:PrStGen1}
\min\limits_{x \in \R^{d_x} }\{ F\left( x \right):=\vp\left( x \right)+\psi\left( x \right)\} .
\end{equation}
In this section we describe this method in the inexact oracle model, so that we can apply it in the system of inner-outer loops to propose accelerated methods for saddle-point problems.

To motivate the study of this section, we slightly rewrite  problem \eqref{eq:problem} in the following way:
\begin{equation}\label{eq:3F}
\min _{x \in  \R^{d_x}} \{ F(x):=\vp(x)+\underbrace{\max _{y \in \R^{d_y}}\{G(x, y)-h(y)\}}_{\psi(x):=G(x, y^*(x))-h(y^*(x))} \;\; \},
\end{equation}
where $y^*(x)$ is the solution for the problem defining $\psi(x)$ for a fixed $x$.
In other words, we can represent  problem \eqref{eq:problem} as an   optimization problem $\min_{x \in \R^{d_x}} \vp(x) + \psi(x)$ with a particular choice of $\vp, \psi$:
\begin{align}\label{eq:psi_example}
    \varphi = f(x), \ \ \psi = \max _{y \in  \R^{d_y}}\left\{G(x, y)-h(y)\right\}.
\end{align}
Importantly, we have no access to the exact gradients of $\psi(x)$ since we can not solve exactly the problem defining $\psi(x)$. At the same time, according to Lemma 2 from \cite{alkousa2020accelerated} we can get (precise definition is given below) an inexact $(\delta, 2 L_{\psi})$ oracle, where $\delta$ depends on the accuracy of the solution of the problem $\max _{y \in  \R^{d_y}}\left\{G(x, y)-h(y)\right\}$. 
Thus, we need to develop an accelerated algorithm for problem \eqref{eq:3F}
which takes into account inaccuracy of the oracles for functions $\vp(x), \psi(x)$ caused by inexact solution to the optimization problem defining $\psi(x)$. 

The situation is even more complicated if we consider problem \eqref{eq:problem} with $m_h>1$ or problem \eqref{eq:problem_Gsum} with $m_G>1$ and apply variance reduction techniques.
Application of known variance reduction methods guarantees us a solution to the problem $\max _{y \in  \R^{d_y}}\left\{G(x, y)-h(y)\right\}$ only with  some high probability $1-\sigma$. Thus, when using the variance reduction setting we obtain an inexact oracle for $\psi(x)$ only with some probability.

To sum up the motivation part, we need to develop a generic acceleration scheme which works with inexact oracles including inexact oracles with high probability. The rest of this section is devoted to the precise definitions of inexact oracles, description of such an accelerated algorithm and stating its convergence properties. Main technical proofs are deferred to the appendices.
Since we believe that the proposed accelerated algorithm with inexact oracles can be of independent interest, we spend some effort to establish more results than we need for the main purpose of this paper. So, first we consider optimization with deterministic oracle, and then move to the setting of probabilistic inexact oracles.


\subsection{Deterministic setting}
Having in mind the above motivation, we introduce necessary notation and definitions. 
We start with a definition, which corresponds to convex functions with Lipschitz-continuous gradient and is a small generalization of inexact oracle introduced in \cite{devolder2014first}.
\begin{definition}
\label{def:delta_L_oracle}
Let $\delta = (\delta_1, \delta_2)$, where $\delta_1, \delta_2 >0$.  Then the pair $(\vp_{\delta, L} (x), \nabla \vp_{\delta, L} (x))$ is called $(\delta, L)$-oracle of  a convex function $ \vp(x) $ at a point $x$, if 
\begin{equation} \label{def:deltaL}
    -\delta_1 \leq \vp(z)-\left(\vp_{\delta, L}(x)+\left\langle \nabla \vp_{\delta, L}(x), z-x\right\rangle\right) \leq \frac{L}{2}\|z-x\|^{2}+\delta_2 \text { for all } z \in \mathbb{R}^{d_x}.
\end{equation} 
With a slight abuse of notation, we use the same notation $(\delta, L)$-oracle for the case $(\delta_1, \delta_2) = (0, \delta)$.
\end{definition}



Our Accelerated Meta-algorithm (AM) is listed below as Algorithm \ref{alg:highorder_inexact}. The method generates three sequences, which are denoted by the same letter $x$ with either no superscript or one of the two superscripts $x^t$, $x^{md}$. Since later we will use this algorithm in a system of inner-outer loops, we will change the letter to denote the sequences, but will not change the superscripts. The idea of the algorithm is inspired by the Monteiro--Svaiter algorithm \cite{monteiro2013accelerated}, but there are several important differences. The first one is that in \eqref{prox_step_inexact} we linearize the function $\vp$ instead of making inexact proximal step for the whole objective $F$ as it is done in \cite{monteiro2013accelerated}. The second difference is that we use inexact oracles for the functions  $\vp$ and $\psi$, and as a corollary inexact oracle for $F$. This affects the measure of inexact solution to problem \eqref{prox_step_inexact} and Step \ref{Step:AM_grad_step} of the algorithm. Thirdly, below we introduce a more convenient in practice way to control the accuracy of the solution to the inexact proximal step \eqref{prox_step_inexact}. To do that we quantify with which accuracy one needs to solve the problem \eqref{prox_step_inexact} in terms of its objective residual, so that the whole Algorithm \ref{alg:highorder_inexact} outputs a solution to the problem \eqref{eq:PrStGen1} with a desired accuracy. This makes it easy to apply Algorithm \ref{alg:highorder_inexact} in a system of inner-outer loops. Finally, the algorithm in \cite{monteiro2013accelerated} is not proved to obtain accelerated linear convergence rate in the case when the objective is strongly convex. For our algorithm we propose an extension which has accelerated linear convergence rate under additional assumption of inexact strong convexity.

\begin{algorithm} [h!]
\caption{Accelerated Meta-algorithm (AM) with inexact $(\delta, L)$-oracles}\label{alg:highorder_inexact}
	\begin{algorithmic}[1]
		\STATE \textbf{Input:} objective $F = \vp+\psi$ where $\vp, \psi$ are convex, 
parameter $H \geq 2L_{\vp}$, inexactness $\delta \geq 0$, starting point $x_0$;
		
		$(\vp_{\delta, L_{\vp}}, \nabla \vp_{\delta, L_{\vp}})$  
		--- $(\delta, L_{\vp})$-oracle of $\vp$, 

		$(\psi_{\delta, L_{\psi}}, \nabla \psi_{\delta, L_{\psi}})$ 
		--- $(\delta, L_{\psi})$-oracle of $\psi$.

		\STATE Set $A_0 = 0, x_0^t =  x_0, x_0^{md}=x_0$.
		\FOR{ $k = 0$ \TO $k = K- 1$}
		\STATE 	Set	 
		$
        a_{k+1} = \frac{1+\sqrt{1+8HA_k}}{4H}, \quad 
        A_{k+1} = A_k+a_{k+1}. 
        $
        \STATE 	Set	
        $
        x_{k}^{md} = \frac{A_k}{A_{k + 1}}x_k^t + \frac{a_{k+1}}{A_{k+1}} x_k. 
        		$
        \label{alg:gradient_step}
		\STATE Find $x_{k+1}^t$ as an approximate solution to the minimization problem 
\begin{equation}
\label{prox_step_inexact}
\hspace{-2em} 
x_{k+1}^t \approx \argmin_{z\in \R^{d_x}} \left\{
\vp_{\delta, L_{\vp}}(x_{k}^{md})+ \langle \nabla \vp_{\delta, L_{\vp}}(x_{k}^{md}), z-x_{k}^{md}\rangle 
+\psi(x) +\frac{H}{2}\|z-x_{k}^{md}\|^{2} \right\} \,,
\end{equation}
such that 


\begin{align}
    \left\| \nabla \vp_{\delta, L_{\vp}}(x_{k}^{md}) + \nabla \psi_{\delta,L_{\psi}}\left({x}_{k+1}^t\right) + H({x}_{k+1}^t-x_{k}^{md})
    \right\| \leq  
     \frac{H}{4 }\left\|{x}_{k+1}^t-x_{k}^{md}\right\| -2\sqrt{2\delta_2 L_{\vp}}. \label{eq:prox_step_inexact_crit}
\end{align}
		\STATE \label{Step:AM_grad_step}$x_{k+1} = x_k-a_{k+1} \nabla \vp_{\delta, L}(x_{k+1}^t) - a_{k+1}\nabla \psi_{\delta, L}(x_{k+1}^t)$.
		\ENDFOR
		\RETURN $x_{K}^t$ 
	\end{algorithmic}	
\end{algorithm}

 The next theorem gives the convergence rate of Algorithm \ref{alg:highorder_inexact}  when applied to the problem \eqref{eq:PrStGen1}. 
\begin{theorem} \label{theoremCATDinexact}
Assume that the starting point $x_0$ of Algorithm~\ref{alg:highorder_inexact} satisfies $\|x_0 - x_{\ast}\| \leq R$ for some $R >0$, and that the parameter $H$ is chosen to satisfy $H\ge 2 L_{\vp}$.
Assume also that the algorithm uses  $ (\delta, L_{\vp}) $-oracle of convex function $ \vp (x)$ and $ (\delta, L_{\psi}) $-oracle of convex function $ \psi (x)$, and that the auxiliary subproblem \eqref{prox_step_inexact} is solved inexactly in each iteration in such a way that the inequality \eqref{eq:prox_step_inexact_crit} holds.
Then, for all $k\geq 0$, the sequence $x_k^t$ generated by Algorithm~\ref{alg:highorder_inexact} satisfies
 \begin{equation} \label{speedCATDinexact}
 F(x_k^t) - F(x_{\ast}) \leq \frac{4 H R^2}{k^2} + 2\left(\sum_{i = 1}^k A_{i} \right) \frac{\delta_2}{A_k} +\delta_1+ \left( \sum_{i=1}^{k-1} A_{i} \right)  \frac{\delta_1}{A_k}.
 \end{equation}
\end{theorem}
We prove this theorem in Appendix~\ref{Appendix_A}.


We now move further to the strongly-convex setting, which will allow us to solve strongly-convex-strongly-concave saddle-point problems in later sections. 
The next definition is an extension of Definition \eqref{def:delta_L_oracle} and \cite{devolder2013exactness} corresponding to strongly convex functions with Lipschitz-continuous gradient.
\begin{definition}\label{def:deltaLmu_oracle}
Let $\delta = (\delta_1, \delta_2)$, where $\delta_1, \delta_2 >0$.  Then the pair  
$(\vp_{\delta, L,\mu} (x), \nabla \vp_{\delta, L,\mu} (x))$ is called $(\delta, L,\mu)$-oracle of a convex function $ \vp $ at a point $x$, if 
\begin{equation} \label{def:deltaLmu}
     \frac{\mu}{2}\|z-x\|^{2}-\delta_1 \leq \vp(z)-\left(\vp_{\delta, L,\mu}(x)+\left\langle \nabla \vp_{\delta, L,\mu}(x), z-x\right\rangle\right) \leq \frac{L}{2}\|z-x\|^{2}+\delta_2 \text { for all } z \in \mathbb{R}^{d_x}
\end{equation}
With a slight abuse of notation, we use the same notation $(\delta, L)$-oracle for the case $(\delta_1, \delta_2) = (0, \delta)$.
\end{definition}
It is straightforward that a $(\delta, L,\mu)$-oracle is also a $(\delta, L)$-oracle, and, thus, we can use $(\delta, L,\mu)$-oracle in Algorithm~\ref{alg:highorder_inexact}.  


Next we consider the case when $F(x)$ in \eqref{eq:PrStGen1} is convex and admits a $(\delta, L, \mu)$-oracle. Then, we use the convergence rate result in Theorem \ref{theoremCATDinexact} and obtain linear convergence rate by applying the restart technique. The restarted algorithm is listed as Algorithm \ref{alg:restarts_inexact_notconvex}, and its convergence rate when applied to the problem \eqref{eq:PrStGen1} is given in Theorem \ref{CATDrestarts_inexact_notconvex}.

\begin{algorithm} [h!]
\caption{Restarted AM (R-AM)}\label{alg:restarts_inexact_notconvex}
	\begin{algorithmic}[1]
		\STATE \textbf{Input:} objective $F = \vp+\psi$ 
		 admits $(\delta, L, \mu)$-oracle, 
		parameter $H \geq 2 L_{\vp}$, inexactness $\delta \geq 0$, starting point $z_0$;
		
		$(\vp_{\delta, L_{\vp}}, \nabla \vp_{\delta, L_{\vp}})$  
		--- $(\delta, L_{\vp})$-oracle of convex function  $\vp$, 

		$(\psi_{\delta, L_{\psi}}, \nabla \psi_{\delta, L_{\psi}})$ 
		--- $(\delta, L_{\psi})$-oracle of convex function $\psi$.

		\FOR{$k = 0, $ \TO $K$}
		\STATE Set 
		\begin{equation}
		\label{numberofrestarts} 
		    N_k=\max \left\{ \left\lceil \left( \frac{128 H }{\mu} \right)^{\frac{1}{2}} \right\rceil, 1\right\}.
		\end{equation}
		\STATE Set $z_{k+1} := x_{N_k}^t$ as the output of Algorithm \ref{alg:highorder_inexact} started from $z_k$ and run for $N_k$ steps.
		\ENDFOR
		\RETURN $z_{K}$ 
	\end{algorithmic}	
\end{algorithm}

\begin{theorem}
\label{CATDrestarts_inexact_notconvex}
Assume that the starting point $z_0$ of Algorithm~\ref{alg:restarts_inexact_notconvex} satisfies $\|z_0 - x_{\ast}\| \leq R$ for some $R >0$, and that the parameter $H$ is chosen to satisfy $H\ge 2 L_{\vp}$.
Further, assume that  $(\delta, L, \mu)$-oracle of $F(x)$, $(\delta,L_{\vp})$-oracle of convex function $\vp(x)$, $(\delta,L_{\psi})$-oracle of convex function $\psi(x)$ are available, and, in each iteration of Algorithm~\ref{alg:highorder_inexact} which is used as a building block of Algorithm~\ref{alg:restarts_inexact_notconvex},  the auxiliary subproblem \eqref{prox_step_inexact} is solved inexactly in such a way that the inequality \eqref{eq:prox_step_inexact_crit} holds.
Finally, assume that the oracle inexactness $\delta_1, \delta_2$ are chosen to satisfy
\begin{align}
    \label{delta_CATD1}
    \forall k: \delta_1 + \delta_2 + 2\left(\sum_{i = 1}^k A_{i} \right) \frac{\delta_2}{A_k} + \left(\sum_{i = 1}^{k-1} A_{i} \right) \frac{\delta_1}{A_k} \leq \frac{\varepsilon}{2},
\end{align} 
\begin{align}\label{delta_CATD2}
    \frac{4\sqrt{2\delta_2 L}}{\mu} \leq \varepsilon/2,
\end{align}
where $\varepsilon$ is the desired accuracy of the solution to problem \eqref{eq:PrStGen1}.
Then, under the listed assumptions, Algorithm \ref{alg:restarts_inexact_notconvex} with 
$K = 2\log_2 \frac{\mu R_0^2}{4\varepsilon}$
guarantees that its output point $z_K$ is an $\varepsilon$-solution to problem \eqref{eq:PrStGen1}, i.e. $F(z_K)-F(x^*) \leq \varepsilon$. Moreover, the total number $N_F$ of calls to inexact oracles both for $\vp$ and for $\psi$ satisfies the following inequality 
\begin{equation}
    N_F \leq  \left( 16\sqrt{2}\sqrt{\frac{H}{\mu}} + 2\right) \log_2 \frac{\mu R_0^2}{\varepsilon} = \widetilde{O} \left(\max\left\{ \sqrt{\frac{H}{\mu}},1\right\} \right).
\end{equation}
\end{theorem}
We prove this theorem in Appendix~\ref{Appendix_B}.

As we see from the above theorems, to ensure that AM and R-AM algorithms provide an $\varepsilon$-solution to problem \eqref{eq:PrStGen1}, we need to guarantee that the oracle error $\delta = (\delta_1, \delta_2)$ is sufficiently small and that the auxiliary problem \eqref{prox_step_inexact} is solved inexactly in such a way that the inequality \eqref{eq:prox_step_inexact_crit} is satisfied.
For our purposes it is more convenient to consider inexact solution of the problem \eqref{prox_step_inexact} not in terms of the inequality \eqref{eq:prox_step_inexact_crit}, but rather in terms of the objective residual in this problem bounded by some tolerance $\tilde{\varepsilon}_f$. Next we provide sufficient conditions on the values of $\delta$ and $\tilde{\varepsilon}_f$ which guarantee that the conditions of the above theorems hold and that R-AM is guaranteed to find an $\varepsilon$-solution to problem \eqref{eq:PrStGen1}.


\begin{theorem} \label{AM:comfortable_view}
Assume that the starting point $z_0$ of Algorithm~\ref{alg:restarts_inexact_notconvex} applied to problem \eqref{eq:PrStGen1} satisfies $\|z_0 - x_{\ast}\| \leq R$ for some $R >0$, and that the parameter $H$ is chosen to satisfy $H\ge 2 L_{\vp}$.
Further, assume that  $F(x)$ is convex, $(\delta, L, \mu)$-oracle of $F(x)$, $(\delta,L_{\vp})$-oracle of convex function $\vp(x)$, $(\delta,L_{\psi})$-oracle of convex function $\psi(x)$ are available, and, in each iteration of Algorithm~\ref{alg:highorder_inexact} which is used as a building block of Algorithm~\ref{alg:restarts_inexact_notconvex},  the auxiliary subproblem \eqref{prox_step_inexact} is solved inexactly in such a way that the inexact solution $x_{k+1}^t$  satisfies
\begin{align}
&\left( 
\langle \nabla \vp_{\delta, L_{\vp}}(x_{k}^{md}), x_{k+1}^t-x_{k}^{md}\rangle 
+\psi(x_{k+1}^t) +\frac{H}{2}\|x_{k+1}^t-x_{k}^{md}\|^{2} 
\right) \notag \\
& \hspace{2em}- \min_{z \in \R^{d_x}} \left(
\langle \nabla \vp_{\delta, L_{\vp}}(x_{k}^{md}), z-x_{k}^{md}\rangle 
+\psi(x) +\frac{H}{2}\|z-x_{k}^{md}\|^{2} \right) \leq \tilde{\varepsilon}_f, \;\; \text{where} \label{eq:prox_step_inexact_crit_objective}
 \end{align}
\begin{align}
    \label{varepsilon_poly_CATD}
    \tilde{\varepsilon}_f \leq \frac{\varepsilon \mu^2}{864^2(L+H)^2}.
\end{align}
Finally, assume that the oracle errors $\delta_1$, $\delta_2$ satisfy 
 \begin{align}
 \label{delta_poly_CATD}
    \delta_1, \delta_2 \leq \min \left\{\frac{\varepsilon \mu}{ 864^2 L_{\vp}},\frac{\varepsilon \mu}{ 864^2L_{\psi}}, \frac{\varepsilon \mu^2}{ 864^2(L+H)^2},  \frac{\varepsilon^{3/2}}{5 \sqrt{8 H R^2}}  \right\}.
\end{align}

Then, under the listed assumptions, Algorithm \ref{alg:restarts_inexact_notconvex} with 
$K = 2\log_2 \frac{\mu R_0^2}{4\varepsilon}$
guarantees that its output point $z_K$ is an $\varepsilon$-solution to problem \eqref{eq:PrStGen1}, i.e. $F(z_K)-F(x^*) \leq \varepsilon$. Moreover, the total number $N_F$ of calls to inexact oracles both for $\vp$ and for $\psi$ satisfies the following inequality 
\begin{equation}
    N_F \leq  \left( 16\sqrt{2}\sqrt{\frac{H}{\mu}} + 2\right) \log_2 \frac{\mu R_0^2}{\varepsilon} = \widetilde{O} \left(\max\left\{ \sqrt{\frac{H}{\mu}},1\right\} \right).
\end{equation}

\end{theorem}
We prove this theorem in Appendix~\ref{Appendix_C}.

An important feature of the above bounds on $\delta_1$, $\delta_2$ and $\tilde{\varepsilon}_f$ is that they depend polynomially on the target accuracy $\varepsilon$. This means that if we can control these errors by some algorithms which have complexity logarithmically depending on $\delta_1$, $\delta_2$ and $\tilde{\varepsilon}_f$, then the total complexity of the whole algorithm R-AM will be logarithmic in the target accuracy $\varepsilon$, which makes it reasonable to apply this algorithm in a system of inner-outer loops. In the next subsection we extend the above theory for stochastic setting.


\subsection{Stochastic setting}
As it was discussed at the beginning of this section, we would like to apply stochastic variance reduction methods or other randomized methods in order to provide an inexact solution to the auxiliary problem \eqref{prox_step_inexact} and in order to obtain inexact oracle for $F$.
In the former case inequality \eqref{eq:prox_step_inexact_crit_objective} can be guaranteed only with some probability. To illustrate the latter case, we consider function $\psi$ in  \eqref{eq:psi_example} with $h$ given in \eqref{eq:problem} with $m_h\gg1$, i.e.
\begin{align}
\psi = \max _{y \in  \R^{d_y}}\left\{G(x, y)-\frac{1}{m_h} \sum_{i = 1}^{m_h} h_i(y)\right\}.
\end{align}
According to Lemma 2 from \cite{alkousa2020accelerated} we can get an inexact $(\delta, 2 L_{\psi})$-oracle, where $\delta$ depends on the accuracy of the solution of this maximization problem. If we solve this maximization problem by a randomized method, we can obtain  inexact $(\delta, 2 L_{\psi})$-oracle only with some probability.
Thus, below we give a formal generalization of the results obtained in the previous subsection to a stochastic setting. We start with the definition of probabilistic inexact oracle.

\begin{definition}
\label{def:d_L_mu_oracle_prob}
Let $\delta = (\delta_1, \delta_2)$, where $\delta_1, \delta_2 >0$.
Then the pair  
$(\vp_{\delta, L,\mu} (x), \nabla \vp_{\delta, L,\mu} (x))$ is called $(\delta,\sigma_0, L,\mu)$-oracle of a convex function $ \vp $ at a point $x$, if
\begin{equation} \label{def:deltaLmuprob}
     \frac{\mu}{2}\|z-x\|^{2}-\delta_1 \leq \vp(z)-\left(\vp_{\delta, L,\mu}(x)+\left\langle \nabla \vp_{\delta, L,\mu}(x), z-x\right\rangle\right) \leq \frac{L}{2}\|z-x\|^{2}+\delta_2,  \text{for all} \in \mathbb{R}^{d_x}  \text{\;\; w.p. } 1-\sigma_{0}
\end{equation}
In the case of $\mu=0$, we say that $(\vp_{\delta, L} (x), \nabla \vp_{\delta, L} (x))$ is called $(\delta,\sigma_0, L)$-oracle of a function $ \vp $ at a point $x$. 
With a slight abuse of notation, we use the same notation $(\delta , \sigma_0 , L, \mu)$-oracle for the case $(\delta_1, \delta_2) = (0, \delta)$.\\
One should distinguish the following notation: the $(\delta,\sigma_0, L)$-oracle of a function $ \vp $ in the sense of Definition~\ref{def:d_L_mu_oracle_prob} and $(\delta, L, \mu)$-oracle of a function $ \vp $ in the sense of Definition~\ref{def:deltaLmu_oracle}.  
\end{definition} 

The following is a simple lemma, which states that such defined inexact oracle is additive.
\begin{lemma}\label{lem:sum_oracle}
Let the following assumptions hold.
\begin{enumerate}
    \item $(\vp_{\delta_{\vp}, L_{\vp},\mu_{\vp}} (x), \nabla \vp_{\delta_{\vp}, L_{\vp},\mu_{\vp}} (x))$ is $(\delta_{\vp},\sigma_{\vp}, L_{\vp}, \mu_{\vp})$-oracle for a convex function $\vp$,
    \item $(\psi_{\delta_{\psi}, L_{\psi},\mu_{\psi}} (x), \nabla \psi_{\delta_{\psi}, L_{\psi},\mu_{\psi}} (x))$ is $(\delta_{\psi},\sigma_{\psi}, L_{\psi}, \mu_{\psi})$-oracle for a convex function $\psi$.
\end{enumerate}
 Then $(\vp_{\delta_{\vp}, L_{\vp},\mu_{\vp}} (x)+\psi_{\delta_{\psi}, L_{\psi},\mu_{\psi}} (x), \nabla \vp_{\delta_{\vp}, L_{\vp},\mu_{\vp}} (x)+\nabla \psi_{\delta_{\psi}, L_{\psi},\mu_{\psi}} (x))$ is $(\delta_{\vp}+\delta_{\psi},\sigma_{\vp}+\sigma_{\psi}, L_{\vp}+L_{\psi}, \mu_{\vp}+\mu_{\psi})$-oracle for $\vp+\psi$.
 \end{lemma}
 We provide the proof of this lemma in the Appendix~\ref{Appendix_D}.

To illustrate why such inexact oracle appears to be useful in the setting of saddle-point problems, we provide the following Lemma, which extends the results of \cite{Alkousa2019,hien2020inexact} to our stochastic setting and which will be very important for the derivations in the next section. This Lemma contains some novelty in comparison with the literature: it is proved in the stochastic setting.

 \begin{lemma}\label{lemma:obt_delta_oracle}
Let us consider the function 
\begin{align}\label{eq:lem_obt_oracle_g}
g(x)= \max _{y \in  \R^{d_y}}\left\{\hat{S}(x,y) = F(x,y) - w(y) \right\},
\end{align}
where $F(x, y)$ is convex in $x$, concave in $y$ and is $L_F$-smooth as a function of $(x,y)$,  $w(y)$ is $\mu_y$-strongly convex.
Then $g(x)$ is $L_g$-smooth with $L_g=L_F+\frac{2L_F^2}{\mu_y}$ and $y^*(\cdot)$ is $\frac{2L_F}{\mu_y}$ Lipschitz continuous, where the point $y^*$ is defined as
\begin{equation}
    y^*(x) := \arg\max_{y \in \mathbb{R}^{d_y}} \hat{S}(x,y) \quad
     \forall x \in \mathbb{R}^{d_x},
\end{equation}
Moreover, if a point $\tilde{y}_{\delta/2}(x)$ is a $\left(\delta/2,\sigma\right)$-solution to \eqref{eq:lem_obt_oracle_g}, i.e. satisfies inequality
\begin{equation}\label{eq:lem_obt_oracle_delta}
    \max _{y \in \R^{d_y}}\{\hat{S}(x,y)\}-\hat{S}\left(x, \tilde{y}_{\delta/2}(x)\right) \leq \delta/2 \text{\;\; w.p. } 1-\sigma,
\end{equation}
then $ \nabla_{x} F\left(x, \tilde{y}_{\delta/2}(x)\right)$ is $(\delta,\sigma, 2 L_g)$-oracle of $g$.
\end{lemma}
We prove this lemma in Appendix~\ref{Appendix_obtain_oracle}.

Armed with Definition \ref{def:d_L_mu_oracle_prob} we can now formulate the following theorem, which is a generalization of Theorem \ref{AM:comfortable_view}, and which is the main result of this section. This theorem provides the iteration complexity of Algorithm \ref{alg:restarts_inexact_notconvex} to obtain an $(\varepsilon,\sigma)$-solution of problem \eqref{eq:PrStGen1} in the stochastic setting under the assumptions of probabilistic inexact oracles for $\vp$, $\psi$ in the sense of Definition \ref{def:d_L_mu_oracle_prob} and also under the assumption that the auxiliary problem \eqref{prox_step_inexact}, which needs to be solved many times in each iteration of Algorithm \ref{alg:restarts_inexact_notconvex}, is solved inexactly with accuracy controlled in a probabilistic sense.


\begin{theorem} \label{AM:comfortable_view_with_prob}
Consider the optimization problem \eqref{eq:PrStGen1}
\begin{equation*}
    \min_{x \in \mathbb{R}^{d_x}}F(x) = \vp(x)  + \psi(x),
\end{equation*}
where $F(x)$ is convex. Let the target accuracy $\varepsilon >0$ and the target confidence level $\sigma \in (0,1)$ be given. Let also be given $H\ge 2 L_{\vp}$, starting point $z_0$ and  a number $R_0 >0$ such that $\|z_0 - x_{\ast}\| \leq R_0$, where $x_{\ast}$ is the solution to \eqref{eq:PrStGen1}.
Let the following two main assumptions of this theorem hold.
\begin{enumerate}
    \item (Inexact oracle.) 
    Inexact $(\delta,\sigma_{0}, L, \mu )$-oracle of $F(x)$, $(\delta,\sigma_{0},L_{\vp})$-oracle of convex function $\vp(x)$, $(\delta,\sigma_{0},L_{\psi})$-oracle of convex function $\psi(x)$ are available, where 
    $\delta_1(\varepsilon), \delta_2(\varepsilon)$ satisfy the following polynomial dependency on $\varepsilon$  
     \begin{align}
     \label{delta_poly_CATD_with_prob}
        \delta_1(\varepsilon), \delta_2(\varepsilon) \leq \min \left\{\frac{\varepsilon \mu}{ 864^2 L_{\vp}},\frac{\varepsilon \mu}{ 864^2L_{\psi}}, \frac{\varepsilon \mu^2}{ 864^2(L+H)^2},  \frac{\varepsilon^{3/2}}{5 \sqrt{8 H R_0^2}}  \right\},
    \end{align}
    and $\sigma_{0}(\varepsilon,\sigma)$ satisfy the following polynomial dependency on $\varepsilon$ and $\sigma$ 
    \begin{align}
        \label{sigma_0_poly_CATD_with_prob}
        & \sigma_{0}(\varepsilon, \sigma) \leq \frac{\sigma}{2\left( 16\sqrt{2}\sqrt{\frac{H}{\mu}} + 2\right) \log_2 \frac{\mu R_0^2}{\varepsilon}}.
    \end{align}
    \item (Inexact solution of the auxiliary problem \eqref{prox_step_inexact}.) Algorithm \ref{alg:restarts_inexact_notconvex} is applied to solve problem \eqref{eq:PrStGen1} and, in each iteration of Algorithm \ref{alg:highorder_inexact} used as a building block in Algorithm \ref{alg:restarts_inexact_notconvex}, an $(\tilde{\varepsilon}_f,\tilde{\sigma})$-solution to the auxiliary problem \eqref{prox_step_inexact} is available, i.e., with probability at least $1-\tilde{\sigma}$
    \begin{align}
    &\left( 
    \langle \nabla \vp_{\delta, L_{\vp}}(x_{k}^{md}), x_{k+1}^t-x_{k}^{md}\rangle 
    +\psi(x_{k+1}^t) +\frac{H}{2}\|x_{k+1}^t-x_{k}^{md}\|^{2} 
    \right) \notag \\
    & \hspace{2em}- \min_{z \in \R^{d_x}} \left(
    \langle \nabla \vp_{\delta, L_{\vp}}(x_{k}^{md}), z-x_{k}^{md}\rangle 
    +\psi(x) +\frac{H}{2}\|z-x_{k}^{md}\|^{2} \right) \leq \tilde{\varepsilon}_f,
    \end{align}
    where $\tilde{\varepsilon}_f(\varepsilon)$ and $\tilde{\sigma}(\varepsilon, \sigma)$ satisfy the following polynomial dependencies on $\varepsilon$ and $\sigma$
    \begin{align}
        \label{varepsilon_poly_CATD_with_prob}
        \tilde{\varepsilon}_f(\varepsilon) \leq \frac{\varepsilon \mu^2}{864^2(L+H)^2},
    \end{align}
    \begin{align}
        \label{sigma_poly_CATD_with_prob}
        & \tilde{\sigma}(\varepsilon, \sigma) \leq\frac{\sigma}{2\left( 16\sqrt{2}\sqrt{\frac{H}{\mu}} + 2\right) \log_2 \frac{\mu R_0^2}{\varepsilon}}.
    \end{align}
\end{enumerate}
Then, under the listed assumptions, Algorithm \ref{alg:restarts_inexact_notconvex} with 
$K = 2\log_2 \frac{\mu R_0^2}{4\varepsilon}$
guarantees that its output point $z_K$ is an $(\varepsilon,\sigma)$-solution to problem \eqref{eq:PrStGen1}. Moreover, the number $N_F$ of the calls to inexact oracle both for $\vp$ and for $\psi$
satisfies the following inequality 
   \begin{align}\label{iter_poly_CATD_with_prob}
    N_F \leq \left( 16\sqrt{2}\sqrt{\frac{H}{\mu}} + 2\right) \log_2 \frac{\mu R_0^2}{\varepsilon}=\widetilde{O} \left(\max\left\{ \sqrt{\frac{H}{\mu}},1\right\} \right),
\end{align}
and the number of times the auxiliary problem \eqref{prox_step_inexact} is solved is also equal to $N_F$.
\end{theorem}
We prove this theorem in Appendix~\ref{Appendix_C}.

\begin{remark}
We state the above theorem in the full generality. In the next sections we use its particular version with $\delta_1=0$.

\demo
\end{remark}

\section{Accelerated Framework for Saddle-Point Problems} \label{section:saddleFramework}


In this section we consider saddle-point problem of the following general form 
    \begin{equation} \label{eq:framework_funct}
    \min _{x \in \mathbb{R}^{d_x}} \max _{y \in \mathbb{R}^{d_y}} \left\{f(x)+ G(x, y) - h(y)\right\} 
\end{equation}
and develop a general accelerated optimization framework for its solution. In the following sections we use this general framework to develop accelerated methods for saddle-point problems in the form of problem \eqref{eq:framework_funct}, but with some additional assumptions about the structure of the functions $G$ and $h$. In particular, we consider problem \eqref{eq:problem} in Section~\ref{section:saddle} and problem \eqref{eq:problem_Gsum}  in Section~\ref{section:G_sum} 
As it was discussed before, our general framework consists of several inner-outer loops, which require to solve optimization problems with some special structure. 
Thus, the general framework in this section is developed under two additional assumptions on two problems with a special structure (see Assumptions~\ref{assumpt:framework_oracle}, \ref{assumpt:framework_oracle_x} below), which we need to solve in two loops of the framework. Then, in the following sections we show, how these assumptions can be satisfied, which allows to obtain the main results as a corollary of the main theorem of this section.
So, the plan of this section is, first, to introduce the main assumptions on the problem \eqref{eq:framework_funct} and two additional assumptions for the sake of generality of the framework. Second, we discuss the structure of the problem \eqref{eq:framework_funct} and slightly reformulate it in an equivalent way. Then, we describe the main part of the framework by giving details of each loop, and finish with the main complexity theorem.


\subsection{Preliminaries}


We start with the main assumptions, which are used to develop the general framework of this section. The first assumption is on the functions $f, G, h$ in problem \eqref{eq:framework_funct}.
\begin{assumption}\label{assumpt:framework} 
\begin{enumerate}
    \item Function $f$ is $L_f$-smooth, $\mu_x$-strongly convex and there exists a basic oracle $O_{f}$ for $f$ such that $\tau_f$ calls of this basic oracle produce the gradient $\nabla f(x)$.
    \item Function $G(x, y)$  is $L_G$-smooth, i.e. for each $x = (x_1, x_2), y = (y_1, y_2) \in  \mathbb{R}^{d_x} \times \mathbb{R}^{d_y}$
\begin{equation}
    \|\nabla G (x_1,x_2)-\nabla G (y_1,y_2)\|\leq L_G \|(x_1, x_2) - (y_1, y_2)\|,
\end{equation}
there exist a basic oracle $O^{x}_{G}$ for $G(\cdot,y)$ such that $\tau_G$ calls of this basic oracle produce the gradient $\nabla_x G(x,y)$. and a basic oracle $O^{y}_{G}$ for $G(x,\cdot )$ such that $\tau_G$ calls of this basic oracle produce the gradient $\nabla_y G(x,y)$.

\item Function $h$ is $L_h$-smooth, $\mu_y$-strongly convex and there exists a basic oracle $O_{h}$ for $h$ such that $\tau_h$ calls of this basic oracle produce the gradient $\nabla h(y)$.
\end{enumerate}
\end{assumption}
\begin{remark}
If the problem \eqref{eq:framework_funct} is not strongly-convex-strongly-concave, then one can apply standard reduction by regularization scheme, i.e. add a small strongly-convex-strongly-concave regularizer, solve the new strongly-convex-strongly-concave problem using the methods we develop and then prove that the obtained solution also approximates the solution of the initial convex-concave problem since the regularization was small. See the details in \cite{Alkousa2019}. 

\demo
\end{remark}

Our plan is to apply the general framework of this section to solve, in particular, problem \eqref{eq:problem}. This problem formulation is not symmetric w.r.t. the variables $x$ and $y$ since different assumptions are imposed on function $f$ and function $h$. Our preliminary derivations, which we do not report here, showed that better complexity bounds are obtained if we first change the order of maximization in $y$ and minimization in $x$, multiply the objective by minus one, and write the following problem which is equivalent to \eqref{eq:framework_funct}
\begin{equation} \label{eq:main30}
     \min _{y \in \mathbb{R}^{d_y}} \left\{ h(y) + \max _{x \in \mathbb{R}^{d_x}} \left\{- G(x, y)-f(x)\right\}  \right\}. 
\end{equation} 
This reformulation allows to solve problem \eqref{eq:main30} by an algorithm which consists of a series of inner-outer loops, where in each loop Algorithm \ref{alg:restarts_inexact_notconvex} is applied to solve some auxiliary problem which has the form \eqref{eq:PrStGen1}. The above equivalent reformulation of \eqref{eq:framework_funct} naturally leads to the following definition of approximate optimality.
\begin{definition}\label{def:saddle_solution}
Let $\varepsilon >0$ and $\sigma \in (0,1)$. By an $\varepsilon$-solution to problem \eqref{eq:framework_funct} we mean a point $\hat{y}$ such that
\begin{equation}\label{eq:saddle_epsilon_solution}
h(\hat{y}) + \max_{x \in \R^{d_x}}\{-G(x,\hat{y})-f(x)\} - \min_{y\in \R^{d_y}}\max_{x \in \R^{d_x}}\{h(y) -G(x,y)-f(x)\} \leq \varepsilon
\end{equation}
 We say that a random point $\hat{y}$ is an $(\varepsilon,\sigma)$-solution to the problem \eqref{eq:framework_funct} if $\mathbb{P}\{ \eqref{eq:saddle_epsilon_solution} \text{ holds True} \} \geq 1-\sigma$.
\end{definition}

Definition~\ref{def:saddle_solution} specifies only the $y$-part of an approximate solution to \eqref{eq:framework_funct} and is motivated by considering reformulation \eqref{eq:main30} as a minimization problem.  The next Lemma~\ref{lem:saddle_definitions} shows how to obtain an approximate solution to \eqref{eq:framework_funct} in more common form with both $x$- and $y$-part when a solution in the sense of Definition~\ref{def:saddle_solution} is available.
\begin{lemma} 
\label{lem:saddle_definitions}
Let us consider problem \eqref{eq:framework_funct} under Assumption \ref{assumpt:framework}.
Let a  pair $(\hat{x}, \hat{y})$ satisfy
\begin{enumerate}
    \item $\hat{y}$ is an $(\varepsilon_y, \sigma_y)$-solution to the problem \eqref{eq:framework_funct}, i.e. \eqref{eq:saddle_epsilon_solution} holds.
    \item $\hat{x}$ is an $(\varepsilon_x, \sigma_x)$-solution to problem $\max_{x\in \R^{d_x}} \{-G(x, \hat{y}) - f(x)\}$;
\end{enumerate}
Then the following inequalities hold with probability $1 - \sigma_y - \sigma_x$  
\begin{gather}
     \|\hat{y} - y_{\star}\|^2 \leq \frac{2\varepsilon_y}{\mu_y},\\
     \|\hat{x} - x_{\ast}\|^2 \leq 8\left(\frac{L_G}{\mu_x}\right)^2\|\hat{y} - y_{\ast}\|^2 + \frac{4\varepsilon_x}{\mu_x},\\
     \max_{x \in \R^{d_x}}\min_{y \in \R^{d_y}} \{h(y) -G(x, y) - f(x) \} -  \min_{y \in \R^{d_y}} \{h(y)-G(\hat{x}, y) - f(\hat{x}) \}  \nonumber\\ 
     \leq2\left(L_f+L_G+\frac{2L_G^2}{\mu_y} \right)\left(\frac{\varepsilon_x}{\mu_x}+\left(\frac{L_G}{\mu_x}\right)^2\frac{4\varepsilon_y}{\mu_y}\right),
\end{gather}
where $(x_{\ast}, y_{\ast})$ is the saddle point 
for problem \eqref{eq:framework_funct}.
\end{lemma}
\begin{proof}
We let $\Phi(y) = \max_{x\in \R^{d_x}} \{h(y) -G(x, y) - f(x) \}$ and note that $\Phi(y)$ is $\mu_y$-strongly convex. Under Assumption~\ref{assumpt:framework} the function $h(y) -G(x, y) - f(x)$ has unique saddle point $(x_{\ast}, y_{\ast})$. Then, with probability $1-\sigma_y$ we have
\begin{align}
    \|\hat{y} - y_{\ast}\|^2 \leq \frac{2}{\mu_y}\left( \max_{x \in \R^{d_x}} \{h(\hat{y}) -G(x, \hat{y}) - f(x) \} - \min_{y \in \R^{d_y}} \max_{x \in \R^{d_x}} \{h(y) -G(x, y) - f(x) \}  \right) \leq \frac{2\varepsilon_y}{\mu_y}.
\end{align}
We denote $x_{\ast}(\hat{y}) = \argmax_{x\in \R^{d_x}}\{h(\hat{y}) -G(x, \hat{y}) - f(x)\}$, then according to  Lemma~\ref{lemma:obt_delta_oracle} $x_{\ast}(y)$ is $ 2L_G/\mu_x$ Lipschitz continuous. Since $\{h(\hat{y}) -G(x, \hat{y}) - f(x) \}$ is $\mu_x$-strongly concave, we obtain  that the inequality
\begin{align}
    \|\hat{x} - x_{\ast}\|^2 \leq 2\|\hat{x} - x_{\ast}(\hat{y})\|^2 + 2\|x_{\ast}(\hat{y}) - x_{\ast}(y_{\ast})\|^2 \leq \frac{4\varepsilon_x}{\mu_x} + 8\left(\frac{L_G}{\mu_x}\right)^2\|\hat{y} - y_{\ast}\|^2 
\end{align}
holds true with probability $1 - \sigma_x - \sigma_y$. By consecutive application of Lemma~\ref{lem:sum_oracle} and Lemma~\ref{lemma:obt_delta_oracle}  we can obtain that $\Psi(x) = \min_{y \in \R^{d_y}} \{h(y) - G(x, y) -f(x)\}$ is concave and $L_f+L_G+\frac{2L_G^2}{\mu_y}$-smooth. Whence,
\begin{align}
    & \max_{x \in \R^{d_x}}\min_{y \in \R^{d_y}} \{h(y) -G(x, y) - f(x) \} -  \min_{y \in \R^{d_y}} \{h(y)-G(\hat{x}, y) - f(\hat{x}) \} = \Psi(x_{\ast}) - \Psi(\hat{x}) \\\nonumber
    &\leq  \frac{L_f+L_G+\frac{2L_G^2}{\mu_y}}{2} \|\hat{x} - x_{\ast}\|^2 \leq 2\frac{L_f+L_G+\frac{2L_G^2}{\mu_y}}{\mu_x}\varepsilon_x + 8 \left(\frac{L_G}{\mu_x}\right)^2 \frac{L_f+L_G+\frac{2L_G^2}{\mu_y}}{\mu_y}\varepsilon_y,
\end{align}
with probability $1 - \sigma_x - \sigma_y$.
In the first inequality we used that $x_*$ is the optimal point, and, hence, $\nabla \Psi(x_*)=0$. 
\qed
\end{proof}

By exchanging the variables $x$ and $y$ we can obtain the useful Corollary~\ref{corollary:change_maxmin} from Lemma~\ref{lem:saddle_definitions}, which we use below in one of the loops of our general scheme. 
\begin{corollary} 
\label{corollary:change_maxmin}
Let us consider the problem 
\begin{align} \label{problem_inverse_corollary}
    \min _{x \in \R^{d_x}}\left\{ f (x) +\max _{y \in  \R^{d_y}}\left\{F(x, y)-w(y) \right\}\right\},
\end{align}
where functions $f, F, w$ are smooth with Lipschitz constants of the gradient being $L_f, L_F, L_w$ respectively and functions $f, w$ are $\mu_x, \mu_y$-strongly convex respectively.
Let a  pair $(\hat{x}, \hat{y})$ satisfy
\begin{enumerate}
    \item $\hat{x}$ is an $(\varepsilon_x, \sigma_x)$-solution to the problem \eqref{problem_inverse_corollary}, i.e. inequality
    \begin{align*}
        f(\hat{x}) + \max_{y \in \R^{d_y}}\{F(\hat{x},y)-w(y)\} - \min_{x\in \R^{d_x}}\max_{y \in \R^{d_y}}\{f(x) +F(x,y)-w(y)\} \leq \varepsilon_x
    \end{align*}
    holds with probability $1-\sigma_x$.
    \item $\hat{y}$ is an $(\varepsilon_y, \sigma_y)$-solution to problem $\max_{y\in \R^{d_y}} \{F(\hat{x}, y) - w(y)\}$.
\end{enumerate}
Then the following inequalities hold with probability $1 - \sigma_y - \sigma_x$  
\begin{gather}
     \|\hat{x} - x_{\star}\|^2 \leq \frac{2\varepsilon_x}{\mu_x},\\
     \|\hat{y} - y_{\ast}\|^2 \leq 8\left(\frac{L_F}{\mu_y}\right)^2\|\hat{x} - x_{\ast}\|^2 + \frac{4\varepsilon_y}{\mu_y},\\
    w(\hat{y}) + \max_{x \in \R^{d_x}} \{ - F(x, \hat{y}) - f(x)\} - \min_{y \in \R^{d_y}}\max_{x \in \R^{d_x}} \{w(y) -F(x, y) - f(x) \} \label{eq:cor_1_1}
       \\ 
      = \max_{y \in \R^{d_y}}\min_{x \in \R^{d_x}} \{f(x) +F(x, y) - w(y) \} -  \min_{x \in \R^{d_x}} \{f(x)+F(x, \hat{y}) - w(\hat{y}) \} 
     \\   \leq 
     2\left(L_w+L_F+\frac{2L_F^2}{\mu_x}\right)\left(\frac{\varepsilon_y}{\mu_y} + 4 \left(\frac{L_F}{\mu_y}\right)^2 \frac{\varepsilon_x}{\mu_x} \right), \label{eq:cor_1_2}
\end{gather}
where $(x_{\ast}, y_{\ast})$ is the saddle point for problem \eqref{problem_inverse_corollary}.
\end{corollary}

The next two assumptions are made for the sake of obtaining a general framework. In this section we assume that two functions which are defined via auxiliary maximization problems and which appear in the loops of our general framework, can be equipped with an inexact oracle. In the following sections in different settings we show how to satisfy these two assumptions and apply the general framework.
\begin{assumption} \label{assumpt:framework_oracle} 
Let $\varepsilon >0$ and $\sigma \in (0,1)$, and a function $g$ be defined as
\begin{align}\label{eq:framework_g_obt_oracle}
g(x) = \max _{y \in  \R^{d_y}}\left\{G(x, y)-h(y) -\frac{H}{2}\|y-y_0\|^2\right\},
\end{align}
 where  $G(x,y)$, $h(y)$ 
 satisfy Assumption~\ref{assumpt:framework}, $H>0$, and  $y_0$ is some fixed point in  $\R^{d_y}$.\\
 Then, we assume that, for 
 any $\delta\left(\varepsilon\right)=\bf{poly}\left(\varepsilon\right)$ and any $\sigma_0\left(\varepsilon, \sigma \right)=\bf{poly}\left(\varepsilon, \sigma \right)$, it is possible to evaluate a $\left(\delta(\varepsilon)/2,\sigma_{0}\left(\varepsilon, \sigma \right)\right)$-solution to this problem and $\left(\delta\left(\varepsilon\right),\sigma_{0}\left(\varepsilon, \sigma \right),2L_G + 4\frac{L_G^2}{\mu_y+H}\right)$-oracle for the function $g$ in the sense of Definition \ref{def:d_L_mu_oracle_prob} with $\delta_1=0$. Moreover, we assume that this solution can be evaluated using 
 $\mathcal{N}_G^y\left( \tau_G, H\right) \mathcal{K}_G^y\left(\varepsilon, \sigma\right)$ calls of the basic oracle $O_G^y$ of $G(x,\cdot)$, $\mathcal{N}_{h}\left( \tau_h, H\right) \mathcal{K}_h\left(\varepsilon, \sigma\right)$ calls of the basic oracle $O_h$ of $h$ and this inexact oracle can be evaluated using 
 $\mathcal{N}_G^y\left( \tau_G, H\right) \mathcal{K}_G^y\left(\varepsilon, \sigma\right)$ calls of the basic oracle $O_G^y$ of $G(x,\cdot)$, $\tau_G$ calls of the basic oracle $O_G^x$ of $G(\cdot ,y)$ and $\mathcal{N}_{h}\left( \tau_h, H\right) \mathcal{K}_h\left(\varepsilon, \sigma\right)$ calls of the basic oracle $O_h$ of $h$, where $\mathcal{K}_G^y\left(\varepsilon, \sigma\right)=\widetilde{O}(1)$ and $\mathcal{K}_h\left(\varepsilon, \sigma\right)=\widetilde{O}(1)$.
 
\end{assumption}

\begin{assumption}\label{assumpt:framework_oracle_x} 
Let $\varepsilon >0$ and $\sigma \in (0,1)$, and a function $r$ be defined as
\begin{align}
\label{eq:framework_r_obt_oracle}
r(y)=  \min_{x \in  \R^{d_x}}\left\{G(x, y)+ f(x) \right\},
\end{align}
 where  $G(x,y),f(x)$ satisfy Assumption~\ref{assumpt:framework}.\\
Then, we assume that, for any $\delta\left(\varepsilon\right)=\bf{poly}\left(\varepsilon\right)$ and any $\sigma_0\left(\varepsilon, \sigma \right)=\bf{poly}\left(\varepsilon, \sigma \right)$,
 it is possible to evaluate a $\left(\delta(\varepsilon)/2,\sigma_{0}\left(\varepsilon, \sigma \right)\right)$-solution to this problem and $\left(\delta\left(\varepsilon\right),\sigma_0\left(\varepsilon, \sigma \right),2L_G + 4\frac{L_G^2}{\mu_x}\right)$-oracle for the function $r$ in the sense of Definition \ref{def:d_L_mu_oracle_prob} with $\delta_1=0$.
 Moreover, we assume that this solution can be evaluated using  $\mathcal{N}_G^x\left( \tau_G\right)\mathcal{K}_G^x\left(\varepsilon, \sigma\right)$ calls of the basic oracle $O_G^x$ for $G(\cdot,y)$, $\mathcal{N}_{f}\left( \tau_f\right)\mathcal{K}_f\left(\varepsilon, \sigma\right)$ calls of the basic oracle $O_f$ for $f$ and this inexact oracle can be evaluated using $\tau_G$ calls of the basic oracle $O_G^y$ of $G(x, \cdot )$, $\mathcal{N}_G^x\left( \tau_G\right)\mathcal{K}_G^x\left(\varepsilon, \sigma\right)$ calls of the basic oracle $O_G^x$ for $G(\cdot,y)$ and $\mathcal{N}_{f}\left( \tau_f\right)\mathcal{K}_f\left(\varepsilon, \sigma\right)$ calls of the basic oracle $O_f$ for $f$, where $\mathcal{K}_G^x\left(\varepsilon, \sigma\right)=\widetilde{O}(1)$ and $\mathcal{K}_f\left(\varepsilon, \sigma\right)=\widetilde{O}(1)$.
\end{assumption}

We use the above two assumptions to develop in this section a general algorithmic framework for problem \eqref{eq:framework_funct}. In the next sections we consider more specific problem formulations \eqref{eq:problem} and \eqref{eq:problem_Gsum} and show, how an application of some particular algorithms for solving maximization problems \eqref{eq:framework_g_obt_oracle} and \eqref{eq:framework_r_obt_oracle} allows us to ensure that Assumptions  \ref{assumpt:framework_oracle} and \ref{assumpt:framework_oracle_x} hold. For now, let us shortly illustrate how this can be achieved by a simple example. Assume, for simplicity, that in \eqref{eq:framework_g_obt_oracle} $h=0$ and the full gradients $\nabla_x G(x,y)$, $\nabla_y G(x,y)$ are available meaning that in Assumption \ref{assumpt:framework} $\tau_G=1$.
Then, the objective in the maximization problem \eqref{eq:framework_g_obt_oracle} has $L_G$-smooth in $y$ part $G(x,y)$ and $H$-strongly concave part $-\frac{H}{2}\|y-y_0\|^2$. Thus, if we apply accelerated gradient method for composite optimization \cite{nesterov2013gradient}, we obtain that a $\delta(\varepsilon)/2$-solution $\tilde{y}_{\delta(\varepsilon)/2}(x)$ to this problem can be obtained in $O\left(\sqrt{\frac{L_G}{H}}\ln\frac{1}{\delta(\varepsilon)}\right)$ iterations of the accelerated method. Each iteration requires to evaluate $\nabla_y G(x,y)$, which means that the number of calls of the basic oracle $O_G^y$ for $G(x,\cdot)$  is $O\left(\tau_G\sqrt{\frac{L_G}{H}}\ln\frac{1}{\delta(\varepsilon)}\right)$. Since $\delta\left(\varepsilon\right)=\bf{poly}\left(\varepsilon\right)$, we obtain that the the number of $O_G^y$ calls is $\mathcal{N}_G^y\left( \tau_G, H\right) \mathcal{K}_G^y\left(\varepsilon, \sigma\right) =O\left(\tau_G\sqrt{\frac{L_G}{H}}\ln\frac{1}{\varepsilon}\right)$, i.e. $\mathcal{K}_G^y\left(\varepsilon, \sigma\right)=\widetilde{O}(1)$. Moreover, by Lemma \ref{lemma:obt_delta_oracle}, $ \nabla_{x} G\left(x, \tilde{y}_{\delta(\varepsilon)/2}(x)\right)$ is $(\delta(\varepsilon), 2 L_G+\frac{2L_G^2}{H})$-oracle for the function $g$, which means that we need also $\tau_G$ calls of the basic oracle $O_G^x$ for $G(\cdot,y)$. Thus, Assumption \ref{assumpt:framework_oracle} holds.



\subsection{General framework for saddle-point problems.}
Next, we describe in detail the resulting structure of our framework which consists of three inner-outer loops. We also summarize the steps of the algorithm in Table~\ref{tabl:saddleproblem_steps}. 
In each loop we apply Algorithm \ref{alg:restarts_inexact_notconvex} with different value of parameter $H$, which defines its complexity. In the next subsection we carefully choose the value of this parameter in each level of the loops. Later, in the next sections this allows us to obtain the desired results on near-optimal complexity bounds for saddle-point problems \eqref{eq:problem} or \eqref{eq:problem_Gsum}.
Further, in each loop we have a target accuracy $\varepsilon$ and a confidence level $\sigma$ which define the required quality of the solution to an optimization problem in this loop. These quantities define the inexactness of the oracle in this loop via inequalities \eqref{delta_poly_CATD_with_prob} and \eqref{sigma_0_poly_CATD_with_prob} and the target accuracy and confidence level for the optimization problem in the next loop via \eqref{varepsilon_poly_CATD_with_prob}, \eqref{sigma_poly_CATD_with_prob}. Due to inexact strong convexity provided by $(\delta,\sigma,L,\mu)$-oracle, Algorithm \ref{alg:restarts_inexact_notconvex} has logarithmic dependence of the complexity on the target accuracy and confidence level (see Theorem \ref{AM:comfortable_view_with_prob}). Since the dependencies on the target accuracy and confidence level in \eqref{delta_poly_CATD_with_prob},  \eqref{sigma_0_poly_CATD_with_prob}, \eqref{varepsilon_poly_CATD_with_prob} and \eqref{sigma_poly_CATD_with_prob} are polynomial, we obtain that the dependency of the complexity in each loop on the target accuracy and confidence level in the first loop, i.e. target accuracy and confidence level for the solution to problem \eqref{eq:framework_funct}, is logarithmic. We hide such logarithmic factors in $\widetilde{O}$ notation.

 \paragraph{\textbf{Loop 1}}\label{subsec_first} $\;$\\
The goal of Loop 1 is to find an $(\varepsilon,\sigma)$-solution of  problem \eqref{eq:main30}, which is considered as a minimization problem in $y$ with the objective given in the form of auxiliary maximization problem in $x$. 
Finding an $(\varepsilon,\sigma)$-solution of this minimization problem gives an approximate solution to the saddle-point problem  \eqref{eq:framework_funct} which is understood in the sense of Definition \ref{def:saddle_solution}.

To solve problem \eqref{eq:main30}, we would like to apply Algorithm \ref{alg:restarts_inexact_notconvex} with 
\begin{align}\label{eq:framework_aux_step1}
    \varphi = 0, \ \ \psi= h(y) + \max_{x \in  \R^{d_x}}\left\{-G(x, y)-f(x)\right\}.
\end{align}
The function $\varphi$ is, clearly, convex and is known exactly.
What makes solving problem \eqref{eq:main30} not straightforward is that the exact value of $\psi$ is not available. At the same time we can construct an inexact oracle for this function.
First, the function $h$ is $\mu_y$-strongly convex, $L_h$-smooth and its exact gradient is available. Second, thanks to Assumption \ref{assumpt:framework_oracle_x}, it is possible to construct a $\left(\delta^{(1)}\left(\varepsilon\right),\sigma^{(1)}_0\left(\varepsilon,\sigma\right),2L_G + 4\frac{L_G^2}{\mu_x}\right)$-oracle for the function $r(y)= \max_{x \in  \R^{d_x}}\left\{-f(x)-G(x, y)\right\}$ for any $\delta^{(1)}\left(\varepsilon\right)=\bf{poly}\left(\varepsilon \right)$ and $\sigma^{(1)}_0\left(\varepsilon,\sigma\right)=\bf{poly}\left(\varepsilon,\sigma \right)$.
Combining these two parts and using Lemma~\ref{lem:sum_oracle}, we obtain that we can construct a $\left(\delta^{(1)}\left(\varepsilon\right),\sigma^{(1)}_0\left(\varepsilon,\sigma\right), L_h + 2L_G + 4\frac{L_G^2}{\mu_x}, \mu_y\right)$-oracle for $\psi$.
Thus, we can apply Algorithm  \ref{alg:restarts_inexact_notconvex} with parameter $H=H_1$, which will be chosen later, to solve problem \eqref{eq:main30}.
Moreover, since Assumption \ref{assumpt:framework_oracle_x} requires  $\delta^{(1)}\left(\varepsilon\right)=\bf{poly}\left(\varepsilon\right)$ and  $\sigma^{(1)}_0\left(\varepsilon,\sigma\right)=\bf{poly}\left(\varepsilon, \sigma \right)$,
which holds for the dependencies in \eqref{delta_poly_CATD_with_prob} and \eqref{sigma_0_poly_CATD_with_prob}, we can choose $\delta^{(1)}\left(\varepsilon\right)$ and $\sigma^{(1)}_0\left(\varepsilon,\sigma\right)$ such that \eqref{delta_poly_CATD_with_prob} and \eqref{sigma_0_poly_CATD_with_prob} hold. So, the first main assumption of Theorem~\ref{AM:comfortable_view_with_prob} holds. 
At the same time, according to Assumptions \ref{assumpt:framework} and \ref{assumpt:framework_oracle_x}, constructing inexact oracle for $\psi$ requires $\tau_h$ calls of the basic oracle for $h$,
$\tau_G$ calls of the basic oracle of $G(x, \cdot )$, $\mathcal{N}_G^x\left( \tau_G\right) \mathcal{K}_G^x\left(\varepsilon, \sigma\right)$ calls of the basic oracle for $G(\cdot,y)$, $\mathcal{N}_{f}\left( \tau_f\right)\mathcal{K}_f\left(\varepsilon, \sigma\right)$ calls of the basic oracle for $f$.

Let us discuss the second main assumption of Theorem~\ref{AM:comfortable_view_with_prob}. 
To ensure that this assumption holds, we need in each iteration of Algorithm \ref{alg:highorder_inexact}, used as a building block in Algorithm \ref{alg:restarts_inexact_notconvex}, to find an $\left(\tilde{\varepsilon}^{(1)}_f\left(\varepsilon\right),\tilde{\sigma}^{(1)}\left(\varepsilon,\sigma\right)\right)$-solution to the auxiliary problem \eqref{prox_step_inexact}, where $\tilde{\sigma}^{(1)}\left(\varepsilon,\sigma\right), \tilde{\varepsilon}^{(1)}_f\left(\varepsilon\right)$ satisfy 
 inequalities \eqref{varepsilon_poly_CATD_with_prob}, \eqref{sigma_poly_CATD_with_prob}.
 For the particular definitions of $\vp$, $\psi$  \eqref{eq:framework_aux_step1} in this Loop, this problem has the following form:  
\begin{equation} \label{eq:sub1}
 y_{k+1}^t = \arg\min _{y \in \R^{d_y}}\left\{ h(y) +\max_{x \in  \R^{d_x}}\left\{-G(x, y)-f(x)\right\}+\frac{H_1}{2} \|y - y_k^{md}\|^2\right\}.
\end{equation}
Below, in the next \hyperref[subsec_second]{paragraph "Loop 2"},  we explain how to solve this auxiliary problem to obtain its $\left(\tilde{\varepsilon}^{(1)}_f\left(\varepsilon\right),\tilde{\sigma}^{(1)}\left(\varepsilon,\sigma\right)\right)$-solution.
To summarize Loop 1, both main assumptions of Theorem~\ref{AM:comfortable_view_with_prob} hold and we can use it to guarantee that we obtain an $(\varepsilon,\sigma)$-solution of problem \eqref{eq:main30}. This requires $\widetilde{O}\left(1+\left( \frac{H_1 }{\mu_{\vp} + \mu_{\psi}} \right)^{\frac{1}{2}}\right)=\widetilde{O}\left(1+\left( \frac{H_1 }{\mu_{y}} \right)^{\frac{1}{2}}\right)$ 
calls to the inexact oracles for $\vp$ and for $\psi$, and the same number of times solving the auxiliary problem \eqref{eq:sub1}.
Combining this oracle complexity with the cost of calculating inexact oracles for $\vp$ and for $\psi$, we obtain that solving problem \eqref{eq:main30} requires $\widetilde{O}\left(1+\left( \frac{H_1 }{\mu_{y}} \right)^{\frac{1}{2}}\right)\tau_h$ calls of the basic oracle for $h$, $\widetilde{O}\left(1+\left( \frac{H_1 }{\mu_{y}} \right)^{\frac{1}{2}}\right)\tau_G$ calls of the basic oracle of $G(x, \cdot )$,
$\widetilde{O}\left(1+\left( \frac{H_1 }{\mu_{y}} \right)^{\frac{1}{2}}\right)\mathcal{N}_G^x\left( \tau_G\right)\mathcal{K}_G^x\left(\varepsilon, \sigma\right)$ calls of the basic oracle for $G(\cdot,y)$, $\widetilde{O}\left(1+\left( \frac{H_1 }{\mu_{y}} \right)^{\frac{1}{2}}\right)\mathcal{N}_{f}\left( \tau_f\right)\mathcal{K}_f\left(\varepsilon, \sigma\right)$ calls of the basic oracle for $f$. 
The only remaining thing is to provide an inexact solution to problem \eqref{eq:sub1} and, next, we move to the Loop 2 to explain how to guarantee this. Note that we need to solve problem \eqref{eq:sub1} $\widetilde{O}\left(1+\left( \frac{H_1 }{\mu_{y}} \right)^{\frac{1}{2}}\right)$ times.

\paragraph{\textbf{Loop 2}}\label{subsec_second}$\;$\\
As mentioned in the previous Loop 1,
in each iteration of Algorithm \ref{alg:restarts_inexact_notconvex} in Loop 1 
we need many times to find  an $(\varepsilon'_2,\sigma'_2)$-solution of the auxiliary problem \eqref{eq:sub1}, where we denoted for simplicity $\sigma'_2=\tilde{\sigma}^{(1)}\left(\varepsilon,\sigma\right)$ and $\varepsilon'_2=\tilde{\varepsilon}^{(1)}_f\left(\varepsilon\right)$. To do this, we reformulate problem \eqref{eq:sub1} by changing the order of minimization and maximization as follows:
\begin{align}
\min _{y \in \R^{d_y}}\left\{ h(y) +\frac{H_1}{2} \|y - y_k^{md}\|^2 +\max_{x \in  \R^{d_x}}\left\{-G(x, y)-f(x)\right\}\right\} \\
=\min _{y \in \R^{d_y}}\max_{x \in  \R^{d_x}}\left\{ h(y)  -G(x, y)-f(x) +\frac{H_1}{2} \|y - y_k^{md}\|^2\right\} \\
 = \max_{x \in  \R^{d_x}} \min _{y \in \R^{d_y}}\left\{ h(y)  -G(x, y)-f(x) +\frac{H_1}{2} \|y - y_k^{md}\|^2\right\}\\
= - \min _{x \in \R^{d_x}}\left\{ f (x) +\max _{y \in  \R^{d_y}}\left\{G(x, y)-h(y) - \frac{H_1}{2} \|y - y_k^{md}\|^2\right\}\right\} \label{eq:problem_step2}   
\end{align}
and obtain an $(\varepsilon'_2,\sigma'_2)$-solution of the problem \eqref{eq:sub1} by solving minimization problem \eqref{eq:problem_step2}. Assume that we can find an $(\varepsilon_2,\sigma_2)$-solution $\hat{x}$ of the minimization problem \eqref{eq:problem_step2} in the sense of Definition \ref{def:saddle_solution}. Then, according to Assumption \ref{assumpt:framework_oracle}, we can also obtain a point $\hat{y}$ which is $(\bar{\delta}(\varepsilon_2)/2,\bar{\sigma}_0(\sigma_2))$-solution to the problem 
\begin{align}\label{eq:loop2_probl_for_changing_delta_L}
    \max _{y \in  \R^{d_y}}\left\{G(x, y)-h(y) - \frac{H_1}{2} \|y - y_k^{md}\|^2\right\},
\end{align}
where $\bar{\delta}(\varepsilon_2),\bar{\sigma}_0(\sigma_2)$ satisfy the following polynomial dependencies
\begin{align}
& \bar{\delta}(\varepsilon_2) \leq \frac{H_1+\mu_y}{4\mu_x\left(\frac{H_1+\mu_y}{4L_G}\right)^2} \varepsilon_2, \;\; \bar{\sigma}_0(\sigma_2) \leq \sigma_2.
\label{eq:loop2_1}
\end{align}
If we choose $\varepsilon_2,\sigma_2,\bar{\delta}(\varepsilon_2),\bar{\sigma}_0(\sigma_2)$ satisfying
\begin{align}
    & \varepsilon_2\leq  \left(\frac{H_1+\mu_y}{4L_G}\right)^2 \frac{\mu_x}{L_h+H_1+L_G+\frac{2L_G^2}{\mu_x}} \varepsilon'_2  \label{eq:step_2_choos_eps_fin},\\
    &\sigma_2\leq \frac{\sigma'_2}{2},\;\;\;  \label{eq:step_2_choos_sigma_fin}\\
    &\bar{\sigma}_0(\sigma_2)  \stackrel{\eqref{eq:loop2_1}}{\leq} \sigma_2 \leq \frac{\sigma'_2}{2}, \;\;\; \bar{\delta}(\varepsilon_2) \leq \frac{H_1+\mu_y}{4\mu_x\left(\frac{H_1+\mu_y}{4L_G}\right)^2} \varepsilon_2\stackrel{\eqref{eq:loop2_1}}{\leq} \frac{H_1+\mu_y}{4L_h+4H_1+4L_G+\frac{8L_G^2}{\mu_x}}\varepsilon'_2 \label{eq:step_2_choos_delta_fin},
\end{align}
then
\begin{align}
    &2\frac{L_h+H_1+L_G+\frac{2L_G^2}{\mu_x}}{H_1+\mu_y}\bar{\delta}(\varepsilon_2) + 8 \left(\frac{L_G}{H_1+\mu_y}\right)^2 \frac{L_h+H_1+L_G+\frac{2L_G^2}{\mu_x}}{\mu_x}\varepsilon_2\leq \varepsilon'_2, \label{eq:step_2_choos_eps}\\
    &\sigma_2+\bar{\sigma}_0(\sigma_2) \leq \sigma'_2. \label{eq:step_2_choos_sigma}
\end{align}
Thus, applying Corollary~\ref{corollary:change_maxmin} to minimization problem \eqref{eq:problem_step2} with $F(x,y)=G(x,y)$, $w(y)=h(y)+\frac{H_1}{2} \|y - y_k^{md}\|^2$, $\varepsilon_x=\varepsilon_2$, $\sigma_x=\sigma_2$, $\varepsilon_y=\bar{\delta}(\varepsilon_2)$, $\sigma_y=\bar{\sigma}_0(\sigma_2)$ we obtain (see \eqref{eq:cor_1_1}, \eqref{eq:cor_1_2}) that $\hat{y}$ satisfies inequality 
\[
h(\hat{y}) +\frac{H_1}{2} \|\hat{y} - y_k^{md}\|^2
 + \max_{x \in \R^{d_x}} \{ - G(x, \hat{y}) - f(x)\} - \min_{y \in \R^{d_y}}\max_{x \in \R^{d_x}} \{h(y) +\frac{H_1}{2} \|y - y_k^{md}\|^2 -G(x, y) - f(x) \} \leq \varepsilon'_2
\]
with probability $\sigma'_2$. Thus, by Definition \ref{def:saddle_solution} it is an $(\varepsilon'_2,\sigma'_2)$-solution of the problem \eqref{eq:sub1}.
By Assumption \ref{assumpt:framework_oracle}, calculation of $\hat{y}$ requires  $\mathcal{N}_G^y\left( \tau_G, H\right) \mathcal{K}_G^y\left(\varepsilon_2, \sigma_2\right)$ calls of the basic oracle $O_G^y$ of $G(x,\cdot)$, $\tau_G$ calls of the basic oracle $O_G^x$ of $G(\cdot ,y)$ and $\mathcal{N}_{h}\left( \tau_h, H\right) \mathcal{K}_h\left(\varepsilon_2, \sigma_2\right)$ calls of the basic oracle $O_h$ of $h$.
Our next step is to provide an $(\varepsilon_2,\sigma_2)$-solution to minimization problem \eqref{eq:problem_step2}, for which we again apply Algorithm  \ref{alg:restarts_inexact_notconvex}, but this time with 
\begin{align}\label{eq:framework_aux_step2}
    \varphi =\max _{y \in  \R^{d_y}}\left\{G(x, y)-h(y) - \frac{H_1}{2} \|y - y_k^{md}\|^2\right\} , \ \ \psi = f(x).
\end{align}
The function $\psi$  is $\mu_x$-strongly convex, $L_f$-smooth and its exact gradient is available.
What makes solving problem \eqref{eq:problem_step2} not straightforward is that the exact value of $\varphi$ is not available. At the same time we can construct an inexact oracle for this function.
Thanks to Assumption \ref{assumpt:framework_oracle}, it is possible to construct a $\left(\delta^{(2)}\left(\varepsilon_2\right), \sigma^{(2)}_0\left(\varepsilon_2,\sigma_2\right), 2L_G + 4\frac{L_G^2}{H_1+\mu_y}\right)$-oracle for the function 
$\varphi$ for any $\delta^{(2)}\left(\varepsilon_2\right)=\textbf{poly}\left(\varepsilon_2 \right)$ and $\sigma^{(2)}_0\left(\varepsilon_2,\sigma_2\right)=\textbf{poly}\left(\varepsilon_2,\sigma_2 \right)$.
Using Lemma \ref{lem:sum_oracle}, we obtain that we can construct \\ a  $\left(\delta^{(2)}\left(\varepsilon_2\right),\sigma^{(2)}_0\left(\varepsilon_2,\sigma_2\right), L_f + 2L_G + 4\frac{L_G^2}{H_1+\mu_y}, \mu_x\right)$-oracle for the function $\vp+\psi$.
Thus, we can apply Algorithm  \ref{alg:restarts_inexact_notconvex} with parameter $H=H_2 \geq 2L_G+4\frac{L_G^2}{H_1+\mu_y}$, which will be chosen later, to solve the problem \eqref{eq:problem_step2}.
Moreover, since Assumption \ref{assumpt:framework_oracle} requires  $\delta^{(2)}\left(\varepsilon_2\right)=\text{\bf{poly}}\left(\varepsilon_2 \right)$ and  $\sigma^{(2)}_0\left(\varepsilon_2,\sigma_2\right)=\text{\bf{poly}}\left(\varepsilon_2,\sigma_2 \right)$, which holds for the dependencies in \eqref{delta_poly_CATD_with_prob} and \eqref{sigma_0_poly_CATD_with_prob}, we can choose $\delta^{(2)}\left(\varepsilon_2\right)$ and $\sigma^{(2)}_0\left(\varepsilon_2,\sigma_2\right)$ such that \eqref{delta_poly_CATD_with_prob} and \eqref{sigma_0_poly_CATD_with_prob} hold. So, the first main assumption of Theorem~\ref{AM:comfortable_view_with_prob} holds. 
At the same time, according to Assumptions \ref{assumpt:framework} and \ref{assumpt:framework_oracle}, constructing inexact oracle for $\vp$ requires $\mathcal{N}_G^y\left( \tau_G, H_1\right) \mathcal{K}_G^y\left(\varepsilon_2, \sigma_2\right)$ calls of the basic oracle for $G(x,\cdot)$, $ \tau_G$ calls of the basic oracle for $G(\cdot, y)$, $\mathcal{N}_{h}\left( \tau_h, H_1\right)\mathcal{K}_h\left(\varepsilon_2, \sigma_2\right)$ calls of the basic oracle for $h$, and constructing exact oracle for $\psi=f$ requires $\tau_f$ calls of the basic oracle for $f$.

Let us discuss the second main assumption of Theorem~\ref{AM:comfortable_view_with_prob}. 
To ensure that this assumption holds, we need in each iteration of Algorithm \ref{alg:highorder_inexact}, used as a building block in Algorithm \ref{alg:restarts_inexact_notconvex}, to find $\left(\tilde{\varepsilon}^{(2)}_f\left(\varepsilon_2\right),\tilde{\sigma}^{(2)}\left(\varepsilon_2,\sigma_2\right)\right)$-solution to the auxiliary problem \eqref{prox_step_inexact}, where $\tilde{\sigma}^{(2)}\left(\varepsilon_2,\sigma_2\right), \tilde{\varepsilon}^{(2)}_f\left(\varepsilon_2\right)$ satisfy 
 inequalities \eqref{varepsilon_poly_CATD_with_prob}, \eqref{sigma_poly_CATD_with_prob}.
 For the particular definitions of $\vp$, $\psi$  \eqref{eq:framework_aux_step2} in this Loop, this problem has the following form:  
\begin{align}
x_{l+1}^t &=\arg\min _{u \in  \R^{d_x}} \{ \la \nabla \varphi_{\delta^{(2)},2L_{\vp}}(x_l^{md}),x-x_l^{md}\ra+\psi(x) + \frac{H_2}{2}\|x-x_l^{md}\|_2^2 \}\\
&= \arg\min _{x \in \R^{d_x}}\left\{ \la \nabla  g_{\delta^{(2)},2L_{g}}(x_l^{md}),x-x_l^{md}\ra +f(x)+ \frac{H_2}{2}\|x - x_l^{md}\|^2\right\}\label{eq:sub2},
\end{align}
where $g(x)= \max _{y \in  \R^{d_y}}\left\{G(x, y)+h(y) - \frac{H_1}{2} \|y - y_k^{md}\|^2 \right\},L_g=L_G+2\frac{L_G^2}{H_1+\mu_y}$.
Below, in the next  \hyperref[subsec_third]{paragraph "Loop 3"},  we explain how to solve this auxiliary problem to obtain its \\ $\left(\tilde{\varepsilon}^{(2)}_f\left(\varepsilon_2\right),\tilde{\sigma}^{(2)}\left(\varepsilon_2,\sigma_2\right)\right)$-solution.

To summarize Loop 2, both main assumptions of Theorem~\ref{AM:comfortable_view_with_prob} hold and we can use it to guarantee that we obtain an $(\varepsilon'_2,\sigma'_2)$-solution of the auxiliary problem \eqref{eq:sub1}. This requires one time to solve the problem~\eqref{eq:loop2_probl_for_changing_delta_L}, which, by Assumption \ref{assumpt:framework_oracle} has the same cost as evaluating inexact oracle for the function $\varphi$.
Further, we need $O\left(\left( 1+\left( \frac{H_2 }{\mu_{\vp} + \mu_{\psi}} \right)^{\frac{1}{2}}\right) \log \varepsilon_2^{-1}\right)=O\left(\left(1+\left( \frac{H_2 }{\mu_{x}} \right)^{\frac{1}{2}} \right)\log \varepsilon_2^{-1}\right)$ calls to the inexact oracles for $\vp$ and for $\psi$, and the same number of times solving the auxiliary problem \eqref{eq:sub2}.
Combining this oracle complexity with the cost of calculating inexact oracles for $\vp$ and for $\psi$, we obtain that solving problem \eqref{eq:problem_step2} requires $O\left(\left(1+\left( \frac{H_2 }{\mu_{x}} \right)^{\frac{1}{2}} \right)\log \varepsilon_2^{-1}\right)\tau_f$ calls of the basic oracle for $f$, $O\left(\left(1+\left( \frac{H_2 }{\mu_{x}} \right)^{\frac{1}{2}} \right)\log \varepsilon_2^{-1}\right)\mathcal{N}_G^y\left( \tau_G, H_1\right) \mathcal{K}_G^y\left(\varepsilon_2, \sigma_2\right)$  calls of the basic oracle for $G(x,\cdot)$, $O\left(\left(1+\left( \frac{H_2 }{\mu_{x}} \right)^{\frac{1}{2}} \right)\log \varepsilon_2^{-1}\right) \tau_G$ calls of the basic oracle for $G(\cdot, y)$,\\
$O\left(\left(1+\left( \frac{H_2 }{\mu_{x}} \right)^{\frac{1}{2}} \right)\log \varepsilon_2^{-1}\right)\mathcal{N}_{h}\left( \tau_h, H_1\right)\mathcal{K}_h\left(\varepsilon_2, \sigma_2\right)$ calls of the basic oracle for $h$. The only remaining thing is to provide an inexact solution to problem \eqref{eq:sub2} and, next, we move to Loop 3 to explain how to guarantee this. Note that we need to solve problem \eqref{eq:sub2} $O\left(\left(1+\left( \frac{H_2 }{\mu_{x}} \right)^{\frac{1}{2}} \right)\log \varepsilon_2^{-1}\right)$ times.

\paragraph{\textbf{Loop 3}}\label{subsec_third}$\;$\\
As mentioned in the previous Loop 2,
in each iteration of Algorithm \ref{alg:restarts_inexact_notconvex} in 
Loop 2 we need to find many times an $(\varepsilon_3,\sigma_3)$-solution of the auxiliary problem \eqref{eq:sub2}, where we denoted for simplicity $\sigma_3=\tilde{\sigma}^{(2)}\left(\varepsilon_2,\sigma_2\right)$ and $\varepsilon_3=\tilde{\varepsilon}^{(2)}_f\left(\varepsilon_2\right)$.
To solve problem \eqref{eq:sub2}, we would like to apply Algorithm  \ref{alg:restarts_inexact_notconvex} with
\begin{align}\label{eq:framework_aux_step3} 
\varphi=f(x), \;\;\; \psi =\la \nabla  g_{\delta^{(2)},2L_{g}}(x_l^{md}),x-x_l^{md}\ra + \frac{H_2}{2}\|x - x_l^{md}\|^2,
\end{align}
where $g(x)= \max _{y \in  \R^{d_y}}\left\{G(x, y)+h(y) - \frac{H_1}{2} \|y - y_k^{md}\|^2 \right\},L_g=L_G+2\frac{L_G^2}{H_1+\mu_y}$.

The function $\vp$  is $\mu_x$-strongly convex, $L_f$-smooth and its exact gradient is available. The function $\psi$ is, clearly, $H_2$-strongly convex, $H_2$-smooth and its exact gradient is available. Also we can obtain the exact gradient for the function $\vp +\psi$.
Thus, we can apply Algorithm  \ref{alg:restarts_inexact_notconvex} with parameter $H=H_3 \geq L_f$, which will be chosen later, to solve problem \eqref{eq:sub2}.
The first main assumption of Theorem~\ref{AM:comfortable_view_with_prob}, clearly, holds. 
At the same time, 
constructing exact oracle for $\varphi =f$ requires 
$\tau_f$ calls of the basic oracle for $f$. At the same time, no calls to the oracle for $G(\cdot,y),G(x,\cdot), h$ are needed.

Let us discuss the second main assumption of Theorem~\ref{AM:comfortable_view_with_prob}. 
To ensure that this assumption holds, we need in each iteration of Algorithm \ref{alg:highorder_inexact}, used as a building block in Algorithm \ref{alg:restarts_inexact_notconvex}, to find $\left(\tilde{\varepsilon}^{(3)}_f\left(\varepsilon_3\right),\tilde{\sigma}^{(3)}\left(\varepsilon_3,\sigma_3\right)\right)$-solution to the auxiliary problem \eqref{prox_step_inexact}, where $\tilde{\sigma}^{(3)}\left(\varepsilon_3,\sigma_3\right), \tilde{\varepsilon}^{(3)}_f\left(\varepsilon_3\right)$ satisfy 
 inequalities \eqref{varepsilon_poly_CATD_with_prob}, \eqref{sigma_poly_CATD_with_prob}.
 For the particular definitions of $\vp$, $\psi$ in \eqref{eq:framework_aux_step3} in this Loop, this problem has the following form:  
\begin{align}
    &u^t_{m+1} = \arg\min _{u \in  \R^{d_x}} \{ \la \nabla \varphi(u_m^{md}),u-u_m^{md}\ra+\psi(u) + \frac{H_3}{2}\|u-u_m^{md}\|_2^2 \}\nonumber \\
    &= \arg\min _{u \in  \R^{d_x}} \{ \la \nabla  f(u_m^{md}),u-u_m^{md}\ra+\la \nabla  g_{\delta^{(2)},2L_{g}}(x_l^{md}),u-x_l^{md}\ra + \frac{H_2}{2}\|u - x_l^{md}\|^2 +\frac{H_3}{2}\|u-u_m^{md}\|_2^2\},\label{eq:SeparateX01}
\end{align}
where $g(x)= \max _{y \in  \R^{d_y}}\left\{G(x, y)+h(y) - \frac{H_1}{2} \|y - y_k^{md}\|^2 \right\},L_g=L_G+2\frac{L_G^2}{H_1+\mu_y}$.
This quadratic auxiliary problem \eqref{eq:SeparateX01} can be solved explicitly and exactly since at the point it needs to be solved, $\nabla  g_{\delta^{(2)},2L_{g}}(x_l^{md})$ is already calculated. Thus, the second main assumption of Theorem~\ref{AM:comfortable_view_with_prob} is satisfied with  $\tilde{\sigma}^{(3)}\left(\varepsilon_3,\sigma_3\right)=0$ and $ \tilde{\varepsilon}^{(3)}_f\left(\varepsilon_3\right)=0$, which clearly satisfy \eqref{delta_poly_CATD_with_prob} and \eqref{sigma_0_poly_CATD_with_prob}.


To summarize Loop 3, both main assumptions of Theorem~\ref{AM:comfortable_view_with_prob} hold and we can use it to guarantee that we obtain an $(\varepsilon_3,\sigma_3)$-solution of the auxiliary problem \eqref{eq:sub2}. This requires 
$O\left(\left(1+\left( \frac{H_3 }{\mu_{\vp} + \mu_{\psi}} \right)^{\frac{1}{2}}\right)\log \varepsilon_3^{-1} \right)=O\left(\left(1+\left( \frac{H_3 }{H_2} \right)^{\frac{1}{2}}\right)\log \varepsilon_3^{-1}\right)$ calls to the inexact oracles for $\vp$ and for $\psi$, and the same number of times solving the auxiliary problem \eqref{eq:SeparateX01}.
Combining this oracle complexity with the cost of calculating inexact oracles for $\vp$ and for $\psi$, 
we obtain that solving problem \eqref{eq:sub2} requires   
$O\left(\left(1+\left( \frac{H_3 }{H_2} \right)^{\frac{1}{2}}\right)\log \varepsilon_3^{-1}\right)\tau_f$ calls of the basic oracle for $f$. 

\begingroup
\setlength{\tabcolsep}{6pt} 
\renewcommand{\arraystretch}{1.5} 
\begin{table}\centering
\begin{tabular}{|c|c|c|c|c|c|}
\hline
            & \begin{tabular}[c]{@{}c@{}}\textbf{Goal} \end{tabular} & \begin{tabular}[c]{@{}c@{}} \textbf{$\varphi ,\psi$}\end{tabular} & \begin{tabular}[c]{@{}c@{}}$\mu$ in Th.\ref{AM:comfortable_view_with_prob} \end{tabular} & \begin{tabular}[c]{@{}c@{}}\textbf{Iteration number}\\
            \textbf{of Algorithm \ref{alg:highorder_inexact}}\\
            (Th. \ref{AM:comfortable_view_with_prob})
            \end{tabular} & \begin{tabular}[c]{@{}c@{}}\textbf{Each iteration}  \\\textbf{requires}\end{tabular}                                                                                                          \\ \hline
Loop 1  & \begin{tabular}[c]{@{}c@{}} $(\varepsilon,\sigma)$-solution \\of problem $\eqref{eq:main30}$ \end{tabular}                                        & $\eqref{eq:framework_aux_step1}$                                                 & $\mu_y$                                                                                  & $\widetilde{O} \left(1+ \sqrt{H_1/\mu_y} \right)$                 & \begin{tabular}[c]{@{}c@{}}Find  $(\varepsilon_1,\sigma_1)$-solution of $\eqref{eq:sub1}$\\ and calculate \\$\left(\delta^{(1)},  L_{\psi}\right)$-oracle of $\psi(y)$\end{tabular}       \\ \hline
Loop 2 & \begin{tabular}[c]{@{}c@{}} $(\varepsilon_1,\sigma_1)$-solution \\of problem $\eqref{eq:problem_step2}$ \end{tabular}                                  & $\eqref{eq:framework_aux_step2}$                                                 & $\mu_x$                                                                                  & $\widetilde{O}(1+\sqrt{H_2/\mu_x})$                               & \begin{tabular}[c]{@{}c@{}}Find  $(\varepsilon_2,\sigma_2)$-solution of $\eqref{eq:sub2}$\\ and calculate \\ $\left(\delta^{(2)}, L_{\vp}\right)$-oracle of $\vp(x)$\end{tabular}        \\ \hline
Loop 3  & \begin{tabular}[c]{@{}c@{}} $(\varepsilon_2,\sigma_2)$-solution \\of problem $\eqref{eq:sub2}$  \end{tabular}                                           & $\eqref{eq:framework_aux_step3}$                                                 & $H_2$                                                                                    & $\widetilde{O}(1+\sqrt{H_3/H_2})$                                 & \begin{tabular}[c]{@{}c@{}}Find  $(\varepsilon_3,\sigma_3)$-solution of $\eqref{eq:SeparateX01}$\end{tabular} \\ \hline
\end{tabular}

\caption{Summary of the three loops of the general framework described above. 
}
\label{tabl:saddleproblem_steps}
\end{table}

\subsection{Complexity of the general framework}

Below we formally finalize in Theorem \ref{theorem:general_framework} the analysis of the general framework by carefully combining the bounds obtained in \hyperref[subsec_first]{Loop 1}-\hyperref[subsec_third]{Loop 3} to obtain the final bounds for the total number of oracle calls for each part $f$, $G$, $h$ of the objective in problem \eqref{eq:framework_funct}. We will use Theorem \ref{theorem:general_framework} in the following sections to obtain complexity results for problems with structure as in \eqref{eq:problem} and \eqref{eq:problem_Gsum}. 


\begin{theorem} \label{theorem:general_framework}
Let Assumptions~\ref{assumpt:framework}, \ref{assumpt:framework_oracle}, \ref{assumpt:framework_oracle_x} hold. Then, execution of the general optimization framework described in \hyperref[subsec_first]{Loop 1}-\hyperref[subsec_third]{Loop 3}
with 
$$
H_1 = 2L_G, H_2 = 2\left(L_G + \frac{2L_G^2}{\mu_y + H_1}\right), H_3 = 2L_f
$$
generates an $(\varepsilon,\sigma)$-solution to the problem \eqref{eq:framework_funct} in the sense of Definition \ref{def:saddle_solution}.
Moreover, for the number of basic oracle calls it holds that

\begin{align}
    &\text{Number of calls of basic oracle $O_{f}$ for }  f \text{ is}:\nonumber\\ 
    & \widetilde{O} \left(\left( 1 + \sqrt{\frac{L_G}{\mu_y}}\right) \left( \mathcal{N}_f\left( \tau_f\right) + \left( 1 + \sqrt{\frac{L_G}{\mu_x}}\right)\left( 1 + \sqrt{\frac{L_f}{L_G}}\right)\cdot \tau_f\right)
    \right),\label{eq:framework_f}\\
    &\text{Number of calls of basic oracle  $O_{h}$ for } h \text{ is}:\nonumber\\ 
    &\widetilde{O} \left( \left( 1 + \sqrt{\frac{L_G}{\mu_y}}\right) \left( \tau_h +   \left( 1 + \sqrt{\frac{L_G}{\mu_x}}\right)\mathcal{N}_h\left( \tau_h, 2L_G\right)\right)
    \right),\label{eq:framework_h}\\
    &\text{Number of calls of basic oracle $O^{x}_{G}$ for } G(\cdot, y)\text{ is}:\nonumber\\ 
    & \widetilde{O} \left( \left( 1 + \sqrt{\frac{L_G}{\mu_y}}\right) \left( \mathcal{N}_G^x\left( \tau_G\right) + \left( 1 + \sqrt{\frac{L_G}{\mu_x}}\right) \tau_G\right)
    \right),\label{eq:framework_G_x}\\ 
    &\text{Number of calls of basic oracle $O^{y}_{G}$ for } G(x,\cdot)\text{ is}:\nonumber\\  
    & \widetilde{O} \left( \left( 1 + \sqrt{\frac{L_G}{\mu_y}}\right) \left( \tau_G + \left( 1 + \sqrt{\frac{L_G}{\mu_x}}\right) \mathcal{N}_G^y\left( \tau_G, 2L_G\right)\right)
    \right).
    \label{eq:framework_G_y}
\end{align}
\end{theorem}

\begin{proof}
By construction, as an output of Loop 1 we obtain an $(\varepsilon,\sigma)$-solution to the problem \eqref{eq:framework_funct} according to Definition \ref{def:saddle_solution}.

We prove the estimates for the numbers of oracle calls in two steps.
The first step is to formally prove that  in each loop the dependence of the number of oracle calls on the target accuracy $\varepsilon$ and a confidence level $\sigma$ is logarithmic. The second step is to multiply the estimates for the number of oracle calls between loops and choose the parameters $H_1$, $H_2$, $H_3$.

\textbf{Step 1. Polynomial dependence.}
The goal of this technical step is to prove that 
\begin{align}
&\varepsilon_i(\varepsilon) = \textbf{poly}\left(\varepsilon\right), \sigma_i\left(\varepsilon,\sigma\right) = \textbf{poly}\left(\varepsilon, \sigma \right),
\tilde{\sigma}^{(i)}\left(\varepsilon,\sigma\right) = \textbf{poly}\left(\varepsilon, \sigma \right), \sigma_0^{(i)}\left(\varepsilon,\sigma\right) = \textbf{poly}\left(\varepsilon, \sigma \right),\label{proof_polynomial_dependence}\\  &\tilde{\varepsilon}^{(i)}_f\left(\varepsilon\right) = \textbf{poly}\left(\varepsilon\right), \delta^{(i)} \left(\varepsilon\right) = \textbf{poly}\left(\varepsilon \right), \varepsilon_2' = \textbf{poly}\left(\varepsilon\right), \sigma_2' = \textbf{poly}\left(\varepsilon, \sigma\right), \bar{\delta}(\varepsilon_2)=\textbf{poly}\left(\varepsilon\right),\bar{\sigma}_0(\sigma_2)= \textbf{poly}\left(\varepsilon, \sigma\right)\nonumber
\end{align}
where $i = 1, 2, 3$.
For $i = 1$, according the polynomial dependencies \eqref{delta_poly_CATD_with_prob}, \eqref{sigma_0_poly_CATD_with_prob}, \eqref{varepsilon_poly_CATD_with_prob}, \eqref{sigma_poly_CATD_with_prob}  we obtain the polynomial dependencies
 \begin{align*}
\varepsilon_1(\varepsilon) = \textbf{poly}\left(\varepsilon\right), \sigma_1\left(\varepsilon,\sigma\right) = \textbf{poly}\left(\varepsilon, \sigma \right),
\tilde{\sigma}^{(1)}\left(\varepsilon,\sigma\right) = \textbf{poly}\left(\varepsilon, \sigma \right), \sigma_0^{(1)}\left(\varepsilon,\sigma\right) = \textbf{poly}\left(\varepsilon, \sigma \right),\\ \nonumber \tilde{\varepsilon}^{(1)}_f\left(\varepsilon\right) = \textbf{poly}\left(\varepsilon\right), \delta^{(1)} \left(\varepsilon\right) = \textbf{poly}\left(\varepsilon \right).
\end{align*}
 Now using that $\varepsilon_2' = \tilde{\varepsilon}^{(1)}_f$, $ \sigma_2' = \tilde{\sigma}^{(1)}$ and \eqref{eq:step_2_choos_eps_fin}, \eqref{eq:step_2_choos_sigma_fin} we have that $\varepsilon_2\left(\varepsilon\right) = \textbf{poly}\left(\varepsilon\right), \sigma_2\left(\varepsilon, \sigma\right) = \textbf{poly}\left(\varepsilon, \sigma\right)$. Further, by \eqref{eq:loop2_1}, $\bar{\delta}(\varepsilon)=\textbf{poly}\left(\varepsilon\right),\bar{\sigma}_0\left(\varepsilon, \sigma\right)= \textbf{poly}\left(\varepsilon, \sigma\right)$. Using the same argument as for $i = 1$, according the polynomial dependencies \eqref{delta_poly_CATD_with_prob}, \eqref{sigma_0_poly_CATD_with_prob}, \eqref{varepsilon_poly_CATD_with_prob}, \eqref{sigma_poly_CATD_with_prob} we obtain the polynomial dependencies
 \begin{align*}
\varepsilon_2(\varepsilon) = \textbf{poly}\left(\varepsilon\right), \sigma_2\left(\varepsilon,\sigma\right) = \textbf{poly}\left(\varepsilon, \sigma \right),
\tilde{\sigma}^{(2)}\left(\varepsilon,\sigma\right) = \textbf{poly}\left(\varepsilon, \sigma \right), \tilde{\varepsilon}^{(2)}_f\left(\varepsilon\right) = \textbf{poly}\left(\varepsilon\right), \\ 
\delta^{(2)}\left(\varepsilon\right) = \textbf{poly}\left(\varepsilon\right), \sigma_0^{(2)}\left(\varepsilon,\sigma\right) = \textbf{poly}\left(\varepsilon, \sigma \right).
\end{align*}
Taking into account that $\varepsilon_3 = \tilde{\varepsilon}^{(2)}_f, \sigma_3 = \tilde{\sigma}^{(2)}$, the polynomial dependencies \eqref{delta_poly_CATD_with_prob}, \eqref{sigma_0_poly_CATD_with_prob}, \eqref{varepsilon_poly_CATD_with_prob},\eqref{sigma_poly_CATD_with_prob} we obtain 
 \begin{align*}
\varepsilon_3(\varepsilon) = \textbf{poly}\left(\varepsilon\right), \sigma_3\left(\varepsilon,\sigma\right) = \textbf{poly}\left(\varepsilon, \sigma \right),
\tilde{\sigma}^{(3)}\left(\varepsilon,\sigma\right) = \textbf{poly}\left(\varepsilon, \sigma \right), \sigma_0^{(3)}\left(\varepsilon,\sigma\right) = \textbf{poly}\left(\varepsilon, \sigma \right),\\ \nonumber \tilde{\varepsilon}^{(3)}_f\left(\varepsilon\right) = \textbf{poly}\left(\varepsilon\right), \delta^{(3)} \left(\varepsilon\right) = \textbf{poly}\left(\varepsilon \right).
\end{align*}
 This finishes the proof of polynomial dependence. Thus, due to \eqref{proof_polynomial_dependence} in each loop when  Assumptions~\ref{assumpt:framework_oracle}, \ref{assumpt:framework_oracle_x} are applied, the dependencies $\mathcal{K}_G^y, \mathcal{K}_h, \mathcal{K}_G^x,\mathcal{K}_f$ have only logarithmic dependence on the target accuracy $\varepsilon$ and confidence level $\sigma$, i.e.
 \begin{align*}
     \mathcal{K}_G^y\left(\varepsilon, \sigma\right)=\widetilde{O}(1),\; \mathcal{K}_h\left(\varepsilon, \sigma\right)=\widetilde{O}(1),\; \mathcal{K}_G^x\left(\varepsilon, \sigma\right)=\widetilde{O}(1), \; \mathcal{K}_f\left(\varepsilon, \sigma\right)=\widetilde{O}(1),\\
     O(\log \varepsilon_1^{-1})=\widetilde{O}(1), \;O(\log \varepsilon_2^{-1})=\widetilde{O}(1), \;O(\log \varepsilon_3^{-1})=\widetilde{O}(1).
 \end{align*}
\textbf{Step 2. Final estimates.}
We have already counted the number of oracles calls for each oracle in each loop \hyperref[subsec_first]{Loop 1}-\hyperref[subsec_third]{Loop 3}, see the last paragraph of the description of each loop. 
We start with the number of basic oracle calls of $f$, which is called in each step of all the three loops.
Thus, the total number is
\begin{align*}
    \text{\# of calls in Loop1 + (\# of steps in Loop 1)$\cdot$(\# of calls in Loop 2)} \\ 
     \text{+ (\# of steps in Loop 1)$\cdot$(\# of steps in Loop 2)$\cdot$(\# of calls in Loop 3)}\\
     = \widetilde{O}\left(1+\left( \frac{H_1 }{\mu_{y}} \right)^{\frac{1}{2}}\right)\mathcal{N}_{f}\left( \tau_f\right)\mathcal{K}_f\left(\varepsilon, \sigma\right) + 
     \widetilde{O}\left(1+\left( \frac{H_1 }{\mu_{y}} \right)^{\frac{1}{2}}\right)\cdot \left(\widetilde{O}\left(1+\left( \frac{H_2 }{\mu_{x}} \right)^{\frac{1}{2}}\right)\tau_f \right) \\
     + \widetilde{O}\left(1+\left( \frac{H_1 }{\mu_{y}} \right)^{\frac{1}{2}}\right)\cdot\left(\widetilde{O}\left(1+\left( \frac{H_2 }{\mu_{x}} \right)^{\frac{1}{2}}\right)\right) \cdot \left(\widetilde{O}\left(1+\left( \frac{H_3 }{H_2} \right)^{\frac{1}{2}}\right)\tau_f \right)\\
     =\widetilde{O} \left(\left( 1 + \sqrt{\frac{H_1}{\mu_y}}\right) \left( \mathcal{N}_f\left( \tau_f\right) + \left( 1 + \sqrt{\frac{H_2}{\mu_x}}\right)\left( 1 + \sqrt{\frac{H_3}{H_2}}\right)\cdot \tau_f\right)
    \right),
\end{align*}
where we used that $\mathcal{K}_f\left(\varepsilon, \sigma\right)=\widetilde{O}(1)$.

The basic oracle of $h$ is called in each step  of \hyperref[subsec_first]{Loop 1} and \hyperref[subsec_second]{Loop 2}.
Thus, the total number is
\begin{align*}
     \text{\# of calls in Loop1 + (\# of steps in Loop 1)$\cdot$(\# of calls in Loop 2)} \\ 
     =\widetilde{O}\left(1+\left( \frac{H_1 }{\mu_{y}} \right)^{\frac{1}{2}}\right)\tau_h 
     +  \widetilde{O}\left(1+\left( \frac{H_1 }{\mu_{y}} \right)^{\frac{1}{2}}\right)\cdot \left(\widetilde{O}\left(1+\left( \frac{H_2 }{\mu_{x}} \right)^{\frac{1}{2}}\right)\mathcal{N}_{h}\left( \tau_h, H_1\right)\mathcal{K}_h\left(\varepsilon_2, \sigma_2\right) \right)\\
     =\widetilde{O} \left( \left( 1+ \sqrt{\frac{H_1}{\mu_y}}\right) \left( \tau_h +   \left( 1+\sqrt{\frac{H_2}{\mu_x}}\right)\mathcal{N}_h\left( \tau_h, H_1\right)\right)
    \right),
\end{align*}
where we used that $\mathcal{K}_h\left(\varepsilon, \sigma\right)=\widetilde{O}(1)$.

The basic oracle of $G(\cdot, y)$ is called in each step  of \hyperref[subsec_first]{Loop 1} and \hyperref[subsec_second]{Loop 2}.
Thus, the total number is
\begin{align*}
     \text{\# of calls in Loop1 + (\# of steps in Loop 1)$\cdot$(\# of calls in Loop 2)} \\ 
     =\widetilde{O}\left(1+\left( \frac{H_1 }{\mu_{y}} \right)^{\frac{1}{2}}\right)\mathcal{N}_G^x\left( \tau_G\right)\mathcal{K}_G^x\left(\varepsilon, \sigma\right) 
     +  \widetilde{O}\left(1+\left( \frac{H_1 }{\mu_{y}} \right)^{\frac{1}{2}}\right)\cdot \left(\widetilde{O}\left(1+\left( \frac{H_2 }{\mu_{x}} \right)^{\frac{1}{2}}\right)\tau_G\right)\\
     =\widetilde{O} \left( \left( 1+ \sqrt{\frac{H_1}{\mu_y}}\right) \left( \mathcal{N}_G^x\left( \tau_G\right) +   \left( 1+\sqrt{\frac{H_2}{\mu_x}}\right)\tau_G
    \right)\right),
\end{align*}
where we used that $\mathcal{K}_G^x\left(\varepsilon, \sigma\right)=\widetilde{O}(1)$.

Finally, the basic oracle of $G(x,\cdot)$ is called in each step  of \hyperref[subsec_first]{Loop 1} and \hyperref[subsec_second]{Loop 2}.
Thus, the total number is
\begin{align*}
     \text{\# of calls in Loop1 + (\# of steps in Loop 1)$\cdot$(\# of calls in Loop 2)} \\ 
     =\widetilde{O}\left(1+\left( \frac{H_1 }{\mu_{y}} \right)^{\frac{1}{2}}\right)\tau_G
     +  \widetilde{O}\left(1+\left( \frac{H_1 }{\mu_{y}} \right)^{\frac{1}{2}}\right)\cdot \left(\widetilde{O}\left(1+\left( \frac{H_2 }{\mu_{x}} \right)^{\frac{1}{2}}\right)\mathcal{N}_G^y\left( \tau_G, H_1\right) \mathcal{K}_G^y\left(\varepsilon_2, \sigma_2\right)\right)\\
     =\widetilde{O} \left( \left( 1+ \sqrt{\frac{H_1}{\mu_y}}\right) \left( \tau_G +   \left( 1+\sqrt{\frac{H_2}{\mu_x}}\right)\mathcal{N}_G^y\left( \tau_G, H_1\right) 
    \right)\right),
\end{align*}
where we used that $\mathcal{K}_G^y\left(\varepsilon_2, \sigma_2\right)=\widetilde{O}(1)$.

 The final estimates are obtained by substituting the constants $H_1, H_2, H_3$ given by
$$H_1 = 2L_G, H_2 = 2\left(L_G + \frac{2L_G^2}{\mu_y + H_1}\right)\leq 2\left(L_G + \frac{2L_G^2}{H_1}\right)=4L_G, H_3 = 2L_f.
$$
\end{proof}

\section{Accelerated Method for Saddle-Point Problems} \label{section:saddle}



In this section, we consider problem \eqref{eq:problem} which is problem \eqref{eq:framework_funct} with a specific finite-sum structure of the function $h$ and our goal is to obtain its $(\varepsilon, \sigma)$-solution. To get the final estimates for the number of oracles calls, we need to satisfy Assumptions~\ref{assumpt:framework}, \ref{assumpt:framework_oracle}, \ref{assumpt:framework_oracle_x} which are formulated in Section \ref{section:saddleFramework} where we construct our general framework. 
So, the plan of this section is first to prove Lemma \ref{lem:obtain_oracle} and Corollary \ref{lem:obtain_oracle_x}, which guarantee that  Assumptions~\ref{assumpt:framework_oracle}, \ref{assumpt:framework_oracle_x} hold.
To satisfy Assumption \ref{assumpt:framework_oracle} we use a two-loop procedure with Algorithm \ref{alg:restarts_inexact_notconvex} and stochastic variance reduction method to solve problem \eqref{eq:framework_g_obt_oracle} in order to use the finite-sum structure of the function $h$ and avoid expensive calculation of the gradient of the whole sum in each iteration.
As a corollary, we also show how to satisfy Assumption \ref{assumpt:framework_oracle_x}.
Then, we obtain final estimates for the setting of this section by combining the complexities to satisfy Assumptions~\ref{assumpt:framework_oracle}, \ref{assumpt:framework_oracle_x} with the estimates in Theorem \ref{theorem:general_framework}.

\subsection{Problem statement}
In this section we consider optimization problem of the form \eqref{eq:problem}: 
\begin{equation} 
\label{eq:problem_st_h_sum}
    \min _{x \in \mathbb{R}^{d_x}} \max _{y \in \mathbb{R}^{d_y}} \left\{f(x)+ G(x, y) - h(y)\right\}, \;\;\; h(y):=\frac{1}{m_h} \sum_{i = 1}^{m_h} h_i(y) 
\end{equation}
and develop accelerated optimization methods for its solution
under the following assumptions.
\begin{assumption}\label{assumpt:h_sum} 
\begin{enumerate}
    \item Function $f(x)$ is $L_f$-smooth and $\mu_x$-strongly convex.
    \item Function $G(x, y)$  is $L_G$-smooth, i.e. for each $(x_1, x_2), (y_1, y_2) \in  \mathbb{R}^{d_x} \times \mathbb{R}^{d_y}$
\begin{equation}
    \|\nabla G (x_1,x_2)-\nabla G (y_1,y_2)\|\leq L_G \|(x_1, x_2) - (y_1, y_2)\|.
\end{equation}
\item \label{Asm:h_finite_sum}$m_h\geq 1$ and each function $h_{i}(x)$, $i \in 1, \dots, m_h$ is $L_h^i$-smooth and convex, function $h(y)$ is $\mu_y$-strongly convex. We also define $L_h = \frac{1}{m_h} \sum_{i = 1}^{m_h} L_h^i$ in this case.
\end{enumerate}
\end{assumption}
To fit Assumption~\ref{assumpt:framework} we consider the full gradient oracles $\nabla_xG(x,y)$, $\nabla_yG(x,y)$, $\nabla f(x)$ as the basic oracles $O_G^x$, $O_G^y$, $O_f$ respectively, and the stochastic gradient oracle $\nabla h_i(y)$ as the basic oracle $O_h$. Then Assumption~\ref{assumpt:h_sum} guarantees that Assumption~\ref{assumpt:framework} holds with \begin{align}\label{eq:assumpt1_constants_h_sum}
    \tau_f = \tau_G = 1, \tau_h = m_h.
\end{align}

\subsection{Preliminaries} \label{saddle_preliminary}
 
 We start with two auxiliary results, which show how Assumptions \ref{assumpt:framework_oracle}, \ref{assumpt:framework_oracle_x} can be satisfied in the setting of this section.
 The first lemma provides complexity for inexact solution of maximization problem \eqref{eq:framework_g_obt_oracle} and the complexity of finding an inexact oracle for function $g$ defined in the same equation.
 
 
 \begin{lemma} \label{lem:obtain_oracle}
Let the function $g$ be defined via maximization problem in \eqref{eq:framework_g_obt_oracle}, i.e.
\begin{align}\label{eq:g_obt_oracle}
g(x)= \max _{y \in  \R^{d_y}}\left\{G(x, y)-h(y) -\frac{H}{2}\|y-y_0\|^2\right\},
\end{align}
 where $G(x,y)$, $h(y)$ are according to \eqref{eq:problem_st_h_sum} 
 and satisfy Assumption~\ref{assumpt:h_sum}, $y_0 \in \R^{d_y}$. 
 Assume also that that $m_h(H+2L_G+\mu_y) \leq L_h$ and $H+\mu_y \leq 4L_G$. Then, organizing computations in two loops and applying Algorithm \ref{alg:restarts_inexact_notconvex} in the outer loop and accelerated variance reduction method L-SVRG from \cite{morin2020sampling} in the inner loop, we guarantee Assumption~\ref{assumpt:framework_oracle} with $\tau_G=1$ basic oracle calls for $G(\cdot,y)$ and the following estimates for the number of basic oracle calls for $G(x,\cdot)$ and $h$ respectively
 \begin{align}\label{eq:complofgG}
&\mathcal{N}_{G}^y\left( \tau_G, H\right) = O\left(1 + \sqrt{L_G/(H+\mu_y)}\right),\\
&\mathcal{N}_{h}\left( \tau_h, H\right) = O\left( \sqrt{\tau_h L_h/(H+\mu_y)}\right). \label{eq:complofgh}
\end{align}
 

\end{lemma}

\begin{proof}
To satisfy Assumption~\ref{assumpt:framework_oracle} we need to provide an $\left(\delta\left(\varepsilon\right)/2 ,\sigma_0\left(\varepsilon,\sigma\right)\right)$-solution to the problem \eqref{eq:g_obt_oracle} and $\left(\delta\left(\varepsilon\right) ,\sigma_0\left(\varepsilon,\sigma\right), 2L_g\right)$-oracle of $g$ in \eqref{eq:g_obt_oracle}, where $L_g = L_G+2L_G^2/(\mu_y+H)$.\\
By Lemma~\ref{lemma:obt_delta_oracle} with $F(x,y)=G(x,y)$, $w(y)=h(y)+\frac{H}{2}\|y-y_0\|^2$, $\delta=\delta\left(\varepsilon\right)$ and $\sigma_0 = \sigma_0\left(\varepsilon,\sigma\right)$ applied to the problem \eqref{eq:g_obt_oracle},
if we find a $(\delta/2,\sigma_0)$-solution $\tilde{y}_{\delta/2}(x)$ of the problem  \eqref{eq:g_obt_oracle},
then $ \nabla_{x} G\left(x, \tilde{y}_{\delta/2}(x)\right)$ is $(\delta,\sigma_0,2 L_g)$-oracle of $g$ and its calculation requires $\tau_G=1$ calls of the oracle $\nabla_x G(\cdot,y)$.
To finish the proof, we now focus on obtaining a $(\delta/2,\sigma_0)$-solution $\tilde{y}_{\delta/2}(x)$ of the problem  \eqref{eq:g_obt_oracle}, for which we construct a two-loop procedure described below. 

\paragraph{Loop 1}\label{lemma_loop1}$\;$\\
The goal of Loop 1 is to find an $\left(\delta\left(\varepsilon\right)/2,\sigma_0\left(\varepsilon,\sigma\right)\right)$-solution of  problem \eqref{eq:g_obt_oracle} as a maximization problem in $y$.
To obtain such an approximate solution, we change the sign of this optimization problem and apply Algorithm \ref{alg:restarts_inexact_notconvex} with  
\begin{align}\label{aux_step4}
    \varphi = -G(x,y) ,\;\;\;\; \psi = h(y) + \frac{H}{2}\|y-y_0\|^2.
\end{align}
Function $\varphi$ is convex and has $L_G$-Lipschitz continuous gradient, function $\psi$ is $ H+\mu_y$-strongly convex and has $L_h+H$-Lipschitz continuous gradient. Thus, we can apply Algorithm  \ref{alg:restarts_inexact_notconvex} with exact oracles and parameter $H_1\geq 2L_G$, which will be chosen later, to solve problem \eqref{eq:g_obt_oracle}.
To satisfy the conditions of Theorem~\ref{AM:comfortable_view_with_prob}, which gives the complexity of Algorithm  \ref{alg:restarts_inexact_notconvex}, we, first, observe that the oracles of $\varphi$ and $\psi$ are exact and, second, observe that we need in each iteration of Algorithm \ref{alg:highorder_inexact}, used as a building block in Algorithm \ref{alg:restarts_inexact_notconvex}, to find an $\left(\tilde{\varepsilon}^{(1)}_f\left(\delta/2\right),\tilde{\sigma}^{(1)}\left(\delta/2,\sigma_0\right)\right)$-solution to the auxiliary problem \eqref{prox_step_inexact}, which in this case has the following form:  
\begin{align} 
&z_{k+1}^t = \arg\min _{z \in  \R^{d_y}} \{ \la \nabla \varphi(z_k^{md}),z-z_k^{md}\ra+\psi(z) + \frac{H_1}{2}\|z-z_k^{md}\|_2^2 \} \nonumber
\\
&= \arg\min _{z \in  \R^{d_y}} \{ -\la \nabla_{z}  G(x, z_k^{md}),z-z_k^{md}\ra+h(z) + \frac{H}{2}\|z-y_0\|^2+\frac{H_1}{2}\|z-z_k^{md}\|_2^2\} \label{eq:190},
\end{align}
 where $\tilde{\sigma}^{(1)}\left(\delta/2,\sigma_0\right), \tilde{\varepsilon}^{(1)}_f\left(\delta/2\right)$ need to satisfy  inequalities \eqref{varepsilon_poly_CATD_with_prob}, \eqref{sigma_poly_CATD_with_prob}. 
 Below, in the \hyperref[lemma_loop2]{paragraph "Loop 2"}, we explain how to solve this auxiliary problem by a variance reduction method in such a way that these inequalities hold. 
 

To summarize Loop 1, both main assumptions of Theorem~\ref{AM:comfortable_view_with_prob} hold and we can use it to guarantee that we obtain an $(\delta/2,\sigma_0)$-solution of problem \eqref{eq:g_obt_oracle}. Due to polynomial dependencies  $\delta\left(\varepsilon\right)=\bf{poly}\left(\varepsilon\right)$, $\sigma_0\left(\varepsilon, \sigma \right)=\bf{poly}\left(\varepsilon, \sigma \right)$ this requires $\widetilde{O}\left(1+\left( \frac{H_1 }{\mu_{\vp} + \mu_{\psi}} \right)^{\frac{1}{2}}\right)=\widetilde{O}\left(1+\left( \frac{H_1 }{\mu_{y}+H} \right)^{\frac{1}{2}}\right)$ 
calls to the (exact) oracles for $\vp$ and for $\psi$, and the same number of times solving the auxiliary problem \eqref{eq:190}.
Combining this oracle complexity with the cost of calculating (exact) oracles for $\vp$ and for $\psi$, we obtain that solving problem \eqref{eq:g_obt_oracle} requires 
$\widetilde{O}\left(1+\left( \frac{H_1 }{\mu_{y}+H} \right)^{\frac{1}{2}}\right)$ calls of the basic oracle for $G(x,\cdot)$ and $\widetilde{O}\left(m_h+m_h\left( \frac{H_1 }{\mu_{y}+H} \right)^{\frac{1}{2}}\right)$ of the basic oracles for $h$, i.e. stochastic gradients $\nabla h_i$. 
The only remaining thing is to provide an inexact solution to problem \eqref{eq:190} and, next, we move to Loop 2 to explain how to guarantee this. Note that we need to solve problem \eqref{eq:190} $\widetilde{O}\left(1+\left( \frac{H_1 }{\mu_{y}+H} \right)^{\frac{1}{2}}\right)$ times.

\paragraph{Loop 2}\label{lemma_loop2}$\;$\\
We solve problem \eqref{eq:190} by the algorithm L-SVRG proposed in \cite{morin2020sampling}, which complexity is stated in Lemma~\ref{L-SVRG}, see Appendix~\ref{Appendix_D}. 
As mentioned in the previous Loop 1,
in each iteration of Algorithm \ref{alg:restarts_inexact_notconvex} in Loop 1 
we need many times to find an $(\varepsilon_2,\sigma_2)$-solution of the auxiliary problem \eqref{eq:190}, where for simplicity we denote $\sigma_2=\tilde{\sigma}^{(1)}\left(\delta/2,\sigma_0\right)$ and $\varepsilon_2=\tilde{\varepsilon}^{(1)}_f\left(\delta/2\right)$.\\
To obtain such an approximate solution, we apply L-SVRG from \cite{morin2020sampling}  with  (see Lemma~\ref{L-SVRG} from Appendix~\ref{Appendix_D}) 
\begin{align}\label{aux_step5}
    \varphi = \frac{1}{m_h}\sum_{i = 1}^{m_h} \underbrace{\left( h_i(z) + \frac{H}{2}\|z-y_0\|^2+\frac{H_1}{2}\|z-z_k^{md}\|_2^2\right)}_{\vp_i(z)},\;\;\;\; \psi = -\la \nabla_{z}  G(x, z_k^{md}),z-z_k^{md}\ra.
\end{align}
Functions $\varphi_i$ are convex 
and have $L_h^i+H+H_1$-Lipschitz continuous gradient for all $i = 1,\dots,m_h$, function $\psi$ is convex, $0$-smooth and prox-friendly. Also function $\vp$ is $\mu_y+H+H_1$-strongly convex.
Thus, all the conditions of Lemma~\ref{L-SVRG} from Appendix~\ref{Appendix_D} are satisfied and we can apply L-SVRG from \cite{morin2020sampling} to solve problem \eqref{eq:190}. From this lemma we get an estimate $\widetilde{O} \left(m_h+\sqrt{\frac{m_h (L_h+H+H_1)}{\mu_y+H+H_1}}\right)$ for the number of calls of the basic oracle for $h$. 

To summarize Loop 2, the assumptions of Lemma~\ref{L-SVRG} from Appendix~\ref{Appendix_D} hold and we can use it to guarantee that we obtain an $(\varepsilon_2,\sigma_2)$-solution of problem \eqref{eq:190}. According to the polynomial dependences \eqref{varepsilon_poly_CATD_with_prob}, \eqref{sigma_poly_CATD_with_prob} we obtain that 
$$
\sigma_2 = \tilde{\sigma}^{(1)}\left(\delta/2,\sigma_0\right) = \textbf{poly}(\delta/2,\sigma_0), \;\; \varepsilon_2 = \tilde{\varepsilon}^{(1)}_f\left(\delta/2,\sigma_0\right) = \textbf{poly}(\delta/2,\sigma_0).
$$
Using conditions $\delta\left(\varepsilon\right)=\bf{poly}\left(\varepsilon\right)$, $\sigma_0\left(\varepsilon, \sigma \right)=\bf{poly}\left(\varepsilon, \sigma \right)$ in the formulation of Asumption~\ref{assumpt:framework_oracle} we obtain that the dependencies 
$$
\sigma_2\left(\varepsilon,\sigma\right), \tilde{\sigma}^{(1)}\left(\varepsilon,\sigma\right), \varepsilon_2 \left(\varepsilon,\sigma\right), \tilde{\varepsilon}^{(1)}_f\left(\varepsilon,\sigma\right)
$$
are polynomial. Then, we can use notation $\widetilde{O}(\cdot)$   without specifying what precision we mean and implying that the logarithmic part depends on the initial  $\varepsilon, \sigma$. Finally, according to Lemma~\ref{L-SVRG} from Appendix~\ref{Appendix_D} an $(\varepsilon_2,\sigma_2)$-solution of problem \eqref{eq:190} requires $\widetilde{O} \left(m_h+\sqrt{\frac{m_h (L_h+H+H_1)}{\mu_y+H+H_1}}\right)$ 
 calls of the basic oracle for $h$, i.e. stochastic gradients $\nabla h_i$, and the same number of times solving the auxiliary problem of the form $\arg \min_y\{\psi(y)+\frac{1}{2\alpha }\|y-\bar{y}\|_2^2 \}$. 
 This problem is solved explicitly since $\psi(y)$ is a linear function. 
 

\paragraph{Combining the estimates of both loops}$\;$\\ 
Combining the estimates of the above \hyperref[lemma_loop1]{paragraph "Loop 1"} and \hyperref[lemma_loop2]{paragraph "Loop 2"} we see that, finding a point $\tilde{y}_{\delta/2}(x)$ which is an $\left(\delta\left(\varepsilon\right)/2 ,\sigma_0\left(\varepsilon,\sigma\right)\right)$-solution to the problem \eqref{eq:g_obt_oracle} 
requires the following number of calls of the basic oracles of $G(x,\cdot)$ and $h$ respectively 
\begin{align}\label{eq:obt_complofgG}
& \widetilde{O}\left(1 + \sqrt{H_1/(H+\mu_y)}\right),\\
& \text{\# of calls in Loop1 + (\# of steps in Loop 1)$\cdot$(\# of calls in Loop 2)} \\ \nonumber &=\widetilde{O}\left( m_h + m_h\sqrt{H_1/(H+\mu_y)} + \left( 1+ \sqrt{H_1/(H+\mu_y)}\right)\left(m_h+\sqrt{\frac{m_h (L_h+H+H_1)}{\mu_y+H+H_1}}\right)\right). \label{eq:obt_complofgh}
\end{align}
 Finding $\left(\delta\left(\varepsilon\right) ,\sigma_0\left(\varepsilon,\sigma\right), 2L_g\right)$-oracle of $g$ by calculating $ \nabla_{x} G\left(x, \tilde{y}_{\delta/2}(x)\right)$ requires additionally $\tau_G=1$ 
calls of the basic oracle for $G(\cdot,y)$. 
Since in Assumption~\ref{assumpt:framework_oracle} we denote the dependence on the target accuracy $\varepsilon$ and confidence level $\sigma$ by a separate quantities denoted by $\mathcal{K}(\varepsilon,\sigma)$ and in this case it is logarithmic, choosing $H_1 = 2L_G$ we get the final estimates for  $\mathcal{N}_{G}^y$ and $\mathcal{N}_{h}$ to guarantee that  Assumption~\ref{assumpt:framework_oracle} holds:
\begin{align}
&\mathcal{N}_{G}^y = O\left(1 + \sqrt{L_G/(H+\mu_y)}\right),\\ \nonumber 
&\mathcal{N}_{h} = O\left( m_h + \left(1+ \sqrt{2L_G/(H+\mu_y)}\right)\left(m_h+\sqrt{\frac{m_h (L_h+H+2L_G)}{\mu_y+H+2L_G}}\right)\right)=\\ \nonumber
&O\left( m_h + \left(1+ \sqrt{2L_G/(H+\mu_y)}\right)\left(m_h+\sqrt{\frac{m_h L_h}{\mu_y+H+2L_G}}+\sqrt{\frac{m_h(H+2L_G)}{\mu_y+H+2L_G}}\right)\right)=\\ \nonumber
&O\left( m_h + \left(1+ \sqrt{2L_G/(H+\mu_y)}\right)\left(m_h+\sqrt{\frac{m_h L_h}{\mu_y+H+2L_G}}\right)\right)=\\ \nonumber
&O\left( m_h + \sqrt{2L_G/(H+\mu_y)}\sqrt{\frac{m_h L_h}{2L_G}}\right)=\\ 
&O\left( m_h + \sqrt{\frac{m_h L_h}{H+\mu_y}}\right)=
O\left( \sqrt{m_hL_h/(H+\mu_y)}\right),
\end{align}
where we used that, by the assumptions of this Lemma, $1 \leq 4L_G/(H+\mu_y)$, $m_h(H+2L_G+\mu_y) \leq L_h$ and  $\forall a, b \geq 0\;\; \sqrt{a+b}\leq \sqrt{a}+\sqrt{b}$.
\qed
\end{proof}

By changing the variables $x$ and $y$ in Lemma \ref{lem:obtain_oracle} and choosing $H = 0$ we obtain the simple Corollary \ref{lem:obtain_oracle_x} which ensures Assumption ~\ref{assumpt:framework_oracle_x}.
\begin{corollary} \label{lem:obtain_oracle_x}
Let the function $r$ be defined via maximization problem in \eqref{eq:framework_r_obt_oracle}, i.e.
\begin{align}
r(y)= \min _{x \in  \R^{d_x}}\left\{G(x, y)+f(x)\right\},
\end{align}
 where  $G(x,y),f(y)$ are according to \eqref{eq:problem_st_h_sum} 
 and satisfy Assumption~\ref{assumpt:h_sum}. Assume also that $2L_G+\mu_x \leq L_f$ and $\mu_x \leq 4L_G$.
 Then, organizing computations in two loops and applying Algorithm \ref{alg:restarts_inexact_notconvex} in the outer loop and 
 accelerated variance reduction method L-SVRG from \cite{morin2020sampling} in the inner loop, we guarantee Assumption~\ref{assumpt:framework_oracle_x} 
 with $\tau_G=1$ basic oracle calls for $G(x,\cdot)$ and the following estimates for the number of basic oracle calls for
 $G(\cdot,y)$, $f$ respectively
 \begin{align}\label{eq:complofgG_x}
&\mathcal{N}_{G}^x\left( \tau_G\right) = O\left(1 + \sqrt{L_G/\mu_x}\right),\\
&\mathcal{N}_{f}\left( \tau_f\right) = O\left( \sqrt{ L_f/\mu_x}\right). \label{eq:complofgf_x}
\end{align}
 

\end{corollary}

\subsection{Final estimates} \label{saddle_final}
We are now in a position to state the final result of this section for the complexity estimates when solving problem \eqref{eq:problem_st_h_sum}. Assumption \ref{assumpt:h_sum} with \eqref{eq:assumpt1_constants_h_sum} guarantee that Assumption~\ref{assumpt:framework} holds.
Lemma~\ref{lem:obtain_oracle} and Corollary~\ref{lem:obtain_oracle_x} guarantee that Assumptions~\ref{assumpt:framework_oracle}, \ref{assumpt:framework_oracle_x} hold.
Thus, all the conditions of Theorem \ref{theorem:general_framework} are satisfied and we obtain the following result for solving problem \eqref{eq:problem_st_h_sum} with our system of inner-outer loops.

\begin{theorem}\label{theorem:h-sum}
Assume that for problem \eqref{eq:problem_st_h_sum} Assumption \ref{assumpt:h_sum} holds and additionally 
$m_h (4L_G + \mu_y) \leq L_h$, $ 2L_G + \mu_x   \leq L_f, \mu_y  \leq L_G$, $\mu_x  \leq L_G$. 
Then the described in Section \ref{section:saddleFramework} general framework combined with the algorithms described in the previous subsection find an $(\varepsilon, \sigma)$-solution to problem \eqref{eq:problem_st_h_sum} with the following number of basic oracle calls
\begin{align}
    &\nabla f \text{-oracle calls}: \widetilde{O} \left( \sqrt{\frac{L_G L_f}{\mu_x\mu_y}} 
    \right),\\
        &\nabla h_i \text{-oracle calls}: \widetilde{O} \left( \sqrt{\frac{m_h L_G L_h}{\mu_x \mu_y}}
    \right),\\
    &\nabla_x G \text{-oracle calls}: \widetilde{O} \left(  \sqrt{\frac{L_G^2}{\mu_x\mu_y}}
    \right),\\ 
    &\nabla_y G \text{-oracle calls}: \widetilde{O} \left(  \sqrt{\frac{L_G^2}{\mu_x\mu_y}} 
    \right).
\end{align}
\end{theorem}

\begin{proof}
Assumption \ref{assumpt:h_sum} with \eqref{eq:assumpt1_constants_h_sum} guarantee that Assumption~\ref{assumpt:framework} holds. 
Further, assumption  $\mu_y  \leq L_G$ and the choice $H = 2L_G$ guarantee that $\mu_y + H \leq 4L_G$. This inequality, assumption that $m_h (4L_G + \mu_y) \leq L_h$ and the choice $H = 2L_G$ allow to apply Lemma~\ref{lem:obtain_oracle} and conclude that Assumption \ref{assumpt:framework_oracle} holds with the number of oracle calls given by \eqref{eq:complofgG} and \eqref{eq:complofgh}.
Assumptions $ 2L_G + \mu_x   \leq L_f$ and $\mu_x \leq L_G$ by Corollary~\ref{lem:obtain_oracle_x} guarantee that Assumption \ref{assumpt:framework_oracle_x} holds with the number of oracle calls given by \eqref{eq:complofgG_x} and \eqref{eq:complofgf_x}. 
Applying Theorem~\ref{theorem:general_framework} and combining its complexity estimates, we obtain the final complexity bounds as follows.

Number of basic oracle calls of $f$:
\begin{align*}
\widetilde{O} \left(\left( 1 + \sqrt{\frac{L_G}{\mu_y}}\right) \left( \sqrt{ \frac{L_f}{\mu_x}} + \left( 1 + \sqrt{\frac{L_G}{\mu_x}}\right)\left( 1 + \sqrt{\frac{L_f}{L_G}}\right)\right)
    \right) 
    = \widetilde{O} \left(\left( \sqrt{\frac{L_G}{\mu_y}}\right) \left( \sqrt{ \frac{L_f}{\mu_x}} + \left( \sqrt{\frac{L_G}{\mu_x}}\right)\left( \sqrt{\frac{L_f}{L_G}}\right)\right)
    \right)\\
    =\widetilde{O} \left(\left( \sqrt{\frac{L_f L_G}{\mu_x \mu_y}}\right)\right),
\end{align*}
where we used that, by the assumptions of this Theorem, $1 \leq L_G/\mu_y $, $1 \leq L_G/\mu_x$ and $1 \leq L_f/L_G$.

Number of basic oracle calls of $h$:
\begin{align*}
\widetilde{O} \left( \left( 1 + \sqrt{\frac{L_G}{\mu_y}}\right) \left( m_h +   \left( 1 + \sqrt{\frac{L_G}{\mu_x}}\right) \sqrt{\frac{m_h L_h}{2L_G+\mu_y}}\right)
    \right) 
=\widetilde{O} \left( \left( \sqrt{\frac{L_G}{\mu_y}}\right) \left( m_h +   \left( \sqrt{\frac{L_G}{\mu_x}}\right)\left( \sqrt{\frac{m_h L_h}{2L_G+\mu_y}}\right)\right)
    \right)=\\
    \widetilde{O} \left( \max \left\{ \underbrace{m_h \sqrt{\frac{L_G}{\mu_y}}}_{=\widetilde{O}\left(\sqrt{m_hL_h/\mu_y}\right)}, \sqrt{\frac{m_h L_G L_h}{\mu_x \mu_y}} \right\}\right)=\widetilde{O} \left( \sqrt{\frac{m_h L_G L_h}{\mu_x \mu_y}}\right),
\end{align*}
where we used that, by the assumptions of this Theorem, $1 \leq L_G/\mu_y $, $1 \leq L_G/\mu_x$ and 
\[
m_h (4L_G + \mu_y) \leq L_h \Rightarrow \sqrt{m_h L_G} \leq \sqrt{L_h}
\]

Number of basic oracle calls of  $G(\cdot, y)$:
\begin{align*}
 \widetilde{O} \left( \left( 1 + \sqrt{\frac{L_G}{\mu_y}}\right) \left( 1 + \sqrt{\frac{L_G}{\mu_x}} + \left( 1 + \sqrt{\frac{L_G}{\mu_x}}\right) \right)
    \right) = \widetilde{O} \left(   \sqrt{\frac{L_G^2}{\mu_x \mu_y}}  \right),
\end{align*}
where we used that, by the assumptions of this Theorem, $1 \leq L_G/\mu_y $ and $1 \leq L_G/\mu_x$.

Number of basic oracle calls of  $G(x,\cdot)$:
\begin{align*}
 \widetilde{O} \left( \left( 1 + \sqrt{\frac{L_G}{\mu_y}}\right) \left( 1 + \left( 1 + \sqrt{\frac{L_G}{\mu_x}}\right)  \left(1 + \sqrt{\frac{L_G}{2L_G+\mu_y}} \right)\right)
    \right) \\
 =\widetilde{O} \left( \left( 1 + \sqrt{\frac{L_G}{\mu_y}}\right) \left( 1 + \sqrt{\frac{L_G}{\mu_x}}  \right)
    \right)=\widetilde{O} \left(   \sqrt{\frac{L_G^2}{\mu_x \mu_y}}  \right),
\end{align*}
where we used that, by the assumptions of this Theorem, $1 \leq L_G/\mu_y $ and $1 \leq L_G/\mu_x$.
\qed \end{proof}

An important particular case, for which we state the following corollary, is when does not have the finite-sum, i.e. $m_h=1$.

\begin{corollary}(\textbf{Particular case $m_h=1$}) \label{corollary:m_h=1}
Let the assumptions of Theorem \ref{theorem:h-sum} hold and additionally $m_h=1$.
Then the described in Section \ref{section:saddleFramework} general framework combined with the algorithms described in the previous subsection find an $(\varepsilon, \sigma)$-solution to problem \eqref{eq:problem_st_h_sum} with the following number of basic oracle calls
\begin{align}
    &\nabla f \text{-oracle calls}: \widetilde{O} \left( \sqrt{\frac{L_G L_f}{\mu_x\mu_y}} 
    \right),\\
        &\nabla h \text{-oracle calls}: \widetilde{O} \left( \sqrt{\frac{ L_G L_h}{\mu_x \mu_y}}
    \right),\\
    &\nabla_x G \text{-oracle calls}: \widetilde{O} \left(  \sqrt{\frac{L_G^2}{\mu_x\mu_y}}
    \right),\\ 
    &\nabla_y G \text{-oracle calls}: \widetilde{O} \left( \sqrt{\frac{L_G^2}{\mu_x\mu_y}}
    \right).
\end{align}
\end{corollary}

\section{Accelerated Methods for Saddle-Point Problems with Finite-Sum Structure}
\label{section:G_sum}

In this section, we consider problem \eqref{eq:problem_Gsum}, which is problem \eqref{eq:framework_funct} with a specific finite-sum structure of the function $G$. The algorithms in this section are, in fact, deterministic, i.e. correspond to confidence levels $\sigma=0$. Thus, our goal is to obtain an  $\varepsilon$-solution of problem \eqref{eq:problem_Gsum}. 
As in the previous section, we use the general framework described in Section \ref{section:saddleFramework}, but in a simpler setting of all the confidence levels $\sigma$ being equal to zero.
To obtain the final estimates for the number of basic oracles calls, we need to satisfy Assumptions~\ref{assumpt:framework}, \ref{assumpt:framework_oracle}, \ref{assumpt:framework_oracle_x} which are formulated in Section \ref{section:saddleFramework}, where we construct our general framework. 
The proof that these assumptions hold and the proof of the resulting complexity bounds follow mostly the same lines as for the case of problem \eqref{eq:problem_st_h_sum} under Assumption \ref{assumpt:h_sum} in the previous section, but are rather technical. Thus, in this section we only state the main results and the proofs are deferred to Appendix~\ref{Appendix_Framework_inverse}  and Appendix~\ref{Appendix_h_not_sum}.


\subsection{Problem statement}
In this section we consider optimization problem of the form \eqref{eq:problem_Gsum}: 
\begin{equation} \label{eq:problem_Gsum_minmax}
    \min _{x \in \mathbb{R}^{d_x}} \max _{y \in \mathbb{R}^{d_y}} \left\{f(x)+ G(x, y) - h(y)\right\}, \;\;\; G(x,y):=\frac{1}{m_G} \sum_{i = 1}^{m_G} G_i(x,y).
\end{equation}
and develop accelerated optimization methods for its solution under the following assumptions.


\begin{assumption}\label{assumpt:G_sum}
\begin{enumerate}
    
    \item Function $f(x)$ is $\mu_x$-strongly convex, and function $h(y)$ is $\mu_y$-strongly convex.
    \item $m_G\geq 1$ and each function $G_i(x, y)$, $i \in 1, \dots, m_G$  is convex in $x$ and concave in $y$, and $L_G^i$-smooth, i.e. for each $x = (x_1, x_2), y = (y_1, y_2) \in  \mathbb{R}^{d_x} \times \mathbb{R}^{d_y}$
\begin{equation}
    \|\nabla G_i (x_1,x_2)-\nabla G_i (y_1,y_2)\|\leq L_G^i \|(x_1, x_2) - (y_1, y_2)\|.
\end{equation}
We also define $L_G = \frac{1}{m_G} \sum_{i = 1}^{m_G} L_G^i$.
\item One of the following three statements holds for the functions $f(x)$, $h(y)$
    \begin{enumerate}
    \item Function  $f(x)$ is $L_f$-smooth and function $h(y)$ is $L_h$-smooth;
    \item  Function $f(x)$ is $L_f$-smooth, function $h(y)$ is $L_h$-smooth and prox-friendly; 
    \end{enumerate}
\end{enumerate}




\end{assumption}

Under Assumption~\ref{assumpt:G_sum}.2 it is easy to see that the function $G(x,y)$ in problem \eqref{eq:problem_Gsum_minmax} is $L_G$-smooth. Indeed,
\begin{align*}
    \|\nabla G (x_1,x_2)-\nabla G (y_1,y_2)\|\leq& \frac{1}{m_G}\sum_{i=1}^{m_G}\|\nabla G_i (x_1,x_2)-\nabla G_i (y_1,y_2)\|\\
    \leq& \frac{1}{m_G}\sum_{i=1}^{m_G}
    L_G^i \|(x_1, x_2) - (y_1, y_2)\|=
    L_G \|(x_1, x_2) - (y_1, y_2)\|.
\end{align*}
where $x = (x_1, x_2), y = (y_1, y_2) \in  \mathbb{R}^{d_x} \times \mathbb{R}^{d_y}$.
To further fit Assumption~\ref{assumpt:framework} we consider the full gradient oracles $\nabla h(y)$, $\nabla f(x)$ as the basic oracles $O_h$, $O_f$ respectively, and the stochastic gradient oracle $\nabla_xG_i(x,y)$, $\nabla_yG_i(x,y)$ as the basic oracles $O_G^x$, $O_G^y$ respectively. Then Assumption~\ref{assumpt:G_sum} guarantees that Assumption~\ref{assumpt:framework} holds with
\begin{align}\label{eq:assumpt1_constants_G_sum}
    \tau_f = \tau_h = 1, \tau_G = m_G.
\end{align}

\subsection{Complexity estimates}

In this section we consider problem \eqref{eq:problem_Gsum_minmax} under one of the two different Assumptions \ref{assumpt:G_sum}.3(a) or (b) and mostly follow the lines of derivations described in Section \ref{section:saddle} with appropriate changes caused by the different problem statement.
In particular, we change the order of the loops in the general framework described in the Section~\ref{section:saddleFramework} as well as in the proof of Lemma~\ref{lem:obtain_oracle} and Corollary~\ref{lem:obtain_oracle_x} depending on which is larger $L_h$ or $L_G$ and $L_f$ or $L_G$.
This eventually allows to avoid assumptions of the form $4L_G + \mu_y \leq L_h$, $ 2L_G + \mu_x \leq L_f$, which are used in  Theorem~\ref{theorem:h-sum}.
The proof of the resulting complexity bounds follows mostly the same ideas as for the  case of problem \eqref{eq:problem_st_h_sum} under Assumption \ref{assumpt:h_sum}, but is rather technical. Thus, in this section we only state the result and the proofs are deferred to appendices.
In Appendix~\ref{Appendix_Framework_inverse} we propose a variation of the general framework described in Section~\ref{section:saddleFramework}, but with the change of the order of Loop 2 and Loop 3. As a result, we prove Theorem \ref{theorem:general_framework_inverse} which is a counterpart of Theorem~\ref{theorem:general_framework}. In Appendix~\ref{Appendix_h_not_sum} we prove Lemma~\ref{lem:obtain_oracle_mh=1} and Corollary~\ref{lem:obtain_oracle_x_mh=1}, which generalize Lemma~\ref{lem:obtain_oracle} and Corollary~\ref{lem:obtain_oracle_x} in two aspects. First, we consider the function $G$ given in \eqref{eq:problem_Gsum_minmax}. Second, we do not use the assumption $m_h(H+2L_G+\mu_y) \leq L_h$ of Lemma~\ref{lem:obtain_oracle} and $2L_G+\mu_x \leq L_f$ of Corollary~\ref{lem:obtain_oracle_x}.



We start with considering problem \eqref{eq:problem_Gsum_minmax} under Assumption \ref{assumpt:G_sum}.1,2,3(a). 
This assumption combined with \eqref{eq:assumpt1_constants_G_sum} guarantees that Assumption~\ref{assumpt:framework} holds.
Lemma~\ref{lem:obtain_oracle_mh=1} and Corollary~\ref{lem:obtain_oracle_x_mh=1} guarantee that Assumptions~\ref{assumpt:framework_oracle}, \ref{assumpt:framework_oracle_x} hold.
This allows to combine Lemma~\ref{lem:obtain_oracle_mh=1} and Corollary~\ref{lem:obtain_oracle_x_mh=1} with either Theorem~\ref{theorem:general_framework} if $L_f \geq L_G$, or
Theorem \ref{theorem:general_framework_inverse} if $L_f \leq L_G$. 
The resulting complexity estimates for solving problem \eqref{eq:problem_Gsum_minmax} with our system of inner-outer loops are given in the next theorem which is proved in Appendix~\ref{Appendix_h_not_sum}.
Notice that in this case the algorithm is fully deterministic and we find an $\varepsilon$-solution to problem \eqref{eq:problem_Gsum_minmax}.

\begin{theorem}\label{theorem:G-sum_noprox}
Assume that for problem \eqref{eq:problem_Gsum_minmax} Assumption \ref{assumpt:G_sum}.1,2,3(a) holds and additionally $\mu_x\leq L_G$, $\mu_x\leq L_f$ and $\mu_y\leq L_G$.
Then using general framework from Section~\ref{section:saddleFramework}, general framework from Appendix~\ref{Appendix_Framework_inverse}, Lemma~\ref{lem:obtain_oracle_mh=1} and  Corollary~\ref{lem:obtain_oracle_x_mh=1} for each relation between $L_h, L_G$ and $L_f, L_G$ respectively we provide an algorithm, which finds an $\varepsilon$-solution to problem \eqref{eq:problem_Gsum_minmax} with the following number of basic oracle calls
\begin{align}
    &\nabla f \text{-oracle calls}: \widetilde{O} \left( \sqrt{\frac{L_G L_f}{\mu_x\mu_y}} 
    \right),\label{theorem:G-sum_f}\\
        &\nabla h \text{-oracle calls}: \widetilde{O} \left( \max \left\{ \sqrt{\frac{ L_G L_h}{\mu_x \mu_y}}, \sqrt{\frac{ L_G^2}{\mu_x \mu_y}} \right\}\right),\label{theorem:G-sum_h}\\
    &\nabla_x G_i \text{-oracle calls}: \widetilde{O} \left(m_G  \sqrt{\frac{L_G^2}{\mu_x\mu_y}}
    \right),\label{theorem:G-sum_G_x}\\ 
    &\nabla_y G_i \text{-oracle calls}: \widetilde{O} \left(m_G  \sqrt{\frac{L_G^2}{\mu_x\mu_y}} 
    \right).\label{theorem:G-sum_G_y}
\end{align}
\end{theorem}
We prove this theorem in Appendix~\ref{Appendix_h_not_sum}.

We would like to emphasize that even though we do not use variance reduction techniques in the algorithm described in Theorem \ref{theorem:G-sum_noprox}, under assumption \ref{assumpt:G_sum}.1,2,3(a) our bounds are better than the bounds obtained by variance reduction method proposed in \cite{NIPS2016_1aa48fc4}. To solve the problem \eqref{eq:problem_Gsum_minmax} by the algorithm of \cite{NIPS2016_1aa48fc4}, we need to restate this problem as
\[
\min _{x \in \mathbb{R}^{d_x}} \max _{y \in \mathbb{R}^{d_y}} \left\{\frac{1}{m_G} \sum_{i = 1}^{m_G} \left(\tilde{G}_i(x, y):=f(x)+ G_i(x, y) - h(y)\right)\right\}
\]
with the objective being $L_{\tilde{G}} = \max\{L_G+L_f, L_G+L_h\}$-smooth.
The algorithm in \cite{NIPS2016_1aa48fc4} does not propose a way to separate the complexities for different parts of the objective and the resulting number of oracle calls for each part is the same
\begin{align} \label{estimates_bach}
    \nabla f, \nabla h, \nabla_x G_i, \nabla_y G_i \text{-oracle calls}: \widetilde{O} \left( \sqrt{m_G} \frac{L_{\tilde{G}}}{\min\{\mu_x, \mu_y\}} \right).
\end{align}
%
Comparing these estimates with the estimates of Theorem~\ref{theorem:G-sum_noprox}, we make two important observations.
\begin{itemize}
    \item Due to our approach with complexity separation the estimates from Theorem~\ref{theorem:G-sum_noprox} on the number of oracle calls for $f$ and $h$ are always better than the corresponding estimates in \eqref{estimates_bach} at least by a factor $\sqrt{m_G}$.
    \item At first sight, the estimates on the number of calls of $\nabla_x G_i$ and $\nabla_y G_i$ from Theorem~\ref{theorem:G-sum_noprox} seem worse than the corresponding estimates in \eqref{estimates_bach} due to the additional factor $\sqrt{m_G}$. However, this is not the case, for example, when $L_f$ or $L_h$ are large enough leading to $L_{\tilde{G}} \gg L_G$. This can be demonstrated by taking $m_G L_G \leq L_f$, then the estimates on the number of calls of $\nabla_x G_i$ and $\nabla_y G_i$ in Theorem~\ref{theorem:G-sum_noprox} become $\sqrt{L_f^2/\mu_x\mu_y}$, which is smaller than the estimates in \eqref{estimates_bach}.
\end{itemize}
An interesting open question is whether we can improve the complexity bounds in Theorem~\ref{theorem:G-sum_noprox} by applying variance reduction methods to ensure Assumptions \ref{assumpt:framework_oracle}, \ref{assumpt:framework_oracle_x}. We conjecture that it is possible to improve the bounds \eqref{theorem:G-sum_G_x} and \eqref{theorem:G-sum_G_y} to $\widetilde{O} \left(\sqrt{\frac{m_GL_G^2}{\mu_x\mu_y}} \right)$.

As a particular case of problem \eqref{eq:problem_Gsum_minmax} we can consider problem \eqref{eq:problem_st_h_sum} with $m_h=1$. This allows to
relax the assumptions $m_h (4L_G + \mu_y) \leq L_h$, $ 2L_G + \mu_x   \leq L_f, \mu_y  \leq L_G$ made in Corollary \ref{corollary:m_h=1} and obtain the following corollary of the previous theorem.
Notice that again in this case the algorithm is fully deterministic and we find an $\varepsilon$-solution to problem \eqref{eq:problem_st_h_sum}.

\begin{corollary}\label{theorem:h-not-sum}
Assume that for problem \eqref{eq:problem_st_h_sum} Assumption \ref{assumpt:h_sum} holds and additionally 
$m_h =1$, $\mu_x\leq L_G$, $\mu_x\leq L_f$ and $\mu_y\leq L_G$. 
Then, using the general framework from Section~\ref{section:saddleFramework}, the general framework from Appendix~\ref{Appendix_Framework_inverse} and Lemma~\ref{lem:obtain_oracle_mh=1} with Corollary~\ref{lem:obtain_oracle_x_mh=1} for each relation between $L_h, L_G$ and $L_f, L_G$ respectively, we provide an algorithm, which finds an $\varepsilon$-solution to problem \eqref{eq:problem_st_h_sum} with the following number of basic oracle calls

\begin{align}
    &\nabla f \text{-oracle calls}: \widetilde{O} \left( \sqrt{\frac{L_G L_f}{\mu_x\mu_y}} 
    \right),\label{theorem:h-not-sum_f}\\
        &\nabla h \text{-oracle calls}: \widetilde{O} \left( \max \left\{ \sqrt{\frac{ L_G L_h}{\mu_x \mu_y}}, \sqrt{\frac{ L_G^2}{\mu_x \mu_y}} \right\}\right),\label{theorem:h-not-sum_h}\\
    &\nabla_x G \text{-oracle calls}: \widetilde{O} \left(  \sqrt{\frac{L_G^2}{\mu_x\mu_y}}
    \right),\label{theorem:h-not-sum_G_x}\\ 
    &\nabla_y G \text{-oracle calls}: \widetilde{O} \left(  \sqrt{\frac{L_G^2}{\mu_x\mu_y}} 
    \right).\label{theorem:h-not-sum_G_y}
\end{align}
\end{corollary}


We now turn to the problem \eqref{eq:problem_Gsum_minmax} under Assumption \ref{assumpt:G_sum}.1,2,3(b). 
This assumption combined with \eqref{eq:assumpt1_constants_G_sum} guarantees that Assumption~\ref{assumpt:framework} holds.
The part 3(b)  allows a simple  construction, which is given in the proof of Lemma~\ref{lem:obtain_oracle_mh=1_h-prox} in Appendix~\ref{Appendix_h_not_sum}, to guarantee Assumption \ref{assumpt:framework_oracle}. The main difference with Lemma \ref{lem:obtain_oracle_mh=1} is that due to the prox-friendliness of $h$ the second loop is not needed and it is sufficient to apply just Algorithm~\ref{alg:restarts_inexact_notconvex} to solve problem \eqref{eq:framework_g_obt_oracle} in Assumption \ref{assumpt:framework_oracle}.  Corollary~\ref{lem:obtain_oracle_x_mh=1} guarantees that Assumptions~\ref{assumpt:framework_oracle_x} holds.
This allows to combine Lemma~\ref{lem:obtain_oracle_mh=1_h-prox} and Corollary~\ref{lem:obtain_oracle_x_mh=1} with either Theorem~\ref{theorem:general_framework} if $L_f \geq L_G$, or
Theorem \ref{theorem:general_framework_inverse} if $L_f \leq L_G$. 
The resulting complexity estimates for solving problem \eqref{eq:problem_Gsum_minmax} with our system of inner-outer loops are given in the next theorem which is proved in Appendix~\ref{Appendix_h_not_sum}.
Notice that in this case the algorithm is fully deterministic and we find an $\varepsilon$-solution to problem \eqref{eq:problem_Gsum_minmax}.



\begin{theorem}\label{theorem:G-sum_h_prox}
Assume that for problem \eqref{eq:problem_Gsum_minmax} Assumption \ref{assumpt:G_sum}.1,2,3(b) holds and additionally $\mu_x\leq L_G$, $\mu_x\leq L_f$ and $\mu_y\leq L_G$.
Then, using the general framework from Section~\ref{section:saddleFramework}, the general framework from Appendix~\ref{Appendix_Framework_inverse} and Lemma~\ref{lem:obtain_oracle_mh=1_h-prox} with Corollary~\ref{lem:obtain_oracle_x_mh=1} for each relation between $L_h, L_G$ and $L_f, L_G$ respectively, we provide an algorithm, which finds an $\varepsilon$-solution to problem \eqref{eq:problem_Gsum_minmax} with the following number of basic oracle calls
\begin{align}
    &\nabla f \text{-oracle calls}: \widetilde{O} \left( \sqrt{\frac{L_G L_f}{\mu_x\mu_y}} 
    \right),\\
        &\nabla h \text{-oracle calls}: \widetilde{O} \left( \sqrt{\frac{L_G}{\mu_y}} 
    \right),\\
    &\nabla_x G_i \text{-oracle calls}: \widetilde{O} \left( m_G \sqrt{\frac{  L_G^2}{\mu_x\mu_y}} 
    \right),\\ 
    &\nabla_y G_i \text{-oracle calls}: \widetilde{O} \left( m_G \sqrt{\frac{L_G^2}{\mu_x\mu_y}}
    \right).
\end{align}
\end{theorem}
We prove this theorem in Appendix~\ref{Appendix_h_not_sum}.

\begin{remark}
In this remark using the results from \cite{song2021variance} we show how we can utilise our approach to solve the problems of structured nonsmooth convex finite-sum optimization that appears widely in machine learning applications, including support vector machines and least absolute deviation. 

We consider large-scale regularized nonsmooth convex empirical risk minimization (ERM) of linear predictors in machine learning. Let $b_i\in \mathbb{R}^n_{x}$, $i=1,2,\dots n$, be sample vectors with n typically large; $f_i: \mathbb{R} \to \mathbb{R}$, $i=1,2,\dots n$, be possibly nonsmooth convex loss functions associated with the linear predictor $\la b_i,x \ra$.
The problem we study is:
\begin{equation} 
\label{eq:remark_ml_problem}
    \min _{x \in \mathbb{R}^{d_x}} \max _{y \in \mathbb{R}^{d_y}} \left\{\frac{1}{n}\sum_{i=1}^{n}f_i(\la b_i, x\ra)+ G(x, y) - h(y)\right\}. \;\;
\end{equation}
We require that the convex conjugates of the functions $f_i$, defined by $f_i^*(z_i) := \max_{\xi_i} (\xi_i z_i − f_i(\xi_i))$, admit efficiently computable proximal operators. Thus, we can rewrite the function $\frac{1}{n}\sum_{i=1}^{n}f_i(\la b_i, x\ra)$ in the following way:
\begin{equation}
\label{eq:remark_ml_g}
\frac{1}{n}\sum_{i=1}^{n}f_i(\la b_i, x\ra)=\frac{1}{n}\sum_{i=1}^{n} \max_{z_i} (z_i \la b_i, x\ra − f^*_i(z_i))=\max_{z\in \mathbb{R}^n} \left\{ \la z,Bx \ra -\frac{1}{n}\sum_{i=1}^{n} f^*_i(z_i) \right\}
\end{equation}
where $y=(y_1,\dots , y_n)$, $B=\frac{1}{n} [b_1,\dots ,b_n]^T $. Then by substitution of the equation \eqref{eq:remark_ml_g} into the problem \eqref{eq:remark_ml_problem}, we obtain:
\begin{equation} 
\label{eq:remark_ml_problem_primal_dual}
    \min _{x \in \mathbb{R}^{d_x}}\left\{ \max _{y \in \mathbb{R}^{d_y}} \left\{ G(x, y) - h(y)\right\} +\max_{z\in \mathbb{R}^n} \left\{ \la z,Bx \ra -\frac{1}{n}\sum_{i=1}^{n} f^*_i(z_i) \right\} \right\}. \;\;
\end{equation}
We can use  another notation $\eta =(y,z)$ and rewrite the problem  \eqref{eq:remark_ml_problem_primal_dual} as follow:
\begin{equation} 
\label{eq:remark_ml_problem_primal_dual_rewr}
    \min _{x \in \mathbb{R}^{d_x}}\left\{ \max _{\eta=(y,z) \in \mathbb{R}^{d_y+n}} \left\{ G(x, y) - h(y)+ \la z,Bx \ra -\frac{1}{n}\sum_{i=1}^{n} f^*_i(z_i) \right\} \right\},
\end{equation}
which we can solve using the general framework from Section~\ref{section:saddleFramework} under the differences assumptions. It is worth mentioning that the function $f^*(z)=\frac{1}{n}\sum_{i=1}^{n} f^*_i(z_i)$ is separable and admits an efficiently computable proximal operator. Thus primal-dual problem \eqref{eq:remark_ml_g} has significantly lower complexity than the saddle-point problem \eqref{eq:remark_ml_problem}. That means we can use primal-dual approach with no care that the saddle-problem \eqref{eq:remark_ml_problem_primal_dual_rewr} become more complex.

\demo
\end{remark}

\section {Accelerated Proximal Variance-Reduction Method for Saddle-Point Problems}
\label{S:G_sum_prox-friendly}


In this section we consider problem \eqref{eq:problem_Gsum} (which is problem \eqref{eq:problem_Gsum_minmax}), but under assumption that $f$ and $h$ are prox-friendly. This does not allow us to use Algorithm \ref{alg:restarts_inexact_notconvex} since it requires to evaluate inexact gradients for $f$ and $h$ (see step \ref{Step:AM_grad_step} of this algorithm). Thus, we exploit that  $f$ and $h$ are prox-friendly and utilize proximal variance reduction methods to avoid calculation of the gradients for these two functions.  
We start with describing two building blocks for our algorithm: the Catalyst framework \cite{catalyst2015nips} adapted and slightly generalized for our setting and variance reduction algorithm SAGA proposed in \cite{NIPS2016_1aa48fc4}, which we also adapt to our setting. 
The former algorithm is an optimization algorithm, the latter is designed for saddle-point problems, and we use these algorithms in the system of inner-outer loops as in the previous sections. Thus, we need to connect the output of these algorithms with the requirements of outer loops. To do this we prove several technical lemmas. Finally in the last subsection we collect all the pieces together and describe the loops of our algorithm as well as present its complexity theorem. 



\subsection{Problem statement}
In this section we consider problem \eqref{eq:problem_Gsum_minmax} under the following assumption.
\begin{assumption}\label{assumpt:G_sum_1}
\begin{enumerate}
    
    \item $f(x)$ is $\mu_x$-strongly convex, $h(y)$ is $\mu_y$-strongly convex.
    \item Each functions $G_i(x, y)$, $i \in 1, \dots, m_G$  is convex-concave and $L_G^i$-smooth, i.e. for each $(x_1, x_2), (y_1, y_2) \in  \mathbb{R}^{d_x} \times \mathbb{R}^{d_y}$
\begin{equation}
    \|\nabla G_i (x_1,x_2)-\nabla G_i (y_1,y_2)\|\leq L_G^i \|(x_1, x_2) - (y_1, y_2)\|.
\end{equation}
\item  $f(x), h(y)$ are prox-friendly (smoothness is not required).
\end{enumerate}
\end{assumption}

We also use slightly different, more convenient for the setting of this section, and more classical definition of an inexact solution to problem \eqref{eq:problem_Gsum_minmax}.
\begin{definition}\label{def:solution_saddle-point_problem}
A point $(\hat{x}, \hat{y})$ is called an $(\varepsilon, \sigma)$ solution to the saddle-point problem \eqref{eq:problem_Gsum_minmax}, if with probability at least $1 - \sigma$, the following inequality is true
\begin{equation}\label{def:saddle_point_solution}
    \max_{y \in \R^{d_y}}\left\{f(\hat{x}) + G(\hat{x}, y) - h(y)\right\} - \min_{x \in \R^{d_x}}\left\{f(x) + G(x, \hat{y}) - h(\hat{y})\right\} \leq \varepsilon.
\end{equation}
\end{definition}
Note that since the saddle-point problem is strongly-convex-strongly-concave, the quantity in the l.h.s. of \eqref{def:saddle_point_solution} is correctly defined.



\subsection{Algorithmic Building Blocks} \label{subsec:proxfr_f_h}
In this subsection we consider the algorithms are used in general algorithm to find an $(\varepsilon, \sigma)$ solution to the problem 
\eqref{eq:problem_Gsum_minmax} under the Assumption~\ref{assumpt:G_sum_1}. In each paragraph we describe the problem is solved by this algorithm with certain assumptions and formulate convergence rate and complexity theorems.

\paragraph{\textbf{The Catalyst metaalgorithm \cite{catalyst2015nips,catalyst2017}.}}
Let us consider the problem 
\begin{equation}\label{problem:Catalyst}
    \min_{x \in \R^{d_x}}\left\{F(x) := \varphi(x) + \psi(x)\right\}
\end{equation}
under the following assumption:
\begin{assumption}\label{assumpt:Catalyst}
\begin{enumerate}
    \item $\varphi(x)$ is convex;
    \item $\varphi(x)$ has Lipschitz continuous derivatives with constant $L$;
    \item $\psi(x)$ is $\mu$-strongly convex (may not be differentiable).
\end{enumerate}
\end{assumption}
To solve the problem \eqref{problem:Catalyst} under the Assumption \ref{assumpt:Catalyst} we can apply the Catalyst algorithm from \cite{catalyst2015nips,catalyst2017}. In the Theorem \ref{corollary:Catalyst_sigma} we show how $\left(\varepsilon_k\right)_{k \geq 0}$ or $\left(\delta_k\right)_{k \geq 0}$ are chosen to get optimal complexity of finding an $(\varepsilon, \sigma)$ solution to this problem which is understood in the sense of Definition \ref{def:solution_saddle-point_problem}. 

\begin{algorithm}[h]
	\caption{Catalyst \cite{catalyst2015nips,catalyst2017}\label{alg:Catalyst}}
	\begin{algorithmic}[1]
		\STATE {\bf Input:} Initial estimate $x_0 \in \mathbb{R}^{d_x},$ smoothing parameter $H$, strong convexity parameter $\mu$, optimization method $\mathcal{M}$ and a stopping criterion based on a sequence of accuracies $\left(\varepsilon_k\right)_{k\leq 0}$, or $\left(\delta_k\right)_{k\leq 0}$, or a fixed budget $T$.
		\STATE Initialize $q = \frac{\mu}{\mu + H}$,  $x_0^{md} = x_0$, $\alpha_0 = \sqrt{q}$;
		\WHILE{the desired accuracy is not achieved}
			\STATE Find an approximate solution of the following problem using $\mathcal{M}$ 
			
			    $$x_{k} \approx \underset{x \in \mathbb{R}^{d_x}}{\arg\min }\left\{S_k(x) :=  \varphi(x)+\psi(x)+\frac{H}{2}\|x-x_{k-1}^{md}\|_2^{2}\right\} $$
			   using one of the following stopping criteria:
			   \begin{enumerate}
			       \item \textit{absolute accuracy}: find $x_k$ such that $S_k(x_k) - S_k(x_k^*) \leq \varepsilon_k$, where $x_k^* = \arg \underset{x \in \R^{d_x}}{\min}S_k(x)$;
			       \item \textit{relative accuracy}: find $x_k$ such that $S_k(x_k) - S_k(x_k^*) \leq \frac{H\delta_k}{2}\|x_k - x_{k-1}^{md}\|_2^2$, where $x_k^* = \arg \underset{x \in \R^{d_x}}{\min}S_k(x)$; 
			       \item \textit{fixed budget}: run  $\mathcal{M}$ for $T$ iterations and output $x_k$.
			   \end{enumerate}
			   
			\STATE Update $\alpha_k \in (0,1)$ from equation $\alpha_k^2 = (1-\alpha_k)\alpha_{k-1}^2 + q\alpha_k$;
			\STATE Compute $x_k^{md}$ with Nesterov's extrapolation step
			
			$$
			    x_k^{md}= x_k + \beta_k(x_k - x_{k-1}) \; \; \; 
			\text{with} \; \; \;  \beta_k = \frac{\alpha_{k-1}(1-\alpha_{k-1})}{\alpha_{k-1}^2+\alpha_k}$$
			
		\ENDWHILE
		\STATE {\bf Output:} $x_k$ (final estimate).
	\end{algorithmic}
\end{algorithm}

\begin{theorem}[Theorem 3.1 from \cite{catalyst2015nips}] \label{th:Catalyst_absolute_accuracy}
\\
Choose 
\begin{equation}
    \varepsilon_k = \frac{2}{9}(F(x_0) - F(x^*))(1-\rho)^k \ \ \ \ \text{with} \ \ \ \ \rho \leq \sqrt{q}
\end{equation}
Then, the Catalyst algorithm (Algorithm \ref{alg:Catalyst}) with absolute accuracy generate iterates $(x_k)_{k \geq 0}$ such that 

\begin{equation}
    F(x_k) - F(x^*) \leq C(1 - \rho)^{k+1}(F(x_0) - F(x^*)) \ \ \ \ \text{with} \ \ \ \ C = \frac{8}{(\sqrt{q} - \rho)^2}.
\end{equation}
\end{theorem}

\begin{theorem}[Proposition 8 from \cite{catalyst2017}] \label{th:Catalyst_relative_accuracy}
\\
Choose 
\begin{equation}
    \delta_k = \frac{\sqrt{q}}{2 - \sqrt{q}}
\end{equation}
Then, the Catalyst algorithm (Algorithm \ref{alg:Catalyst}) with relative accuracy generate iterates $(x_k)_{k \geq 0}$ such that 

\begin{equation}
    F(x_k) - F(x^*) \leq 2\left(1 - \frac{\sqrt{q}}{2}\right)^{k}\left(F(x_0) - F(x^*)\right).
\end{equation}
\end{theorem}

\begin{corollary}\label{corollary:accuracy_catalyst}
Choose  
\begin{equation}
    \varepsilon_k = \frac{2}{9}(F(x_0) - F(x^*))(1-\rho)^k \ \ \ \ \text{with} \ \ \ \ \rho = 0.9\sqrt{q}
\end{equation}
in the Catalyst algorithm (Algorithm \ref{alg:Catalyst}) with absolute accuracy, or
\begin{equation}
    \delta_k = \frac{\sqrt{q}}{2 - \sqrt{q}}
\end{equation}
in the Catalyst algorithm (Algorithm \ref{alg:Catalyst}) with relative accuracy. Then, after the number of iterations  
\begin{equation}
    \mathcal{N} = \widetilde{O}\left(\max\left\{1, \sqrt{\frac{H}{\mu}}\right\}\right)
\end{equation} of the Catalyst algorithm (Algorithm \ref{alg:Catalyst}) we get $x_\mathcal{N}$ such that $F(x_\mathcal{N}) - F(x^*) \leq \varepsilon$
\end{corollary}

\begin{proof}
\begin{enumerate}
    \item
    \textsc{Absolute accuracy.}
\\
By the Theorem \ref{th:Catalyst_absolute_accuracy} the number of iterations $\mathcal{N}$ of the Catalyst algorithm (Algorithm \ref{alg:Catalyst}) with absolute accuracy to guarantee an accuracy of $\varepsilon$ needs to satisfy
\begin{equation}
     F(x_k) - F(x^*) \leq C(1 - \rho)^{\mathcal{N}+1}(F(x_0) - F(x^*)) \leq C(1-\rho) e^{-\rho\mathcal{N}}(F(x_0) - F(x^*)) \leq \varepsilon,
\end{equation}
which gives
\begin{equation}
    \mathcal{N} = \left\lceil \frac{1}{\rho}\ln{\frac{C(1-\rho)(F(x_0) - F(x^*))}{\varepsilon}}\right\rceil = \left\lceil \frac{1}{\rho}\ln{\frac{8(1-\rho)(F(x_0) - F(x^*))}{(\sqrt{q} - \rho)^2\varepsilon}} \right\rceil
\end{equation}
Choose $\rho = 0.9\sqrt{q}$
\begin{multline}\label{eq:absolute_accuracy}
    \mathcal{N} = \left\lceil \frac{1}{0.9 \sqrt{q}}\ln{\frac{8(1-0.9\sqrt{q})(F(x_0) - F(x^*))}{(0.1\sqrt{q})^2\varepsilon}}\right\rceil =
    \\
    \left\lceil \frac{\sqrt{\mu + H}}{0.9 \sqrt{\mu}}\ln{\frac{8(1-0.9\sqrt{\mu/(\mu + H)})(F(x_0) - F(x^*))(H+\mu)}{0.01\mu\varepsilon}}\right\rceil =
    \\
    \widetilde{O}\left(\sqrt{1 + \frac{H}{\mu}} \right)= \widetilde{O}\left(\max\left\{1, \sqrt{\frac{H}{\mu}}\right\}\right)
\end{multline}
\item
    \textsc{Relative accuracy.}
    By the Theorem \ref{th:Catalyst_relative_accuracy} the number of iterations $\mathcal{N}$ of the Catalyst algorithm (Algorithm \ref{alg:Catalyst}) with relative accuracy to guarantee an accuracy of $\varepsilon$ needs to satisfy
\begin{equation}
     F(x_k) - F(x^*) \leq 2\left(1 - \frac{\sqrt{q}}{2}\right)^{\mathcal{N}}(F(x_0) - F(x^*)) \leq 2e^{-\frac{\sqrt{q}}{2}\mathcal{N}}(F(x_0) - F(x^*)) \leq \varepsilon,
\end{equation}
which gives
\begin{multline}\label{eq:relative_accuracy}
    \mathcal{N} = \left\lceil \frac{2}{\sqrt{q}}\ln{\frac{2(F(x_0) - F(x^*))}{\varepsilon}}\right\rceil = \left\lceil \frac{2\sqrt{H+\mu}}{\sqrt{\mu}}\ln{\frac{2(F(x_0) - F(x^*))}{\varepsilon}} \right\rceil =
    \\
     \widetilde{O}\left(\sqrt{1 + \frac{H}{\mu}} \right)= \widetilde{O}\left(\max\left\{1, \sqrt{\frac{H}{\mu}}\right\}\right)
\end{multline}
\end{enumerate}
\smartqed 
\end{proof}

In each iteration of the Catalyst algorithm we need to solve the problem 

\begin{equation} \label{eq:Catalyst_M}
    \min_{x \in \R^{d_x}}{S_k(x)} = \min_{x \in \R^{d_x}}\left\{ F(x) + \frac{H}{2}\|x - x_{k-1}^{md}\|^2\right\}
\end{equation}
where $F(x) := \varphi(x) + \psi(x)$, with an inner method $\mathcal{M}$.

Assume that $\mathcal{M}$ is linearly convergent for strongly convex problems with parameter $\tau_{\mathcal{M}}$ according to 
\begin{equation}\label{eq:M_det}
        S(z_t) - S(z^*)\leq C_{\mathcal{M}}(1-\tau_{\mathcal{M}})^t(S(z_0) - S(z^*)),
    \end{equation} 
in the deterministic case or according to 
\begin{equation}\label{eq:M_rand}
        \mathbb{E}[S(z_t) - S(z^*)]\leq C_{\mathcal{M}}(1-\tau_{\mathcal{M}})^t(S(z_0) - S(z^*)),
    \end{equation}
in the randomized case. 
\begin{theorem}[Lemma 11 from \cite{catalyst2017}]\label{th:M_abs}
 Assume further that $(\varepsilon_k)_{k\geq0}$ in the Catalyst algorithm (Algorithm \ref{alg:Catalyst}) with absolute accuracy are chosen according to the Corollary \ref{corollary:accuracy_catalyst}. At iteration $k$ of this algorithm we consider the following function \eqref{eq:Catalyst_M}, which we minimize with $\mathcal{M}$, producing a sequence $(z_t)_{t\geq0}$. Then, the complexity $T_{k} = \inf\{t \geq 0, S_{k}(z_t) - S_{k}(z^*) \leq \varepsilon_{k}\}$ satisfies  

\begin{enumerate}
    \item If $\mathcal{M}$ is deterministic and satisfies \eqref{eq:M_det}, we have
    \begin{equation}
        T_{k}(\varepsilon_{k}) \leq \frac{1}{\tau_{\mathcal{M}}}\ln\left(\frac{C_{\mathcal{M}}C_k}{\varepsilon_{k}}\right), \ \ \ \text{where} \ \ \ C_k = (S_{k}(z_0) - S_{k}(z_{k}^*)).
    \end{equation}
    \item If $\mathcal{M}$ is randomized and satisfies \eqref{eq:M_rand},
     we have
    \begin{equation}\label{constant:C_1}
        \mathbb{E}[T_{k}(\varepsilon_{k})] \leq \frac{1}{\tau_{\mathcal{M}}}\ln\left(\frac{C_{\mathcal{M}}C_k}{\varepsilon_{k}}\right) + 1, \ \ \ \text{where} \ \ \ C_k = \frac{2(S_{k}(z_0) - S_{k}(z_{k}^*))}{\tau_{\mathcal{M}}}.
    \end{equation}
\end{enumerate}
\end{theorem}

\begin{theorem}[Corollary 16 from \cite{catalyst2017}]\label{th:M_rel}
Assume further that $(\delta_k)_{k\geq0}$ in the Catalyst algorithm (Algorithm \ref{alg:Catalyst}) with relative accuracy are chosen according to the Corollary \ref{corollary:accuracy_catalyst}. At iteration $k$ of this algorithm we consider the following function \eqref{eq:Catalyst_M}, which we minimize with $\mathcal{M}$, producing a sequence $(z_t)_{t\geq0}$. Then, the complexity $T_{k} = \inf\{t \geq 0, S_{k}(z_t) - S_{k}(z^*) \leq \frac{H\delta_{k}}{2}\|z_t - x_k^{md}\|_2^2\}$ is satisfies  

\begin{enumerate}
    \item If $\mathcal{M}$ is deterministic and satisfies \eqref{eq:M_det}, we have
    \begin{equation}
        T_{k}(\delta_{k}) \leq \frac{1}{\tau_{\mathcal{M}}}\ln\left(\frac{C_{\mathcal{M}}C_k}{\delta_{k}}\right) \ \ \ \text{where} \ \ \ C_k = \frac{4(L+H)}{H}.
    \end{equation}
    \item If $\mathcal{M}$ is randomized and satisfies \eqref{eq:M_rand}, we have
    \begin{equation}\label{constant:C_2}
        \mathbb{E}[T_{k}(\delta_{k})] \leq \frac{1}{\tau_{\mathcal{M}}}\ln\left(\frac{C_{\mathcal{M}}C_k}{\delta_{k}}\right) + 1, \ \ \ \text{where} \ \ \ C_k = \frac{8(L+ H)}{\tau_{\mathcal{M}}H}.
    \end{equation}
\end{enumerate}
\end{theorem}

\begin{corollary}\label{corollary:sigma_k}
 Assume further that $(\varepsilon_k)_{k\geq0}$ or $(\delta_k)_{k\geq0}$ in the Catalyst algorithm (Algorithm \ref{alg:Catalyst}) with absolute or relative accuracy are chosen according to the Corollary \ref{corollary:accuracy_catalyst}. At iteration $k$ of this algorithm we consider the following function \eqref{eq:Catalyst_M}, which we minimize with an randomized method $\mathcal{M}$, producing a sequence $(z_t)_{t\geq0}$. Then, after 
\begin{enumerate}
    \item Absolute accuracy case:
    \begin{equation}
        T_{k}(\varepsilon_k \sigma_k) = O\left(\frac{1}{\tau_{\mathcal{M}}} \ln{\frac{C_{\mathcal{M}}C_k}{\varepsilon_k \sigma_k}}\right),
    \end{equation}
    where $C_k$ is the constant defined in \eqref{constant:C_1}.
    \item Relative accuracy case:
    \begin{equation}
        T_{k}(\delta_k \sigma_k) = O\left(\frac{1}{\tau_{\mathcal{M}}} \ln{\frac{C_{\mathcal{M}}C_k}{\delta_k \sigma_k}}\right),
    \end{equation}
    where $C_k$ is the constant defined in \eqref{constant:C_2}.
\end{enumerate} iterations of the randomized method $\mathcal{M}$ we get an $(\varepsilon_k, \sigma_k)$ solution to the problem \eqref{eq:Catalyst_M} which is understood in the sense of Definition \ref{def:solution_saddle-point_problem}.
\end{corollary}

\begin{proof} To solve the problem \eqref{eq:Catalyst_M} we apply an randomized method $\mathcal{M}$.
\begin{enumerate}
    \item
    \textsc{Absolute accuracy.}
    \\
     By the Theorem \ref{th:M_abs} after $T_k(\varepsilon_k)=O \left( \frac{1}{\tau_{\mathcal{M}}}\ln{\frac{C_{\mathcal{M}}C_k}{\varepsilon_k}}\right)$ iterations of the randomized method $\mathcal{M}$ we obtain that we can find $z_{T_k}(\varepsilon_k)$ such that
\begin{equation}
    \mathbb{E}(S_k(z_{T_k}(\varepsilon_k)) - S_k(z^*)) \leq \varepsilon_k
\end{equation}

Since $S_k(z_{T_k}(\varepsilon_k)) - S_k(z^*) \geq  0$, with an arbitrary $\sigma_k \in (0, 1)$ we can apply Markov inequality:
\begin{equation}
\mathbb{P}(S_k(z_{T_k}(\varepsilon_k\sigma_k)) - S_k(z^*) \leq \varepsilon_k) \geq 1 - \frac{\mathbb{E}(S_k(z_{T_k}(\varepsilon_k)) - S_k(z^*))}{\varepsilon_k} \geq 1 - \sigma_k,
\end{equation}
where 
\begin{equation}
    \frac{\mathbb{E}(S_k(z_{T_k}(\varepsilon_k)) - S_k(z^*))}{\varepsilon_k} \leq \sigma_k \ \ \ \rightarrow \ \ \ \mathbb{E}(S_k(z_{T_k}(\varepsilon_k)) - S_k(z^*)) \leq \varepsilon_k\sigma_k.
\end{equation}
Then, after $ T_k(\varepsilon_k\sigma_k) = O \left( \frac{1}{\tau_{\mathcal{M}}}\ln{\frac{C_{\mathcal{M}}C_k}{\varepsilon_k\sigma_k}}\right)$ iterations of the randomized method $\mathcal{M}$ we can find $(\varepsilon_k, \sigma_k)$ solution of the problem \eqref{eq:Catalyst_M} with absolute accuracy.
\item \textsc{Relative accuracy.}
    \\
    By the Theorem \ref{th:M_rel} after $T_k(\delta_k)=O \left( \frac{1}{\tau_{\mathcal{M}}}\ln{\frac{C_{\mathcal{M}}C_k}{\delta_k}}\right)$ iterations of the randomized method $\mathcal{M}$ we obtain that we can find $z_{T_k}(\delta_k)$ such that
\begin{equation}
    \mathbb{E}(S_k(z_{T_k}(\delta_k)) - S_k(z^*)) \leq \frac{\delta_k H}{2}\|z_{T_k}(\delta_k) - x_{k-1}^{md}\|_2^2
\end{equation}

Since $S_k(z_{T_k}(\delta_k)) - S_k(z^*) \geq  0$, with an arbitrary $\sigma_k \in (0, 1)$ we can apply Markov inequality:
\begin{equation}
\mathbb{P}\left(S_k(z_{T_k}(\delta_k\sigma_k)) - S_k(z^*) \leq \frac{\delta_k H}{2}\|z_{T_k}(\delta_k\sigma_k) - x_{k-1}^{md}\|_2^2\right) \geq 1 - \frac{2\mathbb{E}(S_k(z_{T_k}(\delta_k \sigma_k)) - S_k(z^*))}{\delta_k H \|z_{T_k}(\delta_k \sigma_k) - x_{k-1}^{md}\|_2^2} \geq 1 - \sigma_k,
\end{equation}
where 
\begin{equation}
    \frac{2\mathbb{E}(S_k(z_{T_k}(\delta_k \sigma_k)) - S_k(z^*))}{\delta_k H \|z_{T_k}(\delta_k \sigma_k) - x_{k-1}^{md}\|_2^2} \leq \sigma_k \ \ \ \rightarrow \ \ \ 
    \mathbb{E}(S_k(z_{T_k}(\delta_k\sigma_k)) - S_k(z^*)) \leq \frac{\delta_k\sigma_k H}{2}\|z_{T_k}(\delta_k \sigma_k) - x_{k-1}^{md}\|_2^2
\end{equation}
Then, after $ T_k(\delta_k\sigma_k) = O \left( \frac{1}{\tau_{\mathcal{M}}}\ln{\frac{C_{\mathcal{M}}C_k}{\delta_k\sigma_k}}\right)$ iterations of the randomized method $\mathcal{M}$ we can find an $(\varepsilon_k, \sigma_k)$ solution of the problem \eqref{eq:Catalyst_M} with relative accuracy, where $\varepsilon_k = \frac{\delta_k H}{2}\|z_{T_k}(\delta_k \sigma_k) - x_{k-1}^{md}\|_2^2$.
\end{enumerate}
\qed \end{proof}

\begin{theorem}\label{corollary:Catalyst_sigma}
Choose  
\begin{enumerate}
    \item  
    \begin{equation}\label{varepsilon_k_Catalyst_absolute}
    \varepsilon_k = \frac{2}{9}(F(x_0) - F(x^*))(1-\rho)^k \ \ \ \ \text{with} \ \ \ \ \rho = 0.9\sqrt{q}
\end{equation}
 and 
    \begin{equation}\label{sigma_k_Catalyst_absolute}
    \sigma_k \leq \frac{\ln{\frac{1}{1-\sigma}}}{\frac{\sqrt{\mu + H}}{0.9 \sqrt{\mu}}\ln{\frac{8(1-0.9\sqrt{\mu/(\mu + H)})(F(x_0) - F(x^*))(H+\mu)}{0.01\mu\varepsilon}}}
\end{equation} 
    in the absolute accuracy case.
    \item 
    \begin{equation}
    \delta_k = \frac{\sqrt{q}}{2 - \sqrt{q}}
\end{equation}
 and 
    \begin{equation}
    \sigma_k \leq \frac{\ln{\frac{1}{1-\sigma}}}{\frac{2\sqrt{H+\mu}}{\sqrt{\mu}}\ln{\frac{2(F(x_0) - F(x^*))}{\varepsilon}}}
\end{equation}
in the relative accuracy case.
\end{enumerate} 
 Then, after 
\begin{equation}
     \mathcal{N} = \widetilde{O}\left(\max\left\{1, \sqrt{\frac{H}{\mu}}\right\}\right)
\end{equation}
the number of iterations of the Catalyst algorithm (Algorithm \ref{alg:Catalyst}) with absolute or relative accuracy, we find an $(\varepsilon, \sigma)$ solution to the original problem \eqref{problem:Catalyst} under the Assumption \ref{assumpt:Catalyst} which is understood in the sense of Definition \ref{def:solution_saddle-point_problem}. 
\end{theorem}

\begin{proof} 
In each iterations of the Catalyst algorithm we need to solve the problem \eqref{eq:Catalyst_M} with absolute or relative accuracy. To solve this problem, we apply an randomized method $\mathcal{M}$ and by Corollary \ref{corollary:sigma_k} we get an $(\varepsilon_k, \sigma_k)$ solution of the problem \eqref{eq:Catalyst_M}. Then, by Corollary \ref{corollary:accuracy_catalyst}, after
\begin{equation}
    \mathcal{N} = \left\lceil \frac{\sqrt{\mu + H}}{0.9 \sqrt{\mu}}\ln{\frac{8(1-0.9\sqrt{\mu/(\mu + H)})(F(x_0) - F(x^*))(H+\mu)}{0.01\mu\varepsilon}}\right\rceil
\end{equation}
the number of iterations of the Catalyst algorithm (Algorithm \ref{alg:Catalyst}) with absolute accuracy or after
\begin{equation}
    \mathcal{N} = \left\lceil \frac{2\sqrt{H+\mu}}{\sqrt{\mu}}\ln{\frac{2(F(x_0) - F(x^*))}{\varepsilon}} \right\rceil
\end{equation}
the number of iterations of the Catalyst algorithm (Algorithm \ref{alg:Catalyst}) with relative accuracy, we solve the original problem with probability $\prod\limits_{i=1}^{\mathcal{N}}{(1-\sigma_i)}$. Choose $\sigma_1 = \dots = \sigma_{\mathcal{N}} = \sigma_{\mathcal{M}}$ then, we solve the original problem with probability 
\begin{equation}
    1 - \sigma \leq (1 - \sigma_{\mathcal{M}})^{\mathcal{N}} \leq e^{-\mathcal{N}\sigma_{\mathcal{M}}}  \ \ \ \rightarrow \ \ \  \sigma_{\mathcal{M}} \leq \frac{\ln{\frac{1}{1-\sigma}}}{\mathcal{N}} .
\end{equation}
If we choose $(\varepsilon_k)_{k\geq 0}$ according to the Corollary \ref{corollary:accuracy_catalyst} and
\begin{equation}
    \sigma_k \leq \frac{\ln{\frac{1}{1-\sigma}}}{\frac{\sqrt{\mu + H}}{0.9 \sqrt{\mu}}\ln{\frac{8(1-0.9\sqrt{\mu/(\mu + H)})(F(x_0) - F(x^*))(H+\mu)}{0.01\mu\varepsilon}}}
\end{equation}
in the Catalyst algorithm (Algorithm \ref{alg:Catalyst}) with absolute accuracy 
or $(\delta_k)_{k\geq 0}$ according to the Corollary \ref{corollary:accuracy_catalyst} in the Catalyst algorithm (Algorithm \ref{alg:Catalyst}) with relative accuracy and 
\begin{equation}
    \sigma_k \leq \frac{\ln{\frac{1}{1 - \sigma}}}{\frac{2\sqrt{H+\mu}}{\sqrt{\mu}}\ln{\frac{2(F(x_0) - F(x^*))}{\varepsilon}}}.
\end{equation}
Then, after 
\begin{equation}
     \mathcal{N} = \widetilde{O}\left(\max\left\{1, \sqrt{\frac{H}{\mu}}\right\}\right)
\end{equation}
the number of iterations of the Catalyst algorithm (Algorithm \ref{alg:Catalyst}) with absolute or relative accuracy, we find an $(\varepsilon, \sigma)$ solution of the original problem \eqref{problem:Catalyst} which is understood in the sense of Definition \ref{def:solution_saddle-point_problem}. 
\qed \end{proof}

\paragraph{\textbf{The SAGA algorithm.}}
Let us consider the problem 
\begin{equation}\label{problem:SAGA}
    \min_{x \in \R^{d_x}}\max_{y \in \R^{d_y}}\left\{K(x,y) + M(x,y) \right\}.
\end{equation}
under the following assumption
\begin{assumption}\label{assumption:SAGA}
 \begin{enumerate}
     \item $M$ is $(\mu_x, \mu_y)$-strongly convex-concave. Moreover, we assume that we may compute the proximal operator of $M$:
     \begin{equation}
         prox_M^{\lambda}(x', y') = \arg \min_{x \in \R^{d_x}} \max_{y \in \R^{d_y}} \left\{\lambda M(x,y) + \frac{\mu_x}{2}\|x-x'\|_2^2 -  \frac{\mu_y}{2}\|y-y'\|_2^2  \right\};
     \end{equation}
     \item  $K$ is convex-concave and has Lipschitz-continuous gradients;
     \item The vector-valued function $B(x, y) = (\nabla_x K(x, y), −\nabla_y K(x, y)) \in \R^{d_x + d_y}$ may be split into a family of vector-valued functions as $B = \underset{i \in \mathcal{J}}{\sum}B_i$, where the only constraint is that each $B_i$ is Lipschitz-continuous (with constant $L_i$).
 \end{enumerate}
 
\end{assumption}
To solve the problem \eqref{problem:SAGA} under the Assumption \ref{assumption:SAGA} we can apply the SAGA algorithm from \cite{NIPS2016_1aa48fc4}.  
\begin{algorithm}[h]
\caption{SAGA: Online Stochastic Variance Reduction for Saddle Points \cite{NIPS2016_1aa48fc4} \label{alg:SAGA}}
\begin{algorithmic}[1]
		\STATE {\bf Input:} Functions $(K_i)_{i\geq 0}$, probabilities $(\pi_i)_{i \geq 0}$, smoothness $\bar{L}(\pi)$ and $L$, iterate $(x_0, y_0)$, number of iterations t, number of updates per iteration (mini-batch size) $m$. 
		\STATE Set $\lambda = \left(\max\left\{\frac{3|\mathcal{J}|}{2m} - 1, L^2 + \frac{3\bar{L}^2}{m}\right\}\right)^{-1}$; 
		\STATE Initialize $w^i = B_i(x_0, y_0)$ for all $i \in \mathcal{J}$ and $W = \underset{i \in \mathcal{J}}{\sum}w^i$;
		\FOR{$l = 1$ to $t$}
			\STATE Sample $i_1, \dots , i_m \in \mathcal{J}$ from the probability vector $(\pi_i)_{i\geq 0}$ with replacement; 
			   
			\STATE 
			\begin{equation}
			(x_l, y_l) = prox_M^{\lambda}\left\{(x_{l-1}, y_{l-1}) − \lambda\begin{psmallmatrix}
            \frac{1}{\mu_x}& 0
            \\
            0 & \frac{1}{\mu_y}
            \end{psmallmatrix} 
            \left(W + \frac{1}{m}\sum_{k=1}^m \left\{ \frac{1}{\pi_{i_k}}v_k − \frac{1}{\pi_{i_k}}w^{i_k}\right\}\right)\right\};
            \end{equation}
		    \STATE (optional) Sample $i_1, \dots , i_m \in \mathcal{J}$ uniformly with replacement;
		    \STATE (optional) Compute $v_k = B_{i_k}(x_l, y_l)$ for $k \in \{1, \dots, m\}$;
		    \STATE Replace 
		$W =  W − \sum_{k=1}^m\{w^{i_k} − v_k\}$ and $w^{i_k} = v_k$ for $k \in \{1, \dots ,m\}$. 
		\ENDFOR
		\STATE {\bf Output:} Approximate solution $(x_t, y_t)$.
	\end{algorithmic}
\end{algorithm}

\begin{theorem}[Theorem 2 from \cite{NIPS2016_1aa48fc4}, Appendix D.2]\label{theorem:SAGA}
Under the Assumption \ref{assumption:SAGA}. After $t$ iterations of the SAGA algorithm (Algorithm \ref{alg:SAGA}) (with the option of resampling when using non-uniform sampling), we have
\begin{equation*}
    \mathbb{E}\|z_t-z^*\|^2 \leq 
    2\left(1 - \frac{1}{4} \left(\max\left\{\frac{3|\mathcal{J}|}{2m}, 1+ \frac{L^2}{\mu^2} + \frac{3\bar{L}^2}{m\mu^2}\right\}\right)^{-1}\right)^t\|z_0-z^*\|^2 
\end{equation*}
\end{theorem}

\begin{remark}
The constants $L, \bar{L}, \mu$ depend on the type of the problem. For more details see Appendix A, D of the Article \cite{NIPS2016_1aa48fc4}. We will define these constants for our problem below.

\demo
\end{remark}
\begin{lemma}\label{lemma:SAGA_sp}
Let us consider the following special case to the problem \eqref{problem:SAGA} with $M(x,y) = f(x) - h(y), K(x,y) = \frac{1}{m_G}\sum_{i=1}^{m_G}G_i(x,y)$:
\begin{equation}\label{problem:SAGA_sp}
\min_{x \in \R^{d_x}}\max_{y \in \R^{d_y}}\left\{f(x) + \frac{1}{m_G}\sum_{i=1}^{m_G}G_i(x,y) - h(y)\right\}, 
\end{equation}
under the Assumption \ref{assumpt:G_sum_1}. This problem is satisfy to the Assumption \ref{assumption:SAGA}.
\end{lemma}
\begin{proof}
\begin{enumerate}
    \item $f(x)$ is $\mu_x$-strongly convex, $-h(y)$ is $\mu_y$-strongly concave, then $M(x,y)$ is $(\mu_x, \mu_y)$ strongly convex-concave. 
    \begin{multline*}
        prox_M^{\lambda}(x', y') = \arg \min_{x \in \R^{d_x}} \max_{y \in \R^{d_y}} \left\{\lambda M(x,y) + \frac{\mu_x}{2}\|x-x'\|_2^2 -  \frac{\mu_y}{2}\|y-y'\|_2^2  \right\} =
        \\
        \arg \min_{x \in \R^{d_x}} \max_{y \in \R^{d_y}} \left\{\lambda (f(x) - h(y)) + \frac{\mu_x}{2}\|x-x'\|_2^2 -  \frac{\mu_y}{2}\|y-y'\|_2^2  \right\} =
        \\
        \left(\arg \min_{x \in \R^{d_x}} \left\{\lambda f(x) + \frac{\mu_x}{2}\|x-x'\|_2^2\right\}, \arg\max_{y \in \R^{d_y}} \left\{ -\lambda h(y) - \frac{\mu_y}{2}\|y-y'\|_2^2  \right\}\right) =
        \left( prox_f^{\lambda}(x'), prox_h^{\lambda}(y')\right)
    \end{multline*}
$f(x), h(y)$ are proximal-friendly, then $\left( prox_f^{\lambda}(x'), prox_h^{\lambda}(y')\right)$ is easy to compute, then $ prox_M^{\lambda}(x', y')$ is easy to compute. We have shown that Assumption \ref{assumption:SAGA}.1 is fulfilled.
\item $K(x,y) = \frac{1}{m_G}\sum_{i=1}^{m_G}G_i(x,y)$ is convex-concave.  We have shown that Assumption \ref{assumption:SAGA}.2 is fulfilled.
\item 
\begin{multline*}
    B(x,y) = (\nabla_x K(x,y), -\nabla_y K(x,y)) = \left(\nabla_x \frac{1}{m_G} \sum_{i=1}^{m_G} G_i(x,y), -\nabla_y \frac{1}{m_G} \sum_{i=1}^{m_G} G_i(x,y)\right) =
    \\
    \frac{1}{m_G} \sum_{i=1}^{m_G} (\nabla_x G_i(x,y), -\nabla_y G_i(x,y))
\end{multline*}

$B(x,y) = \sum_{i=1}^{m_G}B_i(x,y)$ where $B_i = \frac{1}{m_G}(\nabla_x G_i(x,y), -\nabla_y G_i(x,y)) = \frac{1}{m_G} \nabla G_i(x,y)$.

For each $(x_1, x_2), (y_1, y_2)$:
\begin{equation*}
    \|B_i(x_1,x_2) - B_i(y_1, y_2)\|_2 = \frac{1}{m_G} \|\nabla G_i(x_1, x_2) - \nabla G_i(y_1,y_2)\|_2 \leq \frac{L_i}{m_G}\|(x_1, x_2) - (y_1, y_2)\|_2
\end{equation*}
Then, for each $i \in \overline{\{1,m_G\}}, B_i(x,y)$ is Lipschitz-continuous with constant $\frac{L_i}{m_G}$. 
We have shown that Assumption \ref{assumption:SAGA}.3 is fulfilled.
\end{enumerate}
\qed \end{proof}

In Lemma \ref{lemma:SAGA_for_sp_accuracy} we show how the number of iterations in the SAGA algorithm (Algorithm \ref{alg:SAGA}) is chosen to find an $(\varepsilon, \sigma)$ solution to the problem \eqref{problem:SAGA_sp} which is understood in the sense of Definition \ref{def:solution_saddle-point_problem} .

\begin{lemma}\label{lemma:SAGA_for_sp_accuracy}
Choose
\begin{equation*}
    \varepsilon' \leq \min\left\{\varepsilon, \frac{\varepsilon}{\left(4L_G + \frac{4L_G^2}{\mu_y} + \frac{4L_G^2}{\mu_x}\right)}, \frac{\varepsilon^4}{4(M - f(x^*) - h(y^*))^2}\right\}, \ \ \ \sigma' \leq \sigma
\end{equation*}
where 
     $$\sup \{f(x) + h(y): (x,y) \in \mathcal{B}_2((x^*, y^*),\varepsilon )\} \leq M \ \ (M \  \text{is finite}), \ \ \ L_G = \frac{1}{m_G}\sum_{i=1}^{m_G}L_G^i. $$
After 
\begin{equation}
    \mathcal{N} = O\left(m_G + \frac{L_G^2}{(\min \{\mu_x, \mu_y\})^2}\ln{\frac{2\|z_0 - z^*\|_2^2}{\varepsilon' \sigma'}}\right)
\end{equation}
the number of iterations of the SAGA algorithm (Algorithm \ref{alg:SAGA}) (with $m=1$, $\pi_i = \frac{L_G^i}{\sum_{i=1}^{m_G}L_G^i}$ and the option of resampling when using non-uniform sampling), we have an $(\varepsilon, \sigma)$ solution to the saddle-point problem \eqref{problem:SAGA_sp} under the Assumption \ref{assumpt:G_sum_1} which is understood in the sense of Definition \ref{def:solution_saddle-point_problem}.
\end{lemma}
\begin{proof}
\begin{enumerate}
    \item Let us define the constants $L, \bar{L}, \mu$ for the problem \eqref{problem:SAGA_sp}. This constants are used in Appendix A, D.2 of the article \cite{NIPS2016_1aa48fc4}. For beginning, let us define the operators $A(x,y), B(x,y)$, which used in the Appendix A of the article \cite{NIPS2016_1aa48fc4}. When we compute the $prox_{M}^{\lambda}(x',y')= \left( prox_f^{\lambda}(x'), prox_h^{\lambda}(y')\right)$, we find $(x^*, y^*)$ such that:
    \begin{equation*}
        \lambda\partial f(x^*) + \mu_x \|x^* - x'\|
 = 0 \ \ \ \Rightarrow \ \ \  \frac{\lambda}{\mu_x}\partial f(x^*) +  \|x^* - x'\|
 = 0,   
 \end{equation*}
 \begin{equation*}
        - \lambda\partial h(y^*) - \mu_y \|y^* - y'\|
 = 0 \ \ \ \Rightarrow \ \ \   \frac{\lambda}{\mu_y}\partial h(y^*) +  \|y^* - y'\|
 = 0.   
 \end{equation*}
 Then, 
 \begin{equation*}
     (x^*, y^*) = ( I + \lambda A) ^{-1}(x', y'), \ \ \ \text{where} \ \ \ A = \left(\frac{1}{\mu_x}\partial f(x), \frac{1}{\mu_y}\partial h(y)\right).
 \end{equation*}
 Let us define the operator $B(x,y)$:
 \begin{equation*}
     B(x,y) = \left(\frac{1}{\mu_x}\nabla_x G(x,y), -\frac{1}{\mu_y}\nabla_y G(x,y) \right).
 \end{equation*} 
 
 \begin{equation*}
     B(x,y) = \frac{1}{m_G}\sum_{i=1}^{m_G}\left(\frac{1}{\mu_x}\nabla_x G_i(x,y), -\frac{1}{\mu_y}\nabla_y G_i(x,y) \right) = \sum_{i=1}^{m_G}B_i(x,y),
 \end{equation*} 
 where $B_i(x,y) = \frac{1}{m_G}\left(\frac{1}{\mu_x}\nabla_x G_i(x,y), -\frac{1}{\mu_y}\nabla_y G_i(x,y) \right)$.
 
 Let us define the constants $\mu, L, \bar{L}$:
 \begin{enumerate}
     \item $\mu$ is monotones constant of the operator A. Using that $f(x)$ is $\mu_x$-strongly convex, $h(y)$ is $\mu_y)$-strongly convex, we have:
     \begin{multline*}
         (A(z) - A(z'))^{T}(z - z') =\frac{1}{\mu_x}(\partial f(x) - \partial f(x'))(x - x') +  \frac{1}{\mu_y}(\partial h(y) - \partial h(y'))(y - y') \geq 
         \\
         \|x - x'\|_2^2 + \|y - y'\|_2^2 \geq \|z - z'\|_2^2.
     \end{multline*}
     Then, $A(x,y)$ is $\mu$-monotone with $\mu = 1$.
     \item $L$ is Lipschitz constant of $B(x,y)$ with respect to the Euclidean norm, $z = (x,y), z' = (x', y') \in \R^{d_x+d_y}$:
     
     \begin{multline*}
         \|B(x,y) - B(x',y')\|_2 \leq 
         \\
         \frac{1}{m_G}\sum_{i = 1}^{m_G}\left(\frac{1}{\mu_x}\left\| \nabla_x G_i(x,y) - \nabla_x G_i(x', y')\right\|_2 + \frac{1}{\mu_y}\left\| \nabla_y G_i(x,y) - \nabla_y G_i(x', y')\right\|_2 \right)
         \\
         \leq\frac{1}{m_G}\sum_{i = 1}^{m_G}L_G^i\left(\frac{1}{\mu_x} + \frac{1}{\mu_y}\right)\left\|z - z'\right\|_2= L_G\left(\frac{1}{\mu_x} + \frac{1}{\mu_y}\right)
         \leq \frac{2L_G}{\min\{\mu_x, \mu_y\}}.
     \end{multline*}
     Then, $L \leq  \frac{2L_G}{\min\{\mu_x, \mu_y\}}$.
     \item $z = (x,y), z' = (x', y') \in \R^{d_x+d_y}$:
     \begin{multline*}
     \bar{L}^2 = \sup_{z,z' \in \R^{d_x+d_y}}\frac{1}{\|z-z'\|_2^2}\sum_{i=1}^{m_G}\frac{1}{\pi_i}\left\|B_i(x,y) - B_i(x', y')\right\|_2^2 = 
     \\
     \sup_{z,z' \in \R^{d_x+d_y}}\frac{1}{\|z-z'\|_2^2}\sum_{i=1}^{m_G}\frac{\sum_{i=1}^{m_G}L_G^i}{L_G^i}\left\|B_i(x,y) - B_i(x', y')\right\|_2^2  =
     \\
      \sup_{z,z' \in \R^{d_x+d_y}}\frac{1}{\|z-z'\|_2^2}\sum_{i=1}^{m_G}\frac{\sum_{i=1}^{m_G}L_G^i}{m_G^2L_G^i}\left(\frac{1}{\mu_x^2}\left\|\nabla_x G_i(x,y) - \nabla_x G_i(x', y')\right\|_2^2  +
      \frac{1}{\mu_y^2}\left\|\nabla_y G_i(x,y) - \nabla_y G_i(x', y')\right\|_2^2\right)
      \\
     \leq \sup_{z,z' \in \R^{d_x+d_y}}\frac{1}{\|z-z'\|_2^2}\frac{L_G}{m_G}\sum_{i = 1}^{m_G}L_G^i\left(\frac{1}{\mu_x^2} +
     \frac{1}{\mu_y^2}\right)\|z - z'\|_2^2
     \leq
     \frac{2L_G^2}{\min\{\mu_x, \mu_y\}^2}.
     \end{multline*}
     Then, $\bar{L}^2 \leq \frac{2L_G^2}{\min\{\mu_x, \mu_y\}^2}.$
 \end{enumerate}

    \item By the Lemma \ref{lemma:SAGA_sp}, the problem \eqref{problem:SAGA_sp} under the Assumption \ref{assumpt:G_sum_1} is satisfy to the Assumption \ref{assumption:SAGA}. Then, by Theorem \ref{theorem:SAGA}, after $t$ iteration of the SAGA algorithm (Algorithm \ref{alg:SAGA}) with $m=1$, $\pi_i = \frac{1}{m_G}$ and the option of resampling when using non-uniform sampling, we have
    \begin{equation*}
        \mathbb{E}\|z_t - z^*\|_2^2 \leq 
    2\left(1 - \frac{1}{4}\left(\max\left\{\frac{3|\mathcal{J}|}{2}, 1+ \frac{L^2}{\mu^2} + \frac{3\bar{L}^2}{\mu^2}\right\}\right)^{-1}\right)^t\|z_0-z^*\|_2^2,
    \end{equation*}
     where $|\mathcal{J}| = m_G$ and $\mu, L, \bar{L}$ are defined  above.
     \\
     Let us define $\eta$:
     \begin{equation*}
         \eta = \left(\max\left\{\frac{3|\mathcal{J}|}{2}, 1+ \frac{L^2}{\mu^2} + \frac{3\bar{L}^2}{\mu^2}\right\}\right)^{-1} = \left(\max\left\{\frac{3m_G}{2},  \frac{3\bar{L}_G^2}{\min\{\mu_x, \mu_y\}^2}\right\}\right)^{-1}.
     \end{equation*}
     Then,
     \begin{equation*}
         \mathbb{E}\|z_t - z^*\|_2^2 \leq 
    2\left(1 - \frac{\eta}{4}\right)^t\|z_0-z^*\|_2^2 \leq 2 e^{-\frac{\eta}{4}t}\|z_0-z^*\|_2^2 \leq \varepsilon'
     \end{equation*}
     Then, after 
     \begin{equation*}
        \mathcal{N} = \left\lceil \frac{4}{\eta} \ln{\frac{2\|z_0 - z^*\|_2^2}{\varepsilon'}}\right\rceil 
     \end{equation*}
     iterations of the SAGA algorithm (Algorithm \ref{alg:SAGA}) with parameters, which was defined above, we get $\hat{z} = (\hat{x}, \hat{y})$:
     \begin{equation*}
        \mathbb{E}\|\hat{z} - z^*\|_2^2 = \mathbb{E}\left(\|\hat{x} - x^*\|_2^2 + \|\hat{y} - y^*\|_2^2 \right)\leq \varepsilon'. 
     \end{equation*}
     $\|\hat{z} - z^*\|_2^2 \geq 0$, we can apply Markov's inequality:
     \begin{equation*}
         \mathbb{P}(\|\hat{z} - z^*\|_2^2 \leq \varepsilon' ) \geq 1 - \frac{\mathbb{E}\|\hat{z} - z^*\|_2^2}{\varepsilon'} \geq 1 - \sigma',
     \end{equation*}
     where 
     \begin{equation*}
          \frac{\mathbb{E}\|\hat{z} - z^*\|_2^2}{\varepsilon'} \leq \sigma' \ \ \ \Rightarrow \ \ \ \mathbb{E}\|\hat{z} - z^*\|_2^2 \leq \varepsilon'\sigma'.
     \end{equation*}
     Then, after 
     \begin{equation*}
        \mathcal{N} = \left\lceil \frac{4}{\eta} \ln{\frac{2\|z_0 - z^*\|_2^2}{\varepsilon' \sigma'}}\right\rceil = O\left(m_G + \frac{L_G^2}{\min\{\mu_x, \mu_y\}^2} \ln{\frac{2\|z_0 - z^*\|_2^2}{\varepsilon' \sigma'}}\right)
     \end{equation*}
     iterations of the SAGA algorithm (Algorithm \ref{alg:SAGA}) we  can find $\hat{z} = (\hat{x}, \hat{y})$ with probability at least $1-\sigma'$  such that $\mathbb{E}\|\hat{z} - z^*\|_2^2 = \mathbb{E}\left(\|\hat{x} - x^*\|_2^2 + \|\hat{y} - y^*\|_2^2 \right) \leq \varepsilon'$. Let us suppose $\varepsilon \geq \varepsilon'$
     \item Then, with probability $1 - \sigma', \hat{z} \in \mathcal{B}_2(z^*, \varepsilon)$, where
     \begin{equation*}
     \mathcal{B}_2(z^*, \varepsilon) = \left\{z \in \R^{d_x + d_y}: \|z - z^*\|_2^2 \leq \varepsilon \right\}.
     \end{equation*}
     Then, with probability $1 - \sigma'$ by Theorem 3.1.8 from \cite{nesterov2004} and using that $f(x), h(y)$ are convex, we have:
     \begin{equation}\label{eq:f+h_locally_lipschitz}
         f(\hat{x}) - f(x^*) + h(\hat{y}) - h(y^*) = |f(\hat{x}) + h(\hat{y}) - f(x^*) - h(y^*)| \leq
         \frac{M - f(x^*) - h(y^*)}{\varepsilon}\|\hat{z} - z^*\|_2,
     \end{equation}
     where 
     $$\sup \{f(x) + h(y): (x,y) \in \mathcal{B}_2((x^*, y^*),\varepsilon )\} \leq M \ \ (M \  \text{is finite.})$$
     Let us define 
     $$g(x) = \max_{y\in \R^{d_y}}\{G(x,y) - h(y)\},$$ which is $L_G + \frac{2L_G^2}{\mu_y}$-smooth (Lemma \ref{lemma:obt_delta_oracle}) and $$w(y) = -\min_{x \in \R^{d_x}}\{f(x) + G(x,y)\} = \max_{x \in \R^{d_x}}\{- f(x) - G(x,y)\},$$ which is $L_G + \frac{2L_G}{\mu_x}$-smooth (Lemma \ref{lemma:obt_delta_oracle}).
     Then, with probability $1 -\sigma'$:
 \begin{multline*}
    \max_{y \in \R^{d_y}}\left\{f(\hat{x}) + G(\hat{x}, y) - h(y)\right\} - \min_{x \in \R^{d_x}}\left\{f(x) + G(x, \hat{y}) - h(\hat{y})\right\} = 
    \max_{y \in \R^{d_y}}\left\{f(\hat{x}) + G(\hat{x}, y) - h(y)\right\} -
    \\
    \{f(x^*) + G(x^*, y^*) - h(y^*)\} +
    \{f(x^*) + G(x^*, y^*) - h(y^*)\} - 
    \min_{x \in \R^{d_x}}\left\{f(x) + G(x, \hat{y}) - h(\hat{y})\right\} =
    \\
    f(\hat{x}) - f(x^*) + g(\hat{x}) - g(x^*) + h(\hat{y}) - h(y^*)+w(\hat{y}) - w(y^*)
   \leq
   \\
   \frac{M - f(x^*) - h(y^*)}{\varepsilon}\|\hat{z} - z^*\|_2 + \left(L_G + \frac{2L_G^2}{\mu_y}\right)\|\hat{x} - x^*\|_2^2 + \left( L_G + \frac{2L_G^2}{\mu_x}\right)\|\hat{y} - y^*\|_2^2  \leq
   \\
   \frac{M - f(x^*) - h(y^*)}{\varepsilon}\|\hat{z} - z^*\|_2 + \left(2L_G + \frac{2L_G^2}{\mu_y} + \frac{2L_G^2}{\mu_x}\right)\|\hat{z} - z^*\|_2^2 \leq
   \\
   \frac{M - f(x^*) - h(y^*)}{\varepsilon}\sqrt{\varepsilon'} + \left(2L_G + \frac{2L_G^2}{\mu_y} + \frac{2L_G^2}{\mu_x}\right)\varepsilon'
\end{multline*}
Choose 
\begin{equation*}
    \varepsilon' \leq \min\left\{\varepsilon, \frac{\varepsilon}{\left(4L_G + \frac{4L_G^2}{\mu_y} + \frac{4L_G^2}{\mu_x}\right)}, \frac{\varepsilon^4}{4(M - f(x^*) - h(y^*))^2}\right\}, \sigma \leq \sigma',
\end{equation*}
$\varepsilon'$ is satisfy the inequality $\varepsilon' \leq \varepsilon$.
Then, with probability $1-\sigma' \geq 1 - \sigma$
\begin{multline*}
    \frac{M - f(x^*) - h(y^*)}{\varepsilon}\sqrt{\varepsilon'} + \left(2L_G + \frac{2L_G^2}{\mu_y} + \frac{2L_G^2}{\mu_x}\right)\varepsilon' \leq
    \\
    \frac{M - f(x^*) - h(y^*)}{\varepsilon}\frac{\varepsilon^2}{2(M - f(x^*) - h(y^*))} + \left(2L_G + \frac{2L_G^2}{\mu_y} + \frac{2L_G^2}{\mu_x}\right)\varepsilon' = \frac{\varepsilon}{2} + \left(2L_G + \frac{2L_G^2}{\mu_y} + \frac{2L_G^2}{\mu_x}\right)\varepsilon' \leq
    \\
    \frac{\varepsilon}{2} + \left(2L_G + \frac{2L_G^2}{\mu_y} + \frac{2L_G^2}{\mu_x}\right)\frac{\varepsilon}{2\left(2L_G + \frac{2L_G^2}{\mu_y} + \frac{2L_G^2}{\mu_x}\right)} = \varepsilon
\end{multline*}
     \end{enumerate}
     We have shown that after
     \begin{equation*}
        \mathcal{N}  = O\left(m_G + \frac{L_G^2}{\min\{\mu_x, \mu_y\}^2}\ln{\frac{2\|z_0 - z^*\|_2^2}{\varepsilon' \sigma'}} \right)
     \end{equation*}
     iterations of the SAGA algorithm we get $\hat{z} = (\hat{x}, \hat{y}))$, which is $(\varepsilon, \sigma)$ solution to the problem \eqref{problem:SAGA_sp} which is understood in the sense of Definition \ref{def:solution_saddle-point_problem}.
\qed \end{proof}
\subsection{Preliminaries}
In this subsection we formulate three theorems about equivalent optimization problem are used in the loops of general algorithm of this section.

We can rewrite the problem \eqref{eq:problem_Gsum_minmax}
\begin{equation}\label{problem_min_h_f_prox}
    \min _{x \in \mathbb{R}^{d_x}} \max _{y \in \mathbb{R}^{d_y}} \left\{f(x)+ G(x, y) - h(y)\right\} = \min _{x \in \mathbb{R}^{d_x}}  \left\{f(x)+ \max _{y \in \mathbb{R}^{d_y}}\{G(x, y) - h(y)\}\right\}.
\end{equation}

In the following lemma we show that if we find $\hat{x}$ is an $(\varepsilon_x, \sigma_x)$-solution to the problem \eqref{problem_min_h_f_prox} which is understood in the sense of Definition \ref{def:saddle_solution}  and  $\hat{y}$ is an $(\varepsilon_y, \sigma_y)$-solution to the problem $\max_{y\in \R^{d_y}} \{G(\hat{x}, y) - h(y)\}$ then $(\hat{x}, \hat{y})$ is an $(\varepsilon, \sigma)$ solution to the problem \eqref{eq:problem_Gsum_minmax} which is understood in the sense of Definition \ref{def:solution_saddle-point_problem} where dependencies $\varepsilon_x(\varepsilon)$, $\varepsilon_y(\varepsilon)$, $\sigma_x(\sigma)$, $\sigma_y(\sigma)$ are polynomial.

\begin{lemma}\label{lemma:for_loop_1}
Let us consider the problem \eqref{problem_min_h_f_prox} under the Assumption \ref{assumpt:G_sum_1}.
Let a  pair $(\hat{x}, \hat{y})$ satisfy
\begin{enumerate}
    \item $\hat{x}$ is an $(\varepsilon_x, \sigma_x)$-solution to the problem \eqref{problem_min_h_f_prox}, i.e. \eqref{eq:saddle_epsilon_solution} holds.
    \item $\hat{y}$ is an $(\varepsilon_y, \sigma_y)$-solution to the problem $\max_{y\in \R^{d_y}} \{G(\hat{x}, y) - h(y)\}$,
\end{enumerate}
with 
\begin{equation*}
    \varepsilon_y \leq \min \left\{\frac{\mu_y\varepsilon}{8}, \frac{\varepsilon^4\mu_y}{72\left(M_h - h(y^*)\right)^2}, \frac{\varepsilon\mu_y}{24\left(L_G + \frac{2L_G^2}{\mu_x}\right)}\right\}, \ \ \
    \varepsilon_x \leq \min\left\{ \frac{\varepsilon_y \mu_x \mu_y}{4L_G^2}, \frac{\varepsilon}{3}\right\}, \ \ \ \sigma_x \leq \frac{\sigma}{2}, \ \ \ \sigma_y \leq \frac{\sigma}{2},
\end{equation*}
where $\sup\{h(y): y \in \mathcal{B}_2(y^*, \varepsilon)\} \leq M_h$, $M_h$ is finite. 

Then, $(\hat{x}, \hat{y})$ is $(\varepsilon, \sigma)$-solution to the problem \eqref{eq:problem_Gsum_minmax} under the Assumption \ref{assumpt:G_sum_1}, i.e. \eqref{def:saddle_point_solution} holds.
\end{lemma}
\begin{proof}
We let $\Psi(x) = \max_{y\in \R^{d_y}} \{f(x) + G(x, y) - h(y) \}$ and note that $\Psi(x)$ is $\mu_x$-strongly convex. Under Assumption~\ref{assumpt:G_sum_1} the function $f(x) +G(x, y) - h(y)$ has unique saddle point $(x^*, y^*)$. Then, with probability $1-\sigma_x$ we have
\begin{align*}
    \|\hat{x} - x^*\|_2^2 \leq \frac{2}{\mu_x}\left( \max_{y \in \R^{d_y}} \{f(\hat{x}) +G(\hat{x}, y) - h(y) \} - \min_{x \in \R^{d_x}} \max_{y \in \R^{d_y}} \{f(x) +G(x, y) - h(y) \}  \right) \leq \frac{2\varepsilon_x}{\mu_x}.
\end{align*}
We denote $y^*(\hat{x}) = \arg \max_{y\in \R^{d_y}}\{f(\hat{x}) +G(\hat{x}, y) - h(y)\}$, then according to  Lemma~\ref{lemma:obt_delta_oracle} $y^*(x)$ is $ 2L_G/\mu_y$ Lipschitz continuous. Since $\{f(\hat{x}) +G(\hat{x}, y) - h(y) \}$ is $\mu_y$-strongly concave, we obtain  that the inequality
\begin{align*}
    \|\hat{y} - y^*\|_2^2 \leq 2\|\hat{y} - y^*(\hat{x})\|_2^2 + 2\|y^*(\hat{x}) - y^*(x^*)\|_2^2 \leq \frac{4\varepsilon_y}{\mu_y} + 8\left(\frac{L_G}{\mu_y}\right)^2\|\hat{x} - x^*\|^2 
\end{align*}
holds true with probability $1 - \sigma_x - \sigma_y$. 
The function $f(x)$ is convex and 
\begin{equation*}
\|\hat{y} - y^*\|_2^2 \leq  \frac{4\varepsilon_y}{\mu_y} + 8\left(\frac{L_G}{\mu_y}\right)^2\|\hat{x} - x^*\|^2 \leq \frac{4\varepsilon_y}{\mu_y} + 16\left(\frac{L_G}{\mu_y}\right)^2\frac{\varepsilon_x}{\mu_x} \leq \varepsilon \Rightarrow \hat{y} \in \mathcal{B}_2(y^*, \varepsilon),
\end{equation*}
under assumption $\frac{4\varepsilon_y}{\mu_y} + 16\left(\frac{L_G}{\mu_y}\right)^2\frac{\varepsilon_x}{\mu_x} \leq \varepsilon$.
Then, by lemma 3.1.8 from \cite{nesterov2004} $h(y)$ is locally Lipschitz continuous and:
\begin{equation*}
    h(\hat{y}) - h(y^*) \leq \frac{M_h - h(y^*)}{\varepsilon}\|\hat{y} - y^*\|_2,
\end{equation*}
where $\sup\{h(y): y \in \mathcal{B}_2(y^*, \varepsilon)\} \leq M_h$, $M_h$ is finite. 
By Lemma \ref{lemma:obt_delta_oracle}, $w(y) = - \min_{x \in \R^{d_x}}\{f(x) + G(x, y)\}$ is $L_G + \frac{2L_G^2}{\mu_x}$-smooth. 
Let us define $\Phi(y) = \min_{x \in \R^{d_x}} \{f(x) + G(x, y) -h(y)\}$:
\begin{equation*}
    \Phi(y^*) - \Phi(\hat{y}) = h(\hat{y}) - h(y^*) + w(\hat{y}) - w(y^*) \leq \frac{M_h - h(y^*)}{\varepsilon}\|\hat{y} - y^*\|_2 + \left(L_G + \frac{2L_G^2}{\mu_x}\right)\|\hat{y} - y^*\|_2^2
\end{equation*}
Whence,
\begin{multline*}
     \min_{x \in \R^{d_x}}\max_{y \in \R^{d_y}} \{f(x) +G(x, y) - h(y) \} -  \min_{x \in \R^{d_x}} \{f(x)+G(x, \hat{y}) - h(\hat{y}) \} = \Phi(y^*) - \Phi(\hat{y}) \leq
    \\
     \frac{M_h - h(y^*)}{\varepsilon^2}\|\hat{y} - y^*\|_2 + \left(L_G + \frac{2L_G^2}{\mu_x}\right)\|\hat{y} - y^*\|_2^2 \leq
     \\
     \frac{M_h - h(y^*)}{\varepsilon}\sqrt{\frac{4\varepsilon_y}{\mu_y} + 16\left(\frac{L_G}{\mu_y}\right)^2\frac{\varepsilon_x}{\mu_x}} + \left(L_G + \frac{2L_G^2}{\mu_x}\right)\left(\frac{4\varepsilon_y}{\mu_y} + 16\left(\frac{L_G}{\mu_y}\right)^2\frac{\varepsilon_x}{\mu_x}\right),
\end{multline*}
with probability $1 - \sigma_x - \sigma_y$.
Then, 
\begin{multline*}
     \max_{y \in \R^{d_y}}\left\{f(\hat{x}) + G(\hat{x}, y) - h(y)\right\} - \min_{x \in \R^{d_x}}\left\{f(x) + G(x, \hat{y}) - h(\hat{y})\right\} =
    \max_{y \in \R^{d_y}}\left\{f(\hat{x}) + G(\hat{x}, y) - h(y)\right\} -
    \\
    \{f(x^*) + G(x^*, y^*) - h(y^*)\} +
    \{f(x^*) + G(x^*, y^*) - h(y^*)\} - 
    \min_{x \in \R^{d_x}}\left\{f(x) + G(x, \hat{y}) - h(\hat{y})\right\} \leq
    \\
    \varepsilon_x + \frac{M_h - h(y^*)}{\varepsilon}\sqrt{\frac{4\varepsilon_y}{\mu_y} + 16\left(\frac{L_G}{\mu_y}\right)^2\frac{\varepsilon_x}{\mu_x}} + \left(L_G + \frac{2L_G^2}{\mu_x}\right)\left(\frac{4\varepsilon_y}{\mu_y} + 16\left(\frac{L_G}{\mu_y}\right)^2\frac{\varepsilon_x}{\mu_x}\right)
\end{multline*}
Choose 
\begin{equation*}
    \varepsilon_y \leq \min \left\{\frac{\mu_y\varepsilon}{8}, \frac{\varepsilon^4\mu_y}{72\left(M_h - h(y^*)\right)^2}, \frac{\varepsilon\mu_y}{24\left(L_G + \frac{2L_G^2}{\mu_x}\right)}\right\}, \ \ \
    \varepsilon_x \leq \min\left\{ \frac{\varepsilon_y \mu_x \mu_y}{4L_G^2}, \frac{\varepsilon}{3}\right\}, \ \ \ \sigma_x \leq \frac{\sigma}{2}, \ \ \ \sigma_y \leq \frac{\sigma}{2}.
\end{equation*}
Then, with probability $1 - \sigma_x - \sigma_y \geq 1 - \sigma$:
\begin{multline*}
    \varepsilon_x + \frac{M_h - h(y^*)}{\varepsilon}\sqrt{\frac{4\varepsilon_y}{\mu_y} + 16\left(\frac{L_G}{\mu_y}\right)^2\frac{\varepsilon_x}{\mu_x} } + \left(L_G + \frac{2L_G^2}{\mu_x}\right)\left(\frac{4\varepsilon_y}{\mu_y} + 16\left(\frac{L_G}{\mu_y}\right)^2\frac{\varepsilon_x}{\mu_x}\right) \leq
    \\
    \frac{\varepsilon}{3} + \frac{M_h - h(y^*)}{\varepsilon}\sqrt{\frac{4\varepsilon_y}{\mu_y} + 16\left(\frac{L_G}{\mu_y}\right)^2\frac{\varepsilon_y \mu_x \mu_y}{4L_G^2\mu_x}} + 
    \left(L_G + \frac{2L_G^2}{\mu_x}\right)\left(\frac{4\varepsilon_y}{\mu_y} + 16\left(\frac{L_G}{\mu_y}\right)^2\frac{\varepsilon_y \mu_x \mu_y}{4L_G^2\mu_x}\right) =
    \\
    \frac{\varepsilon}{3} + \frac{M_h - h(y^*)}{\varepsilon}\sqrt{\frac{8\varepsilon_y}{\mu_y}} + \left(L_G + \frac{2L_G^2}{\mu_x}\right)\left(\frac{8\varepsilon_y}{\mu_y}\right) \leq
    \\
    \frac{\varepsilon}{3} + \frac{M_h - h(y^*)}{\varepsilon}\sqrt{\frac{8\varepsilon^4\mu_y}{\mu_y72(M_h - h(y^*))^2}} + \left(L_G + \frac{2L_G^2}{\mu_x}\right)\left(\frac{8\varepsilon\mu_y}{\mu_y24\left(L_G + \frac{2L_G^2}{\mu_x}\right)}\right) = \varepsilon.
\end{multline*}
In the first inequality, we use $\varepsilon_x \leq \frac{\varepsilon_y \mu_x \mu_y}{4L_G^2}, \varepsilon_x \leq \frac{\varepsilon}{3}$, in the second inequality we use that $\varepsilon_y \leq \frac{\varepsilon^4\mu_y}{72\left(M_h - h(y^*)\right)^2}$, $\varepsilon_y \leq \frac{\varepsilon\mu_y}{24\left(L_G + \frac{2L_G^2}{\mu_x}\right)}$.

\begin{equation*}
    \frac{4\varepsilon_y}{\mu_y} + 16\left(\frac{L_G}{\mu_y}\right)^2\frac{\varepsilon_x}{\mu_x} \leq \frac{8\varepsilon_y}{\mu_y} \leq  \varepsilon.
\end{equation*}
Then the assumption $\frac{4\varepsilon_y}{\mu_y} + 16\left(\frac{L_G}{\mu_y}\right)^2\frac{\varepsilon_x}{\mu_x} \leq \varepsilon$ is true.
\qed \end{proof}

We can rewrite the problem \eqref{problem_min_h_f_prox}
\begin{multline}\label{problem:max_x_min_y}
    \min _{x \in \mathbb{R}^{d_x}}  \left\{f(x)+ \max _{y \in \mathbb{R}^{d_y}}\{G(x, y) - h(y)\}\right\} = \min _{x \in \mathbb{R}^{d_x}} \max _{y \in \mathbb{R}^{d_y}} \left\{f(x) + G(x, y) - h(y)\right\} =
    \\
    \min _{x \in \mathbb{R}^{d_x}} - \min _{y \in \mathbb{R}^{d_y}} \left\{ - f(x) - G(x, y) + h(y)\right\} = - \max _{x \in \mathbb{R}^{d_x}} \min _{y \in \mathbb{R}^{d_y}} \left\{ h(y) - G(x, y) - f(x)\right\} =
    \\
    - \min _{y \in \mathbb{R}^{d_y}} \left\{ h(y) + \max _{x \in \mathbb{R}^{d_x}} \left\{ - G(x, y) - f(x)\}\right\}\right\}
\end{multline}

In the following lemma we show that if we find $\hat{y}$ is an $(\varepsilon_y, \sigma_y)$-solution to the problem \eqref{problem:max_x_min_y} which is understood in the sense of Definition \ref{def:saddle_solution}  and  $\hat{x}$ is an $(\varepsilon_x, \sigma_x)$-solution to the problem $\max_{x\in \R^{d_x}} \{ - G(x, \hat{y}) - f(x)\}$ then $\hat{x}$ is an $\left(\varepsilon_x', \sigma_x'\right)$ solution to the problem \eqref{problem_min_h_f_prox} which is understood in the sense of Definition \ref{def:saddle_solution} and $\hat{y}$ is an $\left(\varepsilon_y', \sigma_y'\right)$ solution to the problem $\max_{y\in \R^{d_y}} \{G(\hat{x}, y) - h(y)\}$ where dependencies $\varepsilon_x\left(\varepsilon_x \right)$, $\varepsilon_y\left(\varepsilon_x', \varepsilon_y'\right)$, $\sigma_x\left(\sigma_x'\right)$, $\sigma_y\left(\sigma_x', \sigma_y'\right)$ are polynomial.

\begin{lemma}\label{lemma:for_loop_2}
Let us consider the problem $\min_{y \in \R^{d_y}}\{h(y) + \max_{x \in \R^{d_x}}\{- f(x) - G(x,y)\} \}$ under the Assumption \ref{assumpt:G_sum_1}. Let a pair $(\hat{x}, \hat{y})$ satisfy
\begin{enumerate}
    \item $\hat{y}$ is an $(\varepsilon_y, \sigma_y)$-solution to this problem, i.e. \eqref{eq:saddle_epsilon_solution} holds.
    \item $\hat{x}$ is an $(\varepsilon_x, \sigma_x)$-solution to the problem $\max_{x\in \R^{d_x}} \{- f(x) - G(x, \hat{y})\}$,
\end{enumerate}
with
\begin{equation*}
    \varepsilon_x \leq \min \left\{\frac{\mu_x\varepsilon_x'}{8}, \frac{\varepsilon_x'^4\mu_x}{32\left(M_f - f(x^*)\right)^2}, \frac{\varepsilon_x'\mu_x}{16\left(L_G + \frac{2L_G^2}{\mu_y}\right)}\right\},
    \varepsilon_y \leq \min\left\{\frac{\mu_y\varepsilon_y'}{2}, \frac{\varepsilon_x \mu_x \mu_y}{4L_G^2}, \frac{\varepsilon_y'^4\mu_y}{8\left(M_h - h(y^*)\right)^2}, \frac{\varepsilon_y'\mu_y}{2L_G}\right\},
    \end{equation*}
    \begin{equation*}
     \sigma_x \leq \frac{\sigma_x'}{2}, \ \ \ \sigma_y \leq \min\left\{\frac{\sigma_x'}{2}, \sigma_y'\right\}.
\end{equation*}
where $\sup\{f(x): x \in \mathcal{B}_2(x^*, \varepsilon_x')\} \leq M_f$ and $\sup\{h(y): y \in \mathcal{B}_2(y^*, \varepsilon_y')\} \leq M_h$, $M_f, M_h$ are finite. 

Then, $\hat{x}$ is $(\varepsilon_x', \sigma_x')$-solution to the problem \eqref{problem_min_h_f_prox} and $\hat{y}$ is $(\varepsilon_y', \sigma_y')$-solution to the problem $\max_{y\in \R^{d_y}} \{G(\hat{x}, y) - h(y)\}$.
\end{lemma}

\begin{proof}
We let $\Phi(y) = \max_{x\in \R^{d_x}} \{h(y) - G(x, y) - f(x) \}$ and note that $\Phi(y)$ is $\mu_y$-strongly convex. Under Assumption~\ref{assumpt:G_sum_1} the function $h(y) - G(x, y) - f(x)$ has unique saddle point $(x^*, y^*)$. Then, with probability $1-\sigma_y$ we have
\begin{align*}
    \|\hat{y} - y^*\|_2^2 \leq \frac{2}{\mu_y}\left( \max_{x \in \R^{d_x}} \{h(\hat{y}) -G(x, \hat{y}) - f(x) \} - \min_{y \in \R^{d_y}} \max_{x \in \R^{d_x}} \{h(y) - G(x, y) - f(x) \}  \right) \leq \frac{2\varepsilon_y}{\mu_y}.
\end{align*}
We denote $x^*(\hat{y}) = \arg \max_{x\in \R^{d_x}}\{h(\hat{y}) - G(x,\hat{y}) - f(x) \}$, then according to  Lemma~\ref{lemma:obt_delta_oracle} $x^*(y)$ is $ 2L_G/\mu_x$ Lipschitz continuous. Since $\{ h(\hat{y}) - G(x, \hat{y}) - f(x) \}$ is $\mu_x$-strongly concave, we obtain  that the inequality
\begin{align*}
    \|\hat{x} - x^*\|_2^2 \leq 2\|\hat{x} - x^*(\hat{y})\|_2^2 + 2\|x^*(\hat{y}) - x^*(y^*)\|_2^2 \leq \frac{4\varepsilon_x}{\mu_x} + 8\left(\frac{L_G}{\mu_x}\right)^2\|\hat{y} - y^*\|_2^2 
\end{align*}
holds true with probability $1 - \sigma_y - \sigma_x$. The function
$f(x)$ is convex and 
\begin{equation*}
    \|\hat{x} - x^*\|_2^2 \leq  \frac{4\varepsilon_x}{\mu_x} + 8\left(\frac{L_G}{\mu_x}\right)^2\|\hat{y} - y^*\|_2^2 \leq \frac{4\varepsilon_x}{\mu_x} + 16\left(\frac{L_G}{\mu_x}\right)^2\frac{\varepsilon_y}{\mu_y} \leq \varepsilon_x' \Rightarrow \hat{x} \in \mathcal{B}_2(x^*, \varepsilon_x').
\end{equation*}
under the assumption $\frac{4\varepsilon_x}{\mu_x} + 16\left(\frac{L_G}{\mu_x}\right)^2\frac{\varepsilon_y}{\mu_y} \leq \varepsilon_x'$.Then, by lemma 3.1.8 from \cite{nesterov2004} $f(x)$ is locally Lipschitz continuous and:
\begin{equation*}
    f(\hat{x}) - f(x^*) \leq \frac{M_f - f(x^*)}{\varepsilon_x'}\|\hat{x} - x^*\|_2,
\end{equation*}
where $\sup\{f(x): x \in \mathcal{B}_2(x^*, \varepsilon_x')\} \leq M_f$, $M_f$ is finite. 
By Lemma \ref{lemma:obt_delta_oracle}, $g(x) = \max_{y \in \R^{d_y}}\{G(x, y) - h(y)\}$ is $L_G + \frac{2L_G^2}{\mu_y}$-smooth. 
Let us define $\Psi(x) = \max_{y \in \R^{d_y}} \{f(x) + G(x, y) - h(y)\}$:
\begin{equation*}
    \Psi(\hat{x}) - \Psi(x^*) = f(\hat{x}) - f(x^*) + g(\hat{x}) - g(x^*) \leq \frac{M_f - f(x^*)}{\varepsilon_x'}\|\hat{x} - x^*\|_2 + \left(L_G + \frac{2L_G^2}{\mu_y}\right)\|\hat{x} - x^*\|_2^2
\end{equation*}
Whence,
\begin{multline*}
     \max_{y \in \R^{d_y}} \{f(\hat{x}) +G(\hat{x}, y) - h(y) \} -  \max_{y \in \R^{d_y}}\min_{x \in \R^{d_x}} \{f(x)+G(x, y) - h(y) \} = \Psi(\hat{x}) - \Psi(x^*) 
    \\
    \leq  \frac{M_f - f(x^*)}{\varepsilon_x'}\sqrt{\frac{4\varepsilon_x}{\mu_x} + 16\left(\frac{L_G}{\mu_x}\right)^2\frac{\varepsilon_y}{\mu_y} } + \left(L_G + \frac{2L_G^2}{\mu_y}\right)\left(\frac{4\varepsilon_x}{\mu_x} + 16\left(\frac{L_G}{\mu_x}\right)^2\frac{\varepsilon_y}{\mu_y} \right),
\end{multline*}
with probability $1 - \sigma_x - \sigma_y$. The function $h(y)$ is convex and
\begin{align*}
    \|\hat{y} - y^*\|_2^2 \leq  \frac{2\varepsilon_y}{\mu_y} \leq \varepsilon_y' \ \ \ \Rightarrow \ \ \ \hat{y} \in \mathcal{B}_2(y^*, \varepsilon_y')
\end{align*}
with probability $1 - \sigma_y$ and under the assumption $\frac{2\varepsilon_y}{\mu_y} \leq \varepsilon_y'$. Then, by Lemma 3.1.8 from \cite{nesterov2004} h(y) is locally Lipschitz continuous:
\begin{equation*}
    h(\hat{y}) - h(y^*) \leq \frac{M_h - h(y^*)}{\varepsilon_y'}\|\hat{y} - y^*\|_2,
\end{equation*}
where $\sup\{h(y): y \in \mathcal{B}_2(y^*, \varepsilon_y')\} \leq M_h$, $M_h$ is finite. $G(x,y)$ is $L_G$-smooth.
Then,
\begin{multline*}
    \left(G(\hat{x}, y^*) - h(y^*)\right) - \left(G(\hat{x}, \hat{y}) - h(\hat{y})\right) =  h(\hat{y}) - h(y^*) + G(\hat{x}, y^*) - G(\hat{x}, \hat{y}) \leq 
    \\
    \frac{M_h - h(y^*)}{\varepsilon_y'}\|\hat{y} - y^*\|_2 + \frac{L_G}{2}\|\hat{y} - y^*\|_2^2 \leq
    \frac{M_h - h(y^*)}{\varepsilon_y'}\sqrt{\frac{2\varepsilon_y}{\mu_y}} + \frac{L_G}{2}\frac{2\varepsilon_y}{\mu_y}
\end{multline*}
with probability $1 - \sigma_y$.

Choose 
\begin{equation*}
    \varepsilon_x \leq \min \left\{\frac{\mu_x\varepsilon_x'}{8}, \frac{\varepsilon_x'^4\mu_x}{32\left(M_f - f(x^*)\right)^2}, \frac{\varepsilon_x'\mu_x}{16\left(L_G + \frac{2L_G^2}{\mu_y}\right)}\right\}, \ \ \
    \varepsilon_y \leq \min\left\{\frac{\mu_y\varepsilon_y'}{2}, \frac{\varepsilon_x \mu_x \mu_y}{4L_G^2}, \frac{\varepsilon_y'^4\mu_y}{8\left(M_h - h(y^*)\right)^2}, \frac{\varepsilon_y'\mu_y}{2L_G}\right\},
    \end{equation*}
    \begin{equation*}
     \sigma_x \leq \frac{\sigma_x'}{2}, \ \ \ \sigma_y \leq \min\left\{\frac{\sigma_x'}{2}, \sigma_y'\right\}.
\end{equation*}
Then, with probability at least $1 - \sigma_x - \sigma_y \geq 1 - \sigma_x'$:
\begin{multline*}
    \frac{M_f - f(x^*)}{\varepsilon_x'}\sqrt{\frac{4\varepsilon_x}{\mu_x} + 16\left(\frac{L_G}{\mu_x}\right)^2\frac{\varepsilon_y}{\mu_y} } + \left(L_G + \frac{2L_G^2}{\mu_y}\right)\left(\frac{4\varepsilon_x}{\mu_x} + 16\left(\frac{L_G}{\mu_x}\right)^2\frac{\varepsilon_y}{\mu_y}\right) \leq
    \\
    \frac{M_f - f(x^*)}{\varepsilon_x'}\sqrt{\frac{4\varepsilon_x}{\mu_x} + 16\left(\frac{L_G}{\mu_x}\right)^2\frac{\varepsilon_x \mu_y \mu_x}{4L_G^2\mu_y}} + 
    \left(L_G + \frac{2L_G^2}{\mu_y}\right)\left(\frac{4\varepsilon_x}{\mu_x} + 16\left(\frac{L_G}{\mu_x}\right)^2\frac{\varepsilon_x \mu_y \mu_x}{4L_G^2\mu_y}\right) =
    \\
    \frac{M_f - f(x^*)}{\varepsilon_x'}\sqrt{\frac{8\varepsilon_x}{\mu_x}} + \left(L_G + \frac{2L_G^2}{\mu_y}\right)\left(\frac{8\varepsilon_x}{\mu_x}\right) \leq
    \\
   \frac{M_f - f(x^*)}{\varepsilon_x'}\sqrt{\frac{8\varepsilon_x'^4\mu_x}{\mu_x32(M_f - f(x^*))^2}} + \left(L_G + \frac{2L_G^2}{\mu_y}\right)\left(\frac{8\varepsilon_x'\mu_x}{\mu_x16\left(L_G + \frac{2L_G^2}{\mu_y}\right)}\right) = \varepsilon.
\end{multline*}
In the first inequality, we use $\varepsilon_y \leq \frac{\varepsilon_x \mu_x \mu_y}{4L_G^2}$, in the second inequality we use that $\varepsilon_x \leq \frac{\varepsilon_x'^4\mu_x}{32\left(M_f - f(x^*)\right)^2}$, $\varepsilon_x \leq \frac{\varepsilon_x'\mu_x}{16\left(L_G + \frac{2L_G^2}{\mu_y}\right)}$.

\begin{equation*}
    \frac{4\varepsilon_x}{\mu_x} + 16\left(\frac{L_G}{\mu_x}\right)^2\frac{\varepsilon_y}{\mu_y} \leq \frac{8\varepsilon_x}{\mu_x} \leq  \varepsilon_x'.
\end{equation*}
Then the assumption $\frac{4\varepsilon_x}{\mu_x} + 16\left(\frac{L_G}{\mu_x}\right)^2\frac{\varepsilon_y}{\mu_y} \leq \varepsilon_x'$ is true.

With probability $1-\sigma_y \geq 1 - \sigma_y'$:
\begin{equation*}
    \frac{M_h - h(y^*)}{\varepsilon_y'}\sqrt{\frac{2\varepsilon_y}{\mu_y}} + L_G\frac{\varepsilon_y}{\mu_y} \leq \frac{M_h - h(y^*)}{\varepsilon_y'}\sqrt{\frac{2\varepsilon_y'^4\mu_y}{\mu_y8(M_h - h(y^*))^2}} + L_G\frac{\varepsilon_y'\mu_y}{2L_G\mu_y} = \varepsilon_y'
\end{equation*}
in this inequality we use that $\varepsilon_y \leq \frac{\varepsilon_y'^4\mu_y}{8\left(M_h - h(y^*)\right)^2}$, $\varepsilon_y \leq \frac{\varepsilon_y'\mu_y}{2L_G}$.
\begin{equation*}
    \frac{2\varepsilon_y}{\mu_y} \leq \frac{2\varepsilon_y'\mu_y}{2\mu_y} = \varepsilon_y',
\end{equation*}
Then, the assumption $\frac{2\varepsilon_y}{\mu_y} \leq \varepsilon_y'$ is true.
\qed \end{proof}

In the following lemma we show that if we find $(\hat{x}, \hat{y})$ is an $(\varepsilon, \sigma)$ solution to the problem \eqref{eq:problem_Gsum_minmax} which is understood in the sense of Definition \ref{def:solution_saddle-point_problem} then $\hat{y}$ is an $(\varepsilon_y, \sigma_y)$-solution to the problem \eqref{problem:max_x_min_y} which is understood in the sense of Definition \ref{def:saddle_solution}  and  $\hat{x}$ is an $(\varepsilon_x, \sigma_x)$-solution to the problem $\max_{x\in \R^{d_x}} \{ - G(x, \hat{y}) - f(x)\}$ where dependencies $\varepsilon(\varepsilon_x, \varepsilon_y)$,  $\sigma(\sigma_x, \sigma_y)$ are polynomial.

\begin{lemma}\label{lemma:for_loop_3}
If $(\hat{x}, \hat{y})$ is $(\varepsilon, \sigma)$-solution to the saddle point problem \eqref{eq:problem_Gsum_minmax} under the Assumption \ref{assumpt:G_sum_1} which is understood in the sense of Definition \ref{def:solution_saddle-point_problem}, with
\begin{equation*}
    \varepsilon \leq \min\left\{\varepsilon_y, \frac{\varepsilon_x\mu_x}{2},\frac{\varepsilon_x\mu_x}{2L_G}, \frac{\varepsilon_x^4\mu_x}{8(M_f - f(x^*))^2}\right\}, \ \ \ \sigma \leq \min\left\{\sigma_x, \sigma_y\right\}
\end{equation*}
where $\sup\{f(x): x \in \mathcal{B}_2(x^*, \varepsilon_x)\} \leq M_f$, $M_f$ is finite.

Then,
\begin{enumerate}
    \item $\hat{y}$ is $(\varepsilon_y, \sigma_y)$-solution to the problem $\min_{y \in \R^{d_y}}\{h(y) + \max_{x \in \R^{d_x}}\{-G(x,y) - f(x)\}\}$;
    \item $\hat{x}$ is $(\varepsilon_x, \sigma_x)$-solution to the problem $\max_{x \in \R^{d_x}}\{-G(x,\hat{y})-f(x)\}$.
\end{enumerate}
\end{lemma}

\begin{proof} 
 $(\hat{x}, \hat{y})$ is $(\varepsilon, \sigma)$-solution to the saddle point problem \eqref{eq:problem_Gsum_minmax} which is understood in the sense of Definition \ref{def:solution_saddle-point_problem}, then:
 \begin{multline*}
     \max_{y \in \R^{d_y}}\{f(\hat{x}) + G(\hat{x},y) - h(y)\} - \min_{x \in \R^{d_x}}\{f(x) + G(x, \hat{y}) - h(\hat{y})\} =
     \max_{y \in \R^{d_y}}\{f(\hat{x}) + G(\hat{x},y) - h(y)\} -
     \\
     \{f(x^*) +G(x^*, y^*)- h(y^*)\} + \{f(x^*) +G(x^*, y^*) - h(y^*)\} - \min_{x \in \R^{d_x}}\{f(x) + G(x, \hat{y}) - h(\hat{y})\}
     \leq \varepsilon,
 \end{multline*}
 with probability $1 - \sigma$, where $(x^*, y^*)$ is a saddle point of this problem.
 We have shown, that 
 \begin{enumerate}
     \item $\hat{x}$ is $(\varepsilon, \sigma )$-solution to the problem $\min_{x \in \R^{d_x}}\{f(x) + \max_{y \in \R^{d_y}}\{ G(x,y) - h(y)\} \}$;
     \item $\hat{y}$ is $(\varepsilon, \sigma )$-solution to the problem $\max_{y \in \R^{d_y}}\{- h(y) + \min_{x \in \R^{d_x}}\{f(x) + G(x,y)\} \}$.
 \end{enumerate}
 Choose $\varepsilon \leq \varepsilon_y, \sigma \leq \sigma_y$.
 Then, with probability $1 - \sigma \geq 1 - \sigma_y$:
 \begin{multline*}
    h(\hat{y}) + \max_{x \in \R^{d_x}}\{-G(x,\hat{y}) - f(x)\} - \min_{y \in \R^{d_y}}\{h(y) + \max_{x \in \R^{d_x}}\{-G(x,y) - f(x)\}\} =
    \\ 
    \max_{x \in \R^{d_x}}\{h(\hat{y}) - G(x,\hat{y}) - f(x)\} + \max_{y \in \R^{d_y}}\{- h(y) - \max_{x \in \R^{d_x}}\{-G(x,y) - f(x)\}\} =
    \\
     -\min_{x \in \R^{d_x}}\{ - h(\hat{y}) + G(x,\hat{y}) + f(x)\} + \max_{y \in \R^{d_y}}\{- h(y) + \min_{x \in \R^{d_x}}\{G(x,y) + f(x)\}\} =
     \\
     \{f(x^*) + G(x^*, y^*) - h(y^*)\} - \{ - h(\hat{y}) + \min_{x \in \R^{d_x}}\{ f(x) + G(x,\hat{y})\}\} \leq \varepsilon \leq \varepsilon_y
 \end{multline*}
 we have shown that $\hat{y}$ is $(\varepsilon_y, \sigma_y)$-solution to the problem $\min_{y \in \R^{d_y}}\{h(y) + \max_{x \in \R^{d_x}}\{-G(x,y) - f(x)\}\}$.
 
 $\hat{x}$ is $(\varepsilon, \sigma )$-solution to the problem $\min_{x \in \R^{d_x}}\{f(x) + \max_{y \in \R^{d_y}}\{ G(x,y) - h(y)\} \}$ and under assumption \ref{assumpt:G_sum_1} this problem is $\mu_x$ strongly convex. Then, with probability $1 - \sigma$
 \begin{equation*}
   \|\hat{x} - x^*\|_2^2 \leq \frac{2}{\mu_x}\left(\{f(\hat{x}) + \max_{y \in \R^{d_y}}\{ G(\hat{x},y) - h(y)\} \} - \min_{x \in \R^{d_x}}\{f(x) + \max_{y \in \R^{d_y}}\{ G(x,y) - h(y)\} \}\right) \leq \frac{2\varepsilon}{\mu_x} \leq \varepsilon_x
 \end{equation*}
 under the Assumption $\frac{2\varepsilon}{\mu_x} \leq \varepsilon_x$. Then, $\hat{x} \in \mathcal{B}_2(x^*, \varepsilon_x)$. The function $f(x)$ is convex and by Lemma 3.1.8 from \cite{nesterov2004} $f(x)$ is locally Lipschitz continuous:
 \begin{equation*}
    f(\hat{x}) - f(x^*) \leq \frac{M_f - f(x^*)}{\varepsilon_x}\|\hat{x} - x^*\|_2,
\end{equation*}
where $\sup\{f(x): x \in \mathcal{B}_2(x^*, \varepsilon_x)\} \leq M_f$, $M_f$ is finite. Using that $G(x,y)$ is $L_G$-smooth, we get:
\begin{multline*}
    \left\{-G(x^*, \hat{y})-f(x^*)\right\} - \left\{-G(\hat{x}, \hat{y}) - f(\hat{x})\right\} = f(\hat{x}) - f(x^*) + G(\hat{x}, \hat{y}) - G(x^*, \hat{y}) \leq
    \\
    \frac{M_f - f(x^*)}{\varepsilon_x}\|\hat{x} - x^*\|_2 + \frac{L_G}{2}\|\hat{x} - x^*\|_2^2 \leq \frac{M_f - f(x^*)}{\varepsilon_x}\sqrt{\frac{2\varepsilon}{\mu_x}} + \frac{L_G}{2}\frac{2\varepsilon}{\mu_x}
\end{multline*}
Choose 
\begin{equation*}
    \varepsilon \leq \min\left\{\varepsilon_y, \frac{\varepsilon_x\mu_x}{2},\frac{\varepsilon_x\mu_x}{2L_G}, \frac{\varepsilon_x^4\mu_x}{8(M_f - f(x^*))^2}\right\}, \ \ \ \sigma \leq \min\left\{\sigma_x, \sigma_y\right\}.
\end{equation*}
\begin{equation*}
    \frac{2\varepsilon}{\mu_x} \leq \frac{2\varepsilon_x\mu_x}{2\mu_x} = \varepsilon_x,
\end{equation*}
then the assumption $\frac{2\varepsilon}{\mu_x} \leq \varepsilon_x$ is true.
Then, with probability $1 - \sigma \geq 1 - \sigma_x$:
\begin{equation*}
     \frac{M_f - f(x^*)}{\varepsilon_x}\sqrt{\frac{2\varepsilon}{\mu_x}} + L_G\frac{\varepsilon}{\mu_x} \leq
     \frac{M_f - f(x^*)}{\varepsilon_x}\sqrt{\frac{2\varepsilon_x^4\mu_x}{8\mu_x(M_f - f(x^*))^2}} + L_G\frac{\varepsilon_x\mu_x}{2L_G\mu_x} = \varepsilon_x,
\end{equation*}
in this inequality we use that $ \varepsilon \leq \frac{\varepsilon_x\mu_x}{2L_G}, \varepsilon \leq \frac{\varepsilon_x^4\mu_x}{8(M_f - f(x^*))^2}$.
We have shown that $\hat{x}$ is $(\varepsilon_x, \sigma_x)$-solution to the problem $\max_{x \in \R^{d_x}}\{-G(x, \hat{y}) - f(x)\}$.
\qed \end{proof}

\subsection{Accelerated Proximal Method for Saddle-Point Problems}

In this subsection we describe in detail the resulting structure of our algorithm for the setting of this section which consists of three loops. In the first two loops we apply the Catalyst algorithm (Algorithm \ref{alg:Catalyst}) with different value of parameter $H$ ($H_1$ and $H_2$ respectively) which defines its complexity. In the third loop we apply the SAGA algorithm (Algorithm \ref{alg:SAGA}) and we choose the number of iterations which depends on $H_1$ and $H_2$. In the end of this subsection we choose the value of these parameters.  Further, in each loop we have a target accuracy $\varepsilon$ and a confidence level $\sigma$ which defines the required quality of the solution to an optimization problem in this loop. These quantities define the target accuracy and confidence level for an equivalent optimization problem in the next loop using lemmas and theorems were proved in the two previous subsections. Our algorithm in this section has logarithmic dependence of the complexity on the target accuracy and confidence level (see Theorem \ref{theorem:section_6}). We hide such logarithmic factors in $\widetilde{O}$ notation. We conclude this section with the main Theorem \ref{theorem:section_6} which gives the complexity estimates of the proposed algorithm.
\paragraph{\textbf{Loop 1.}\label{subsec_1_f_h_prox}} 
The goal of the Loop 1 is to find an $(\varepsilon,\sigma)$-solution to the problem \eqref{eq:problem_Gsum_minmax} under the Assumption \ref{assumpt:G_sum_1}, which is understood in the sense of the Definition \ref{def:saddle_solution}. By Lemma \ref{lemma:for_loop_1} we find an $(\varepsilon, \sigma)$-solution to this saddle problem, if we find $\hat{x}$, which is $\left(\varepsilon_x^{(1)}, \sigma_x^{(1)}\right) $-solution to the minimization problem
\begin{equation}\label{problem:f_h_prox_loop_1_x}
    \min _{x \in \mathbb{R}^{d_x}}  \left\{f(x)+ \max _{y \in \mathbb{R}^{d_y}}\{G(x, y) - h(y)\}\right\}
\end{equation} under the Assumption \ref{assumpt:G_sum_1} and $\hat{y}$, which is $\left(\varepsilon_y^{(1)}, \sigma_y^{(1)}\right)$-solution to the problem 
\begin{equation}\label{problem:f_h_prox_loop_1_y}
    \max_{y \in \R^{d_y}}\{G(\hat{x}, y) - h(y)\}
\end{equation}
under the Assumption \ref{assumpt:G_sum_1}, we choose $\left(\varepsilon_x^{(1)}, \sigma_x^{(1)}\right)  = \left(\textbf{poly}(\varepsilon), \textbf{poly}(\sigma)\right)$, $\left(\varepsilon_y^{(1)}, \sigma_y^{(1)}\right) = \left(\textbf{poly}(\varepsilon), \textbf{poly}(\sigma)\right)$ which satisfy the following inequalities:
\begin{equation*}
    \varepsilon_y^{(1)} \leq \min \left\{\frac{\mu_y\varepsilon}{8}, \frac{\varepsilon^4\mu_y}{72\left(M_h - h(y^*)\right)^2}, \frac{\varepsilon\mu_y}{24\left(L_G + \frac{2L_G^2}{\mu_x}\right)}\right\}, \ \ \
    \varepsilon_x^{(1)} \leq \min\left\{ \frac{\varepsilon_y^{(1)} \mu_x \mu_y}{4L_G^2}, \frac{\varepsilon}{3}\right\}, \ \ \ \sigma_x^{(1)} \leq \frac{\sigma}{2}, \ \ \ \sigma_y^{(1)} \leq \frac{\sigma}{2},
\end{equation*}
where $\sup\{h(y): y \in \mathcal{B}_2(y^*, \varepsilon)\} \leq M_h$, $M_h$ is finite. To solve the problem \eqref{problem:f_h_prox_loop_1_x} under the Assumption \ref{assumpt:G_sum_1}, we would like to apply the Catalyst algorithm (Algorithm \ref{alg:Catalyst}) with
\begin{equation}\label{eq:varphi_psi_1}
\varphi(x) = \max _{y \in  \R^{d_y}}\left\{G(x, y) - h(y)\right\}, \ \ \ \psi(x) = f(x).
\end{equation}
By Lemma \ref{lemma:obt_delta_oracle}, the function $\varphi(x)$ is convex and has $L_G + \frac {2L_G^2}{\mu_y}$-Lipschitz continuous gradients, where $L_G=\frac{1}{m_G}\sum_{i=1}^{m_G} L_G^i$ and the function $\psi(x)$ is $\mu_x$ strongly convex. Then, these functions are satisfy the Assumption \ref{assumpt:Catalyst}. This allows us to apply the Catalyst algorithm (Algorithm \ref{alg:Catalyst})  to solve the minimization problem \eqref{problem:f_h_prox_loop_1_x}. We apply this algorithm with absolute accuracy and parameters $H = H_1$, which will be chosen later, $\mu = \mu_x$, $\left(\varepsilon_{x_k}^{(1)}\left(\varepsilon_x^{(1)}\right), \sigma_{x_k}^{(1)}\left(\varepsilon_x^{(1)}, \sigma_x^{(1)}\right)\right)_{k \geq 0}$ according to the Theorem \ref{corollary:Catalyst_sigma}, where $\varepsilon_{x_k}^{(1)} = \textbf{poly}\left(\varepsilon_x^{(1)}\right)$ and $\sigma_{x_k}^{(1)} = \textbf{poly}\left(\varepsilon_x^{(1)}, \sigma_x^{(1)}\right)$. We need to find $\hat{x}_k$ is an $\left(\varepsilon_{x_k}^{(1)}\left(\varepsilon_x^{(1)}\right), \sigma_{x_k}^{(1)}\left(\varepsilon_x^{(1)}, \sigma_x^{(1)}\right)\right)$-solution to the inner problem with the inner method $\mathcal{M}$ in each iteration of the Catalyst algorithm. For the particular definitions of $\varphi, \psi$ \eqref{eq:varphi_psi_1} in this Loop, this inner problem has the following form:  
\begin{equation}\label{eq:f_h_prox_2}
 x_{k} = \arg\min _{x \in \R^{d_x}}\left\{ f(x)+ \max _{y \in  \R^{d_y}}\left\{G(x, y)-h(y)\right\}+\frac{H_1}{2} \|x - x_{k-1}^{md}\|_2^2\right\}.
\end{equation} 

Below, in the next \hyperref[subsec_2_f_h_prox]{paragraph "Loop 2"},  we explain how to solve this auxiliary problem to obtain an $\left(\varepsilon_{x_k}^{(1)}\left(\varepsilon_x^{(1)}\right), \sigma_{x_k}^{(1)}\left(\varepsilon_x^{(1)}, \sigma_x^{(1)}\right)\right)$ solution to the problem \eqref{eq:f_h_prox_2} and an $\left(\varepsilon_y^{(1)}, \sigma_y^{(1)}\right)$-solution to the problem
\eqref{problem:f_h_prox_loop_1_y}.

To summarize the \hyperref[subsec_1_f_h_prox]{Loop 1}, the Assumption \ref{assumpt:Catalyst} holds and $\left(\varepsilon_{x_k}^{(1)}\left(\varepsilon_x^{(1)}\right), \sigma_{x_k}^{(1)}\left(\varepsilon_x^{(1)}, \sigma_x^{(1)}\right)\right)_{k \geq 0}$ are satisfy to \eqref{varepsilon_k_Catalyst_absolute} and \eqref{sigma_k_Catalyst_absolute}. Due  to  polynomial  dependencies $ \varepsilon_x^{(1)} =\textbf{poly}(\varepsilon), \sigma_x^{(1)} = \textbf{poly}(\sigma)$ we can use the notation $\widetilde{O}(\cdot)$ in the number of iterations of the Catalyst algorithm (Algorithm \ref{alg:Catalyst}).  Then, we can use the Theorem \ref{corollary:Catalyst_sigma} to guarantee that we find  $\left(\varepsilon_x^{(1)}, \sigma_x^{(1)}\right)$-solution to the problem \eqref{problem:f_h_prox_loop_1_x} in $\mathcal{N}_1 = \widetilde{O} \left(\max\left\{1, \sqrt{\frac{H_1}{\mu_x}}\right\} \right)$ number of iterations of the Catalyst algorithm (Algorithm \ref{alg:Catalyst}).   

\paragraph{\textbf{Loop 2.}}\label{subsec_2_f_h_prox}
The goal of the Loop 2 is to find an $\left(\varepsilon_{x_k}^{(1)}\left(\varepsilon_x^{(1)}\right), \sigma_{x_k}^{(1)}\left(\varepsilon_x^{(1)}, \sigma_x^{(1)}\right)\right)$-solution to the problem \eqref{eq:f_h_prox_2} and an $\left(\varepsilon_y^{(1)}, \sigma_y^{(1)}\right)$-solution to the problem
\eqref{problem:f_h_prox_loop_1_y}. By Lemma \ref{lemma:for_loop_2}, to find these solutions, we need to find $\hat{y}$ is an $\left(\varepsilon_y^{(2)}, \sigma_y^{(2)}\right)$-solution to the problem 
\begin{equation}\label{problem:f_h_prox_loop_2_y}
    \min _{y \in  \R^{d_y}} \left\{ h(y) + \max _{x \in \R^{d_x}}\left\{- \hat{f}(x) - G(x, y)\right\} \right\}
\end{equation}
where $\hat{f}(x) = f(x) + \frac{H_1}{2}\|x - x_{k-1}^{md}\|_2^2$ under the Assumption \ref{assumpt:G_sum_1} and $\hat{x}$ is an $\left(\varepsilon_x^{(2)}, \sigma_x^{(2)}\right)$-solution to the problem 
\begin{equation}\label{problem:f_h_prox_loop_2_x}
    \max_{x \in \R^{d_x}}\left\{-\hat{f}(x) - G(x,\hat{y})\right\},
\end{equation}
under the Assumption \ref{assumpt:G_sum_1}, we choose $\left(\varepsilon_x^{(2)}, \sigma_x^{(2)}\right)  = \left(\textbf{poly}(\varepsilon_x'), \textbf{poly}(\sigma_x')\right)$, $\left(\varepsilon_y^{(2)}, \sigma_y^{(2)}\right) = \left(\textbf{poly}(\varepsilon_y', \varepsilon_x'), \textbf{poly}(\sigma_x', \sigma_y')\right)$ which satisfy the following inequalities:
\begin{equation*}
    \varepsilon_x^{(2)} \leq \min \left\{\frac{(\mu_x + H_1)\varepsilon_x'}{8}, \frac{\varepsilon_x'^4(\mu_x + H_1)}{32\left(M_{\hat{f}} - \hat{f}(x^*)\right)^2}, \frac{\varepsilon_x'(\mu_x + H_1)}{16\left(L_G + \frac{2L_G^2}{\mu_y}\right)}\right\},
\end{equation*}

\begin{equation*}
    \varepsilon_y^{(2)} \leq \min\left\{\frac{\mu_y\varepsilon_y'}{2}, \frac{\varepsilon_x^{(2)} (\mu_x + H_1) \mu_y}{4L_G^2}, \frac{\varepsilon_y'^4\mu_y}{8\left(M_h - h(y^*)\right)^2}, \frac{\varepsilon_y'\mu_y}{2L_G}\right\},
\end{equation*}
    
\begin{equation*}
     \sigma_x^{(2)} \leq \frac{\sigma_x'}{2}, \ \ \ \sigma_y^{(2)} \leq \min\left\{\frac{\sigma_x'}{2}, \sigma_y'\right\},
\end{equation*}
\begin{equation*}
    \left(\varepsilon_x', \sigma_x'\right) = \left(\varepsilon_{x_k}^{(1)}\left(\varepsilon_x^{(1)}\right), \sigma_{x_k}^{(1)}\left(\varepsilon_x^{(1)}, \sigma_x^{(1)}\right)\right), \ \ \ \left(\varepsilon_y', \sigma_y'\right) = \left(\varepsilon_y^{(1)}, \sigma_y^{(1)}\right),
\end{equation*}
where $\sup\{\hat{f}(x): x \in \mathcal{B}_2(x^*, \varepsilon_x')\} \leq M_{\hat{f}}$ and $\sup\{h(y): y \in \mathcal{B}_2(y^*, \varepsilon_y')\} \leq M_h$, $M_f, M_h$ are finite. 

To solve the problem \eqref{problem:f_h_prox_loop_2_y} under the Assumption \ref{assumpt:G_sum_1} we apply the Catalyst algorithm (Algorithm \ref{alg:Catalyst}) with
\begin{equation}\label{eq:varphi_psi_2}
\varphi(y) =  \max _{x \in  \R^{d_x}}\left\{ -f(x) - G(x, y) - \frac{H_1}{2}\|x - x_{k-1}^{md}\|_2^2\right\}, \ \ \ \psi(y) = h(y).
\end{equation}
By the Lemma \ref{lemma:obt_delta_oracle}, the function $\varphi(y)$ is convex and has $H_1 + L_G + \frac {2L_G^2}{\mu_x}$-Lipschitz continuous gradients, where $L_G=\frac{1}{m_G}\sum_{i=1}^{m_G} L_G^i$ and the function $\psi(y)$ is $\mu_y$ strongly convex. Then, these functions are satisfy the Assumption \ref{assumpt:Catalyst}. This allows us to apply the Catalyst algorithm (Algorithm \ref{alg:Catalyst})  to solve the minimization problem \eqref{problem:f_h_prox_loop_2_y}. We apply this algorithm with absolute accuracy and parameters $H = H_2$, which will be chosen later, $\mu = \mu_y$, $\left(\varepsilon_{y_k}^{(2)}\left(\varepsilon_y^{(2)}\right), \sigma_{y_k}^{(2)}\left(\varepsilon_y^{(2)}, \sigma_y^{(2)}\right)\right)_{k \geq 0}$ according to the Theorem \ref{corollary:Catalyst_sigma}, where $\varepsilon_{y_k}^{(2)} = \textbf{poly}\left(\varepsilon_y^{(2)}\right)$ and $\sigma_{y_k}^{(2)} = \textbf{poly}\left(\varepsilon_y^{(2)}, \sigma_y^{(2)}\right)$. We need to find an $\left(\varepsilon_{y_k}^{(2)}\left(\varepsilon_y^{(2)}\right), \sigma_{y_k}^{(2)}\left(\varepsilon_y^{(2)}, \sigma_y^{(2)}\right)\right)$-solution to the inner problem with the inner method $\mathcal{M}$ in each iteration of the Catalyst algorithm. For the particular definitions of $\varphi, \psi$ \eqref{eq:varphi_psi_2} in this Loop, this inner problem has the following form:  
\begin{equation}\label{eq:f_h_prox_3}
 y_{k} = \arg\min _{y \in \R^{d_y}}\left\{ h(y)+ \max _{x \in  \R^{d_x}}\left\{- G(x, y)-f(x) - \frac{H_1}{2} \|x - x_{k-1}^{md}\|_2^2\right\} + \frac{H}{2}\|y - y_{k-1}^{md}\|_2^2\right\}.
\end{equation} 

Below, in the next \hyperref[subsec_3_f_h_prox]{paragraph "Loop 3"},  we explain how to solve this auxiliary problem to obtain an $\left(\varepsilon_{y_k}^{(2)}\left(\varepsilon_y^{(2)}\right), \sigma_{y_k}^{(2)}\left(\varepsilon_y^{(2)}, \sigma_y^{(2)}\right)\right)$ solution to the problem \eqref{eq:f_h_prox_3} and an $\left(\varepsilon_x^{(2)}, \sigma_x^{(2)}\right)$-solution to the problem
\eqref{problem:f_h_prox_loop_2_x}.

To summarize the \hyperref[subsec_2_f_h_prox]{Loop 2}, the Assumption \ref{assumpt:Catalyst} holds and $\left(\varepsilon_{y_k}^{(2)}\left(\varepsilon_y^{(2)}\right), \sigma_{y_k}^{(2)}\left(\varepsilon_y^{(2)}, \sigma_y^{(2)}\right)\right)_{k \geq 0}$ are satisfy to \eqref{varepsilon_k_Catalyst_absolute} and \eqref{sigma_k_Catalyst_absolute}, due to $\varepsilon_{x_k}^{(1)} = \textbf{poly}\left(\varepsilon_x^{(1)}\right) = \textbf{poly}\left(\varepsilon\right)$, $\sigma_{x_k}^{(1)} = \textbf{poly}\left(\varepsilon_x^{(1)}, \sigma_x^{(1)}\right) = \textbf{poly}\left(\varepsilon, \sigma\right)$ and $\varepsilon_y^{(1)} = \textbf{poly}(\varepsilon)$, $\sigma_y^{(1)} = \textbf{poly}(\sigma)$ dependencies $\varepsilon_y^{(2)}(\varepsilon, \sigma), \sigma_y^{(2)}(\varepsilon, \sigma)$  are polynomial and we can use the notation $\widetilde{O}(\cdot)$ in the number of iterations of the Catalyst algorithm (Algorithm \ref{alg:Catalyst}).   Then, we can use the Theorem \ref{corollary:Catalyst_sigma} to guarantee that we find  $\left(\varepsilon_y^{(2)}, \sigma_y^{(2)}\right)$-solution to the problem \eqref{problem:f_h_prox_loop_2_y} in $\mathcal{N}_2 = \widetilde{O} \left(\max\left\{1, \sqrt{\frac{H_2}{\mu_y}}\right\} \right)$ iterations of the Catalyst algorithm (Algorithm \ref{alg:Catalyst}).

\paragraph{\textbf{Loop 3}}\label{subsec_3_f_h_prox}
The goal of the Loop 3 is to find an $\left(\varepsilon_{y_k}^{(2)}\left(\varepsilon_y^{(2)}\right), \sigma_{y_k}^{(2)}\left(\varepsilon_y^{(2)}, \sigma_y^{(2)}\right)\right)$-solution to the problem \eqref{eq:f_h_prox_3} and an $\left(\varepsilon_x^{(2)}, \sigma_x^{(2)}\right)$-solution to the problem
\eqref{problem:f_h_prox_loop_2_x}. By Lemma \ref{lemma:for_loop_3}, to find these solutions, we need to find $(\hat{x}, \hat{y})$ is an $(\varepsilon^{(3)}, \sigma^{(3)})$-solution to the saddle problem 
\begin{equation}\label{problem:f_h_prox_loop_3_x_y}
   \min _{x \in \mathbb{R}^{d_x}} \max _{y \in \mathbb{R}^{d_y}} \left\{\hat{f}(x)+ G(x, y) - \hat{h}(y)\right\},
\end{equation}
under the Assumption \ref{assumpt:G_sum_1}, where 
\begin{equation*}
\hat{f}(x) = f(x) + \frac{H_1}{2}\|x - x_{k-1}^{md}\|_2^2, \ \ \ \hat{h}(y) = h(y) + \frac{H_2}{2}\|y - y_{k-1}^{md}\|_2^2, 
\end{equation*}
we choose $\left(\varepsilon^{(3)}, \sigma^{(3)}\right)  = \left(\textbf{poly}(\varepsilon_x, \varepsilon_y), \textbf{poly}(\sigma_x, \sigma_y)\right)$, which satisfy the following inequalities:
\begin{equation*}
    \varepsilon^{(3)} \leq \min\left\{\varepsilon_y, \frac{\varepsilon_x(\mu_x + H_1)}{2},\frac{\varepsilon_x(\mu_x + H_1)}{2L_G}, \frac{\varepsilon_x^4(\mu_x + H_1)}{8(M_{\hat{f}} - \hat{f}(x^*))^2}\right\},\ \ \ \sigma \leq \min\left\{\sigma_x, \sigma_y\right\}
\end{equation*}

\begin{equation*}
   \left(\varepsilon_x, \sigma_x\right) = \left(\varepsilon_x^{(2)}, \sigma_x^{(2)}\right), \ \ \ \left(\varepsilon_y, \sigma_y\right) = \left(\varepsilon_{y_k}^{(2)}\left(\varepsilon_y^{(2)}\right), \sigma_{y_k}^{(2)}\left(\varepsilon_y^{(2)}, \sigma_y^{(2)}\right)\right),
\end{equation*}
where $\sup\{\hat{f}(x): x \in \mathcal{B}_2(x^*, \varepsilon_x)\} \leq M_{\hat{f}}$, $M_{\hat{f}}$ is finite.
 
 To solve the problem \eqref{problem:f_h_prox_loop_3_x_y} we apply the SAGA 
 algorithm (Algorithm \ref{alg:SAGA}) with:
\begin{equation}
    M(x, y) = \hat{f}(x) - \hat{h}(y), \ \ \ \ K(x, y) =  G(x, y).
\end{equation}

$\hat{f}(x)$ is $\mu_x + H_1$-strongly convex, $\hat{h}(y))$ is $\mu_y + H_2$ strongly convex. Then, the Assumption \ref{assumpt:G_sum_1} is true for this problem. By Lemma \ref{lemma:SAGA_sp} the Assumption \ref{assumption:SAGA} is true for this problem, this allow us to apply the SAGA algorithm (Algorithm \ref{alg:SAGA}) to solve the problem \eqref{problem:f_h_prox_loop_3_x_y}. We apply this algorithm with parameters:
\begin{equation*}
    K_i(x,y) = \frac{1}{m_G}G_i(x,y), \ \ \ \pi_i = \frac{L_G^i}{\sum_{i=1}^{m_G}L_G^i}, \ \ \ \bar{L} = L = \frac{2L_G}{\min\{\mu_x + H_1, \mu_y + H_2\}}, \ \ \ m = 1,
\end{equation*}
where $L_G = \frac{1}{m_G}\sum_{i=1}^{m_G}L_G^i$ and number of iterations 
\begin{equation*}
    \mathcal{N}_3 = \left\lceil \frac{4}{\eta} \ln{\frac{2\|z_0 - z^*\|_2^2}{\varepsilon' \sigma'}}\right\rceil = O\left(m_G + \frac{L_G^2}{\min\{\mu_x + H_1, \mu_y + H_2\}^2} \ln{\frac{2\|z_0 - z^*\|_2^2}{\varepsilon' \sigma'}}\right),
\end{equation*} 
where 
\begin{equation*}
    \eta = \left(\max\left\{\frac{3m_G}{2},  \frac{3L_G^2}{\min\{\mu_x + H_1, \mu_y + H_2\}^2}\right\}\right)^{-1}.
\end{equation*} 
Choosing $(\varepsilon', \sigma')$ according to the Lemma \ref{lemma:SAGA_for_sp_accuracy}, where $(\varepsilon, \sigma) = \left(\varepsilon^{(3)}, \sigma^{(3)}\right)$ and $\varepsilon' = \textbf{poly}(\varepsilon)$, $\sigma' = \textbf{poly}(\sigma)$. 

To summarize the \hyperref[subsec_3_f_h_prox]{Loop 3}, the Assumption \ref{assumption:SAGA} holds and $\left(\varepsilon'\left(\varepsilon^{(3)}\right), \sigma'\left(\sigma^{(3)}\right)\right)$ are satisfy to \eqref{lemma:SAGA_for_sp_accuracy} and \eqref{sigma_k_Catalyst_absolute}, due to $\varepsilon_{x}^{(2)} = \textbf{poly}\left(\varepsilon, \sigma\right)$, $\sigma_{x}^{(2)} = \textbf{poly}\left(\varepsilon, \sigma\right)$, $\varepsilon_{y_k}^{(2)} = \textbf{poly}(\varepsilon_y^{(2)})= \textbf{poly}(\varepsilon, \sigma)$,$\sigma_{y_k}^{(2)} = \textbf{poly}(\varepsilon_y^{(2)}, \sigma_y^{(2)})= \textbf{poly}(\varepsilon, \sigma)$  and $\varepsilon_y^{(3)} = \textbf{poly}\left(\varepsilon_{x}^{(2)}, \varepsilon_{y_k}^{(2)}\right)$, $\sigma_y^{(3)} = \textbf{poly}\left(\sigma_{x}^{(2)}, \sigma_{y_k}^{(2)}\right)$ dependencies $\varepsilon_y^{(3)}(\varepsilon, \sigma)$, $\sigma_y^{(3)}(\varepsilon, \sigma) $  are polynomial and we can use the notation $\widetilde{O}(\cdot)$ in the number of iterations of the SAGA algorithm (Algorithm \ref{alg:SAGA}).  Then, we can use the Lemma \ref{lemma:SAGA_for_sp_accuracy} to guarantee that we find  $\left(\varepsilon^{(3)}, \sigma^{(3)}\right)$-solution to the problem \eqref{problem:f_h_prox_loop_3_x_y} in
\begin{equation*}
    \mathcal{N}_3 = \widetilde{O} \left(m_G + \frac{L_G^2}{(\min \{\mu_x + H_1, \mu_y + H_2\})^2}\right)
\end{equation*}
iterations of the SAGA algorithm (Algorithm \ref{alg:SAGA}).
In each iteration of the SAGA algorithm (Algorithm \ref{alg:SAGA}) we make no more than $m\mathcal{N}_3 = \mathcal{N}_3$ the number of oracle calls of $\nabla_x G(x,y)$, $\nabla_y G(x, y)$ and calculations of proximal operator for the functions $\hat{f}(x)$, $\hat{h}(y)$.



To summarize these 3 loops we can formulate the following main theorem of this section:

\begin{theorem}\label{theorem:section_6}
Suppose saddle problem of the form \eqref{eq:problem_Gsum_minmax} under the Assumption \ref{assumpt:G_sum_1} and supposition that $ \mu_x \sqrt{m_G} \leq L_G$ and $\mu_y \sqrt{m_G} \leq L_G$ .
Then we can find the $(\varepsilon, \sigma)$-solution to the problem \eqref{eq:problem_Gsum} and evaluate the number of oracle calls. Namely, after 3 loops of Algorithm from Section~\ref{subsec:proxfr_f_h} one can obtain next estimates on the number of oracles calls of $\nabla_x G(x, y)$, $\nabla_y G(x, y)$ and calculations of \eqref{f_x:prox-friendly} for functions $f(x)$ and $h(y)$:
\begin{equation}
     \widetilde{O} \left(  \sqrt{\frac{m_G  L_G^2}{\mu_x\mu_y}}
    \right).
\end{equation}
\end{theorem}

\begin{proof}

\textbf{Step 1. Polynomial dependence.} In \hyperref[subsec_1_f_h_prox]{Loop 1} we find $\hat{x}$ is $\left(\varepsilon_x^{(1)}, \sigma_x^{(1)}\right)$ solution to the problem \eqref{problem:f_h_prox_loop_1_x}, where $\left(\varepsilon_x^{(1)}, \sigma_x^{(1)}\right)  = \left(\textbf{poly}(\varepsilon), \textbf{poly}(\sigma)\right)$. We solve this problem with Catalyst algorithm and by Theorem \ref{corollary:Catalyst_sigma} we can find $\hat{x}$ after 
\begin{equation*}
    \mathcal{N}_1 = O\left(\max\left\{1, \frac{H_1}{\mu_x}\right\}\ln{\frac{1}{\varepsilon_x^{(1)}}}\right) = O\left(\max\left\{1, \frac{H_1}{\mu_x}\right\}\ln{\frac{1}{\varepsilon}}\right) = \widetilde{O}\left(\max\left\{1, \frac{H_1}{\mu_x}\right\}\right)
\end{equation*}
the number of iterations of the Catalyst algorithm (Algorithm \ref{alg:Catalyst}) we find $\hat{x}$ is $\left(\varepsilon_x^{(1)}, \sigma_x^{(1)}\right)$ solution to the problem \eqref{problem:f_h_prox_loop_1_x}.

In \hyperref[subsec_2_f_h_prox]{Loop 2} we find $\hat{y}$ is $\left(\varepsilon_y^{(2)}, \sigma_y^{(2)}\right)$ solution to the problem \eqref{problem:f_h_prox_loop_2_y}, where 
\begin{equation*}
    \left(\varepsilon_y^{(2)}, \sigma_y^{(2)}\right) = \left(\textbf{poly}\left(\varepsilon_{x_k}^{(1)}, \varepsilon_y^{(1)}\right), \textbf{poly}\left(\sigma_{x_k}^{(1)}, \sigma_y^{(1)}\right)\right).
\end{equation*}. 

\begin{equation*}
    \varepsilon_{x_k}^{(1)} = \textbf{poly}\left(\varepsilon_x^{(1)}\right) = \textbf{poly}\left(\varepsilon\right), \sigma_{x_k}^{(1)} = \textbf{poly}\left(\varepsilon_x^{(1)}, \sigma_x^{(1)}\right) = \textbf{poly}\left(\varepsilon, \sigma\right)
\end{equation*}
\begin{equation*}
    \varepsilon_y^{(1)} = \textbf{poly}(\varepsilon), \sigma_y^{(1)} = \textbf{poly}(\sigma).
\end{equation*} 
Then, dependencies $\varepsilon_y^{(2)}(\varepsilon, \sigma)$, $\sigma_y^{(2)}(\varepsilon, \sigma)$ are polynomial. We solve this problem with Catalyst algorithm and by Theorem \ref{corollary:Catalyst_sigma} we can find $\hat{y}$ after 
\begin{equation*}
    \mathcal{N}_2 = O\left(\max\left\{1, \frac{H_2}{\mu_y}\right\}\ln{\frac{1}{\varepsilon_y^{(2)}}}\right) = O\left(\max\left\{1, \frac{H_2}{\mu_2}\right\}\ln{\frac{1}{\varepsilon\sigma}}\right) = \widetilde{O}\left(\max\left\{1, \frac{H_2}{\mu_y}\right\}\right)
\end{equation*}
the number of iterations of the Catalyst algorithm (Algorithm \ref{alg:Catalyst}) we find $\hat{y}$ is $\left(\varepsilon_y^{(2)}, \sigma_y^{(2)}\right)$ solution to the problem \eqref{problem:f_h_prox_loop_2_y}.

In \hyperref[subsec_3_f_h_prox]{Loop 3} we find $(\hat{x},\hat{y})$ is $\left(\varepsilon^{(3)}, \sigma^{(3)}\right)$ solution to the problem \eqref{problem:f_h_prox_loop_3_x_y}, where 
\begin{equation*}
    \left(\varepsilon^{(3)}, \sigma^{(3)}\right) = \left(\textbf{poly}\left(\varepsilon_x^{(2)}, \varepsilon_{y_k}^{(2)}\right), \textbf{poly}\left(\sigma_x^{(2)}, \sigma_{y_k}^{(2)}\right)\right).
\end{equation*}. 

\begin{equation*}
    \varepsilon_{y_k}^{(2)} = \textbf{poly}\left(\varepsilon_y^{(2)}\right) = \textbf{poly}\left(\varepsilon, \sigma\right), \sigma_{y_k}^{(2)} = \textbf{poly}\left(\varepsilon_y^{(2)}, \sigma_y^{(2)}\right) = \textbf{poly}\left(\varepsilon, \sigma\right)
\end{equation*}
\begin{equation*}
    \varepsilon_x^{(2)} = )\textbf{poly}\left(\varepsilon_{x_k}^{(1)}, \sigma_{x_k}^{(1)}\right)= \textbf{poly}\left(\varepsilon_{x}^{(1)}, \sigma_{x}^{(1)}\right) = \textbf{poly}\left(\varepsilon, \sigma\right),
\end{equation*}
\begin{equation*}
    \sigma_x^{(2)} = \textbf{poly}(\sigma_{x_k}^{(1)}) = \textbf{poly}\left(\varepsilon_{x}^{(1)}, \sigma_{x}^{(1)}\right) = ) = \textbf{poly}\left(\varepsilon, \sigma\right).
\end{equation*} 
Then, dependencies $\varepsilon^{(3)}(\varepsilon, \sigma)$, $\sigma^{(3)}(\varepsilon, \sigma)$ are polynomial. We solve this problem with Catalyst algorithm and by Lemma \ref{lemma:SAGA_for_sp_accuracy} we can find $(\hat{x}, \hat{y})$ after 
\begin{multline*}
    \mathcal{N}_3 = O\left(m_G + \frac{L_G^2}{(\min \{\mu_x + H_1, \mu_y + H_2\})^2}\ln{\frac{1}{\varepsilon_y^{(3)}\sigma_y^{(3)}}}\right) =
    \\
    O\left(m_G + \frac{L_G^2}{(\min \{\mu_x + H_1, \mu_y + H_2\})^2}\ln{\frac{1}{\varepsilon\sigma}}\right) = \widetilde{O}\left(m_G + \frac{L_G^2}{(\min \{\mu_x + H_1, \mu_y + H_2\})^2}\right)
\end{multline*}
the number of iterations of the SAGA algorithm (Algorithm \ref{alg:SAGA}) we find $(\hat{x}, \hat{y})$ is $\left(\varepsilon_y^{(3)}, \sigma_y^{(3)}\right)$ solution to the problem \eqref{problem:f_h_prox_loop_3_x_y}.

\textbf{Step 2. Final estimates.}
We make oracle calls of $\nabla_x G(x,y)$, $\nabla_y G(x, y)$ and calculations of proximal operator for the functions $\hat{f}(x)$, $\hat{h}(y)$ only in \hyperref[subsec_3_f_h_prox]{Loop 3} and we make it no more than $m\mathcal{N}_3 = \mathcal{N}_3$ times in each iteration of the SAGA algorithm (Algorithm \ref{alg:SAGA}).  Then, after 3 loops of Algorithm from Section~\ref{subsec:proxfr_f_h} one can obtain next estimates on the number of oracles calls of $\nabla_x G(x, y)$, $\nabla_y G(x, y)$ and calculations of $prox_{\hat{f}}^{\lambda}(x')$, $prox_{\hat{h}}^{\lambda}(y')$:

\begin{equation*} 
     \widetilde {{\rm O}}\left( \max\left\{1, {\sqrt \frac{H_1} {\mu_x}}\right\} \right) 
     \cdot 
     \widetilde {{\rm O}}\left( \max\left\{1, {\sqrt \frac{H_2}{\mu_y}}\right\} \right) 
     \cdot
    \widetilde{O}\left(m_G + \frac{L_G^2}{\min(H_1 + \mu_x, H_2 + \mu_y)^2}\right).
\end{equation*}

Choose $H_1= \max\left\{\mu_x, \frac{L_G}{\sqrt{m_G}}\right\}, H_2= \max\left\{\mu_y, \frac{L_G}{\sqrt{m_G}}\right\}$
then:

\begin{multline*} \label{eq:f_x,h_y:prox-friendly_estimate}
    \widetilde {{\rm O}}\left( \max\left\{1, {\sqrt \frac{H_1} {\mu_x}}\right\} \right) 
     \cdot 
     \widetilde {{\rm O}}\left( \max\left\{1, {\sqrt \frac{H_2 }{\mu_y} }\right\} \right) 
     \cdot
    \widetilde{O}\left(m_G + \frac{L_G^2}{\min(H_1 + \mu_x, H_2 + \mu_y)^2}\right) \leq
    \\
     \widetilde {{\rm O}}\left( \max\left\{1, {\sqrt \frac{L_G }{\mu_x\sqrt{m_G}} }\right\} \right) \cdot
    \widetilde {{\rm O}}\left( \max\left\{1, {\sqrt \frac{L_G }{\mu_y\sqrt{m_G}} }\right\} \right) 
    \cdot
    \widetilde {{\rm O}}\left( m_G + L_G^2 \max\left(\frac{1}{H_1},\frac{1}{H_2}\right)^2  \right) \leq 
    \\
    \widetilde {{\rm O}}\left( \max\left\{1, {\sqrt \frac{L_G }{\mu_x\sqrt{m_G}} }\right\} \right) \cdot
    \widetilde {{\rm O}}\left( \max\left\{1, {\sqrt \frac{L_G }{\mu_y\sqrt{m_G}} }\right\} \right) 
     \cdot
    \widetilde {{\rm O}}\left( m_G + L_G^2 \frac{m_G}{L_G^2}   \right) =
    \\
    \widetilde {{\rm O}}\left( \max\left\{1, {\sqrt \frac{L_G }{\mu_x\sqrt{m_G}} }\right\} \right) \cdot
    \tilde {{\rm O}}\left( \max\left\{1, {\sqrt \frac{L_G }{\mu_y\sqrt{m_G}} }\right\} \right) 
     \cdot
    \widetilde {{\rm O}}\left( m_G \right) =
    \\
    \widetilde {{\rm O}}\left( \max\left\{m_G , m_G^{\frac{3}{4}}\sqrt{\frac{L_G}{\mu_x}},m_G^{\frac{3}{4}}\sqrt{\frac{L_G}{\mu_y}}, \frac{L_G\sqrt{m_G}}{\sqrt{\mu_x\mu_y} } \right\}\right) = \widetilde{O}\left(\sqrt{\frac{m_GL_G^2}{\mu_x\mu_y}}\right)
\end{multline*}

In the second inequality we used that $\frac{1}{H_1}, \frac{1}{H_2} \leq \frac{\sqrt{m_G}}{L_G}$, in the last equality we used that $\mu_x\sqrt{m_G} \leq L_G$ and $\mu_y\sqrt{m_G} \leq L_G$.
 
 To compute
 \begin{multline*}
     prox_{\hat{f}}^{\lambda}(x') = \arg \min_{x \in \R^{d_x}} \left\{\lambda\left(f(x) + \frac{H_1}{2}\|x - x'\|_2^2\right) + \frac{H_1 + \mu_x}{2}\|x - x'\|_2^2\right\} =
     \\
     \arg \min_{x \in \R^{d_x}} \left\{f(x) + \left(\frac{H_1}{2} + \frac{H_1 + \mu_x}{2\lambda}\right)\|x - x'\|_2^2\right\},
 \end{multline*}
 and
 \begin{multline*}
     prox_{\hat{h}}^{\lambda}(y') = \arg \max_{y \in \R^{d_y}} \left\{\lambda\left(- h(y) - \frac{H_2}{2}\|y - y'\|_2^2\right) - \frac{H_2 + \mu_y}{2}\|y - y'\|_2^2\right\} =
     \\
     \arg \min_{y \in \R^{d_y}} \left\{h(y) + \left(\frac{H_2}{2} + \frac{H_2 + \mu_y}{2\lambda}\right)\|y - y'\|_2^2\right\}
 \end{multline*}
 we should compute \eqref{f_x:prox-friendly} for the function $f(x)$ with $c_1 = 0$, $c_2 = \frac{H_1}{2} + \frac{H_1 + \mu_x}{2\lambda}$ and \eqref{f_x:prox-friendly} for the function $h(y)$ with $c_1 = 0$, $c_2 = \frac{H_2}{2} + \frac{H_2 + \mu_y}{2\lambda}$
 Then, after 3 loops of Algorithm from Section~\ref{subsec:proxfr_f_h} one can obtain next estimates on the number of oracles calls of $\nabla_x G(x, y)$, $\nabla_y G(x, y)$ and calculations of \eqref{f_x:prox-friendly} for functions $f(x)$ and $h(y)$:
\begin{equation*}
     \widetilde{O} \left(  \sqrt{\frac{m_G  L_G^2}{\mu_x\mu_y}}
    \right).
\end{equation*}
 \end{proof}

\section{Conclusions}\label{sec:conclusions}


%
%


\bibliographystyle{spmpsci}
\bibliography{references}






\appendix
\section{Proof of Theorem \ref{theoremCATDinexact}}\label{Appendix_A}


In this Appendix~\ref{Appendix_A} we rename the sequence of points $(x_k^{md}, x_k^t, x_k)$ (see listing of the Algorithm~\ref{alg:highorder_inexact}) to $(\tilde{x}_k, y_k, x_k)$.
We use the following definition to simplify calculations.
\begin{definition}
Let $(\varphi_{\delta, L_{\varphi}} (x), \nabla \varphi_{\delta, L_{\varphi}} (x))$ be a $(\delta, L_{\varphi})$ - oracle of function $\varphi$ at \pd{a} point $x$, then $\Omega_ {1,\delta, L_{\varphi}} \left (\varphi, z, x \right) $ is the following linear function of $z$:
\begin{equation}
    \Omega_{1, \delta, L_{\varphi}}\left(\varphi,x, z\right)=\varphi_{\delta, L_{\varphi}}(x)+ \langle \nabla \varphi_{\delta, L_{\varphi}} (x), z-x\rangle
\end{equation}
\end{definition}

To prove the Theorem \ref{theoremCATDinexact}, we need the following Theorem \ref{thm:MS}, which is based on Theorem 2.1 from \cite{bubeck2019near}.
\begin{theorem} \label{thm:MS}
Let $(y_k)_{k \geq 1}$~--- be a sequence in $\R^d$, and  $(\lambda_k)_{k \geq 1}$~--- a sequence in $\R_+$. Define $(a_k)_{k \geq 1}$ such that $\lambda_k A_k = a_k^2$ and $A_k = \sum_{i=1}^k a_i$. Define also for any $k\geq 0,x_k = x_0 - \sum_{i=1}^k a_i (\nabla \vp_{\delta, L_{\vp}}(y_i)+\nabla \psi_{\delta, L_{\psi}}(y_i))$ 
and $\tilde{x}_k := \frac{a_{k+1}}{A_{k+1}} x_{k} + \frac{A_k}{A_{k+1}} y_k$. Finally assume
if for some $\sigma \in [0,1]$ 
\begin{equation} \label{eq:igdrefined}
\|y_{k+1} - (\tilde{x}_k - \lambda_{k+1} (\nabla \vp_{\delta, L_{\vp}}(y_{k+1}) + \nabla \psi_{\delta, L_{\psi}}(y_{k+1})))\| \leq \sigma \cdot \|y_{k+1} - \tilde{x}_k\| \,,
\end{equation}
then one has for any $x \in \R^d$,
\begin{equation} \label{eq:rate}
F(y_k) - F(x) \leq \frac{ \|x - x_0\|^2}{2 A_k} + 2\left(\sum_{i = 1}^k A_{i} \right) \delta_2/A_k +\delta_1+ \left( \sum_{i=1}^{k-1} A_{i} \right)  \delta_1/A_k\,,
\end{equation}
\end{theorem}
To prove this Theorem we introduce auxiliaries Lemmas based on lemmas 2.2-2.5 and 3.1 from \cite{bubeck2019near}.

Consider a linear combination of gradients:
\[x_k =x_0 - \sum_{i=1}^k a_i (\nabla \vp_{\delta, L_{\vp}}(y_i) + \nabla \psi_{\delta, L_{\psi}}(y_i))\]
where coefficients $(a_i)_{i\geq1} \geq 0$ and points $(y_i)_{i\geq1}$ is not defined yet. A key observation for such a linear combination of gradients is that it minimizes the approximate lower bound of $F$.

\begin{lemma} \label{lem:basic1}
Let $\xi_0(x) = \frac{1}{2} \|x-x_0\|^2$ and define by induction  $\xi_{k}(x) = \xi_{k-1}(x) + a_{k} \left( \Omega_{1,\delta, L_{\vp}}(\vp, y_{k}, x)+ \Omega_{1,\delta, L_{\psi}}(\psi, y_{k}, x) \right) =\xi_{k-1}(x) + a_{k} \Omega_{1,2\delta, L_{\vp}+L_{\psi}}(F, y_{k}, x)$. Then $x_k =x_0 - \sum_{i=1}^k a_i (\nabla \vp_{\delta, L_{\vp}}(y_i) + \nabla \psi_{\delta, L_{\psi}}(y_i))$ is the minimizer of $\xi_k$, and 
$\xi_k(x) \leq A_k F(x) + \frac{1}{2} \|x-x_0\|^2 +A_k \delta_1$, where $A_k = \sum_{i=1}^k a_i$. 
\end{lemma}
\begin{proof}
Since $\xi_k(x)$ is strongly convex and smooth then expression
\begin{equation}
    \nabla \xi_k (x)=0
\end{equation}
is the criterion of minimum.

The sequence $x_k$ is satisfied
\begin{align}
    \nabla \xi_k (x_k)=\nabla \left(\left[\sum_{i=1}^k a_{i} \Omega_{1,2\delta, L_{\vp}+L_{\psi}}(F, y_{k}, x)\right] + \frac{1}{2} \|x_k-x_0\|^2 \right)=\\
    =\left[\sum_{i=1}^k a_{i} \left( \nabla \vp_{\delta, L_{\vp}}(y_i) + \nabla \psi_{\delta, L_{\psi}}(y_i)\right) \right]+x_k-x_0 =0.
\end{align}
Therefore, $x_k$ is a minimizer of the function $\xi_k$. Let us prove now that
\begin{equation}
    \Omega_{1,2\delta, L_{\vp}+L_{\psi}}(F, y_{k}, x)\leq F(x)+ \delta_1.
\end{equation}
From the definition of $\Omega_{1,2\delta, L_{\vp}+L_{\psi}}(F, y_{k}, x)$ we obtain
\begin{equation}
    \Omega_{1,2\delta, L_{\vp}+L_{\psi}}(F, y_{k}, x)=F_{2\delta,  L_{\vp}+L_{\psi}}(y_{i})+ \la \nabla F_{2\delta,  L_{\vp}+L_{\psi}} (y_i), x-y_i \ra \leq F(x)+\delta_1.
\end{equation}
Using $\xi_k(x)=\left[\sum_{i=1}^k a_{i} \Omega_{1,2\delta, L_{\vp}+L_{\psi}}(F, y_{k}, x)\right] + \frac{1}{2} \|x-x_0\|^2$ we obtain the statement of the theorem.
\qed 
\end{proof}
The next idea is to produce a control sequence $(z_k)_{k\geq1}$ demonstrating that $\xi_k$ is not too far below $A_k F$. From this we can directly yield a convergence rate for $z_k$.

\begin{lemma} \label{lem:basic2}
Let $(z_k)$ be a sequence such that 
\begin{equation} \label{eq:tosatisfy}
\xi_k(x_k) - A_k F(z_k) \geq - 2\left(\sum_{i = 1}^k A_{i} \right) \delta_2 -\left(\sum_{i = 1}^{k-1} A_{i} \right) \delta_1\,.
\end{equation}
Then one has for any $x$,
\begin{equation} \label{eq:tosatisfy2}
F(z_k) \leq F(x) + \frac{\|x-x_0\|^2}{2 A_k} + 2\left(\sum_{i = 1}^k A_{i} \right) \delta_2/A_k +\delta_1+ \left(\sum_{i = 1}^{k-1} A_{i} \right) \delta_1/A_k \,.
\end{equation}
\end{lemma}

\begin{proof}
Using Lemma \ref{lem:basic1} we obtain 
\begin{align}
    A_k F(z_k) \leq \xi_k(x_k) + 2\left(\sum_{i = 1}^k A_{i} \right) \delta_2+\left(\sum_{i = 1}^{k-1} A_{i} \right) \delta_1 \leq \xi_k(x) + 2\left(\sum_{i = 1}^k A_{i} \right) \delta_2 +\left(\sum_{i = 1}^{k-1} A_{i} \right) \delta_1\\
    \leq A_k F(x) +\frac{1}{2}\|x-x_0\|^2 + 2\left(\sum_{i = 1}^k A_{i} \right) \delta_2 +\left(\sum_{i = 1}^{k-1} A_{i} \right) \delta_1+A_k \delta_1\,.
\end{align}
\qed 
\end{proof}
Our aim now to get sequences $(a_k, y_k, z_k)$, satisfying  \eqref{eq:tosatisfy}.

\begin{lemma} \label{lem:basic3}
One has for any $x,z_k \in \mathbb{R}^d$ and $k\in \mathbb{N}$
\begin{align*}
& \xi_{k+1}(x) - A_{k+1} F(y_{k+1}) - (\xi_k(x_k) - A_k F(z_k)) \\
& \geq A_{k+1} \la \nabla \vp_{\delta, L_{\vp}}(y_{k+1}) +\nabla \psi_{\delta, L_{\psi}}(y_{k+1}),\frac{a_{k+1}}{A_{k+1}} x + \frac{A_k}{A_{k+1}} z_k - y_{k+1}\ra + \frac{1}{2} \|x -x_k\|^2-2A_{k+1} \delta_2- A_{k}  \delta_1\,.
\end{align*}
\end{lemma}

\begin{proof}
Firstly from $H( \xi_k) = I$ using that $x_k$ is a minimizer of $\xi_k(x)$ we get 
\[
\xi_k(x) = \xi_k(x_k) + \frac{1}{2} \|x- x_k\|^2,\] 
and
\[ \xi_{k+1}(x) = \xi_k(x_k) + \frac{1}{2} \|x-x_k\|^2 + a_{k+1} \Omega_{1, 2\delta, L_{\vp}+L_{\psi}} (F, y_{k+1}, x) \,,
\]
we can rewrite this as follows
\begin{equation} \label{eq:ind1}
\xi_{k+1}(x) - \xi_k(x_k) = a_{k+1} \Omega_{1, 2\delta, L_{\vp}+L_{\psi}}(F, y_{k+1}, x) + \frac{1}{2} \|x-x_k\|^2 \,.
\end{equation}
Now using \eqref{def:deltaL}:
\begin{equation} \label{eq:lem3}
    \Omega_{1,2\delta, L_{\vp}+L_{\psi}}(F, y_{k+1}, z_k)=F_{2\delta, L_{\vp}+L_{\psi}}(y_{k+1})+ \la \nabla F_{2\delta, L_{\vp}+L_{\psi}} (y_{k+1}), 
    z_k-y_{k+1}\ra \leq F(z_k)+  \delta_1
\end{equation}
 we obtain:

\begin{eqnarray*}
&&a_{k+1}\Omega_{1, 2\delta, L_{\vp}+L_{\psi}}(F, y_{k+1}, x)  =  A_{k+1} \Omega_{1,2\delta, L_{\vp}+L_{\psi}}(F, y_{k+1}, x)  \\
&-& A_k \Omega_{1, 2\delta, L_{\vp}+L_{\psi}}(F, y_{k+1}, x) =  A_{k+1} \Omega_{1, 2\delta, L_{\vp}+L_{\psi}}(F, y_{k+1}, x)   \\
 &-& A_k \la \nabla F_{2\delta, L_{\vp}+L_{\psi}}(y_{k+1}), x - z_k \ra - A_k \Omega_{1, 2\delta, L_{\vp}+L_{\psi}}(F, y_{k+1}, z_k)  \\
 & = & A_{k+1} \Omega_{1, 2\delta, L_{\vp}+L_{\psi}}\left(F, y_{k+1}, x - \frac{A_k}{A_{k+1}} (x - z_k) \right )-  A_k \Omega_{1, 2\delta, L_{\vp}+L_{\psi}}(F, y_{k+1}, z_k)  \\
 &=& A_{k+1} F_{2\delta, L_{\vp}+L_{\psi}}(y_{k+1})+  A_{k+1} \la \nabla F_{2\delta, L_{\vp}+L_{\psi}}(y_{k+1}),\left( x - \frac{A_k}{A_{k+1}} (x - z_k)\right) - y_{k+1}  \ra  \\
 &-& A_k \Omega_{1, 2\delta, L_{\vp}+L_{\psi}}(F, y_{k+1}, z_k) \stackrel{\eqref{eq:lem3}}{\geq}  A_{k+1} F_{2\delta, L_{\vp}+L_{\psi}}(y_{k+1}) - A_k F(z_k)- A_{k}  \delta_1\\
 &+ &A_{k+1} \la \nabla \vp_{\delta, L_{\vp}}(y_{k+1})+\nabla \psi_{\delta, L_{\psi}}(y_{k+1}),\frac{a_{k+1}}{A_{k+1}} x + \frac{A_k}{A_{k+1}} z_k - y_{k+1} \ra\\
  &  \stackrel{\eqref{def:deltaL}}{\geq}   & A_{k+1} F(y_{k+1})-2{A_{k+1} \delta_2} - A_k F(z_k)- A_{k}  \delta_1\\
 &+ &A_{k+1}  \la \nabla \vp_{\delta, L_{\vp}}(y_{k+1})+\nabla \psi_{\delta, L_{\psi}}(y_{k+1}),\frac{a_{k+1}}{A_{k+1}} x + \frac{A_k}{A_{k+1}} z_k - y_{k+1} \ra
 \,,
\end{eqnarray*}
which concludes the proof.
\qed 
 

\end{proof}

\begin{lemma} \label{lem:basic4}
Denoting  
\begin{equation}\label{eq:lambda}
    \lambda_{k+1} := \frac{a_{k+1}^2}{A_{k+1}}
\end{equation}
 and $\tilde{x}_k := \frac{a_{k+1}}{A_{k+1}} x_{k} + \frac{A_k}{A_{k+1}} y_k$, one has:
\begin{align*}
& \xi_{k+1}(x_{k+1}) - A_{k+1} F(y_{k+1}) - (\xi_k(x_k) - A_k F(y_k)) \geq \\
& \frac{A_{k+1}}{2 \lambda_{k+1}} \bigg( \|y_{k+1} - \tilde{x}_k\|^2 - \|y_{k+1} - (\tilde{x}_k - \lambda_{k+1} (\nabla \vp_{\delta, L_{\vp}}(y_{k+1}))+\nabla \psi_{\delta, L_{\psi}}(y_{k+1})) \|^2 \bigg) - 2A_{k+1} \delta_2  - A_{k}  \delta_1\,.
\end{align*}
In particular, we have in light of \eqref{eq:igdrefined}
$$\xi_{k}(x_{k})-A_{k}F(y_{k})\geq\frac{1-\sigma^{2}}{2}\sum_{i=1}^{k}\frac{A_{i}}{\lambda_{i}}\|y_{i}-\tilde{x}_{i-1}\|^{2} - 2\left(\sum_{i = 1}^k A_{i} \right) \delta_2 - \left( \sum_{i=1}^{k-1} A_{i} \right)  \delta_1.$$
\end{lemma}

\begin{proof}
We apply Lemma \ref{lem:basic3} with $z_k = y_k$ and $x=x_{k+1}$ 
, and 
note that (with $\tilde{\zeta} := \frac{a_{k+1}}{A_{k+1}} x + \frac{A_k}{A_{k+1}} y_k$): 
\begin{align*}
& \la \nabla \vp_{\delta, L_{\vp}}(y_{k+1})+ \nabla \psi_{\delta, L_{\psi}}(y_{k+1}), \frac{a_{k+1}}{A_{k+1}} x + \frac{A_k}{A_{k+1}} y_k - y_{k+1} \ra  + \frac{1}{2 A_{k+1}} \|x - x_k\|^2 \\
& = \la \nabla \vp_{\delta, L_{\vp}}(y_{k+1})+\nabla \psi_{\delta, L_{\psi}}(y_{k+1}), \tilde{\zeta} - y_{k+1}\ra  + \frac{1}{2 A_{k+1}} \left\|\frac{A_{k+1}}{a_{k+1}} \left(\tilde{\zeta} - \frac{A_k}{A_{k+1}} y_k \right) - x_k \right\|^2 \\
& = \la \nabla \vp_{\delta, L_{\vp}}(y_{k+1})+\nabla \psi_{\delta, L_{\psi}} (y_{k+1}), \tilde{\zeta} - y_{k+1} \ra + \frac{A_{k+1}}{2 a_{k+1}^2} \left\|\tilde{\zeta} - \left(\frac{a_{k+1}}{A_{k+1}} x_k + \frac{A_k}{A_{k+1}} y_k \right) \right\|^2 \,.
\end{align*}

This yields, using \eqref{eq:lambda}:
\begin{align*}
& \xi_{k+1}(x_{k+1}) - A_{k+1} F(y_{k+1}) - (\xi_k(x_k) - A_k F(y_k)) \\
& \geq A_{k+1} \cdot  \la \nabla \vp_{\delta, L_{\vp}}(y_{k+1})+\nabla \psi_{\delta, L_{\psi}}(y_{k+1})), \tilde{\zeta} - y_{k+1}) \ra + \frac{A_{k+1}}{2 \lambda_{k+1}} \|\tilde{\zeta} - \tilde{x}_k\|^2  -2 A_{k+1} \delta_2 - A_{k}  \delta_1\\
& \geq A_{k+1} \cdot \min_{x \in \R^d} \left\{ \la \nabla \vp_{\delta, L_{\vp}}(y_{k+1})+\nabla \psi_{\delta, L_{\psi}}(y_{k+1}), x - y_{k+1}\ra  + \frac{1}{2 \lambda_{k+1}} \|x - \tilde{x}_k\|^2 \right\} -2 A_{k+1} \delta_2 - A_{k}  \delta_1\,.
\end{align*}
The value of the minimum is easy to compute. Due to the strong convexity of the minimized function and its continuous differentiability, achieving a minimum is equivalent to the condition
\begin{align*}
    & 0=\nabla \left[ \la \nabla \vp_{\delta, L_{\vp}}(y_{k+1})+\nabla \psi_{\delta,L_{\psi}}(y_{k+1}), x - y_{k+1} \ra + \frac{1}{2 \lambda_{k+1}} \|x - \tilde{x}_k\|^2\right]=\\
    & =(\nabla \vp_{\delta, L_{\vp}}(y_{k+1})+\nabla \psi_{\delta, L_{\psi}}(y_{k+1})) + \frac{1}{ \lambda_{k+1}} \left( x - \tilde{x}_k\right) 
\end{align*}
Then
\begin{align*}
    x_{*}= \tilde{x}_k-\lambda_{k+1}(\nabla \vp_{\delta, L_{\vp}}(y_{k+1})+\nabla \psi_{\delta, L_{\psi}}(y_{k+1})) 
\end{align*}
Substituting into the last inequality we obtain the statement of the theorem.
\qed 

\end{proof}
\begin{proof} \textbf{of the Theorem \ref{thm:MS}}

Using Lemma \ref{lem:basic4} we get
\begin{align*}
  \xi_{k}(x_{k})-A_{k}F(y_{k})\geq\frac{1-\sigma^{2}}{2}\sum_{i=1}^{k}\frac{A_{i}}{\lambda_{i}}\|y_{i}-\tilde{x}_{i-1}\|^{2} - 2\left(\sum_{i = 1}^k A_{i} \right) \delta_2 - \left( \sum_{i=1}^{k-1} A_{i} \right)  \delta_1\\
  \geq - 2\left(\sum_{i = 1}^k A_{i} \right) \delta_2 - \left( \sum_{i=1}^{k-1} A_{i} \right)  \delta_1.  
\end{align*}
Applying Lemma \ref{lem:basic2} for $z_k=y_k$ one has for any $ x \in \mathbb{R}^d$:
\begin{equation} 
F(y_k) - F(x) \leq \frac{ \|x-x_0\|^2}{2 A_k} + 2\left(\sum_{i = 1}^k A_{i} \right) \delta_2 /A_k +\delta_1+ \left( \sum_{i=1}^{k-1} A_{i} \right)  \delta_1/A_k \,,
\end{equation}
\qed 
\end{proof}
that conclude the proof.


Now one will formulate the sufficient condition \eqref{eq:prox_step_inexact_crit} for the accuracy of solving auxiliary problems \eqref{prox_step_inexact}.
Let us assume, that auxiliary problems \eqref{prox_step_inexact} can not be solved exactly. Let the algorithm only have an inaccurate solution ${y}_{k+1}$ satisfying 
\begin{align}
    \eqref{eq:prox_step_inexact_crit}:&\|\nabla \left( \Omega_{1,\delta,L_{\vp}}\left(\vp, \tilde{x}_{k}, {y}_{k+1}\right)+\frac{H}{2}\left\|{y}_{k+1}-\tilde{x}_{k}\right\|^2 \right)+\nabla \psi_{\delta,L_{\psi}}\left({y}_{k+1}\right)\|\nonumber\\
    &\leq \frac{H}{4} \|y_{k+1} - \tilde{x}_k\|-2\sqrt{2\delta_2 L_{\vp}} \nonumber 
\end{align}
in this case:

\begin{lemma} \label{lem:controlstepsize}
Assume that $\vp(x)$ has $(\delta,L_{\vp})$~-oracle, $\psi(x)$ has $(\delta,L_{\psi})$~-oracle and the auxiliary subproblem \eqref{prox_step_inexact} is solved inexactly in such a way that the inequality \eqref{eq:prox_step_inexact_crit} holds. If 
$$H\geq 2L_{\vp}$$ then equation \eqref{eq:igdrefined} holds true with $\sigma = 7/8$ for \eqref{prox_step_inexact}. In the case $p = 1$ one can consider $\lambda_{k+1} = \lambda =  \frac{1}{2H}$.
\end{lemma}
\begin{proof}

Using that $\vp$ is equipped with a $(\delta, L_{\vp})$~-oracle and Corollary 4.2. from \cite{devolder2013exactness} one obtains:
\begin{equation}\label{eq:speedstrongconvex}
\|\nabla \vp_{\delta, L_{\vp}}(y) - \nabla_y \Omega_{1, \delta, L_{\vp}}(\vp,x, y)\| \leq L_{\vp} \|y - x\| + 2\sqrt{2 L_{\vp}\delta_2} \,.
\end{equation}
By \eqref{eq:prox_step_inexact_crit} and \eqref{eq:speedstrongconvex} we can get next inequalities:
\begin{align*}
& \|y_{k+1} - (\tilde{x}_k - \lambda_{k+1} (\nabla \vp_{\delta, L_{\vp}}(y_{k+1})+ \nabla \psi_{\delta, L_{\psi}}(y_{k+1}))) \| \\
&= \|y_{k+1} \pm\lambda_{k+1} \nabla_y \Omega_{1,\delta ,L_{\vp}}(\vp, \tilde{x}_k, y_{k+1})\pm H \lambda_{k+1} (y_{k+1}-\tilde{x}_k) \\
&- (\tilde{x}_k - \lambda_{k+1} (\nabla \vp_{\delta, L_{\vp}}(y_{k+1})+ \nabla \psi_{\delta, L_{\psi}}(y_{k+1}))) \| \leq (1-H \lambda_{k+1})\|(y_{k+1}-\tilde{x}_k)  \|\\
&+ \lambda_{k+1}\| \nabla_y \Omega_{1,\delta ,L_{\vp}}(\vp, \tilde{x}_k, y_{k+1}) + H (y_{k+1} - \tilde{x}_k) + \nabla \psi_{\delta, L_{\psi}}(y_{k+1})\| \\
&+ \lambda_{k+1} \|\nabla \vp_{\delta, L_{\vp}}(y_{k+1})- \nabla_y \Omega_{1,\delta, L_{\vp}}(\vp, \tilde{x}_k, y_{k+1}) \| \stackrel{\eqref{eq:prox_step_inexact_crit},\eqref{eq:speedstrongconvex}}{\leq}  (1-H \lambda_{k+1})\|y_{k+1}-\tilde{x}_k \| \\
& +\lambda_{k+1}\frac{H}{4 }\left\|{y}_{k+1}-\tilde{x}_{k}\right\| -2 \lambda_{k+1} \sqrt{2L_{\vp}\delta_2}+  \lambda_{k+1} \left(L_{\vp} \|{y}_{k+1}-\tilde{x}_{k}\| + 2 \sqrt{2 L_{\vp}\delta_2} \right)  \\
&\leq  \left(\frac{5}{8}+\frac{L_{\vp}}{2H}\right)\|y_{k+1}-\tilde{x}_k  \|
 \leq \frac{7}{8}\|y_{k+1}-\tilde{x}_k  \|
\end{align*}
that ends the proof.
\qed 
\end{proof}


Recall from Lemma \ref{lem:basic2} that the rate of convergence of AM-\ref{alg:highorder_inexact} is  $\|x_0 - x^*\|/A_k + 2\left(\sum_{i = 1}^k A_{i} \right) \delta_2/A_k+\delta_1+ \left(\sum_{i = 1}^{k-1} A_{i} \right) \delta_1/A_k$. We now finally give an estimate of $A_k$:
\begin{lemma} \label{lemma:Ak}
Suppose $H \geq 2L_{\vp}$. Then one has, with $c_1 = 4$,
\begin{align} \label{eq:Ak}
    A_k \geq \frac{k^2}{c_1 H}
\end{align}
\end{lemma}
\begin{proof}
In case, when $p = 1$ $\lambda_{k+1}$ are defined as
\begin{align*}
    \lambda_{k+1} = \frac{1}{2H}.
\end{align*}
Inequality \eqref{eq:Ak} holds when $k = 1$.

Let us proof that if \eqref{eq:Ak} holds for $k$ then it holds for $k + 1$. Using the definition of $A_k$%
\begin{align*}
    a_{k+1} = \frac{\lambda_{k+1}+\sqrt{\lambda_{k+1}^2+4\lambda_{k+1}A_k}}{2} 
\text{ , } 
A_{k+1} = A_k+a_{k+1},
\end{align*}
 we obtain
 \begin{align*}
     A_{k+1} \geq \frac{k^2}{2c_1L_{\vp}} + \frac{1}{8L_{\vp}} \left( 1 + \sqrt{1 + \frac{16k^2}{c_1}} \right) \geq \frac{(k+1)^2}{2c_1L_{\vp}}.
 \end{align*}
\
\end{proof}
\begin{proof} \textbf{of the Theorem \ref{theoremCATDinexact}}

To prove the Theorem \ref{theoremCATDinexact} it suffices to combine Lemmas \ref{lem:controlstepsize},\ref{lemma:Ak} with Theorem  \ref{thm:MS}.

\end{proof}
\section{Proof of Theorem \ref{CATDrestarts_inexact_notconvex}}\label{Appendix_B}
In this Appendix~\ref{Appendix_B} we rename the sequence of points $(x_k^{md}, x_k^t, x_k)$ (see listing of the Algorithm~\ref{alg:highorder_inexact}) to $(\tilde{x}_k, y_k, x_k)$.

\begin{proof} 
Firstly, let us choose $\delta$ according to \eqref{delta_CATD1}:
\begin{equation*} 
    \forall k: \delta_1 + \delta_2 + 2\left(\sum_{i = 1}^k A_{i} \right) \delta_2/A_k + 2\left(\sum_{i = 1}^{k-1} A_{i} \right) \delta_1/A_k \leq \frac{\varepsilon}{2},
\end{equation*}
where $\varepsilon$ is solution accuracy in terms of $F(x) - F(x_{\ast}) \leq \varepsilon$.

Then, as $\varepsilon/2 \leq c_1 H R^2/k^2$ with $c_1 = 4$, next inequality holds true
\begin{equation*} 
    \forall k: \delta_1 + \delta_2 + 2\left(\sum_{i = 1}^k A_{i} \right) \delta_2/A_k + 2\left(\sum_{i = 1}^{k-1} A_{i} \right) \delta_1/A_k \leq \frac{c_1 H R^2}{k^2} .
\end{equation*}

From $(\delta, L, \mu)$ - oracle definition \eqref{def:deltaLmu}
 we get 
\begin{align} \label{RestartdeltaLmu}
     \frac{\mu}{2}\|z-x_*\|^{2} - \delta_1 \leq F(z)-\left(F_{\delta, L,\mu}(x_*)+\left\langle \nabla F_{\delta, L,\mu}(x_*), z-x_*\right\rangle\right) = \\ \nonumber
      = (F(z)- F_{\delta, L,\mu}(x_*)) + \left\langle \partial F(x_*) - \nabla F_{\delta, L,\mu}(x_*), z-x_*\right\rangle - \left\langle \partial F(x_*), z-x_*\right\rangle \leq \\ \nonumber
      \leq (F(z)- F(x_*) + \delta_2) + \sqrt{2\delta_2 L} \|z-x_*\|.
\end{align}
Therefore
\begin{equation*}
    \frac{\mu}{2}\|z-x_*\|^{2} - \sqrt{2\delta_2 L} \|z-x_*\| \leq (F(z)- F(x_*) + \delta_1+\delta_2).
\end{equation*}
If $\delta_2$ is small enough such that
\begin{equation*} 
    \frac{4\sqrt{2\delta_2 L}}{\mu} \leq \varepsilon/2,
\end{equation*}
then taking into account that $\forall k: \varepsilon/2 \leq \|z_{k+1}-x_*\|$  we obtain
\begin{equation} \label{DLM:restarts_delta}
    \forall k: \frac{\mu}{4}\|z_{k+1}-x_*\|^{2} \geq \sqrt{2\delta_2 L} \|z_{k+1}-x_*\|,
\end{equation}
which implies the following inequality
\begin{equation}\label{DLM:restarts}
    \frac{\mu}{4}\|z_k-x_*\|^{2} \leq (F(z_k)- F(x_*) + \delta_1+\delta_2).
\end{equation}
Finally, we can conclude that $R_k$ decreases as a geometric progression:
\begin{align*}
    R_{k+1}&=\|z_{k+1}-x_{\ast}\| \overset{\eqref{DLM:restarts}}{\leq}  \left( \frac{4 \left( F(z_{k+1})-F(x_{\ast}) + \delta_1 + \delta_2 \right)}{\mu} \right)^{\frac{1}{2}} \\
    &\overset{\eqref{speedCATDinexact}}{\leq}\left( \frac{4 \left( \frac{c_1 H R_k^2}{N_k^2} + 2\left(\sum_{i = 1}^k A_{i} \right) \delta_1/A_k + 2\left(\sum_{i = 1}^{k-1} A_{i} \right) \delta_2/A_k + \delta_1 + \delta_2 \right)}{\mu} \right)^{\frac{1}{2}}\\
    &\overset{\eqref{delta_CATD1}}{\leq} \left( \frac{4 \left( \frac{2 c_1 H R_{k}^{2}}{N_k^{2}} \right)}{\mu} \right)^{\frac{1}{2}} =\left( \frac{8 c_1 H R_{k}^{2}}{\mu N_k^{2}} \right)^{\frac{1}{2}}
    \overset{\eqref{numberofrestarts}}{\leq} \left( \frac{ R_{k}^{2}}{2^2 }\right)^{\frac{1}{2}} = \frac{ R_{k}}{2}.
\end{align*}
Which in turn guarantees that
\begin{align}
    F(z_K)- F(x_*) \leq \frac{\mu R_0^2}{4\cdot 4^K}.
\end{align}
It is sufficient to choose $K = 2\log_2 \frac{\mu R_0^2}{4\varepsilon}$ in order that $F(z_k)- F(x_*) \leq \varepsilon$.

Now we compute the total number of AM steps.
\begin{align*}
    \sum\limits_{k=0}^K N_k &\leq \sum\limits_{k=0}^K \left( \frac{32 c_1 H }{\mu} \right)^{\frac{1}{2}}+K\leq \sum\limits_{k=0}^K \left( \frac{32 c_1 H }{\mu}  \right)^{\frac{1}{2}}+K\\
    &=  \left( \frac{32 c_1 H }{\mu}  \right)^{\frac{1}{2}}K+K = \left( \sqrt{\frac{128H}{\mu}} + 1\right)\cdot 2\log_2 \frac{\mu R_0^2}{4\varepsilon} \leq \left( 16\sqrt{2}\sqrt{\frac{H}{\mu}} + 2\right) \log_2 \frac{\mu R_0^2}{\varepsilon}
\end{align*}
\qed 
\end{proof}
\section{Proof of Theorem \ref{AM:comfortable_view} and Theorem \ref{AM:comfortable_view_with_prob}}
\label{Appendix_C}

The Theorem \ref{AM:comfortable_view} show that the fulfillment of condition \eqref{eq:prox_step_inexact_crit} keep the linear rate of convergence when solving the auxiliary problems \eqref{prox_step_inexact}. Also in this Appendix~\ref{Appendix_C} we rename the sequence of points $(x_k^{md}, x_k^t, x_k)$ (see listing of the Algorithm~\ref{alg:highorder_inexact}) to $(\tilde{x}_k, y_k, x_k)$.\\
Firstly, based on \eqref{eq:prox_step_inexact_crit} we try to relate the accuracy $\tilde{\varepsilon}$ we need to solve \eqref{prox_step_inexact} in terms of the following criteria:
\begin{align}
    &\|\nabla \left( \Omega_{1,\delta,L_{\vp}}\left(\vp, \tilde{x}_{k}, {y}_{k+1}\right)+\frac{H}{2}\left\|{y}_{k+1}-\tilde{x}_{k}\right\|^2 \right)+\nabla \psi_{\delta,L_{\psi}}\left({y}_{k+1}\right)\|\leq \tilde{\varepsilon}. \label{eq:prox_step_inexact_crit_eps} 
\end{align}
For this we prove the auxiliary lemma for $(\delta,L_{\vp})$~-oracle of $\vp$ and $(\delta,L_{\psi})$~-oracle of $\psi$, that is based on the Lemma 2.1 from \cite{grapiglia2020inexact} .

\begin{lemma}\label{lem:SufficConditForStop}
Let $\tilde{x}_k \in \mathbb{R}^d, H,\Theta >0$.\\
Assume that $\vp(x)$ admits $(\delta,L_{\vp})$~-oracle, $\psi(x)$ admits $(\delta,L_{\psi})$~-oracle. If inquality
\begin{align}\label{eq:lemSufficCondit1}
    &\|\nabla \left( \Omega_{1,\delta,L_{\vp}}\left(\vp, \tilde{x}_{k}, {y}_{k+1}\right)+\frac{H}{2}\left\|{y}_{k+1}-\tilde{x}_{k}\right\|^2 \right) +\nabla \psi_{\delta,L_{\psi}}\left({y}_{k+1}\right)\|\\
    &\leq \min \left\{ \frac{1}{2}, \frac{\Theta}{2\left[ L_{\vp} +H \right]} \right\} \left( \|\nabla \vp_{\delta,L_{\vp}} (y_{k+1}) +\nabla \psi_{\delta,L_{\psi}}\left({y}_{k+1}\right)\|\right)
\end{align}
holds true, then $y_{k+1}$ satisfies 
\begin{align}\label{eq:lemSufficConditRes}
    &\|\nabla \left( \Omega_{1,\delta,L_{\vp}}\left(\vp, \tilde{x}_{k}, {y}_{k+1}\right)+\frac{H}{2}\left\|{y}_{k+1}-\tilde{x}_{k}\right\|^2 \right)+\nabla \psi_{\delta,L_{\psi}}\left({y}_{k+1}\right)\|\\
    &\leq \Theta \|y_{k+1} - \tilde{x}_k\|+\frac{2 \Theta}{\left[ L_{\vp}+H\right]}\sqrt{2 L_{\vp}\delta_2}
\end{align}
\end{lemma}
\begin{proof}
Using that $\vp$ is equipped with a $(\delta, L_{\vp})$~-oracle and Corollary 4.2. from \cite{devolder2013exactness} one obtains:
\begin{equation}\label{eq:lemSufficCondit2}
\|\nabla \vp_{\delta, L_{\vp}}(y) - \nabla_y \Omega_{1, \delta, L_{\vp}}(\vp,x, y)\| \leq L_{\vp} \|y - x\| + 2\sqrt{2 L_{\vp}\delta_2} \,.
\end{equation}
Combining \eqref{eq:lemSufficCondit1} and \eqref{eq:lemSufficCondit2} we obtain
\begin{align*}
    &\|\nabla \vp_{\delta,L_{\vp}} (y_{k+1}) +\nabla \psi_{\delta,L_{\psi}}\left({y}_{k+1}\right)\| \\
    &\leq\|\nabla \vp_{\delta,L_{\vp}} (y_{k+1})+\nabla \psi_{\delta,L_{\psi}}\left({y}_{k+1}\right)-\nabla \psi_{\delta,L_{\psi}}\left({y}_{k+1}\right)-\nabla \Omega_{1,\delta,L_{\vp}}\left(\vp, \tilde{x}_{k}, {y}_{k+1}\right)\| \\
    &+\|\nabla \Omega_{1,\delta,L_{\vp}}\left(\vp, \tilde{x}_{k}, {y}_{k+1}\right)\pm \nabla \psi_{\delta,L_{\psi}}\left({y}_{k+1}\right)-\nabla \left( \Omega_{1,\delta,L_{\vp}}\left(\vp, \tilde{x}_{k}, {y}_{k+1}\right)+\frac{H}{2}\left\|{y}_{k+1}-\tilde{x}_{k}\right\|^2 \right)\|\\
    &+\|\nabla \left( \Omega_{1,\delta,L_{\vp}}\left(\vp, \tilde{x}_{k}, {y}_{k+1}\right)+\frac{H}{2}\left\|{y}_{k+1}-\tilde{x}_{k}\right\|^2 \right) +\nabla \psi_{\delta,L_{\psi}}\left({y}_{k+1}\right)\|\\
    &\stackrel{\eqref{eq:lemSufficCondit2},\eqref{eq:lemSufficCondit1}}{\leq} \left( L_{\vp} \|y_{k+1} - \tilde{x}_k\| + 2\sqrt{2 L_{\vp}\delta_2} \right) + H\|y_{k+1} - \tilde{x}_k\| +\frac{1}{2}\|\nabla \vp_{\delta,L_{\vp}} (y_{k+1})+\nabla \psi_{\delta,L_{\psi}}\left({y}_{k+1}\right)\|.
\end{align*}
Thus,
\begin{equation}
    \frac{\|\nabla \vp_{\delta,L_{\vp}} (y_{k+1})+\nabla \psi_{\delta,L_{\psi}}\left({y}_{k+1}\right)\|}{2}\leq \left[ L_{\vp}+H\right]\|y_{k+1} - \tilde{x}_k\|+2\sqrt{2 L_{\vp}\delta_2}
\end{equation}
which gives
\begin{equation}\label{eq:lemSufficCondit3}
    \frac{\Theta }{2\left[ L_{\vp}+H\right]}\|\nabla \vp_{\delta,L_{\vp}} (y_{k+1})+\nabla \psi_{\delta,L_{\psi}}\left({y}_{k+1}\right)\|\leq \Theta \|y_{k+1} - \tilde{x}_k\|+\frac{2 \Theta}{\left[ L_{\vp}+H\right]}\sqrt{2 L_{\vp}\delta_2}
\end{equation}
Finally, \eqref{eq:lemSufficConditRes} follows directly from the \eqref{eq:lemSufficCondit1} and \eqref{eq:lemSufficCondit3}.
\qed 
\end{proof}

\begin{lemma} \label{lem:ff}
Assume that $H\geq 2 L_{\vp}$, $\vp(x)$ admits $(\delta,L_{\vp})$-oracle, $\psi(x)$ admits $(\delta,L_{\psi})$-oracle, $F(x)$ admits $(2\delta,L_{\vp}+L_{\psi},\mu)$-oracle; ${y}_{k+1},\tilde{x}_k \in \mathbb{R}^d$ and $\varepsilon \in (0,1)$. If inequalities 
\begin{equation}\label{eq:lemmffF}
    \varepsilon\leq F_{2\delta, L_{\vp}+L_{\psi},\mu}(y)-\min_{x \in Q_f} F(x)
\end{equation}
\begin{equation}\label{eq:lemmffDelta}
    \delta_2 \leq \frac{\varepsilon \mu}{ 64^2 \cdot L_{\vp}},
\end{equation}
are satisfied then inequality \eqref{eq:prox_step_inexact_crit} holds true if one solve the auxiliary problem  \eqref{prox_step_inexact} with the accuracy  \todo{Correct this part, hard to understand and put here precise dependence $\tilde{\varepsilon} = \sqrt{2\varepsilon\mu}/64$???}
\begin{equation}\label{eq:lemmffEpsil}
    \tilde{\varepsilon}=\frac{\sqrt{\varepsilon \mu}}{72}
\end{equation}
in terms of criteria \eqref{eq:prox_step_inexact_crit_eps}.
\end{lemma}
\begin{proof}
According to the conditions of the lemma, the problem \eqref{prox_step_inexact} is solved with the accuracy
\begin{equation}
    \eqref{eq:prox_step_inexact_crit_eps}:\|\nabla \left( \Omega_{1,\delta,L_{\vp}}\left(\vp, \tilde{x}_{k}, {y}_{k+1}\right)+\frac{H}{2}\left\|{y}_{k+1}-\tilde{x}_{k}\right\|^2 \right) +\nabla \psi_{\delta,L_{\psi}}\left({y}_{k+1}\right)\|\leq \tilde{\varepsilon}\nonumber
\end{equation}
To prove the lemma, it suffices to show the following chain of inequalities
\begin{align}
    & \tilde{\varepsilon} \leq \min \lbrace \frac{1}{2}, \frac{H}{8\left[ L_{\vp} +H \right]} \rbrace \left( \|\nabla \vp_{\delta,L_{\vp}} (y_{k+1}) +\nabla \psi_{\delta,L_{\psi}}\left({y}_{k+1}\right)\|\right)-\left(2+\frac{H}{2\left[ L_{\vp}+H \right]}\right)\sqrt{2\delta_2 L_{\vp}}\nonumber \\
    &\leq \frac{H}{4} \|y_{k+1} - \tilde{x}_k\|-2\sqrt{2\delta_2 L_{\vp}}\label{eq:lemff1}
\end{align}
Lemma \ref{lem:SufficConditForStop} for $\Theta=\frac{H}{4}$ guarantee that if the next inequality holds true
\begin{align}
    &\|\nabla \left( \Omega_{1,\delta,L_{\vp}}\left(\vp, \tilde{x}_{k}, {y}_{k+1}\right)+\frac{H}{2}\left\|{y}_{k+1}-\tilde{x}_{k}\right\|^2 \right) +\nabla \psi_{\delta,L_{\psi}}\left({y}_{k+1}\right)\| \label{eq:lemff2}\\
    &\leq \min \lbrace \frac{1}{2}, \frac{H}{8\left[ L_{\vp} +H \right]} \rbrace \left( \|\nabla \vp_{\delta,L_{\vp}} (y_{k+1}) +\nabla \psi_{\delta,L_{\psi}}\left({y}_{k+1}\right)\|\right)-\left(2+\frac{H}{2\left[ L_{\vp}+H \right]}\right)\sqrt{2\delta_2 L_{\vp}} \nonumber
\end{align}
then the equation \eqref{eq:prox_step_inexact_crit} is satisfied.\\
If \eqref{eq:lemff2} is sufficient condition for \eqref{eq:prox_step_inexact_crit}, it means that right-hand sides of \eqref{eq:lemff2} less the right-hand sides of \eqref{eq:prox_step_inexact_crit}. From this consequence the next inequality\\
\begin{align}
    &\frac{1}{12} \left( \|\nabla F_{2\delta,L_{\vp}+L_{\psi}} (y_{k+1}) \|\right)-\frac{5}{2}\sqrt{2\delta_2 L_{\vp}} = \frac{1}{12} \left( \|\nabla \vp_{\delta,L_{\vp}} (y_{k+1}) +\nabla \psi_{\delta,L_{\psi}}\left({y}_{k+1}\right)\|\right)-\frac{5}{2}\sqrt{2\delta_2 L_{\vp}} \\
    &\leq \min \lbrace \frac{1}{2}, \frac{H}{8\left[ L_{\vp} +H \right]} \rbrace \left( \|\nabla \vp_{\delta,L_{\vp}} (y_{k+1}) +\nabla \psi_{\delta,L_{\psi}}\left({y}_{k+1}\right)\|\right)-\left(2+\frac{H}{2\left[ L_{\vp}+H \right]}\right)\sqrt{2\delta_2 L_{\vp}}  \\
    &\leq  \frac{H}{4} \|y_{k+1} - \tilde{x}_k\|-2\sqrt{2\delta_2 L_{\vp}} \label{eq:lemff5}
\end{align}
The second inequality of the equation \eqref{eq:lemff1} is satisfied, let us prove the first one.\\
The fact that F has $(2\delta,L_{\vp}+L_{\psi},\mu)$~-oracle guarantee
\begin{equation}\label{eq:lemmff3}
     \frac{\mu}{2}\|x-y\|^{2}+\left(F_{2\delta, L_{\vp}+L_{\psi},\mu}(y)+\left\langle \nabla F_{2\delta, L_{\vp}+L_{\psi},\mu}(y), x-y\right\rangle\right)\leq F(x) \text { for all } x \in Q_f
\end{equation}
Let us minimize the right-hand and left-hand sides of \eqref{eq:lemmff3} with respect to x independently
\begin{align}
    &F^*=\min_{x \in Q_f} F(x) \stackrel{\eqref{eq:lemmff3}}{\geq} F_{2\delta, L_{\vp}+L_{\psi},\mu}(y)+
    \min_{x \in Q_f}\left\{ \frac{\mu}{2}\|x-y\|^{2}+\left\langle \nabla F_{2\delta, L_{\vp}+L_{\psi},\mu}(y), x-y\right\rangle\right\} \nonumber\\
    &=F_{2\delta, L_{\vp}+L_{\psi},\mu}(y)-\frac{1}{2\mu} \|\nabla F_{2\delta, L_{\vp}+L_{\psi},\mu}(y)\|^2\nonumber
\end{align}
Then obtain
\begin{equation}\label{eq:lemmff4}
    \varepsilon\stackrel{\eqref{eq:lemmffF}}{\leq} F_{2\delta, L_{\vp}+L_{\psi},\mu}(y)-F^* \leq \frac{1}{2\mu} \|\nabla F_{2\delta, L_{\vp}+L_{\psi},\mu}(y)\|^2
\end{equation}
Inequality \eqref{eq:lemmff4} guarantee that 
\begin{equation}\label{eq:lemmff6}
    \frac{1}{2}\sqrt{\frac{\varepsilon \mu}{18}} -\frac{5}{2} \sqrt{2\delta_2 L_{\vp}}\leq \frac{1}{12} \|\nabla F_{2\delta, L_{\vp}+L_{\psi},\mu}(y)\|-\frac{5}{2} \sqrt{2\delta_2 L_{\vp}}
\end{equation}
In case of \eqref{eq:lemmffDelta} inequality \eqref{eq:lemmff6} give us guarantees that the first inequality of the equation \eqref{eq:lemff1} holds true
\begin{equation} \label{inexact_precision}
    \tilde{\varepsilon}=\frac{\sqrt{\varepsilon \mu}}{72}\stackrel{\eqref{eq:lemmffDelta}}{\leq} \frac{1}{2}\sqrt{\frac{\varepsilon \mu}{18}} -\frac{5}{2} \sqrt{2\delta_2 L_{\vp}}  \stackrel{\eqref{eq:lemmff6}}{\leq} \frac{1}{12} \|\nabla F_{2\delta, L_{\vp}+L_{\psi},\mu}(y)\|-\frac{5}{2} \sqrt{2\delta_2 L_{\vp}}
\end{equation}
Finally, combine the equations \eqref{eq:lemff5} and \eqref{inexact_precision} obtain the required chain of inequalities \eqref{eq:lemff1}.
\qed 

\end{proof}

Let us prove the Theorem \ref{AM:comfortable_view} using Lemma \ref{lem:ff}:
\begin{proof}
Firstly, let us collect all restrictions on $\delta_1, \delta_2$ and auxiliary problem precision for obtaining convergence of outer Algorithm-\ref{alg:restarts_inexact_notconvex} and fulfillment of the Lemma \ref{lem:ff} together:
\begin{align*}
    \eqref{eq:prox_step_inexact_crit_eps}:  & \|\nabla \left( \Omega_{1,\delta,L_{\vp}}\left(\vp, \tilde{x}_{k}, {y}_{k+1}\right)+\frac{H}{2}\left\|{y}_{k+1}-\tilde{x}_{k}\right\|^2 \right) +\nabla \psi_{\delta,L_{\psi}}\left({y}_{k+1}\right)\|\leq \tilde{\varepsilon},\\
    \eqref{eq:lemmffDelta}:& \delta_2 \leq  \frac{\varepsilon \mu}{ 64^2 \cdot L_{\vp}},\\
    \eqref{delta_CATD1}:&\forall k: \delta_1 + \delta_2+ 2\left(\sum_{i = 1}^k A_{i} \right) \delta_2/A_k + \left( \sum_{i=1}^{k-1} A_{i} \right)  \delta_1/A_k \leq \frac{\varepsilon}{2},\\
    \eqref{delta_CATD2}:& \frac{4\sqrt{2\delta_2 L}}{\mu} \leq \varepsilon/2.
\end{align*}

 Let us have a look at \eqref{eq:prox_step_inexact_crit_eps}. We need obtain the sufficient condition for it in terms of the criterion \eqref{eq:prox_step_inexact_crit_objective}. 
\begin{align*}
    &\|\nabla \left( \Omega_{1,\delta,L_{\vp}}\left(\vp, \tilde{x}_{k}, {y}_{k+1}\right)+\frac{H}{2}\left\|{y}_{k+1}-\tilde{x}_{k}\right\|^2 \right)+\nabla \psi_{\delta,L_{\psi}}\left({y}_{k+1}\right)\|\\ 
    \leq & \|\nabla \Omega_{1,\delta,L_{\vp}}\left(\vp, \tilde{x}_{k}, {y}_{k+1}\right) \pm \partial \vp (\tilde{x}_{\ast})\| + H\|{y}_{k+1}-\tilde{x}_{k}\| + \|\nabla \psi_{\delta,L_{\psi}}\left({y}_{k+1}\right) \pm \partial \psi (\tilde{x}_{\ast})\|\\
    \leq & L_{\vp}\|\tilde{x}_k - x_{\ast}\| + 2\sqrt{L_{\vp}\delta_2} +  H\|{y}_{k+1}-\tilde{x}_{k}\| + L_{\psi}\|y_{k+1} - x_{\ast}\| + 2\sqrt{L_{\psi} \delta_2}\\
    \leq & (L_{\vp} + L_{\psi} + H) \max \left\{\|\tilde{x}_k - x_{\ast}\|, \|y_{k+1} - x_{\ast}\| \right\} + 2\sqrt{L_{\vp}\delta_2} + 2\sqrt{L_{\psi} \delta_2}\\
     \stackrel{\eqref{DLM:restarts}}{\leq} &  (L_{\vp} + L_{\psi} + H)\sqrt{\frac{4(\tilde{\varepsilon}_{f} + \delta_1 + \delta_2)}{\mu}} + 2\sqrt{L_{\vp} \delta_2}+ 2\sqrt{L_{\psi} \delta_2}.\\
\end{align*}
Then, according to \eqref{eq:prox_step_inexact_crit_eps}, the sufficient condition for  \eqref{eq:prox_step_inexact_crit} holds true is
\begin{align*}
    (L_{\vp}+L_{\psi} + H)\sqrt{\frac{4(\tilde{\varepsilon}_{f} + \delta_1 + \delta_2)}{\mu}}  + 2\sqrt{L_{\vp} \delta_2}+ 2\sqrt{L_{\psi} \delta_2}
    \leq  \frac{\sqrt{\varepsilon \mu}}{72}.
\end{align*}

Next, under the assumption $\delta_1 \leq \delta_2$, \eqref{delta_CATD1} is converting into more simple sufficient condition
\begin{align}\label{choose_delta_simple_SumAi}
    &\delta_2 \leq \frac{\varepsilon}{2\left(1 + 4 N \right)} \leq \frac{\varepsilon}{2\left(1 + 4\left(\sum_{i = 1}^k A_{i} \right) /A_k \right)}
\end{align}
where $N$ is the number of outer steps. There was used the fact that $A_{i} \leq A_{i+1}$. Finally, if $\delta_2$ satisfies the inequality
\begin{equation*}
    \delta_2 \leq \frac{\varepsilon^{3/2}}{5 \sqrt{2 c_1 H R^2}}
\end{equation*}
then \eqref{choose_delta_simple_SumAi} holds true. 

If we choose $\delta_1,\delta_2, \tilde{\varepsilon}_f$ such that:
\begin{align*}
    &\delta_1, \delta_2 = \min \left\{\frac{\varepsilon \mu}{ 864^2 L_{\vp}},\frac{\varepsilon \mu}{ 864^2L_{\psi}}, \frac{\varepsilon \mu^2}{ 864^2(L_{\vp}+L_{\psi}+H)^2},  \frac{\varepsilon^{3/2}}{5 \sqrt{2 c_1 H R^2}}  \right\},\\
    & \tilde{\varepsilon}_{f} \leq \frac{\varepsilon \mu^2}{864^2(L_{\vp}+L_{\psi}+H)^2},
\end{align*}
 then all required inequalities are satisfied: 
\begin{align*}
    &(L_{\vp}+L_{\psi} + H)\sqrt{\frac{4(\tilde{\varepsilon}_{f} + \delta_1 + \delta_2)}{\mu}}  + 2\sqrt{L_{\vp} \delta_2}+ 2\sqrt{L_{\psi} \delta_2}
    \leq  \frac{\sqrt{\varepsilon \mu}}{72},\\
    & \delta_2 \leq  \frac{\varepsilon \mu}{ 64^2 \cdot L_{\vp}},\\
     &  \delta_2 \leq \frac{\varepsilon^{3/2}}{5 \sqrt{2 c_1 H R^2}},\\
    & \frac{4\sqrt{2\delta_2 L}}{\mu} \leq \varepsilon/2.
\end{align*}
Also dependences $\delta_1(\varepsilon),\delta_2(\varepsilon), \tilde{\varepsilon}_{f}(\varepsilon)$ are polynomial.
\qed 

\end{proof}

Let’s prove the Theorem \ref{AM:comfortable_view_with_prob} using Theorem \ref{AM:comfortable_view}:
\begin{proof} \;\\
Suppose that at each iteration of the Algorithm \ref{alg:restarts_inexact_notconvex} one have:
\begin{enumerate}
    \item inexact $(\delta, \sigma_{0},\mu_{\vp}, L_{\vp}), (\delta, \sigma_{0},\mu_{\psi}, L_{\psi})$-oracles of ${\vp}, \psi$;
    \item the $(\varepsilon,\sigma_{0})$-solution of auxiliary problem.
\end{enumerate}
Let us estimate the probability $\mathbb{P}$ with which inexact $(\delta, \mu_{\vp}, L_{\vp}), (\delta, \mu_{\psi}, L_{\psi})$-oracles of ${\vp}, \psi$ and the $\varepsilon$-solution of auxiliary problem will be available at all iterations \eqref{iter_poly_CATD_with_prob} of the Algorithm \ref{alg:restarts_inexact_notconvex}
\begin{align}\label{eq:prob_CATD_with_prob}
    & \mathbb{P} =(1-\sigma_0)^{N(\varepsilon)} (1-\tilde{\sigma})^{N(\varepsilon)}\geq 1-N(\varepsilon)\left(\sigma_0+\tilde{\sigma}\right)\stackrel{\eqref{sigma_0_poly_CATD_with_prob},\eqref{sigma_poly_CATD_with_prob},\eqref{iter_poly_CATD_with_prob}}{\geq} 1-\sigma 
\end{align}
Hence with probability \eqref{eq:prob_CATD_with_prob} the conditions of the Theorem \ref{AM:comfortable_view} are satisfied which ends the proof.
\qed 
\end{proof}
\section{L-SVRG}\label{Appendix_D}
In this Appendix~\ref{Appendix_D} we reformulate the convergence results of Algorithm L-SVRG from \cite{morin2020sampling} in terms of large deviations.
\begin{lemma}\label{L-SVRG} (Corollary 5.6 from \cite{morin2020sampling})
We consider the problem 
\begin{align}\label{eq:lsvrg-problem}
    \min _{x \in \mathbb{R}^{d}} F(x) = \vp(x)+\psi(x)
\end{align}
where $\vp$ is of finite sum form
\begin{align*}
    \vp(x)=\frac{1}{n} \sum_{i=1}^{n} \vp_{i}(x)
\end{align*}
and $\psi$ is $L_{\psi}$-smooth, convex and prox-friendly. The function $\vp_i$ is convex and $L_i$-smooth for all $i = 1, \dots, n$. The function $\vp$ is convex, L-smooth with $L \leq \frac{1}{n} \sum_{i=1}^n L_i$ an $\mu$-strongly convex. Then L-SVRG \cite{morin2020sampling} achieves an $(\varepsilon, \sigma)$-solution of \eqref{eq:lsvrg-problem}, i.e. 
\begin{align}
    \mathbb{P} \{F(x_k) - F(x_{\ast}) \geq \varepsilon\} \leq \sigma
\end{align}
within
\begin{align*}
    O\left(\left(\sqrt{n}+\sqrt{2 D_{L} \frac{\bar{L}}{\mu}}\right)^{2} \log \frac{1}{\epsilon \sigma}\right)
\end{align*}
iterations where $\bar{L}=\frac{1}{n} \sum_{i=1}^{n} L_{i}$, $D_{L}=4-3 \frac{\mu}{\bar{L}}$ and $x_{\ast}$ is solution of \eqref{eq:lsvrg-problem}. We note that $1\leq D_{L}\leq 4$.
\end{lemma}
\begin{proof}
According to Corollary 5.6 from \cite{morin2020sampling} we obtain that after $O\left(\left(\sqrt{n}+\sqrt{2 D_{L} \frac{\bar{L}}{\mu}}\right)^{2} \log \frac{1}{\epsilon'}\right)$ steps L-SVRG \cite{morin2020sampling} gives $\varepsilon'$ accurate solution, i.e. \begin{align}\label{eq:solution-lsvrg-arg}
    \mathbb{E}\left\|x^{k}-x_{\ast}\right\|^{2} \leq \epsilon',
\end{align}
holds true. For arbitrary $\varepsilon,\sigma>0$ let us take $x_k$, $\varepsilon' = 2\varepsilon \sigma/L_F$ accurate solution in terms of \eqref{eq:solution-lsvrg-arg}. Then from $L_F = L + L_{\psi}$-smoothness of $F$ we have
\begin{align}\label{eq:solution-lsvrg-funct}
    \E[F(x_k) - F(x_{\ast})] \leq \E \left[ \frac{L_F}{2}\|x_k - x_{\ast}\|^2 \right] \leq \varepsilon\sigma.
\end{align}
Using Markov inequality and \eqref{eq:solution-lsvrg-funct} we obtain that
\begin{align}\label{eq:solution-lsvrg-largediv}
    \mathbb{P} \{F(x_k) - F(x_{\ast}) \geq \varepsilon\} \leq \frac{\E [F(x_k) - F(x_{\ast})]}{\varepsilon} \leq \sigma.
\end{align}
In other words, after 
\begin{align*}
    O\left(\left(\sqrt{n}+\sqrt{2 D_{L} \frac{\bar{L}}{\mu}}\right)^{2} \log \frac{1}{\epsilon'}\right) = O\left(\left(\sqrt{n}+\sqrt{2 D_{L} \frac{\bar{L}}{\mu}}\right)^{2} \log \frac{1}{\epsilon\sigma}\right)
\end{align*}
Algorithm L-SVRG from \cite{morin2020sampling} gives random point $x_k$ such as \eqref{eq:solution-lsvrg-largediv} holds true. In other words, $x_k$ is $(\varepsilon, \sigma)$-solution of \eqref{eq:lsvrg-problem}.
\qed 
\end{proof}


\section{Proof of Lemma~\ref{lem:sum_oracle} and  Lemma~\ref{lemma:obt_delta_oracle}}\label{Appendix_obtain_oracle}

Let us proof Lemma~\ref{lem:sum_oracle}
\begin{proof}
Using the Deffinition~\ref{def:d_L_mu_oracle_prob} for function $\vp$ and $\psi$, we can obtain:
\begin{equation} 
     \frac{\mu_{\vp}}{2}\|x-y\|^{2} \leq \vp(x)-\left(\vp_{\delta_{\vp}, L_{\vp},\mu_{\vp}}(y)+\left\langle \nabla \vp_{\delta_{\vp}, L_{\vp},\mu_{\vp}}(y), x-y\right\rangle\right) \leq \frac{L_{\vp}}{2}\|x-y\|^{2}+\delta_{\vp} \text {\;\; w.p. } 1-\sigma_{\vp}
\end{equation}
\begin{equation} 
     \frac{\mu_{\psi}}{2}\|x-y\|^{2} \leq \psi(x)-\left(\psi_{\delta_{\psi}, L_{\psi},\mu_{\psi}}(y)+\left\langle \nabla \psi_{\delta_{\psi}, L_{\psi},\mu_{\psi}}(y), x-y\right\rangle\right) \leq \frac{L_{\psi}}{2}\|x-y\|^{2}+\delta_{\psi} \text {\;\; w.p. } 1-\sigma_{\psi}
\end{equation}
Let us sum this equations:
\begin{align}
     \frac{\mu_{\vp}}{2}\|x-y\|^{2}+\frac{\mu_{\psi}}{2}\|x-y\|^{2} &\leq 
     \vp(x)+\psi(x)-\left(\vp_{\delta_{\vp}, L_{\vp},\mu_{\vp}}(y)+\psi_{\delta_{\psi}, L_{\psi},\mu_{\psi}}(y)+\left\langle  \nabla \vp_{\delta_{\vp}, L_{\vp},\mu_{\vp}}(y)+\nabla \psi_{\delta_{\psi}, L_{\psi},\mu_{\psi}}(y), x-y\right\rangle\right)\nonumber \\
     \leq
     &\frac{L_{\vp}}{2}\|x-y\|^{2}+\frac{L_{\psi}}{2}\|x-y\|^{2}+\delta_{\vp}+\delta_{\psi} \text {\;\; w.p. } 1-\sigma_{\vp}-\sigma_{\psi}\label{eq:lem_sum_orac_fin}
\end{align} 
The equation \eqref{eq:lem_sum_orac_fin} means that the pair $\left(\vp_{\delta_{\vp}, L_{\vp},\mu_{\vp}}(y)+\psi_{\delta_{\psi}, L_{\psi},\mu_{\psi}}(y), \nabla \vp_{\delta_{\vp}, L_{\vp},\mu_{\vp}}(y)+\nabla \psi_{\delta_{\psi}, L_{\psi},\mu_{\psi}}(y)\right)$ is $(\delta_{\vp}+\delta_{\psi},\sigma_{\vp}+\sigma_{\psi}, L_{\vp}+L_{\psi},\mu_{\vp}+\mu_{\psi})$-oracle for $\vp+\psi$.
\qed 
\end{proof}
Let us proof Lemma~\ref{lemma:obt_delta_oracle}
\begin{proof}
The function $\hat{S}(x,\cdot)$ is $\mu_y$-strongly concave, and $\hat{S}(\cdot,y)$ is differentiable. Therefore, by Demyanov–Danskin’s theorem, for any $x \in \mathbb{R}^{d_x}$, we have 
\begin{equation}\tag{A1}\label{nabla_g(x)}
    \nabla g(x) = \nabla_x \tilde{S}(x,y^*(x)) = \nabla_x F(x,y^*(x)).
\end{equation}

To prove that  $g(\cdot)$ has an $L$--Lipschitz gradient for $L = L_{F} + \frac{2 L_{F}^2}{\mu_y}$, let us prove the Lipschitz condition for $y^*(\cdot)$ with a constant, the function $y^*$ is defined as:
\begin{equation}\label{solution:max_problem_g(x)}
    y^*(x) := \arg\max_{y \in \mathbb{R}^{d_y}} \hat{S}(x,y) \quad
     \forall x \in \mathbb{R}^{d_x},
\end{equation}

Since $\hat{S}(x_1, \cdot)$ is $\mu_y$-strongly concave, for arbitrary $x_1, x_2 \in \mathbb{R}^{d_x}$:
\begin{equation}\tag{A2}\label{ineq:1}
   \|y^*(x_1) - y^*(x_2)\|_2^2 \leq \frac{2}{\mu_y} \left( \hat{S}(x_1, y^*(x_1)) - \hat{S}(x_1, y^*(x_2)) \right).
\end{equation}

On the other hand, $\hat{S}(x_2, y^*(x_1)) - \hat{S}(x_2, y^*(x_2)) \leq 0$, since $y^*(x_2)$ affords the maximum to $\hat{S}(x_2,.)$ on $ \mathbb{R}^{d_y} $. We have
\begin{equation}\tag{A3}\label{ineq:2}
\begin{split}
\hat{S}(x_1, y^*(x_1&)) -  \hat{S}(x_1, y^*(x_2))  \leq  \hat{S}(x_1, y^*(x_1)) - \hat{S}(x_1, y^*(x_2))  -  \hat{S}(x_2, y^*(x_1)) + \hat{S}(x_2, y^*(x_2)) = \\
&  \hspace{-2em} \stackrel{\text{from} \, \eqref{eq:lem_obt_oracle_g}}{=} \left( F(x_1, y^*(x_1)) - F(x_1, y^*(x_2)) \right) - \left( F(x_2, y^*(x_1)) - F(x_2, y^*(x_2)) \right) = \\
& \hspace{-1.5em} = \int_0^1 \langle  \nabla_xF(x_1 + t(x_2 - x_1), y^*(x_1)) - \nabla_xF(x_1 + t(x_2 - x_1), y^*(x_2)), x_2 - x_1 \rangle dt \leq \\
&  \hspace{-1.5em} \leq  \| \nabla_xF(x_1 + t(x_2 - x_1), y^*(x_1)) - \nabla_xF(x_1 + t(x_2 - x_1), y^*(x_1))  \|_2\cdot \|x_2 - x_1\|_2 \leq \\
&   \hspace{-1em} \leq L_F \|y^*(x_1) - y^*(x_2)\|_2\cdot \|x_2 - x_1\|_2.
\end{split}
\end{equation}

Thus, \eqref{ineq:1} and \eqref{ineq:2} imply the inequality
\begin{equation}\tag{A4}\label{ineq:Lip_y_star}
    \|y^*(x_2) - y^*(x_1)\|_2 \leq \frac{2L_{F}}{\mu_y}\|x_2 - x_1\|_2,
\end{equation}
i.e., the function $y^*(\cdot)$ satisfies the Lipschitz condition with a constant $\frac{2L_{F}}{\mu_y}$. Next, from \eqref{nabla_g(x)}, we obtain
\begin{align*}
\|\nabla g(x_1&) - \nabla g(x_2)\|_2 = \|\nabla_x F(x_1,y^*(x_1))-\nabla_x F(x_2,y^*(x_2))\|_2 = \\
&  = \|\nabla_x F(x_1,y^*(x_1)) - \nabla_x F(x_1,y^*(x_2)) + \nabla_x F(x_1,y^*(x_2)) - \nabla_x F(x_2,y^*(x_2))\|_2 \leq \\
&  \leq  \|\nabla_x F(x_1,y^*(x_1)) - \nabla_x F(x_1,y^*(x_2))\|_2 + \|\nabla_x F(x_1,y^*(x_2)) - \nabla_x F(x_2,y^*(x_2))\|_2 \leq \\
&  \hspace{0em} \leq  L_{F}\|y^*(x_1)- y^*(x_2)\|_2 + L_{F} \|x_2 - x_1\|_2  =\\
& \hspace{-1em} \stackrel{ \text{from}\, \eqref{ineq:Lip_y_star}}{=} \left( L_{F} + \frac{2L_{F}^2}{\mu_y} \right) \|x_2 - x_1\|_2.
\end{align*}

This means that $g(\cdot)$ has an $L$--Lipschitz gradient with $L = L_{F} + \frac{2 L_{F}^2}{\mu_y}$.

Let us now prove that $ \nabla_{x} F\left(x, \tilde{y}_{\delta/2}(x)\right)$ is $(\delta, 2 L_g)$-oracle of $g$, i.e.:
\begin{equation}\label{eq:lemma2}
	0 \leq g(z) -  \left[\left\{ F(x, \tilde{y}_{\delta /2}(x))-w(\tilde{y}_{\delta/2}(x)) \right\}
	+\langle \nabla_{x}  F(x, \tilde{y}_{\delta/2}(x)), z-x \rangle\right]
	\leq \frac{2 L}{2}\|z-x\|_{2}^{2}+ \delta,
\end{equation}
First, we prove that, for any $\delta \geq 0$ and $x \in \mathbb{R}^{d_x}$ 
\begin{equation}\tag{A5}\label{ineq:inexact_for_g(x)}
    \|\nabla_x \hat{S}(x, \tilde{y}_{\delta/2}(x)) - \nabla g(x)\|_2 \leq L_{F}\sqrt{\frac{\delta}{\mu_y}}.
\end{equation}

For any $x \in \mathbb{R}^{d_x}$, , it is true that $\nabla_x \hat{S}(x, \tilde{y}_{\delta/2}(x)) = \nabla_x F(x, \tilde{y}_{\delta/2}(x))$. Then,
\begin{align*}
\|\nabla_x \hat{S}(x, \tilde{y}_{\delta/2}(x)) - \nabla g(x)\|_2^2  &= \|\nabla_x F(x, \tilde{y}_{\delta/2}(x)) - \nabla_x F(x, y^*(x))\|_2^2 \leq \\
&  \hspace{0em} \leq  L_{F}^2 \|y^*(x) - \tilde{y}_{\delta/2}(x) \|_2^2 \leq \\
&  \hspace{-1em} \stackrel{ \text{from}\, \eqref{ineq:1}}{\leq}  \frac{2 L_{F}^2}{\mu_y} \left( \hat{S}(x, y^*(x)) - \hat{S}(x, \tilde{y}_{\delta/2}(x))  \right) \leq \\
&  \hspace{-1em} \stackrel{ \text{from}\, \eqref{eq:lem_obt_oracle_delta}}{\leq}  \frac{\delta L_{F}^2}{\mu_y},
\end{align*}
which justifies inequality \eqref{ineq:inexact_for_g(x)}.

Now , due to the $\mu_x$-strong convexity of $\hat{S}(\cdot, \tilde{y}_{\delta/2}(x))$ on $\mathbb{R}^{d_x}$, for arbitrary $x, z \in \mathbb{R}^{d_x}$ it is true that
\begin{align*}
    g(z)  \stackrel{ \text{from}\, \eqref{eq:lem_obt_oracle_g}}{\geq} \hat{S}(z, \tilde{y}_{\delta/2}(x))  \geq \hat{S}(x, \tilde{y}_{\delta/2}(x))+ \langle \nabla_x\hat{S}(x, \tilde{y}_{\delta/2}(x)), z-x \rangle.
\end{align*}

Thus,
\begin{equation*}
    0 \geq
     \hat{S}(x, \tilde{y}_{\delta/2}(x)) - g(z) + \langle \nabla_x\hat{S}(x, \tilde{y}_{\delta/2}(x)), z-x \rangle,
\end{equation*}
which proves the left-hand side of \eqref{eq:lemma2}. To prove the right-hand side of \eqref{eq:lemma2}, note that $g$ is convex and has an $L$--Lipschitz gradient on $\mathbb{R}^{d_x}$. Therefore, for arbitrary $x, z \in \mathbb{R}^{d_x}$, we have
\begin{align*}
    g(z) & \leq g(x) + \langle \nabla g(x), z-x\rangle + \frac{L}{2}\|z-x\|_2^2 \leq\\
    & \hspace{-1em} \stackrel{ \text{from}\, \eqref{eq:lem_obt_oracle_delta}}{\leq}  \hat{S}(x,\tilde{y}_{\delta/2}(x)) + \delta/2 + \frac{L}{2} \|z-x\|_2^2 + \langle \nabla g(x), z-x \rangle + \langle \nabla_x\hat{S}(x,\tilde{y}_{\delta/2}(x)), x-z \rangle - \\ & - \langle \nabla_x\hat{S}(x,\tilde{y}_{\delta/2}(x)), x-z \rangle =\\
    & = \hat{S}(x,\tilde{y}_{\delta/2}(x)) + \delta/2 + \langle \nabla_x\hat{S}(x,\tilde{y}_{\delta/2}(x)), z-x \rangle +\langle \nabla_x\hat{S}(x,\tilde{y}_{\delta/2}(x)) - \nabla g(x), x-z \rangle + \\
    & +\frac{L}{2} \|z-x\|_2^2 \leq\\
    & \hspace{-1em} \stackrel{ \text{from}\, \eqref{ineq:inexact_for_g(x)}}{\leq} \hat{S}(x,\tilde{y}_{\delta/2}(x)) + \delta/2 + \langle \nabla_x\hat{S}(x,\tilde{y}_{\delta/2}(x)), z-x \rangle + L_{F}\sqrt{\frac{\delta}{\mu_y}}\cdot \|z-x\|_2+ \frac{L}{2}\|z-x\|_2^2.
\end{align*}

However,
\begin{equation*}
    L_{F}\sqrt{\frac{\delta}{\mu_y}}\cdot \|z-x\|_2  
    =  \sqrt{\frac{L_{F}^2}{\mu_y} \|z-x\|_2^2 \cdot \delta}
    \leq \frac{L_{F}^2}{2\mu_y}\|z-x\|_2^2 + \delta/2
\end{equation*}
due to the classical inequality between the arithmetic and geometric mean. Therefore,
\begin{equation*}
    g(z)  \leq \hat{S}(x,\tilde{y}_{\delta/2}(x)) + \delta + \langle \nabla_x\hat{S}(x,\tilde{y}_{\delta/2}(x)), z-x \rangle + \frac{L_{F}^2}{\mu_y}\|z-x\|_2^2 +\frac{L}{2}\|z-x\|_2^2,
\end{equation*}
and since $L = L_{F} + \frac{2L_{F}^2}{\mu_y}$, we have $\frac{L_{F}^2}{\mu_y} \leq \frac{L}{2}$;therefore,
\begin{equation*}
    g(z) \leq \hat{S}(x,\tilde{y}_{\delta/2}(x)) + \langle \nabla_x \hat{S}(x,\tilde{y}_{\delta/2}(x)), z-x \rangle + \delta + L \|z-x\|_2^2.
\end{equation*}

Thus, we have
\begin{equation*}
g(z) - \hat{S}(x, \tilde{y}_{\delta/2}(x)) - \langle \nabla_x\hat{S}(x, \tilde{y}_{\delta/2}(x)), z-x \rangle \leq L\|z-x\|_2^2 + \delta,
\end{equation*}
which implies the left-hand side of inequality \eqref{eq:lemma2}.
 
In the statement of Lemma~\ref{lemma:obt_delta_oracle} only  $(\delta/2, \sigma)$-solution to \eqref{eq:lem_obt_oracle_g} is available. In this case the inequality \eqref{eq:lemma2} will be satisfied with probability $1-\sigma$. Then $ \nabla_{x} F\left(x, \tilde{y}_{\delta/2}(x)\right)$ is $(\delta,\sigma, 2 L_g)$-oracle of $g$.
\qed 
\end{proof}
\section{A Variant of Accelerated Framework for Saddle-Point Problems.}\label{Appendix_Framework_inverse}
In this appendix we consider saddle-point problem under the same assumptions as in Section \ref{section:saddleFramework}. 
We describe in detail the structure of a general framework for solving such problems which consists of three inner-outer loops. The only difference compared with the general framework in Section \ref{section:saddleFramework}
is that the order of the Loop 2 and Loop 3 has been reversed. 
We also summarize the steps of the algorithm in Table~\ref{tabl:saddleproblem_steps_inverse}. 
In each loop we apply Algorithm \ref{alg:restarts_inexact_notconvex} with different value of parameter $H$ which defines its complexity. In the subsection after description of the loops we carefully choose the value of this parameter in each level of the loops. Later, in the next Appendix \ref{Appendix_h_not_sum}
we use this general framework in the proof of Theorems~\ref{theorem:G-sum_noprox} and \ref{theorem:G-sum_h_prox} with complexity estimates for problem \eqref{eq:problem_Gsum_minmax} under Assumption~\ref{assumpt:G_sum}, as well as Corollary \ref{theorem:h-not-sum}
with complexity estimates for problem \eqref{eq:problem_st_h_sum} with $m_h=1$.

\subsection{Main loops of the framework}

In each of the three loops of the general framework we have a target accuracy $\varepsilon$ and a confidence level $\sigma$ which define the required quality of the solution to an optimization problem in this loop. These quantities define the inexactness of the oracle in this loop via inequalities \eqref{delta_poly_CATD_with_prob} and \eqref{sigma_0_poly_CATD_with_prob} and the target accuracy and confidence level for the optimization problem in the next loop via \eqref{varepsilon_poly_CATD_with_prob}, \eqref{sigma_poly_CATD_with_prob}. Due to inexact strong convexity provided by $(\delta,\sigma,L,\mu)$-oracle, Algorithm \ref{alg:restarts_inexact_notconvex} has logarithmic dependence of the complexity on the target accuracy and confidence level (see Theorem \ref{AM:comfortable_view_with_prob}). Since the dependencies on the target accuracy and confidence level in \eqref{delta_poly_CATD_with_prob},  \eqref{sigma_0_poly_CATD_with_prob}, \eqref{varepsilon_poly_CATD_with_prob} and \eqref{sigma_poly_CATD_with_prob} are polynomial, we obtain that the dependency of the complexity in each loop on the target accuracy and confidence level in the first loop, i.e. target accuracy and confidence level for the solution to problem \eqref{eq:framework_funct}, is logarithmic. We hide such logarithmic factors in $\widetilde{O}$ notation.

For convenience, we summarize the main details of the loops in Table \ref{tabl:saddleproblem_steps_inverse}.

 \paragraph{\textbf{Loop 1}}\label{subsec_first_inverse} $\;$\\
The goal of Loop 1 is to find an $(\varepsilon,\sigma)$-solution of  problem \eqref{eq:main30}, which is considered as a minimization problem in $y$ with the objective given in the form of auxiliary maximization problem in $x$. 
Finding an $(\varepsilon,\sigma)$-solution of this minimization problem gives an approximate solution to the saddle-point problem  \eqref{eq:framework_funct} which is understood in the sense of Definition \ref{def:saddle_solution}.

To solve problem \eqref{eq:main30}, we would like to apply Algorithm \ref{alg:restarts_inexact_notconvex} with 
\begin{align}\label{eq:framework_aux_step1_inverse}
    \varphi = 0, \ \ \psi= h(y) + \max_{x \in  \R^{d_x}}\left\{-G(x, y)-f(x)\right\}.
\end{align}
The function $\varphi$ is, clearly, convex and is known exactly.
What makes solving problem \eqref{eq:main30} not straightforward is that the exact value of $\psi$ is not available. At the same time we can construct an inexact oracle for this function.
First, the function $h$ is $\mu_y$-strongly convex, $L_h$-smooth and its exact gradient is available. Second, thanks to Assumption \ref{assumpt:framework_oracle_x}, it is possible to construct a $\left(\delta^{(1)}\left(\varepsilon\right),\sigma^{(1)}_0\left(\varepsilon,\sigma\right),2L_G + 4\frac{L_G^2}{\mu_x}\right)$-oracle for the function $r(y)= \max_{x \in  \R^{d_x}}\left\{-f(x)-G(x, y)\right\}$ for any $\delta^{(1)}\left(\varepsilon\right)=\bf{poly}\left(\varepsilon \right)$ and $\sigma^{(1)}_0\left(\varepsilon,\sigma\right)=\bf{poly}\left(\varepsilon,\sigma \right)$.
Combining these two parts and using Lemma~\ref{lem:sum_oracle}, we obtain that we can construct a $\left(\delta^{(1)}\left(\varepsilon\right),\sigma^{(1)}_0\left(\varepsilon,\sigma\right), L_h + 2L_G + 4\frac{L_G^2}{\mu_x}, \mu_y\right)$-oracle for $\psi$.
Thus, we can apply Algorithm  \ref{alg:restarts_inexact_notconvex} with parameter $H=H_1$, which will be chosen later, to solve problem \eqref{eq:main30}.
Moreover, since Assumption \ref{assumpt:framework_oracle_x} requires  $\delta^{(1)}\left(\varepsilon\right)=\bf{poly}\left(\varepsilon\right)$ and  $\sigma^{(1)}_0\left(\varepsilon,\sigma\right)=\bf{poly}\left(\varepsilon, \sigma \right)$,
which holds for the dependencies in \eqref{delta_poly_CATD_with_prob} and \eqref{sigma_0_poly_CATD_with_prob}, we can choose $\delta^{(1)}\left(\varepsilon\right)$ and $\sigma^{(1)}_0\left(\varepsilon,\sigma\right)$ such that \eqref{delta_poly_CATD_with_prob} and \eqref{sigma_0_poly_CATD_with_prob} hold. So, the first main assumption of Theorem~\ref{AM:comfortable_view_with_prob} holds. 
At the same time, according to Assumptions \ref{assumpt:framework} and \ref{assumpt:framework_oracle_x}, constructing inexact oracle for $\psi$ requires $\tau_h$ calls of the basic oracle for $h$,
$\tau_G$ calls of the basic oracle of $G(x, \cdot )$, $\mathcal{N}_G^x\left( \tau_G\right) \mathcal{K}_G^x\left(\varepsilon, \sigma\right)$ calls of the basic oracle for $G(\cdot,y)$, $\mathcal{N}_{f}\left( \tau_f\right)\mathcal{K}_f\left(\varepsilon, \sigma\right)$ calls of the basic oracle for $f$.

Let us discuss the second main assumption of Theorem~\ref{AM:comfortable_view_with_prob}. 
To ensure that this assumption holds, we need in each iteration of Algorithm \ref{alg:highorder_inexact}, used as a building block in Algorithm \ref{alg:restarts_inexact_notconvex}, to find an $\left(\tilde{\varepsilon}^{(1)}_f\left(\varepsilon\right),\tilde{\sigma}^{(1)}\left(\varepsilon,\sigma\right)\right)$-solution to the auxiliary problem \eqref{prox_step_inexact}, where $\tilde{\sigma}^{(1)}\left(\varepsilon,\sigma\right), \tilde{\varepsilon}^{(1)}_f\left(\varepsilon\right)$ satisfy 
 inequalities \eqref{varepsilon_poly_CATD_with_prob}, \eqref{sigma_poly_CATD_with_prob}.
 For the particular definitions of $\vp$, $\psi$  \eqref{eq:framework_aux_step1_inverse} in this Loop, this problem has the following form:  
\begin{equation} \label{eq:sub1_inverse}
 y_{k+1}^t = \arg\min _{y \in \R^{d_y}}\left\{ h(y) +\max_{x \in  \R^{d_x}}\left\{-G(x, y)-f(x)\right\}+\frac{H_1}{2} \|y - y_k^{md}\|^2\right\}.
\end{equation}
Below, in the next \hyperref[subsec_second_inverse]{paragraph "Loop 2"},  we explain how to solve this auxiliary problem to obtain its $\left(\tilde{\varepsilon}^{(1)}_f\left(\varepsilon\right),\tilde{\sigma}^{(1)}\left(\varepsilon,\sigma\right)\right)$-solution.
To summarize Loop 1, both main assumptions of Theorem~\ref{AM:comfortable_view_with_prob} hold and we can use it to guarantee that we obtain an $(\varepsilon,\sigma)$-solution of problem \eqref{eq:main30}. This requires $\widetilde{O}\left(1+\left( \frac{H_1 }{\mu_{\vp} + \mu_{\psi}} \right)^{\frac{1}{2}}\right)=\widetilde{O}\left(1+\left( \frac{H_1 }{\mu_{y}} \right)^{\frac{1}{2}}\right)$ 
calls to the inexact oracles for $\vp$ and for $\psi$, and the same number of times solving the auxiliary problem \eqref{eq:sub1_inverse}.
Combining this oracle complexity with the cost of calculating inexact oracles for $\vp$ and for $\psi$, we obtain that solving problem \eqref{eq:main30} requires $\widetilde{O}\left(1+\left( \frac{H_1 }{\mu_{y}} \right)^{\frac{1}{2}}\right)\tau_h$ calls of the basic oracle for $h$, $\widetilde{O}\left(1+\left( \frac{H_1 }{\mu_{y}} \right)^{\frac{1}{2}}\right)\tau_G$ calls of the basic oracle of $G(x, \cdot )$,
$\widetilde{O}\left(1+\left( \frac{H_1 }{\mu_{y}} \right)^{\frac{1}{2}}\right)\mathcal{N}_G^x\left( \tau_G\right)\mathcal{K}_G^x\left(\varepsilon, \sigma\right)$ calls of the basic oracle for $G(\cdot,y)$, $\widetilde{O}\left(1+\left( \frac{H_1 }{\mu_{y}} \right)^{\frac{1}{2}}\right)\mathcal{N}_{f}\left( \tau_f\right)\mathcal{K}_f\left(\varepsilon, \sigma\right)$ calls of the basic oracle for $f$. 
The only remaining thing is to provide an inexact solution to problem \eqref{eq:sub1_inverse} and, next, we move to the Loop 2 to explain how to guarantee this. Note that we need to solve problem \eqref{eq:sub1_inverse} $\widetilde{O}\left(1+\left( \frac{H_1 }{\mu_{y}} \right)^{\frac{1}{2}}\right)$ times.

\paragraph{\textbf{Loop 2}}\label{subsec_second_inverse}$\;$\\
As mentioned in the previous Loop 1,
in each iteration of Algorithm \ref{alg:restarts_inexact_notconvex} in Loop 1 
we need many times to find  an $(\varepsilon'_2,\sigma'_2)$-solution of the auxiliary problem \eqref{eq:sub1_inverse}, where we denoted for simplicity $\sigma'_2=\tilde{\sigma}^{(1)}\left(\varepsilon,\sigma\right)$ and $\varepsilon'_2=\tilde{\varepsilon}^{(1)}_f\left(\varepsilon\right)$. To do this, we reformulate problem \eqref{eq:sub1_inverse} by changing the order of minimization and maximization as follows:
\begin{align}
\min _{y \in \R^{d_y}}\left\{ h(y) +\frac{H_1}{2} \|y - y_k^{md}\|^2 +\max_{x \in  \R^{d_x}}\left\{-G(x, y)-f(x)\right\}\right\} \\
=\min _{y \in \R^{d_y}}\max_{x \in  \R^{d_x}}\left\{ h(y)  -G(x, y)-f(x) +\frac{H_1}{2} \|y - y_k^{md}\|^2\right\} \\
 = \max_{x \in  \R^{d_x}} \min _{y \in \R^{d_y}}\left\{ h(y)  -G(x, y)-f(x) +\frac{H_1}{2} \|y - y_k^{md}\|^2\right\}\\
= - \min _{x \in \R^{d_x}}\left\{ f (x) +\max _{y \in  \R^{d_y}}\left\{G(x, y)-h(y) - \frac{H_1}{2} \|y - y_k^{md}\|^2\right\}\right\} \label{eq:problem_step2_inverse}   
\end{align}
and obtain an $(\varepsilon'_2,\sigma'_2)$-solution of the problem \eqref{eq:sub1_inverse} by solving minimization problem \eqref{eq:problem_step2_inverse}. Assume that we can find an $(\varepsilon_2,\sigma_2)$-solution $\hat{x}$ of the minimization problem \eqref{eq:problem_step2_inverse} in the sense of Definition \ref{def:saddle_solution}. Then, according to Assumption \ref{assumpt:framework_oracle}, we can also obtain a point $\hat{y}$ which is $(\bar{\delta}(\varepsilon_2)/2,\bar{\sigma}_0(\sigma_2))$-solution to the problem 
\begin{align}\label{eq:loop2_probl_for_changing_delta_L_inverse}
    \max _{y \in  \R^{d_y}}\left\{G(x, y)-h(y) - \frac{H_1}{2} \|y - y_k^{md}\|^2\right\},
\end{align}
where $\bar{\delta}(\varepsilon_2),\bar{\sigma}_0(\sigma_2)$ satisfy the following polynomial dependencies
\begin{align}
& \bar{\delta}(\varepsilon_2) \leq \frac{H_1+\mu_y}{4\mu_x\left(\frac{H_1+\mu_y}{4L_G}\right)^2} \varepsilon_2, \;\; \bar{\sigma}_0(\sigma_2) \leq \sigma_2.
\label{eq:loop2_1_inverse}
\end{align}
If we choose $\varepsilon_2,\sigma_2,\bar{\delta}(\varepsilon_2),\bar{\sigma}_0(\sigma_2)$ satisfying
\begin{align}
    & \varepsilon_2\leq  \left(\frac{H_1+\mu_y}{4L_G}\right)^2 \frac{\mu_x}{L_h+H_1+L_G+\frac{2L_G^2}{\mu_x}} \varepsilon'_2  \label{eq:step_2_choos_eps_fin_inverse},\\
    &\sigma_2\leq \frac{\sigma'_2}{2},\;\;\;  \label{eq:step_2_choos_sigma_fin_inverse}\\
    &\bar{\sigma}_0(\sigma_2)  \stackrel{\eqref{eq:loop2_1_inverse}}{\leq} \sigma_2 \leq \frac{\sigma'_2}{2}, \;\;\; \bar{\delta}(\varepsilon_2) \leq \frac{H_1+\mu_y}{4\mu_x\left(\frac{H_1+\mu_y}{4L_G}\right)^2} \varepsilon_2\stackrel{\eqref{eq:loop2_1_inverse}}{\leq} \frac{H_1+\mu_y}{4L_h+4H_1+4L_G+\frac{8L_G^2}{\mu_x}}\varepsilon'_2 \label{eq:step_2_choos_delta_fin_inverse},
\end{align}
then
\begin{align}
    &2\frac{L_h+H_1+L_G+\frac{2L_G^2}{\mu_x}}{H_1+\mu_y}\bar{\delta}(\varepsilon_2) + 8 \left(\frac{L_G}{H_1+\mu_y}\right)^2 \frac{L_h+H_1+L_G+\frac{2L_G^2}{\mu_x}}{\mu_x}\varepsilon_2\leq \varepsilon'_2, \label{eq:step_2_choos_eps_inverse}\\
    &\sigma_2+\bar{\sigma}_0(\sigma_2) \leq \sigma'_2. \label{eq:step_2_choos_sigma_inverse}
\end{align}
Thus, applying Corollary~\ref{corollary:change_maxmin} to minimization problem \eqref{eq:problem_step2_inverse} with $F(x,y)=G(x,y)$, $w(y)=h(y)+\frac{H_1}{2} \|y - y_k^{md}\|^2$, $\varepsilon_x=\varepsilon_2$, $\sigma_x=\sigma_2$, $\varepsilon_y=\bar{\delta}(\varepsilon_2)$, $\sigma_y=\bar{\sigma}_0(\sigma_2)$ we obtain (see \eqref{eq:cor_1_1}, \eqref{eq:cor_1_2}) that $\hat{y}$ satisfies inequality 
\[
h(\hat{y}) +\frac{H_1}{2} \|\hat{y} - y_k^{md}\|^2
 + \max_{x \in \R^{d_x}} \{ - G(x, \hat{y}) - f(x)\} - \min_{y \in \R^{d_y}}\max_{x \in \R^{d_x}} \{h(y) +\frac{H_1}{2} \|y - y_k^{md}\|^2 -G(x, y) - f(x) \} \leq \varepsilon'_2
\]
with probability $\sigma'_2$. Thus, by Definition \ref{def:saddle_solution} it is an $(\varepsilon'_2,\sigma'_2)$-solution of the problem \eqref{eq:sub1_inverse}.
By Assumption \ref{assumpt:framework_oracle}, calculation of $\hat{y}$ requires  $\mathcal{N}_G^y\left( \tau_G, H\right) \mathcal{K}_G^y\left(\varepsilon_2, \sigma_2\right)$ calls of the basic oracle $O_G^y$ of $G(x,\cdot)$, $\tau_G$ calls of the basic oracle $O_G^x$ of $G(\cdot ,y)$ and $\mathcal{N}_{h}\left( \tau_h, H\right) \mathcal{K}_h\left(\varepsilon_2, \sigma_2\right)$ calls of the basic oracle $O_h$ of $h$.

Our next step is to provide an $(\varepsilon_2,\sigma_2)$-solution to minimization problem \eqref{eq:problem_step2_inverse}, for which we again apply Algorithm  \ref{alg:restarts_inexact_notconvex}, but this time with 
\begin{align}\label{eq:framework_aux_step2_inverse}
    \varphi = f(x), \ \ \psi = \max _{y \in  \R^{d_y}}\left\{G(x, y)-h(y) - \frac{H_1}{2} \|y - y_k^{md}\|^2\right\}.
\end{align}
The function $\varphi$  is $\mu_x$-strongly convex, $L_f$-smooth and its exact gradient is available.
What makes solving problem \eqref{eq:problem_step2_inverse} not straightforward is that the exact value of $\psi$ is not available. At the same time we can construct an inexact oracle for this function.
Thanks to Assumption \ref{assumpt:framework_oracle}, it is possible to construct a $\left(\delta^{(2)}\left(\varepsilon_2\right), \sigma^{(2)}_0\left(\varepsilon_2,\sigma_2\right), 2L_G + 4\frac{L_G^2}{H_1+\mu_y}\right)$-oracle for the function 
$\psi$ for any $\delta^{(2)}\left(\varepsilon_2\right)=\textbf{poly}\left(\varepsilon_2 \right)$ and $\sigma^{(2)}_0\left(\varepsilon_2,\sigma_2\right)=\textbf{poly}\left(\varepsilon_2,\sigma_2 \right)$.
Using Lemma \ref{lem:sum_oracle}, we obtain that we can construct \\ a  $\left(\delta^{(2)}\left(\varepsilon_2\right),\sigma^{(2)}_0\left(\varepsilon_2,\sigma_2\right), L_f + 2L_G + 4\frac{L_G^2}{H_1+\mu_y}, \mu_x\right)$-oracle for the function $\vp+\psi$.
Thus, we can apply Algorithm  \ref{alg:restarts_inexact_notconvex} with parameter $H=H_2 \geq 2L_f$, which will be chosen later, to solve the problem \eqref{eq:problem_step2_inverse}.
Moreover, since Assumption \ref{assumpt:framework_oracle} requires  $\delta^{(2)}\left(\varepsilon_2\right)=\text{\bf{poly}}\left(\varepsilon_2 \right)$ and  $\sigma^{(2)}_0\left(\varepsilon_2,\sigma_2\right)=\text{\bf{poly}}\left(\varepsilon_2,\sigma_2 \right)$, which holds for the dependencies in \eqref{delta_poly_CATD_with_prob} and \eqref{sigma_0_poly_CATD_with_prob}, we can choose $\delta^{(2)}\left(\varepsilon_2\right)$ and $\sigma^{(2)}_0\left(\varepsilon_2,\sigma_2\right)$ such that \eqref{delta_poly_CATD_with_prob} and \eqref{sigma_0_poly_CATD_with_prob} hold. So, the first main assumption of Theorem~\ref{AM:comfortable_view_with_prob} holds. 
At the same time, according to Assumptions \ref{assumpt:framework} and \ref{assumpt:framework_oracle}, constructing inexact oracle for $\psi$ requires $\mathcal{N}_G^y\left( \tau_G, H_1\right) \mathcal{K}_G^y\left(\varepsilon_2, \sigma_2\right)$ calls of the basic oracle for $G(x,\cdot)$, $ \tau_G$ calls of the basic oracle for $G(\cdot, y)$, $\mathcal{N}_{h}\left( \tau_h, H_1\right)\mathcal{K}_h\left(\varepsilon_2, \sigma_2\right)$ calls of the basic oracle for $h$, and constructing exact oracle for $\varphi=f$ requires $\tau_f$ calls of the basic oracle for $f$.

Let us discuss the second main assumption of Theorem~\ref{AM:comfortable_view_with_prob}. 
To ensure that this assumption holds, we need in each iteration of Algorithm \ref{alg:highorder_inexact}, used as a building block in Algorithm \ref{alg:restarts_inexact_notconvex}, to find $\left(\tilde{\varepsilon}^{(2)}_f\left(\varepsilon_2\right),\tilde{\sigma}^{(2)}\left(\varepsilon_2,\sigma_2\right)\right)$-solution to the auxiliary problem \eqref{prox_step_inexact}, where $\tilde{\sigma}^{(2)}\left(\varepsilon_2,\sigma_2\right), \tilde{\varepsilon}^{(2)}_f\left(\varepsilon_2\right)$ satisfy 
 inequalities \eqref{varepsilon_poly_CATD_with_prob}, \eqref{sigma_poly_CATD_with_prob}.
 For the particular definitions of $\vp$, $\psi$  \eqref{eq:framework_aux_step2_inverse} in this Loop, this problem has the following form:  
\begin{align}
x_{l+1}^t &= \arg\min _{x \in \R^{d_x}}\left\{ \la\nabla f (x_l^{md}), x-x_l^{md}\ra \right.\nonumber\\
&\left. +\max _{y \in  \R^{d_y}}\left\{G(x, y)+h(y) - \frac{H_1}{2} \|y - y_k^{md}\|^2 \right\} + \frac{H_2}{2}\|x - x_l^{md}\|^2\right\}\label{eq:sub2_inverse},
\end{align}
Below, in the next \hyperref[subsec_third_inverse]{paragraph "Loop 3"},  we explain how to solve this auxiliary problem to obtain its \\ $\left(\tilde{\varepsilon}^{(2)}_f\left(\varepsilon_2\right),\tilde{\sigma}^{(2)}\left(\varepsilon_2,\sigma_2\right)\right)$-solution.

To summarize Loop 2, both main assumptions of Theorem~\ref{AM:comfortable_view_with_prob} hold and we can use it to guarantee that we obtain an $(\varepsilon'_2,\sigma'_2)$-solution of the auxiliary problem \eqref{eq:sub1_inverse}. This requires one time to solve the problem~\eqref{eq:loop2_probl_for_changing_delta_L_inverse}, which, by Assumption \ref{assumpt:framework_oracle} has the same cost as evaluating inexact oracle for the function $\psi$.
Further, we need $O\left(\left( 1+\left( \frac{H_2 }{\mu_{\vp} + \mu_{\psi}} \right)^{\frac{1}{2}}\right) \log \varepsilon_2^{-1}\right)=O\left(\left(1+\left( \frac{H_2 }{\mu_{x}} \right)^{\frac{1}{2}} \right)\log \varepsilon_2^{-1}\right)$ calls to the inexact oracles for $\vp$ and for $\psi$, and the same number of times solving the auxiliary problem \eqref{eq:sub2_inverse}.
Combining this oracle complexity with the cost of calculating inexact oracles for $\vp$ and for $\psi$, we obtain that solving problem \eqref{eq:problem_step2_inverse} requires $O\left(\left(1+\left( \frac{H_2 }{\mu_{x}} \right)^{\frac{1}{2}} \right)\log \varepsilon_2^{-1}\right)\tau_f$ calls of the basic oracle for $f$, $O\left(\left(1+\left( \frac{H_2 }{\mu_{x}} \right)^{\frac{1}{2}} \right)\log \varepsilon_2^{-1}\right)\mathcal{N}_G^y\left( \tau_G, H_1\right) \mathcal{K}_G^y\left(\varepsilon_2, \sigma_2\right)$  calls of the basic oracle for $G(x,\cdot)$, $O\left(\left(1+\left( \frac{H_2 }{\mu_{x}} \right)^{\frac{1}{2}} \right)\log \varepsilon_2^{-1}\right) \tau_G$ calls of the basic oracle for $G(\cdot, y)$,\\
$O\left(\left(1+\left( \frac{H_2 }{\mu_{x}} \right)^{\frac{1}{2}} \right)\log \varepsilon_2^{-1}\right)\mathcal{N}_{h}\left( \tau_h, H_1\right)\mathcal{K}_h\left(\varepsilon_2, \sigma_2\right)$ calls of the basic oracle for $h$. The only remaining thing is to provide an inexact solution to problem \eqref{eq:sub2_inverse} and, next, we move to Loop 3 to explain how to guarantee this. Note that we need to solve problem \eqref{eq:sub2_inverse} $O\left(\left(1+\left( \frac{H_2 }{\mu_{x}} \right)^{\frac{1}{2}} \right)\log \varepsilon_2^{-1}\right)$ times.

\paragraph{\textbf{Loop 3}}\label{subsec_third_inverse}$\;$\\
As mentioned in the previous Loop 2,
in each iteration of Algorithm \ref{alg:restarts_inexact_notconvex} in 
Loop 2 we need to find many times an $(\varepsilon_3,\sigma_3)$-solution of the auxiliary problem \eqref{eq:sub2_inverse}, where we denoted for simplicity $\sigma_3=\tilde{\sigma}^{(2)}\left(\varepsilon_2,\sigma_2\right)$ and $\varepsilon_3=\tilde{\varepsilon}^{(2)}_f\left(\varepsilon_2\right)$.
To solve problem \eqref{eq:sub2_inverse}, we would like to apply Algorithm  \ref{alg:restarts_inexact_notconvex} with
\begin{align}\label{eq:framework_aux_step3_inverse} 
\varphi = \max _{y \in  \R^{d_y}}\left\{G(x, y)-h(y) - \frac{H_1}{2} \|y - y_k^{md}\|^2 \right\}, \;\;\;
\psi = \la\nabla f (x_l^{md}), x-x_l^{md}\ra + \frac{H_2}{2}\|x - x_l^{md}\|^2.
\end{align}
The function $\psi$ is, clearly, $H_2$-strongly convex, $H_2$-smooth and its exact gradient is available. What makes solving problem \eqref{eq:sub2_inverse} not straightforward is that the exact value of $\varphi$ is not available. At the same time, we can construct an inexact oracle for this function.
Thanks to Assumption \ref{assumpt:framework_oracle}, it is possible to construct a $\left(\delta^{(3)}\left(\varepsilon_3\right), \sigma^{(3)}_0\left(\varepsilon_3,\sigma_3\right), 2L_G + 4\frac{L_G^2}{H_1+\mu_y}\right)$-oracle for the function 
$\vp$ for any $\delta^{(3)}\left(\varepsilon_3\right)=\textbf{poly}\left(\varepsilon_3 \right)$ and $\sigma^{(3)}_0\left(\varepsilon_3,\sigma_3\right)=\textbf{poly}\left(\varepsilon_3,\sigma_3 \right)$. 
Using Lemma \ref{lem:sum_oracle}, we obtain that we can construct \\ a  
$\left(\delta^{(3)}\left(\varepsilon_3\right),\sigma^{(3)}_0\left(\varepsilon_3,\sigma_3\right), H_2+ 2L_G + 4\frac{L_G^2}{H_1+\mu_x}, H_2\right)$-oracle for the function $\vp+\psi$.
Thus, we can apply Algorithm  \ref{alg:restarts_inexact_notconvex} with parameter $H=H_3 \geq 2L_G + 4\frac{L_G^2}{H_1+\mu_y}$, which will be chosen later, to solve problem \eqref{eq:sub2_inverse}.
Moreover, since Assumption \ref{assumpt:framework_oracle} requires  $\delta^{(3)}\left(\varepsilon_3\right)=\bf{poly}\left(\varepsilon_3 \right)$ and  $\sigma^{(3)}_0\left(\varepsilon_3,\sigma_3\right)=\bf{poly}\left(\varepsilon_3,\sigma_3 \right)$, which holds for the dependencies in \eqref{delta_poly_CATD_with_prob} and \eqref{sigma_0_poly_CATD_with_prob}, we can choose $\delta^{(3)}\left(\varepsilon_3\right)$ and $\sigma^{(3)}_0\left(\varepsilon_3,\sigma_3\right)$ such that \eqref{delta_poly_CATD_with_prob} and \eqref{sigma_0_poly_CATD_with_prob} hold. So, the first main assumption of Theorem~\ref{AM:comfortable_view_with_prob} holds. 
At the same time, according to Assumptions \ref{assumpt:framework} and \ref{assumpt:framework_oracle}, constructing inexact oracle for $\varphi$ requires $\mathcal{N}_G^y\left( \tau_G, H_1\right) \mathcal{K}_G^y\left(\varepsilon_3, \sigma_3\right)$ calls of the basic oracle for $G(x,\cdot)$, $ \tau_G$ calls of the basic oracle for $G(\cdot,y)$,
$\mathcal{N}_{h}\left( \tau_h, H_1\right)\mathcal{K}_h\left(\varepsilon_3, \sigma_3\right)$ calls of the basic oracle for $h$. At the same time, no calls to the oracle for $f$ are needed.

Let us discuss the second main assumption of Theorem~\ref{AM:comfortable_view_with_prob}. 
To ensure that this assumption holds, we need in each iteration of Algorithm \ref{alg:highorder_inexact}, used as a building block in Algorithm \ref{alg:restarts_inexact_notconvex}, to find $\left(\tilde{\varepsilon}^{(3)}_f\left(\varepsilon_3\right),\tilde{\sigma}^{(3)}\left(\varepsilon_3,\sigma_3\right)\right)$-solution to the auxiliary problem \eqref{prox_step_inexact}, where $\tilde{\sigma}^{(3)}\left(\varepsilon_3,\sigma_3\right), \tilde{\varepsilon}^{(3)}_f\left(\varepsilon_3\right)$ satisfy 
 inequalities \eqref{varepsilon_poly_CATD_with_prob}, \eqref{sigma_poly_CATD_with_prob}.
 For the particular definitions of $\vp$, $\psi$ in \eqref{eq:framework_aux_step3_inverse} in this Loop, this problem has the following form:  
\begin{multline}
    \label{eq:SeparateX01_inverse}
    u^t_{m+1} = \arg\min _{u \in  \R^{d_x}} \{ \la \nabla \varphi_{\delta^{(3)},2L_{\vp}}(u_m^{md}),u-u_m^{md}\ra+\psi(u) + \frac{H_3}{2}\|u-u_m^{md}\|_2^2 \}
    \\
    = \arg\min _{u \in  \R^{d_x}} \{ \la \nabla \varphi_{\delta^{(3)},2L_{\vp}}(u_m^{md}),u-u_m^{md}\ra
    +\la \nabla  f(x_l^{md}),u-x_l^{md}\ra+\frac{H_2}{2}\|u-x_l^{md}\|_2^2+\frac{H_3}{2}\|u-u_m^{md}\|_2^2\},
\end{multline}
where $L_{\vp}=L_G + \frac{L_G^2}{H_1+\mu_x}$.
This quadratic auxiliary problem \eqref{eq:SeparateX01_inverse} can be solved explicitly and exactly since at the point it needs to be solved, $\nabla \varphi_{\delta^{(3)},2L_{\vp}}(u_m^{md})$ is already calculated. Thus, the second main assumption of Theorem~\ref{AM:comfortable_view_with_prob} is satisfied with  $\tilde{\sigma}^{(3)}\left(\varepsilon_3,\sigma_3\right)=0$ and $ \tilde{\varepsilon}^{(3)}_f\left(\varepsilon_3\right)=0$, which clearly satisfy \eqref{delta_poly_CATD_with_prob} and \eqref{sigma_0_poly_CATD_with_prob}.

To summarize Loop 3, both main assumptions of Theorem~\ref{AM:comfortable_view_with_prob} hold and we can use it to guarantee that we obtain an $(\varepsilon_3,\sigma_3)$-solution of the auxiliary problem \eqref{eq:sub2_inverse}. This requires 
$O\left(\left(1+\left( \frac{H_3 }{\mu_{\vp} + \mu_{\psi}} \right)^{\frac{1}{2}}\right)\log \varepsilon_3^{-1} \right)=O\left(\left(1+\left( \frac{H_3 }{H_2} \right)^{\frac{1}{2}}\right)\log \varepsilon_3^{-1}\right)$ calls to the inexact oracles for $\vp$ and for $\psi$, and the same number of times solving the auxiliary problem \eqref{eq:SeparateX01_inverse}.
Combining this oracle complexity with the cost of calculating inexact oracles for $\vp$ and for $\psi$, 
we obtain that solving problem \eqref{eq:sub2_inverse} requires  $O\left(\left(1+\left( \frac{H_3 }{H_2} \right)^{\frac{1}{2}}\right)\log \varepsilon_3^{-1}\right)\mathcal{N}_G^y\left( \tau_G, H_1\right) \mathcal{K}_G^y\left(\varepsilon_3, \sigma_3\right)$  calls of the basic oracle for $G(x,\cdot)$, $O\left(\left(1+\left( \frac{H_3 }{H_2} \right)^{\frac{1}{2}}\right)\log \varepsilon_3^{-1}\right) \tau_G$ calls of the basic oracle for $G(\cdot,y)$ and 
$O\left(\left(1+\left( \frac{H_3 }{H_2} \right)^{\frac{1}{2}}\right)\log \varepsilon_3^{-1}\right)\mathcal{N}_{h}\left( \tau_h, H_1\right)\mathcal{K}_h\left(\varepsilon_3, \sigma_3\right)$ calls of the basic oracle for $h$.

\begingroup
\setlength{\tabcolsep}{6pt} 
\renewcommand{\arraystretch}{1.5} 
\begin{table}\centering
\begin{tabular}{|c|c|c|c|c|c|}
\hline
            & \begin{tabular}[c]{@{}c@{}}\textbf{Goal} \end{tabular} & \begin{tabular}[c]{@{}c@{}} \textbf{$\varphi ,\psi$}\end{tabular} & \begin{tabular}[c]{@{}c@{}}$\mu$ in Th.\ref{AM:comfortable_view_with_prob} \end{tabular} & \begin{tabular}[c]{@{}c@{}}\textbf{Iteration number}\\
            \textbf{of Algorithm \ref{alg:highorder_inexact}}\\
            (Th. \ref{AM:comfortable_view_with_prob})
            \end{tabular} & \begin{tabular}[c]{@{}c@{}}\textbf{Each iteration}  \\\textbf{requires}\end{tabular}                                                                                                          \\ \hline
Loop 1  & \begin{tabular}[c]{@{}c@{}} $(\varepsilon,\sigma)$-solution \\of problem $\eqref{eq:main30}$ \end{tabular}                                        & $\eqref{eq:framework_aux_step1_inverse}$                                                 & $\mu_y$                                                                                  & $\widetilde{O} \left(1+ \sqrt{H_1/\mu_y} \right)$                 & \begin{tabular}[c]{@{}c@{}}Find  $(\varepsilon_1,\sigma_1)$-solution of $\eqref{eq:sub1_inverse}$\\ and calculate \\$\left(\delta^{(1)},  L_{\psi}\right)$-oracle of $\psi(y)$\end{tabular}       \\ \hline
Loop 2 & \begin{tabular}[c]{@{}c@{}} $(\varepsilon_1,\sigma_1)$-solution \\of problem $\eqref{eq:problem_step2_inverse}$ \end{tabular}                                  & $\eqref{eq:framework_aux_step2_inverse}$                                                 & $\mu_x$                                                                                  & $\widetilde{O}(1+\sqrt{H_2/\mu_x})$                               & \begin{tabular}[c]{@{}c@{}}Find  $(\varepsilon_2,\sigma_2)$-solution of $\eqref{eq:sub2_inverse}$\\ and calculate \\ $\left(\delta^{(2)}, L_{\psi}\right)$-oracle of $\psi(x)$\end{tabular}        \\ \hline
Loop 3  & \begin{tabular}[c]{@{}c@{}} $(\varepsilon_2,\sigma_2)$-solution \\of problem $\eqref{eq:sub2_inverse}$  \end{tabular}                                           & $\eqref{eq:framework_aux_step3_inverse}$                                                 & $H_2$                                                                                    & $\widetilde{O}(1+\sqrt{H_3/H_2})$                                 & \begin{tabular}[c]{@{}c@{}}Find  $(\varepsilon_3,\sigma_3)$-solution of $\eqref{eq:SeparateX01_inverse}$\\ and calculate \\ $\left(\delta^{(3)}, L_{\vp}\right)$-oracle of $\vp(x)$\end{tabular} \\ \hline

\end{tabular}

\caption{Summary of the three loops of the framework described in this Appendix. 
}
\label{tabl:saddleproblem_steps_inverse}
\end{table}

\subsection{Complexity of the framework}

Below we formally finalize in Theorem \ref{theorem:general_framework_inverse} the analysis of the framework by carefully combining the bounds obtained in \hyperref[subsec_first_inverse]{Loop 1}-\hyperref[subsec_third_inverse]{Loop 3} to obtain the final bounds for the total number of oracle calls for each part $f$, $G$, $h$ of the objective in problem \eqref{eq:framework_funct}. 
In the next Appendix~\ref{Appendix_h_not_sum}, we apply Theorem \ref{theorem:general_framework_inverse} to obtain complexity bounds for our framework applied to problem \eqref{eq:problem_st_h_sum} in the case $m_h=1$.

\begin{theorem} \label{theorem:general_framework_inverse}
Let Assumptions~\ref{assumpt:framework}, \ref{assumpt:framework_oracle}, \ref{assumpt:framework_oracle_x} hold. Then, execution of the optimization framework described in \hyperref[subsec_first_inverse]{Loop 1}-\hyperref[subsec_third_inverse]{Loop 3}
with 
$$
H_1 = 2L_G, H_2 = 2L_f, H_3 = 2\left(L_G + \frac{2L_G^2}{\mu_y + H_1}\right)
$$
generates an $(\varepsilon,\sigma)$-solution to the problem \eqref{eq:framework_funct} in the sense of Definition \ref{def:saddle_solution}.
Moreover, for the number of basic oracle calls it holds that

\begin{align}
    &\text{Number of calls of basic oracle $O_{f}$ for }  f \text{ is}:\nonumber\\ 
    & \widetilde{O} \left(\left( 1 + \sqrt{\frac{L_G}{\mu_y}}\right) \left( \mathcal{N}_f\left( \tau_f\right) + \left( 1 + \sqrt{\frac{L_f}{\mu_x}}\right)\cdot \tau_f\right)
    \right),\label{eq:framework_f_inverse}\\
    &\text{Number of calls of basic oracle  $O_{h}$ for } h \text{ is}:\nonumber\\ 
    &\widetilde{O} \left( \left( 1 + \sqrt{\frac{L_G}{\mu_y}}\right) \left( \tau_h +   \left( 1 + \sqrt{\frac{L_f}{\mu_x}}\right)\left(1 + \sqrt{\frac{L_G}{L_f}}\right)\mathcal{N}_h\left( \tau_h, 2L_G\right)\right)
    \right),\label{eq:framework_h_inverse}\\
    &\text{Number of calls of basic oracle $O^{x}_{G}$ for } G(\cdot, y)\text{ is}:\nonumber\\ 
    & \widetilde{O} \left( \left( 1 + \sqrt{\frac{L_G}{\mu_y}}\right) \left( \mathcal{N}_G^x\left( \tau_G\right) + \left( 1 + \sqrt{\frac{L_f}{\mu_x}}\right) \left( 1 + \sqrt{\frac{L_G}{L_f}}  \right)\tau_G\right)
    \right),\label{eq:framework_G_x_inverse}\\ 
    &\text{Number of calls of basic oracle $O^{y}_{G}$ for } G(x,\cdot)\text{ is}:\nonumber\\  
    & \widetilde{O} \left( \left( 1 + \sqrt{\frac{L_G}{\mu_y}}\right) \left( \tau_G + \left( 1 + \sqrt{\frac{L_f}{\mu_x}}\right) \left( 1 + \sqrt{\frac{L_G}{L_f}}  \right)\mathcal{N}_G^y\left( \tau_G, 2L_G\right)\right)
    \right).
    \label{eq:framework_G_y_inverse}
\end{align}
\end{theorem}

\begin{proof}
By construction, as an output of Loop 1 we obtain an $(\varepsilon,\sigma)$-solution to the problem \eqref{eq:framework_funct} according to Definition \ref{def:saddle_solution}.

We prove the estimates of for the numbers of oracle calls in two steps.
The first step is to formally prove that  in each loop the dependence of the number of oracle calls on the target accuracy $\varepsilon$ and a confidence level $\sigma$ is logarithmic. The second step is to multiply the estimates for the number of oracle calls between loops and choose the parameters $H_1$, $H_2$, $H_3$.

\textbf{Step 1. Polynomial dependence.}
Proof of this part is equivalent to the proof of the Theorem~\ref{theorem:general_framework}.

\textbf{Step 2. Final estimates.}

We have already counted the number of oracles calls for each oracle in each loop \hyperref[subsec_first_inverse]{Loop 1}-\hyperref[subsec_third_inverse]{Loop 3}, see the last paragraph of the description of each loop. 
We start with the number of basic oracle calls of $f$, which is called in each step of \hyperref[subsec_first_inverse]{Loop 1} and \hyperref[subsec_second_inverse]{Loop 2}.
Thus, the total number is
\begin{align*}
     \text{\# of calls in Loop1 + (\# of steps in Loop 1)$\cdot$(\# of calls in Loop 2)}\\
     =\widetilde{O}\left(1+\left( \frac{H_1 }{\mu_{y}} \right)^{\frac{1}{2}}\right)\mathcal{N}_{f}\left( \tau_f\right)\mathcal{K}_f\left(\varepsilon, \sigma\right) + 
     \widetilde{O}\left(1+\left( \frac{H_1 }{\mu_{y}} \right)^{\frac{1}{2}}\right)\cdot \left(\widetilde{O}\left(1+\left( \frac{H_2 }{\mu_{x}} \right)^{\frac{1}{2}}\right)\tau_f \right) \\
     =\widetilde{O} \left(\left( 1+ \sqrt{\frac{H_1}{\mu_y}}\right) \left( \mathcal{N}_f\left( \tau_f\right) + \left( 1+\sqrt{\frac{H_2}{\mu_x}}\right)\cdot \tau_f\right)
    \right),
\end{align*}
where we used that $\mathcal{K}_f\left(\varepsilon, \sigma\right)=\widetilde{O}(1)$.

The basic oracle of $h$ is called in each step of all the three loops.
Thus, the total number is
\begin{align*}
     \text{\# of calls in Loop1 + (\# of steps in Loop 1)$\cdot$(\# of calls in Loop 2)} \\ 
     \text{+ (\# of steps in Loop 1)$\cdot$(\# of steps in Loop 2)$\cdot$(\# of calls in Loop 3)}\\
     =\widetilde{O}\left(1+\left( \frac{H_1 }{\mu_{y}} \right)^{\frac{1}{2}}\right)\tau_h 
     +  \widetilde{O}\left(1+\left( \frac{H_1 }{\mu_{y}} \right)^{\frac{1}{2}}\right)\cdot \left(\widetilde{O}\left(1+\left( \frac{H_2 }{\mu_{x}} \right)^{\frac{1}{2}}\right)\mathcal{N}_{h}\left( \tau_h, H_1\right)\mathcal{K}_h\left(\varepsilon_2, \sigma_2\right) \right)\\
     +  \widetilde{O}\left(1+\left( \frac{H_1 }{\mu_{y}} \right)^{\frac{1}{2}}\right)\cdot\left(\widetilde{O}\left(1+\left( \frac{H_2 }{\mu_{x}} \right)^{\frac{1}{2}}\right)\right) \cdot \left(\widetilde{O}\left(1+\left( \frac{H_3 }{H_2} \right)^{\frac{1}{2}}\right)\mathcal{N}_{h}\left( \tau_h, H_1\right)\mathcal{K}_h\left(\varepsilon_3, \sigma_3\right) \right)
     \\
     =\widetilde{O} \left( \left( 1+ \sqrt{\frac{H_1}{\mu_y}}\right) \left( \tau_h +   \left( 1+\sqrt{\frac{H_2}{\mu_x}}\right)\left(\mathcal{N}_h\left( \tau_h, H_1\right) + \left( 1+ \sqrt{\frac{H_3}{H_2}}\right) \cdot \mathcal{N}_h\left( \tau_h, H_1\right)\right)\right)
    \right)\\
    =\widetilde{O} \left( \left( 1+ \sqrt{\frac{H_1}{\mu_y}}\right) \left( \tau_h +   \left( 1+\sqrt{\frac{H_2}{\mu_x}}\right) \left( 1+ \sqrt{\frac{H_3}{H_2}}\right)\mathcal{N}_h\left( \tau_h, H_1\right)
    \right)\right),
\end{align*}
where we used that $\mathcal{K}_h\left(\varepsilon, \sigma\right)=\widetilde{O}(1)$.

The basic oracle of $G(\cdot, y)$ is called in each step of all the three loops.
Thus, the total number is
\begin{align*}
     \text{\# of calls in Loop1 + (\# of steps in Loop 1)$\cdot$(\# of calls in Loop 2)} \\ 
     \text{+ (\# of steps in Loop 1)$\cdot$(\# of steps in Loop 2)$\cdot$(\# of calls in Loop 3)}\\
     =\widetilde{O}\left(1+\left( \frac{H_1 }{\mu_{y}} \right)^{\frac{1}{2}}\right)\mathcal{N}_G^x\left( \tau_G\right)\mathcal{K}_G^x\left(\varepsilon, \sigma\right) 
     +  \widetilde{O}\left(1+\left( \frac{H_1 }{\mu_{y}} \right)^{\frac{1}{2}}\right)\cdot \left(\widetilde{O}\left(1+\left( \frac{H_2 }{\mu_{x}} \right)^{\frac{1}{2}}\right)\tau_G\right)\\
     +  \widetilde{O}\left(1+\left( \frac{H_1 }{\mu_{y}} \right)^{\frac{1}{2}}\right)\cdot\left(\widetilde{O}\left(1+\left( \frac{H_2 }{\mu_{x}} \right)^{\frac{1}{2}}\right)\right) \cdot \left(\widetilde{O}\left(1+\left( \frac{H_3 }{H_2} \right)^{\frac{1}{2}}\right)\tau_G \right)
     \\
     =\widetilde{O} \left( \left( 1+ \sqrt{\frac{H_1}{\mu_y}}\right) \left( \mathcal{N}_G^x\left( \tau_G\right) +   \left( 1+\sqrt{\frac{H_2}{\mu_x}}\right)\left(\tau_G + \left( 1+ \sqrt{\frac{H_3}{H_2}}\right) \tau_G\right)
    \right)\right)\\
    =\widetilde{O} \left( \left( 1+ \sqrt{\frac{H_1}{\mu_y}}\right) \left( \mathcal{N}_G^x\left( \tau_G\right)+   \left( 1+\sqrt{\frac{H_2}{\mu_x}}\right) \left( 1+ \sqrt{\frac{H_3}{H_2}}\right)\tau_G
    \right)\right),
\end{align*}
where we used that $\mathcal{K}_G^x\left(\varepsilon, \sigma\right)=\widetilde{O}(1)$.

Finally, the basic oracle of $G(x,\cdot)$ is called in each step of all the three loops.
Thus, the total number is
\begin{align*}
     \text{\# of calls in Loop1 + (\# of steps in Loop 1)$\cdot$(\# of calls in Loop 2)} \\ 
     \text{+ (\# of steps in Loop 1)$\cdot$(\# of steps in Loop 2)$\cdot$(\# of calls in Loop 3)}\\
     =\widetilde{O}\left(1+\left( \frac{H_1 }{\mu_{y}} \right)^{\frac{1}{2}}\right)\tau_G
     +  \widetilde{O}\left(1+\left( \frac{H_1 }{\mu_{y}} \right)^{\frac{1}{2}}\right)\cdot \left(\widetilde{O}\left(1+\left( \frac{H_2 }{\mu_{x}} \right)^{\frac{1}{2}}\right)\mathcal{N}_G^y\left( \tau_G, H_1\right) \mathcal{K}_G^y\left(\varepsilon_2, \sigma_2\right)\right)\\
     +  \widetilde{O}\left(1+\left( \frac{H_1 }{\mu_{y}} \right)^{\frac{1}{2}}\right)\cdot\left(\widetilde{O}\left(1+\left( \frac{H_2 }{\mu_{x}} \right)^{\frac{1}{2}}\right)\right) \cdot \left(\widetilde{O}\left(1+\left( \frac{H_3 }{H_2} \right)^{\frac{1}{2}}\right)\mathcal{N}_G^y\left( \tau_G, H_1\right) \mathcal{K}_G^y\left(\varepsilon_2, \sigma_2\right) \right)
     \\
     =\widetilde{O} \left( \left( 1+ \sqrt{\frac{H_1}{\mu_y}}\right) \left( \tau_G +   \left( 1+\sqrt{\frac{H_2}{\mu_x}}\right)\left(\mathcal{N}_G^y\left( \tau_G, H_1\right)  + \left( 1+ \sqrt{\frac{H_3}{H_2}}\right) \mathcal{N}_G^y\left( \tau_G, H_1\right) \right)
    \right)\right)\\
    =\widetilde{O} \left( \left( 1+ \sqrt{\frac{H_1}{\mu_y}}\right) \left( \tau_G+   \left( 1+\sqrt{\frac{H_2}{\mu_x}}\right) \left( 1+ \sqrt{\frac{H_3}{H_2}}\right)\mathcal{N}_G^y\left( \tau_G, H_1\right)
    \right)\right),
\end{align*}
where we used that $\mathcal{K}_G^y\left(\varepsilon_2, \sigma_2\right)=\widetilde{O}(1)$.

 The final estimates are obtained by substituting the constants $H_1, H_2, H_3$ given by
$$H_1 = 2L_G, H_2 = 2L_f, H_3 = 2\left(L_G + \frac{2L_G^2}{\mu_y + H_1}\right)\leq 2\left(L_G + \frac{2L_G^2}{H_1}\right)=4L_G.
$$
\end{proof}
\section{Proof of Theorem~\ref{theorem:G-sum_noprox} and Theorem~\ref{theorem:G-sum_h_prox}}\label{Appendix_h_not_sum}

In this appendix we prove Theorems~\ref{theorem:G-sum_noprox},  \ref{theorem:G-sum_h_prox} and Corollary \ref{theorem:h-not-sum} and construct algorithms for problem \eqref{eq:problem_Gsum_minmax}
using the results of Section \ref{section:saddleFramework}, in particular, Theorem \ref{theorem:general_framework}, for the case $L_f \geq L_G$, and the results of the previous Appendix, in particular, Theorem \ref{theorem:general_framework_inverse}. To use these theorems we need to satisfy Assumptions \ref{assumpt:framework_oracle}, \ref{assumpt:framework_oracle_x}, which is done in the first subsection. Then, in the next subsections, we combine the building blocks  to obitan the final results.


\subsection{Algorithms to guarantee Assumptions \ref{assumpt:framework_oracle}, \ref{assumpt:framework_oracle_x}}
We start with two auxiliary results, that show how to satisfy Assumptions \ref{assumpt:framework_oracle}, \ref{assumpt:framework_oracle_x} algorithmically.
The first lemma provides complexity for inexact solution of the maximization problem \eqref{eq:framework_g_obt_oracle} and the complexity of finding an inexact oracle for function $g$ defined in the same equation, thereby proving that Assumption \ref{assumpt:framework_oracle} holds. We underline that the algorithm which guarantees Assumption \ref{assumpt:framework_oracle} depends on whether $L_h \geq L_G$ or $L_h \leq L_G$.
After that we provide a simple corollary to show that Assumption \ref{assumpt:framework_oracle_x} also holds.

 \begin{lemma} \label{lem:obtain_oracle_mh=1}
Let the function $g$ be defined via maximization problem in \eqref{eq:framework_g_obt_oracle}, i.e.
\begin{align}\label{eq:g_obt_oracle_mh=1}
g(x)= \max _{y \in  \R^{d_y}}\left\{G(x, y)-h(y) -\frac{H}{2}\|y-y_0\|^2\right\},
\end{align}
 where $G(x,y)$, $h(y)$ are according to \eqref{eq:problem_Gsum_minmax} 
 and satisfy Assumption \ref{assumpt:G_sum}.1,2,3(a), $y_0 \in \R^{d_y}$.  Then, for each of two cases $L_h \geq L_G$ and $L_h \leq L_G$ we  organize computations in two loops and apply Algorithm \ref{alg:restarts_inexact_notconvex}, so that  Assumption~\ref{assumpt:framework_oracle} holds with $\tau_G$ basic oracle calls for $G(\cdot,y)$ and the following estimates for the number of basic oracle calls for $G(x,\cdot)$ and $h$ respectively
 \begin{align}\label{eq:complofgG_mh=1}
&\mathcal{N}_{G}^y\left( \tau_G, H\right) = O\left(\tau_G + \tau_G\sqrt{L_G/(H+\mu_y)}\right),\\
&\mathcal{N}_{h}\left( \tau_h, H\right) = O\left(\tau_h+\tau_h \sqrt{L_h/(H+\mu_y)}\right). \label{eq:complofgh_mh=1}
\end{align}
We name these algorithms "Sliding $L_h \geq L_G$" and "Sliding $L_h \leq L_G$".
\end{lemma}
\begin{proof}
To satisfy Assumption~\ref{assumpt:framework_oracle} we need to provide an $\left(\delta\left(\varepsilon\right)/2 ,\sigma_0\left(\varepsilon,\sigma\right)\right)$-solution to the problem \eqref{eq:g_obt_oracle_mh=1} and $\left(\delta\left(\varepsilon\right) ,\sigma_0\left(\varepsilon,\sigma\right), 2L_g\right)$-oracle of $g$ in \eqref{eq:g_obt_oracle_mh=1}, where $L_g = L_G+2L_G^2/(\mu_y+H)$.\\
By Lemma~\ref{lemma:obt_delta_oracle} with $F(x,y)=G(x,y)$, $w(y)=h(y)+\frac{H}{2}\|y-y_0\|^2$, $\delta=\delta\left(\varepsilon\right)$ and $\sigma_0 = \sigma_0\left(\varepsilon,\sigma\right)$ applied to the problem \eqref{eq:g_obt_oracle_mh=1},
if we find a $(\delta/2,\sigma_0)$-solution $\tilde{y}_{\delta/2}(x)$ of the problem  \eqref{eq:g_obt_oracle_mh=1},
then $ \nabla_{x} G\left(x, \tilde{y}_{\delta/2}(x)\right)$ is $(\delta,\sigma_0,2 L_g)$-oracle of $g$ and its calculation requires $\tau_G$ calls of the oracle $\nabla_x G(\cdot,y)$.
To finish the proof, we now focus on obtaining a $(\delta/2,\sigma_0)$-solution $\tilde{y}_{\delta/2}(x)$ of the problem  \eqref{eq:g_obt_oracle_mh=1}. For this we consider two cases $L_h \geq L_G$ and $L_h \leq L_G$ and for each one we construct a two-loop procedure described below. We begin with the case $L_h \geq L_G$.

\paragraph{Sliding for $L_h \geq L_G$, Loop 1}\label{lemma_loop1_mh=1}$\;$\\
The goal of Loop 1 is to find an $\left(\delta\left(\varepsilon\right)/2,\sigma_0\left(\varepsilon,\sigma\right)\right)$-solution of  problem \eqref{eq:g_obt_oracle_mh=1} as a maximization problem in $y$.
To obtain such an approximate solution, we change the sign of this optimization problem and apply Algorithm~\ref{alg:restarts_inexact_notconvex} with  
\begin{align}\label{aux_step4_mh=1}
    \varphi = -G(x,y) ,\;\;\;\; \psi = h(y) + \frac{H}{2}\|y-y_0\|^2.
\end{align}
Function $\varphi$ is convex and has $L_G$-Lipschitz continuous gradient, function $\psi$ is $ H+\mu_y$-strongly convex and has $L_h+H$-Lipschitz continuous gradient. Thus, we can apply Algorithm  \ref{alg:restarts_inexact_notconvex} with exact oracles and parameter $H_1\geq 2L_G$, which will be chosen later, to solve problem \eqref{eq:g_obt_oracle_mh=1}.
To satisfy the conditions of Theorem~\ref{AM:comfortable_view_with_prob}, which gives the complexity of Algorithm  \ref{alg:restarts_inexact_notconvex}, we, first, observe that the oracles of $\varphi$ and $\psi$ are exact and, second, observe that we need in each iteration of Algorithm \ref{alg:highorder_inexact}, used as a building block in Algorithm \ref{alg:restarts_inexact_notconvex}, to find an $\left(\tilde{\varepsilon}^{(1)}_f\left(\delta/2\right),\tilde{\sigma}^{(1)}\left(\delta/2,\sigma_0\right)\right)$-solution to the auxiliary problem \eqref{prox_step_inexact}, which in this case has the following form:  
\begin{align} 
&z_{k+1}^t = \arg\min _{z \in  \R^{d_y}} \{ \la \nabla \varphi(z_k^{md}),z-z_k^{md}\ra+\psi(z) + \frac{H_1}{2}\|z-z_k^{md}\|_2^2 \} \nonumber
\\
&= \arg\min _{z \in  \R^{d_y}} \{ -\la \nabla_{z}  G(x, z_k^{md}),z-z_k^{md}\ra+h(z) + \frac{H}{2}\|z-y_0\|^2+\frac{H_1}{2}\|z-z_k^{md}\|_2^2\} \label{eq:190_mh=1},
\end{align}
 where $\tilde{\sigma}^{(1)}\left(\delta/2,\sigma_0\right), \tilde{\varepsilon}^{(1)}_f\left(\delta/2\right)$ need to satisfy  inequalities \eqref{varepsilon_poly_CATD_with_prob}, \eqref{sigma_poly_CATD_with_prob}. 
 Below, in the Loop 2, we explain how to solve this auxiliary problem in such a way that these inequalities hold.

To summarize Loop 1, both main assumptions of Theorem~\ref{AM:comfortable_view_with_prob} hold and we can use it to guarantee that we obtain an $(\delta/2,\sigma_0)$-solution of problem \eqref{eq:g_obt_oracle_mh=1}. Due to polynomial dependencies  $\delta\left(\varepsilon\right)=\bf{poly}\left(\varepsilon\right)$, $\sigma_0\left(\varepsilon, \sigma \right)=\bf{poly}\left(\varepsilon, \sigma \right)$ this requires $\widetilde{O}\left(1+\left( \frac{H_1 }{\mu_{\vp} + \mu_{\psi}} \right)^{\frac{1}{2}}\right)=\widetilde{O}\left(1+\left( \frac{H_1 }{\mu_{y}+H} \right)^{\frac{1}{2}}\right)$ 
calls to the (exact) oracles for $\vp$ and for $\psi$, and the same number of times solving the auxiliary problem \eqref{eq:190_mh=1}.
Combining this oracle complexity with the cost of calculating (exact) oracles for $\vp$ and for $\psi$, we obtain that solving problem \eqref{eq:g_obt_oracle} requires 
$\widetilde{O}\left(1+\left( \frac{H_1 }{\mu_{y}+H} \right)^{\frac{1}{2}}\right)\tau_G$ calls of the basic oracle for $G(x,\cdot)$ and $\widetilde{O}\left(\tau_h\left(1+\left( \frac{H_1 }{\mu_{y}+H} \right)^{\frac{1}{2}}\right)\right)$ of the basic oracles for $h$. 
The only remaining thing is to provide an inexact solution to problem \eqref{eq:190_mh=1} and, next, we move to Loop 2 to explain how to guarantee this. Note that we need to solve problem \eqref{eq:190_mh=1} $\widetilde{O}\left(1+\left( \frac{H_1 }{\mu_{y}+H} \right)^{\frac{1}{2}}\right)$ times.\\
\paragraph{Sliding for $L_h \geq L_G$, Loop 2 }\label{lemma_loop2_mh=1}$\;$\\
As mentioned in the previous Loop 1,
in each iteration of Algorithm \ref{alg:restarts_inexact_notconvex} in 
Loop 1 we need to find many times an $(\varepsilon_2,\sigma_2)$-solution of the auxiliary problem \eqref{eq:190_mh=1}, where we denoted for simplicity $\sigma_2=\tilde{\sigma}^{(1)}\left(\delta/2,\sigma_0\right)$ and $\varepsilon_2=\tilde{\varepsilon}^{(1)}_f\left(\delta/2\right)$.
To solve problem \eqref{eq:190_mh=1}, we would like to apply Algorithm  \ref{alg:restarts_inexact_notconvex} with
\begin{align}\label{aux_step5_mh=1}
    \varphi = h(z) ,\;\;\;\; \psi = -\la \nabla_{z}  G(x, z_k^{md}),z-z_k^{md}\ra + \frac{H}{2}\|z-y_0\|^2+\frac{H_1}{2}\|z-z_k^{md}\|_2^2.
\end{align}
Function $\varphi$ is $\mu_y$-strongly convex and has $L_h$-Lipschitz continuous gradient, function $\psi$ is $ H+H_1$-strongly convex and has $H+H_1$-Lipschitz continuous gradient. Thus, we can apply Algorithm  \ref{alg:restarts_inexact_notconvex} with exact oracles and parameter $H_2\geq 2L_h$, which will be chosen later, to solve problem \eqref{eq:190_mh=1}.
To satisfy the conditions of Theorem~\ref{AM:comfortable_view_with_prob}, which gives the complexity of Algorithm  \ref{alg:restarts_inexact_notconvex}, we, first, observe that the oracles of $\varphi$ and $\psi$ are exact and, second, observe that we need in each iteration of Algorithm \ref{alg:highorder_inexact}, used as a building block in Algorithm \ref{alg:restarts_inexact_notconvex}, to find an $\left(\tilde{\varepsilon}^{(2)}_f\left(\varepsilon_2\right),\tilde{\sigma}^{(2)}\left(\varepsilon_2,\sigma_2\right)\right)$-solution to the auxiliary problem \eqref{prox_step_inexact}, which in this case has the following form:  
\begin{align}
    &u^t_{m+1} = \arg\min _{u \in  \R^{d_x}} \{ \la \nabla \varphi(u_m^{md}),u-u_m^{md}\ra+\psi(u) + \frac{H_2}{2}\|u-u_m^{md}\|_2^2 \}\nonumber \\
    &= \arg\min _{u \in  \R^{d_x}} \{ \la \nabla  h(u_m^{md}),u-u_m^{md}\ra-\la \nabla_{z}  G(x, z_k^{md}),u-z_k^{md}\ra + \frac{H}{2}\|u-y_0\|^2+\frac{H_1}{2}\|u-z_k^{md}\|_2^2 +\frac{H_2}{2}\|u-u_m^{md}\|_2^2\}.\label{eq:SeparateX01_mh=1}
\end{align}
 This quadratic auxiliary problem \eqref{eq:SeparateX01_mh=1} can be solved explicitly and exactly. Thus, the second main assumption of Theorem~\ref{AM:comfortable_view_with_prob} is satisfied with  $\tilde{\sigma}^{(2)}\left(\varepsilon_2,\sigma_2\right) = 0, \tilde{\varepsilon}^{(2)}_f\left(\varepsilon_2\right) = 0$, which clearly satisfy \eqref{delta_poly_CATD_with_prob} and \eqref{sigma_0_poly_CATD_with_prob}.

To summarize Loop 2, both main assumptions of Theorem~\ref{AM:comfortable_view_with_prob} hold and we can use it to guarantee that we obtain an $(\varepsilon_2,\sigma_2)$-solution of the auxiliary problem \eqref{eq:190_mh=1}. This requires 
$O\left(\left(1+\left( \frac{H_2 }{\mu_{\vp} + \mu_{\psi}} \right)^{\frac{1}{2}}\right)\log \varepsilon_2^{-1} \right)=O\left(\left(1+\left( \frac{H_2}{\mu_y+H+H_1} \right)^{\frac{1}{2}}\right)\log \varepsilon_2^{-1}\right)$ calls to the (exact) oracles for $\vp$ and for $\psi$, and the same number of times solving the auxiliary problem \eqref{eq:SeparateX01_mh=1}.
Combining this oracle complexity with the cost of calculating (exact) oracles for $\vp$ and for $\psi$, 
we obtain that solving problem \eqref{eq:190_mh=1} requires   
$O\left(\tau_h\left(1+\left( \frac{H_2}{\mu_y+H+H_1} \right)^{\frac{1}{2}}\right)\log \varepsilon_2^{-1}\right)$ calls of the basic oracle for $h$. Also according to the polynomial dependencies \eqref{varepsilon_poly_CATD_with_prob}, \eqref{sigma_poly_CATD_with_prob} we obtain that 
$$
\sigma_2 = \tilde{\sigma}^{(1)}\left(\delta/2,\sigma_0\right) = \textbf{poly}(\delta/2,\sigma_0), \;\; \varepsilon_2 = \tilde{\varepsilon}^{(1)}_f\left(\delta/2,\sigma_0\right) = \textbf{poly}(\delta/2,\sigma_0).
$$
Using conditions $\delta\left(\varepsilon\right)=\bf{poly}\left(\varepsilon\right)$, $\sigma_0\left(\varepsilon, \sigma \right)=\bf{poly}\left(\varepsilon, \sigma \right)$ in the formulation of Assumption~\ref{assumpt:framework_oracle} we obtain that the dependencies 
$$
\sigma_2\left(\varepsilon,\sigma\right), \tilde{\sigma}^{(1)}\left(\varepsilon,\sigma\right), \varepsilon_2 \left(\varepsilon,\sigma\right), \tilde{\varepsilon}^{(1)}_f\left(\varepsilon,\sigma\right)
$$
are polynomial. Then, we can use notation $\widetilde{O}(\cdot)$   without specifying what precision we mean and implying that the logarithmic part depends on the initial  $\varepsilon, \sigma$. \\
\paragraph{Sliding $L_h \geq L_G$, combining the estimates of both loops}$\;$\\ 
Combining the estimates of the above Loop 1 and Loop 2 we see that, finding a point $\tilde{y}_{\delta/2}(x)$ that is a $\left(\delta\left(\varepsilon\right)/2 ,\sigma_0\left(\varepsilon,\sigma\right)\right)$-solution to the problem \eqref{eq:g_obt_oracle_mh=1} 
requires the following number of calls of the basic oracles of $G(x,\cdot)$ and $h$ respectively
\begin{align}\label{eq:obt_complofgG_mh=1}
& \widetilde{O}\left(\tau_G + \tau_G\sqrt{H_1/(H+\mu_y)}\right),\\
&  \widetilde{O}\left( \tau_h\left(1 + \sqrt{H_1/(H+\mu_y)}\right) + \left(1 + \sqrt{H_1/(H+\mu_y)} \right)\tau_h\left(1+\sqrt{\frac{H_2}{\mu_y+H+H_1}}\right)\right). \label{eq:obt_complofgh_mh=1}
\end{align}
 Finding $\left(\delta\left(\varepsilon\right) ,\sigma_0\left(\varepsilon,\sigma\right), 2L_g\right)$-oracle of $g$ by calculating $ \nabla_{x} G\left(x, \tilde{y}_{\delta/2}(x)\right)$ requires additionally $\tau_G=m_G$ 
calls of the basic oracle for $G(\cdot,y)$. 
Since in Assumption~\ref{assumpt:framework_oracle} we denote the dependence on the target accuracy $\varepsilon$ and confidence level $\sigma$ by a separate quantities denoted by $\mathcal{K}(\varepsilon,\sigma)$ and in this case it is logarithmic, choosing $H_1 = 2L_G$ and $H_2 = 2L_h$ we get the final estimates for  $\mathcal{N}_{G}^y$ and $\mathcal{N}_{h}$ to guarantee that  Assumption~\ref{assumpt:framework_oracle} holds:
\begin{align*}
&\mathcal{N}_{G}^y = O\left(\tau_G + \tau_G\sqrt{\frac{L_G}{\mu_y+H}}\right),\\ 
&\mathcal{N}_{h} = O\left( \tau_h \left( 1 + \sqrt{2L_G/(H+\mu_y)}\right)\left(1+\sqrt{\frac{2L_h}{\mu_y+H+2L_G}}\right)\right)\\ 
& =  O\left( 1 + \sqrt{\frac{2L_G}{\mu_y+H}} + \sqrt{\frac{2L_h}{\mu_y+H}} + \sqrt{\frac{2L_G}{H+\mu_y}} \sqrt{\frac{2L_h}{\mu_y+H+2L_G}} \right)\tau_h = O\left( \tau_h \left(1 +  \sqrt{\frac{L_h}{\mu_y+H}}\right)\right),
\end{align*}
where we used that $L_h\geq L_G$

Our aim now is to obtain the same estimates on $\mathcal{N}_{G}^y$ and $\mathcal{N}_{h}$ for the case when $L_h\leq L_G$. We do this by changing the order of Loop 1 and Loop 2 in the construction of previous Algorithm.

\paragraph{Sliding for $L_h \leq L_G$, Loop 1 }\label{lemma_loop1_mh=1_inverse}$\;$\\
The goal of Loop 1 is to find an $\left(\delta\left(\varepsilon\right)/2,\sigma_0\left(\varepsilon,\sigma\right)\right)$-solution of  problem \eqref{eq:g_obt_oracle_mh=1} as a maximization problem in $y$.
To obtain such an approximate solution, we change the sign of this optimization problem and apply Algorithm~\ref{alg:restarts_inexact_notconvex} with  
\begin{align}\label{aux_step4_mh=1_inverse}
    \varphi = h(y) ,\;\;\;\; \psi = -G(x,y) + \frac{H}{2}\|y-y_0\|^2.
\end{align}
Function $\varphi$ is $\mu_y$-strongly convex and has $L_h$-Lipschitz continuous gradient, function $\psi$ is $H$-strongly convex and has $L_h+H$-Lipschitz continuous gradient. Thus, we can apply Algorithm  \ref{alg:restarts_inexact_notconvex} with exact oracles and parameter $H_1\geq 2L_h$, which will be chosen later, to solve problem \eqref{eq:g_obt_oracle_mh=1}.
To satisfy the conditions of Theorem~\ref{AM:comfortable_view_with_prob}, which gives the complexity of Algorithm  \ref{alg:restarts_inexact_notconvex}, we, first, observe that the oracles of $\varphi$ and $\psi$ are exact and, second, observe that we need in each iteration of Algorithm \ref{alg:highorder_inexact}, used as a building block in Algorithm \ref{alg:restarts_inexact_notconvex}, to find an $\left(\tilde{\varepsilon}^{(1)}_f\left(\delta/2\right),\tilde{\sigma}^{(1)}\left(\delta/2,\sigma_0\right)\right)$-solution to the auxiliary problem \eqref{prox_step_inexact}, which in this case has the following form:  
\begin{align} 
&z_{k+1}^t = \arg\min _{z \in  \R^{d_y}} \{ \la \nabla \varphi(z_k^{md}),z-z_k^{md}\ra+\psi(z) + \frac{H_1}{2}\|z-z_k^{md}\|_2^2 \} \nonumber
\\
&= \arg\min _{z \in  \R^{d_y}} \{ \la \nabla_{z} h(z_k^{md}),z-z_k^{md}\ra-G(x, z) + \frac{H}{2}\|z-y_0\|^2+\frac{H_1}{2}\|z-z_k^{md}\|_2^2\} \label{eq:190_mh=1_inverse},
\end{align}
 where $\tilde{\sigma}^{(1)}\left(\delta/2,\sigma_0\right), \tilde{\varepsilon}^{(1)}_f\left(\delta/2\right)$ need to satisfy  inequalities \eqref{varepsilon_poly_CATD_with_prob}, \eqref{sigma_poly_CATD_with_prob}. 
 Below, in the Loop 2, we explain how to solve this auxiliary problem in such a way that these inequalities hold.

To summarize Loop 1, both main assumptions of Theorem~\ref{AM:comfortable_view_with_prob} hold and we can use it to guarantee that we obtain an $(\delta/2,\sigma_0)$-solution of problem \eqref{eq:g_obt_oracle_mh=1}. Due to polynomial dependencies  $\delta\left(\varepsilon\right)=\bf{poly}\left(\varepsilon\right)$, $\sigma_0\left(\varepsilon, \sigma \right)=\bf{poly}\left(\varepsilon, \sigma \right)$ this requires $\widetilde{O}\left(1+\left( \frac{H_1 }{\mu_{\vp} + \mu_{\psi}} \right)^{\frac{1}{2}}\right)=\widetilde{O}\left(1+\left( \frac{H_1 }{\mu_{y}+H} \right)^{\frac{1}{2}}\right)$ 
calls to the (exact) oracles for $\vp$ and for $\psi$, and the same number of times solving the auxiliary problem \eqref{eq:190_mh=1_inverse}.
Combining this oracle complexity with the cost of calculating (exact) oracles for $\vp$ and for $\psi$, we obtain that solving problem \eqref{eq:g_obt_oracle} requires 
$\widetilde{O}\left(1+\left( \frac{H_1 }{\mu_{y}+H} \right)^{\frac{1}{2}}\right)\tau_G$ calls of the basic oracle for $G(x,\cdot)$ and $\widetilde{O}\left(1+\left( \frac{H_1 }{\mu_{y}+H} \right)^{\frac{1}{2}}\right)\tau_h$ of the basic oracles for $h$. 
The only remaining thing is to provide an inexact solution to problem \eqref{eq:190_mh=1_inverse} and, next, we move to Loop 2 to explain how to guarantee this. Note that we need to solve problem \eqref{eq:190_mh=1_inverse} $\widetilde{O}\left(1+\left( \frac{H_1 }{\mu_{y}+H} \right)^{\frac{1}{2}}\right)$ times.\\

\paragraph{Sliding for $L_h \leq L_G$, Loop 2 }\label{lemma_loop2_mh=1_inverse}$\;$\\
As mentioned in the previous Loop 1,
in each iteration of Algorithm \ref{alg:restarts_inexact_notconvex} in 
Loop 2 we need to find many times an $(\varepsilon_2,\sigma_2)$-solution of the auxiliary problem \eqref{eq:190_mh=1_inverse}, where we denoted for simplicity $\sigma_2=\tilde{\sigma}^{(1)}\left(\delta/2,\sigma_0\right)$ and $\varepsilon_2=\tilde{\varepsilon}^{(1)}_f\left(\delta/2\right)$.
To solve problem \eqref{eq:190_mh=1_inverse}, we would like to apply Algorithm  \ref{alg:restarts_inexact_notconvex} with
\begin{align}\label{aux_step5_mh=1_inverse}
    \varphi = -G(x, z) ,\;\;\;\; \psi = \la \nabla  h( z_k^{md}),z-z_k^{md}\ra + \frac{H}{2}\|z-y_0\|^2+\frac{H_1}{2}\|z-z_k^{md}\|_2^2.
\end{align}
Function $\varphi$ is convex and has $L_G$-Lipschitz continuous gradient, function $\psi$ is $ H+H_1+\mu_y$-strongly convex and has $H+H_1$-Lipschitz continuous gradient. Thus, we can apply Algorithm  \ref{alg:restarts_inexact_notconvex} with exact oracles and parameter $H_2\geq 2L_G$, which will be chosen later, to solve problem \eqref{eq:190_mh=1_inverse}.
To satisfy the conditions of Theorem~\ref{AM:comfortable_view_with_prob}, which gives the complexity of Algorithm  \ref{alg:restarts_inexact_notconvex}, we, first, observe that the oracles of $\varphi$ and $\psi$ are exact and, second, observe that we need in each iteration of Algorithm \ref{alg:highorder_inexact}, used as a building block in Algorithm \ref{alg:restarts_inexact_notconvex}, to find an $\left(\tilde{\varepsilon}^{(2)}_f\left(\varepsilon_2\right),\tilde{\sigma}^{(2)}\left(\varepsilon_2,\sigma_2\right)\right)$-solution to the auxiliary problem \eqref{prox_step_inexact}, which in this case has the following form:  
\begin{align}
    &u^t_{m+1} = \arg\min _{u \in  \R^{d_x}} \{ \la \nabla \varphi(u_m^{md}),u-u_m^{md}\ra+\psi(u) + \frac{H_2}{2}\|u-u_m^{md}\|_2^2 \}\nonumber \\
    &= \arg\min _{u \in  \R^{d_x}} \{ -\la \nabla_u  G(x,u_m^{md}),u-u_m^{md}\ra+\la \nabla  h(z_k^{md}),u-z_k^{md}\ra + \frac{H}{2}\|u-y_0\|^2+\frac{H_1}{2}\|u-z_k^{md}\|_2^2 +\frac{H_2}{2}\|u-u_m^{md}\|_2^2\}.\label{eq:SeparateX01_mh=1_inverse}
\end{align}
 This quadratic auxiliary problem \eqref{eq:SeparateX01_mh=1_inverse} can be solved explicitly and exactly. Thus, the second main assumption of Theorem~\ref{AM:comfortable_view_with_prob} is satisfied with  $\tilde{\sigma}^{(2)}\left(\varepsilon_2,\sigma_2\right) = 0, \tilde{\varepsilon}^{(2)}_f\left(\varepsilon_2\right) = 0$, which clearly satisfy \eqref{delta_poly_CATD_with_prob} and \eqref{sigma_0_poly_CATD_with_prob}.

To summarize Loop 2, both main assumptions of Theorem~\ref{AM:comfortable_view_with_prob} hold and we can use it to guarantee that we obtain an $(\varepsilon_2,\sigma_2)$-solution of the auxiliary problem \eqref{eq:190_mh=1_inverse}. This requires 
$O\left(\left(1+\left( \frac{H_2 }{\mu_{\vp} + \mu_{\psi}} \right)^{\frac{1}{2}}\right)\log \varepsilon_2^{-1} \right)=O\left(\left(1+\left( \frac{H_2}{\mu_y+H+H_1} \right)^{\frac{1}{2}}\right)\log \varepsilon_2^{-1}\right)$ calls to the (exact) oracles for $\vp$ and for $\psi$, and the same number of times solving the auxiliary problem \eqref{eq:SeparateX01_mh=1_inverse}.
Combining this oracle complexity with the cost of calculating (exact) oracles for $\vp$ and for $\psi$, 
we obtain that solving problem \eqref{eq:190_mh=1_inverse} requires   
$O\left(\left(1+\left( \frac{H_2}{\mu_y+H+H_1} \right)^{\frac{1}{2}}\right)\log \varepsilon_2^{-1}\right)\tau_G$ calls of the basic oracle for $G(x, \cdot)$. Also according to the polynomial dependences \eqref{varepsilon_poly_CATD_with_prob}, \eqref{sigma_poly_CATD_with_prob} we obtain that 
$$
\sigma_2 = \tilde{\sigma}^{(1)}\left(\delta/2,\sigma_0\right) = \textbf{poly}(\delta/2,\sigma_0), \;\; \varepsilon_2 = \tilde{\varepsilon}^{(1)}_f\left(\delta/2,\sigma_0\right) = \textbf{poly}(\delta/2,\sigma_0).
$$
Using conditions $\delta\left(\varepsilon\right)=\bf{poly}\left(\varepsilon\right)$, $\sigma_0\left(\varepsilon, \sigma \right)=\bf{poly}\left(\varepsilon, \sigma \right)$ in the formulation of Asumption~\ref{assumpt:framework_oracle} we obtain that the dependencies 
$$
\sigma_2\left(\varepsilon,\sigma\right), \tilde{\sigma}^{(1)}\left(\varepsilon,\sigma\right), \varepsilon_2 \left(\varepsilon,\sigma\right), \tilde{\varepsilon}^{(1)}_f\left(\varepsilon,\sigma\right)
$$
are polynomial. Then, we can use notation $\widetilde{O}(\cdot)$   without specifying what precision we mean and implying that the logarithmic part depends on the initial  $\varepsilon, \sigma$. \\
\paragraph{Sliding for $L_h \leq L_G$, combining the estimates of both loops}$\;$\\ 
Combining the estimates of the above Loop 1 and Loop 2 we see that, finding a point $\tilde{y}_{\delta/2}(x)$ which is an $\left(\delta\left(\varepsilon\right)/2 ,\sigma_0\left(\varepsilon,\sigma\right)\right)$-solution to the problem \eqref{eq:g_obt_oracle_mh=1} 
requires the following number of calls of the basic oracles of $h$ and $G(x,\cdot)$ respectively
\begin{align}\label{eq:obt_complofgG_mh=1_inverse}
& \widetilde{O}\left(1 + \sqrt{H_1/(H+\mu_y)}\right)\tau_h,\\
&  \widetilde{O}\left( \tau_G + \tau_G\sqrt{H_1/(H+\mu_y)}+ \left( 1 + \sqrt{H_1/(H+\mu_y)}\right)\left(\tau_G+\tau_G\sqrt{\frac{H_2}{\mu_y+H+H_1}}\right)\right). \label{eq:obt_complofgh_mh=1_inverse}
\end{align}
 Finding $\left(\delta\left(\varepsilon\right) ,\sigma_0\left(\varepsilon,\sigma\right), 2L_g\right)$-oracle of $g$ by calculating $ \nabla_{x} G\left(x, \tilde{y}_{\delta/2}(x)\right)$ requires additionally $\tau_G$ 
calls of the basic oracle for $G(\cdot,y)$. 
Since in Assumption~\ref{assumpt:framework_oracle} we denote the dependence on the target accuracy $\varepsilon$ and confidence level $\sigma$ by a separate quantities denoted by $\mathcal{K}(\varepsilon,\sigma)$ and in this case it is logarithmic, choosing $H_1 = 2L_h$ and $H_2 = 2L_G$ we get the final estimates for  $\mathcal{N}_{G}^y$ and $\mathcal{N}_{h}$ to guarantee that  Assumption~\ref{assumpt:framework_oracle} holds:
\begin{align*}
&\mathcal{N}_{G}^y =O\left( \left(1 + \sqrt{2L_h/(H+\mu_y)}\right)\left(1+\sqrt{\frac{2L_G}{\mu_y+H+2L_h}}\right)\right)\tau_G\\ 
&=O\left( 1 + \sqrt{\frac{2L_h}{\mu_y+H}} + \sqrt{\frac{2L_G}{\mu_y+H}} + \sqrt{\frac{2L_h}{\mu_y+H}}\sqrt{\frac{2L_G}{\mu_y+H+2L_h}}\right)\tau_G = O\left( \tau_G +  \tau_G\sqrt{\frac{L_G}{\mu_y+H}} \right),\\ 
&\mathcal{N}_{h} = O\left(1 + \sqrt{\frac{L_h}{\mu_y+H}}\right)\tau_h,
\end{align*}
where for the first bound we used that $L_h \leq L_G$.

It is important to note that the estimates on $\mathcal{N}_{G}^y$ and $\mathcal{N}_{h}$ obtained in both cases $L_h \geq L_G$ and $L_h \leq L_G$ are exactly the same. Thus, regardless of the relation between $L_h$ and $L_G$, we obtain the estimates in the statement of the Lemma. Yet, we underline that the algorithm actually depends on whether $L_h \geq L_G$ or $L_h \leq L_G$.
\qed
\end{proof}

We now obtain a simple counterpart of the previous Lemma for the case when  Assumption \ref{assumpt:G_sum}.3(b) holds instead of Assumption \ref{assumpt:G_sum}.3(a). In this case $h$ is prox-friendly and there is no need to consider different cases and just one Loop is enough since the auxiliary problem \eqref{eq:190_mh=1} in Loop 1 can be solved explicitly. 


\begin{lemma} \label{lem:obtain_oracle_mh=1_h-prox}
Let the function $g$ be defined via maximization problem in \eqref{eq:framework_g_obt_oracle}, i.e.
\begin{align}\label{eq:g_obt_oracle_mh=1_h-prox}
g(x)= \max _{y \in  \R^{d_y}}\left\{G(x, y)-h(y) -\frac{H}{2}\|y-y_0\|^2\right\},
\end{align}
 where $G(x,y)$, $h(y)$ are according to \eqref{eq:problem_Gsum_minmax} 
 and satisfy Assumption \ref{assumpt:G_sum}.1,2,3(b),  $y_0 \in \R^{d_y}$.  Then, applying Algorithm \ref{alg:restarts_inexact_notconvex} to this problem, we guarantee that Assumption~\ref{assumpt:framework_oracle} holds with $\tau_G$ basic oracle calls for $G(\cdot,y)$ and the following estimates for the number of basic oracle calls for $G(x,\cdot)$ and $h$ respectively
  \begin{align}\label{eq:complofgG_mh=1_h-prox}
&\mathcal{N}_{G}^y\left( \tau_G, H\right) = O\left(\tau_G + \tau_G\sqrt{L_G/(H+\mu_y)}\right),\\
&\mathcal{N}_{h}\left( \tau_h, H\right) = 0. \label{eq:complofgh_mh=1_h-prox}
\end{align}
\end{lemma}
\begin{proof}
The proof is similar to the proof for the case "Sliding $L_h \geq L_G$" in the proof of Lemma~\ref{lem:obtain_oracle_mh=1} with the only change that the auxiliary problem \eqref{eq:190_mh=1} is solved explicitly thanks to $h$ being prox-friendly. 
\qed
\end{proof}

By changing the variables $x$ and $y$ in Lemma \ref{lem:obtain_oracle_mh=1} and choosing $H = 0$ we obtain the following simple corollary that ensures Assumption ~\ref{assumpt:framework_oracle_x}.
\begin{corollary} \label{lem:obtain_oracle_x_mh=1}
Let the function $r$ be defined via maximization problem in \eqref{eq:framework_r_obt_oracle}, i.e.
\begin{align}
r(y)= \min _{x \in  \R^{d_x}}\left\{G(x, y)+f(x)\right\},
\end{align}
 where  $G(x,y),f(y)$ are according to \eqref{eq:problem_Gsum_minmax} 
 and satisfy Assumption \ref{assumpt:G_sum}.1,2,3(a).
 Then, for each of two cases $L_f \geq L_G$ and $L_f \leq L_G$ we organize computations in two loops and apply Algorithm \ref{alg:restarts_inexact_notconvex}, so that Assumption~\ref{assumpt:framework_oracle_x} holds with $\tau_G$ basic oracle calls for $G(x,\cdot)$ and the following estimates for the number of basic oracle calls for $G(\cdot,y)$ and $f$ respectively
 \begin{align}\label{eq:complofgG_x_mh=1}
&\mathcal{N}_{G}^x\left( \tau_G\right) = O\left(\tau_G + \tau_G\sqrt{L_G/\mu_x}\right),\\
&\mathcal{N}_{f}\left( \tau_f\right) = O\left(\tau_f+ \tau_f\sqrt{L_f/\mu_x}\right). \label{eq:complofgf_x_mh=1}
\end{align}
We name these algorithms "Sliding $L_f \geq L_G$" and "Sliding $L_f \leq L_G$".
\end{corollary}

\subsection{Proof of Theorem~\ref{theorem:G-sum_noprox}}
Finally, we prove Theorem~\ref{theorem:G-sum_noprox} for problem \eqref{eq:problem_Gsum_minmax} by combining the building blocks depending on the relation between $L_f$ and $L_G$ and relation between $L_h$ and $L_G$. If $L_f \geq L_G$ we use the general framework from the main text (see Section \ref{section:saddleFramework} and Theorem \ref{theorem:general_framework}). In the opposite case we apply the variation of this framework described in Appendix \ref{Appendix_Framework_inverse} (see Theorem \ref{theorem:general_framework_inverse}). 
In both cases we use Lemma~\ref{lem:obtain_oracle_mh=1} and Corollary~\ref{lem:obtain_oracle_x_mh=1} to ensure Assumptions \ref{assumpt:framework_oracle}, \ref{assumpt:framework_oracle_x}, but with different order of the loops described inside these Lemma and Corollary depending on the relation between $L_h$ and $L_G$, i.e. we use either sliding $L_h \geq L_G$ or sliding $L_h \leq L_G$ in Lemma~\ref{lem:obtain_oracle_mh=1} and either sliding $L_f \geq L_G$ or sliding $L_f \leq L_G$ in Corollary~\ref{lem:obtain_oracle_x_mh=1}.
For convenience, we summarize which results are used in which case in Table \ref{tabl:saddleproblem_mh=1}.

\begin{proof}[of Theorem~\ref{theorem:G-sum_noprox}]
Assumption \ref{assumpt:G_sum}.1,2,3(a) with \eqref{eq:assumpt1_constants_G_sum}  guarantee that Assumption~\ref{assumpt:framework} holds. 
Further, the choice $H = 2L_G$ in Lemma~\ref{lem:obtain_oracle_mh=1} guarantee that Assumption \ref{assumpt:framework_oracle} holds with the number of oracle calls given by \eqref{eq:complofgG_mh=1} and \eqref{eq:complofgh_mh=1}. Corollary~\ref{lem:obtain_oracle_x_mh=1} guarantee that Assumption \ref{assumpt:framework_oracle_x} holds with the number of oracle calls given by \eqref{eq:complofgG_x_mh=1} and \eqref{eq:complofgf_x_mh=1}. We consider two cases $L_f \geq L_G$ and $L_f \leq L_G$ and, for each case, apply either the general framework from the main text or from the previous appendix. We show that in both cases the estimates are the same and are equal to the ones in the statement of the theorem. In each case we make the derivations with $\sigma =0$ since all the algorithms are deterministic in this case.
\begin{table}\centering
\begin{tabular}{|c|c|c|}
\hline
\textbf{Different regimes} & \textbf{$L_h \geq L_G$} & \textbf{$L_h \leq L_G$}\\ \hline
\textbf{$L_f \leq L_G$} & \begin{tabular}[c]{@{}c@{}} Framework from Appendix~\ref{Appendix_Framework_inverse} (Theorem~\ref{theorem:general_framework_inverse})\\+
Sliding for $L_h \geq L_G$ (Lemma~\ref{lem:obtain_oracle_mh=1})\\+ Sliding for $L_f \leq L_G$ (Corollary~\ref{lem:obtain_oracle_x_mh=1})\end{tabular}
& \begin{tabular}[c]{@{}c@{}} Framework from Appendix~\ref{Appendix_Framework_inverse} (Theorem~\ref{theorem:general_framework_inverse})\\+
Sliding for $L_h \leq L_G$ (Lemma~\ref{lem:obtain_oracle_mh=1})\\+ Sliding for $L_f \leq L_G$ (Corollary~\ref{lem:obtain_oracle_x_mh=1})\end{tabular} \\\hline
\textbf{$L_f \geq L_G$} & \begin{tabular}[c]{@{}c@{}} General Framework (Theorem~\ref{theorem:general_framework})\\+
Sliding for $L_h \geq L_G$ (Lemma~\ref{lem:obtain_oracle_mh=1})\\+ Sliding for $L_f \geq L_G$ (Corollary~\ref{lem:obtain_oracle_x_mh=1})\end{tabular}
& \begin{tabular}[c]{@{}c@{}} General Framework (Theorem~\ref{theorem:general_framework})\\+
Sliding for $L_h \leq L_G$ (Lemma~\ref{lem:obtain_oracle_mh=1})\\+ Sliding for $L_f \geq L_G$ (Corollary~\ref{lem:obtain_oracle_x_mh=1})\end{tabular} \\\hline
\end{tabular}

\caption{Summary of the proof of Theorem~\ref{theorem:h-not-sum}. For each regime we apply the algorithms described in the proofs of the corresponding results listed in the table to obtain the complexity 
estimates \eqref{theorem:h-not-sum_f}-\eqref{theorem:h-not-sum_G_y} for the number of basic oracle calls for each part of the objective $f$, $h$, and $G$.
}
\label{tabl:saddleproblem_mh=1}
\end{table}

 We begin with the case $L_f \geq L_G$. 
\paragraph{Case $L_f \geq L_G$}$\;$\\
Applying Theorem~\ref{theorem:general_framework} with $\tau_f=\tau_h=1$ and $\tau_G = m_G$, Lemma~\ref{lem:obtain_oracle_mh=1} with $H=2L_G$, Corollary~\ref{lem:obtain_oracle_x_mh=1} and combining the complexity estimates in these results,
 we obtain the following final complexity bounds.

Number of basic oracle calls of $f$:
\begin{align*}
\widetilde{O} \left(\left( 1 + \sqrt{\frac{L_G}{\mu_y}}\right) \left(1+ \sqrt{ \frac{L_f}{\mu_x}} + \left( 1 + \sqrt{\frac{L_G}{\mu_x}}\right)\left( 1 + \sqrt{\frac{L_f}{L_G}}\right)\right)
    \right) 
    = \widetilde{O} \left(\left( \sqrt{\frac{L_G}{\mu_y}}\right) \left( \sqrt{ \frac{L_f}{\mu_x}} + \left( \sqrt{\frac{L_G}{\mu_x}}\right)\left( \sqrt{\frac{L_f}{L_G}}\right)\right)
    \right)\\
    =\widetilde{O} \left(\left( \sqrt{\frac{L_f L_G}{\mu_x \mu_y}}\right)\right),
\end{align*}
where we used that, $L_G \leq L_f$ and, by the assumptions of this Theorem, $1 \leq L_G/\mu_y $, $1 \leq L_G/\mu_x$, $1 \leq L_f/\mu_x $.

Number of basic oracle calls of $h$:
\begin{align*}
\widetilde{O} \left( \left( 1 + \sqrt{\frac{L_G}{\mu_y}}\right) \left( 1 +   \left( 1 + \sqrt{\frac{L_G}{\mu_x}}\right) \left( 1+\sqrt{\frac{ L_h}{2L_G+\mu_y}}\right)\right)
    \right) 
=\widetilde{O} \left( \left( \sqrt{\frac{L_G}{\mu_y}}\right) \left( 1 +   \left( \sqrt{\frac{L_G}{\mu_x}}\right)\left(1+ \sqrt{\frac{ L_h}{L_G}}\right)\right)
    \right)=\\
    =\widetilde{O} \left( \max \left\{ \sqrt{\frac{ L_G L_h}{\mu_x \mu_y}}, \sqrt{\frac{ L_G^2}{\mu_x \mu_y}} \right\}\right),
\end{align*}
where we used that $H=2L_G$ in Lemma~\ref{lem:obtain_oracle_mh=1} and, by the assumptions of this Theorem, $1 \leq L_G/\mu_y $, $1 \leq L_G/\mu_x$.

Number of basic oracle calls of  $G(\cdot, y)$:
\begin{align*}
 \widetilde{O} \left( \left( 1 + \sqrt{\frac{L_G}{\mu_y}}\right) \left( m_G + m_G\sqrt{\frac{L_G}{\mu_x}} + m_G\left( 1 + \sqrt{\frac{L_G}{\mu_x}}\right) \right)
    \right) = \widetilde{O} \left(   m_G\sqrt{\frac{L_G^2}{\mu_x \mu_y}}  \right),
\end{align*}
where we used that, by the assumptions of this Theorem, $1 \leq L_G/\mu_y $ and $1 \leq L_G/\mu_x$.

Number of basic oracle calls of  $G(x,\cdot)$:
\begin{align*}
 \widetilde{O} \left( \left( 1 + \sqrt{\frac{L_G}{\mu_y}}\right) \left( m_G + m_G\left( 1 + \sqrt{\frac{L_G}{\mu_x}}\right)  \left(1 + \sqrt{\frac{L_G}{2L_G+\mu_y}} \right)\right)
    \right) \\
 =\widetilde{O} \left(m_G \left( 1 + \sqrt{\frac{L_G}{\mu_y}}\right) \left( 1 + \sqrt{\frac{L_G}{\mu_x}}  \right)
    \right)=\widetilde{O} \left(  m_G \sqrt{\frac{L_G^2}{\mu_x \mu_y}}  \right),
\end{align*}
where we used that $H=2L_G$ in Lemma~\ref{lem:obtain_oracle_mh=1} and, by the assumptions of this Theorem, $1 \leq L_G/\mu_y $ and $1 \leq L_G/\mu_x$.

\paragraph{Case $L_f \leq L_G$}$\;$\\
Applying Theorem~\ref{theorem:general_framework_inverse} with $\tau_f=\tau_h=1$ and $\tau_G = m_G$, Lemma~\ref{lem:obtain_oracle_mh=1} with $H=2L_G$, Corollary~\ref{lem:obtain_oracle_x_mh=1} and combining the, complexity estimates
we obtain the final complexity bounds as follows.

Number of basic oracle calls of $f$:
\begin{align*}
& \widetilde{O} \left(\left( 1 + \sqrt{\frac{L_G}{\mu_y}}\right) \left(1+ \sqrt{ \frac{L_f}{\mu_x}} + \left( 1 + \sqrt{\frac{L_f}{\mu_x}}\right)\right)
    \right) = \widetilde{O} \left(\left( \sqrt{\frac{L_G}{\mu_y}}\right) \left( \sqrt{ \frac{L_f}{\mu_x}} + \left( \sqrt{\frac{L_f}{\mu_x}}\right)\right)
    \right)=\widetilde{O} \left(\left( \sqrt{\frac{L_f L_G}{\mu_x \mu_y}}\right)\right),
\end{align*}
where we used that, by the assumptions of this Theorem, $1 \leq L_G/\mu_y $ and $1 \leq L_f/\mu_x$.

Number of basic oracle calls of $h$:
\begin{align*}
& \widetilde{O} \left( \left( 1 + \sqrt{\frac{L_G}{\mu_y}}\right) \left( 1+   \left( 1 + \sqrt{\frac{L_f}{\mu_x}}\right)\left( 1+\sqrt{\frac{L_G}{L_f}}\right)\left(1 +\sqrt{\frac{ L_h}{2L_G+\mu_y}}\right)\right)
    \right) \nonumber \\
&=\widetilde{O} \left( \left( \sqrt{\frac{L_G}{\mu_y}}\right) \left( 1 +   \left( \sqrt{\frac{L_f}{\mu_x}}\right)\cdot \left( \sqrt{\frac{L_G}{L_f}}\right) \cdot \left(1+ \sqrt{\frac{ L_h}{L_G}}\right)\right)
    \right)=\widetilde{O} \left(\max \left\{ \sqrt{\frac{ L_G^2}{\mu_x \mu_y}},\sqrt{\frac{ L_G L_h}{\mu_x \mu_y}} \right\} \right),
\end{align*}
where we used that $H=2L_G$ in Lemma~\ref{lem:obtain_oracle_mh=1}, $L_G \geq L_f$ and, by the assumptions of this Theorem, $1 \leq L_G/\mu_y $, $1 \leq L_f/\mu_x$.

Number of basic oracle calls of  $G(\cdot, y)$:
\begin{align*}
& \widetilde{O} \left( \left( 1 + \sqrt{\frac{L_G}{\mu_y}}\right) \left( m_G + m_G\sqrt{\frac{L_G}{\mu_x}} + m_G\left( 1 + \sqrt{\frac{L_f}{\mu_x}}\right) \left( 1 + \sqrt{\frac{L_G}{L_f}} \right)\right)
    \right) \nonumber \\
& =\widetilde{O} \left( m_G\left(  \sqrt{\frac{L_G}{\mu_y}}\right) \left(  \sqrt{\frac{L_G}{\mu_x}} + \left(  \sqrt{\frac{L_f}{\mu_x}}\right) \cdot \left(  \sqrt{\frac{L_G}{L_f}}\right) \right)
    \right)=\widetilde{O} \left(m_G \left(  \sqrt{\frac{L_G}{\mu_y}}\right)  \left(  \sqrt{\frac{L_G}{\mu_x}} \right) \right)= \widetilde{O} \left( m_G  \sqrt{\frac{L_G^2}{\mu_x \mu_y}}  \right),
\end{align*}
where we used that $L_G \geq L_f$ and, by the assumptions of this Theorem, $1 \leq L_G/\mu_y $, $1 \leq L_G/\mu_x $, $1 \leq L_f/\mu_x$.

Number of basic oracle calls of  $G(x,\cdot)$:
\begin{align*}
& \widetilde{O} \left( \left( 1 + \sqrt{\frac{L_G}{\mu_y}}\right) \left( m_G + \left( 1 + \sqrt{\frac{L_f}{\mu_x}}\right)  \left(m_G + m_G\sqrt{\frac{L_G}{2L_G+\mu_y}} + \sqrt{\frac{L_G}{L_f}} \cdot m_G\left(1 + \sqrt{\frac{L_G}{2L_G+\mu_y}}\right) \right)\right)
    \right) \nonumber \\
& =\widetilde{O} \left( m_G\left(  \sqrt{\frac{L_G}{\mu_y}}\right) \left( 1 + \left(  \sqrt{\frac{L_f}{\mu_x}}\right) \cdot \left( \sqrt{\frac{L_G}{L_f}}\right)  \right) \right)= \widetilde{O} \left( \max \left\{ m_G\sqrt{\frac{L_G}{\mu_y}}, m_G\sqrt{\frac{ L_G^2}{\mu_x \mu_y}}\right\} \right)=\widetilde{O} \left( m_G  \sqrt{\frac{L_G^2}{\mu_x \mu_y}}  \right),
\end{align*}
where we used that $H=2L_G$ in Lemma~\ref{lem:obtain_oracle_mh=1}, $L_G \geq L_f$ and, by the assumptions of this Theorem, $1 \leq L_G/\mu_y $, $1 \leq L_G/\mu_x $, $1 \leq L_f/\mu_x$.

\qed \end{proof}

\begin{proof}[of Theorem~\ref{theorem:G-sum_h_prox}]
The only difference in the proof of Theorem~\ref{theorem:G-sum_h_prox} from the proof of Theorem~\ref{theorem:G-sum_noprox} is the use of Lemma~\ref{lem:obtain_oracle_mh=1_h-prox} instead of Lemma~\ref{lem:obtain_oracle_mh=1} to satisfy Assumption~\ref{assumpt:framework_oracle}. Thus, applying expressions \eqref{eq:complofgG_mh=1_h-prox}, \eqref{eq:complofgh_mh=1_h-prox} for $\mathcal{N}_{G}^y$ and $\mathcal{N}_{h}$ and following the proof of Theorem~\ref{theorem:G-sum_noprox} without any changes we obtain the same estimates for the number of basic oracle calls of $f, G(\cdot, y), G(x,\cdot)$. Considering $\mathcal{N}_{h} = 0$ and using that, by the assumptions of this Theorem, $1 \leq L_G/\mu_y $, we obtain that the number of basic oracle calls of $h$ is
\begin{align*}
\widetilde{O} \left(\sqrt{\frac{ L_G }{ \mu_y}} \right).
\end{align*}
\qed
\end{proof}

\end{document}